\documentclass[a4paper,12pt]{amsart}
\usepackage{amsthm}
\usepackage{amsmath,amssymb,amsfonts,dsfont}
\usepackage[utf8]{inputenc} 
\usepackage[english]{babel}
\usepackage[T1]{fontenc} 
\usepackage{graphicx}
\usepackage{float}
\usepackage{pdfpages}
\usepackage{enumerate}
\usepackage{ifpdf}
\usepackage{marginnote}
\usepackage[a4paper, margin = 3cm, bottom = 3cm]{geometry}
\usepackage[font=small,labelfont=bf]{caption}
\usepackage{hyperref}	
\hypersetup{pdfborder=0 0 0, 
	    colorlinks=true,
	    citecolor=black,
	    linkcolor=black,
	    urlcolor=black,
	    pdfauthor={Martin Vogel}
	   }

\newcommand{\erw}{\mathds{E}}
\newcommand{\e}{\mathrm{e}}

\newcommand{\tr}{\mathrm{tr\,}}
\newcommand{\supp}{\mathrm{supp\,}}

\newcommand{\Ima}{\mathrm{Im\,}}
\newcommand{\Rea}{\mathrm{Re\,}}
\newcommand{\dist}{\mathrm{dist\,}}
\newcommand{\mO}{\mathcal{O}}
\newcommand{\C}{\mathds{C}}
\newcommand{\N}{\mathds{N}}
\newtheorem{thm}{Theorem}[section]
\newtheorem{cor}[thm]{Corollary}
\newtheorem{prop}[thm]{Proposition}
\newtheorem{lem}[thm]{Lemma}
\theoremstyle{definition}
\newtheorem{defn}[thm]{Definition}
\theoremstyle{remark}
\newtheorem{rem}[thm]{Remark}
\theoremstyle{definition}

\theoremstyle{definition}

\theoremstyle{definition}
\newtheorem{hypo}[thm]{Hypothesis}
\numberwithin{equation}{section} 
\title{The precise shape of the eigenvalue intensity for a class of non-selfadjoint operators 
under random perturbations}
\author{Martin Vogel} 
\address[Martin Vogel]{Institut de Math\'ematiques de Bourgogne - UMR 5584 CNRS, Universit\'e de Bourgogne, 
		       Facult\'e des Sciences Mirande, 9 avenue Alain Savary, 
		       BP 47870 21078 Dijon Cedex.}
\email{martin.vogel@u-bourgogne.fr}
\begin{document}
\begin{abstract}
We consider a non-selfadjoint $h$-differential model operator $P_h$ in the 
semiclassical limit ($h\rightarrow 0$) subject to small random perturbations. 
Furthermore, we let the coupling constant $\delta$ be 
$\e^{-\frac{1}{Ch}}\leq \delta \ll h^{\kappa}$ for constants $C,\kappa>0$ 
suitably large. Let $\Sigma$ be the closure of the range of the principal 
symbol.
Previous results on the same model by Hager, Bordeaux-Montrieux 
and Sj\"ostrand show that if $\delta \gg\e^{-\frac{1}{Ch}}$ there is, with a 
probability close to $1$, a Weyl law for the eigenvalues in the interior of the 
of the pseudospectrum up to a distance 
$\gg\left(-h\ln{\delta h}\right)^{\frac{2}{3}}$ to the boundary of $\Sigma$.
\par
We study the intensity measure of the random point process of eigenvalues 
and prove an $h$-asymptotic formula for the average density of eigenvalues. 
With this we show that there are three distinct regions of different spectral 
behavior in $\Sigma$: 
The interior of the pseudospectrum is solely governed by a Weyl law, 
close to its boundary there is a strong spectral accumulation given by a tunneling 
effect followed by a region where the density decays rapidly. 
  \vskip.5cm
  \par\noindent \textbf{R{\'e}sum{\'e}.} 
Nous consid\'erons un op\'erateur diff\'erentiel non-autoadjoint 
$P_h$ dans la limite semiclassique ($h\rightarrow 0$) soumis 
\`a de petites perturbations al\'eatoires. De plus, nous imposons 
que la constant couplage $\delta$ verifie 
$\e^{-\frac{1}{Ch}}\leq \delta \ll h^{\kappa}$ pour certaines constantes 
$C,\kappa>0$ choisies assez grandes. Soit $\Sigma$ l'adh\'erence de l'image 
du symbole principal de $P_h$. De pr\'ec\'edents r\'esultats 
par Hager, Bordeaux-Montrieux and Sj\"ostrand montrent que, pour le 
m\^eme op\'erateur, si l'on choisit $\delta \gg\e^{-\frac{1}{Ch}}$, 
alors la distribution des valeurs propres est donn\'ee par une 
loi de Weyl jusqu'\`a une distance 
$\gg\left(-h\ln{\delta h}\right)^{\frac{2}{3}}$ du bord de $\Sigma$. 
\par
Dans cet article, nous donnons une formule $h$-asymptotique pour la 
densit\'e moyenne des valeurs propres en \'etduiant le mesure de 
comptage al\'etoire des valeurs propres. 
En \'etudiant cette densit\'e, nous prouvons qu'il y a une loi de 
Weyl \`a l'interieur du $\delta$-pseudospectre, une zone d'accumulation 
des valeurs propres d\^ue \`a un effet tunnel pr\`es du bord du 
pseudospectre suivi par une zone o\`u la densit\'e d\'ecro\^it rapidement.
\end{abstract}
\maketitle
\setcounter{tocdepth}{1}
\tableofcontents
%
%Chapter 1
%
\section{Introduction}\label{sec:Intro}
Over the last twenty years there has been a growing interest in the spectral 
theory of non-self-adjoint operators, as they appear naturally in many areas, 
for example
  \begin{itemize}
   \item in the solvability theory of linear PDE's given by non-normal operators,
   \item in mathematical physics, for example when studying scattering poles 
         (quantum resonances),
   \item in the spectral analysis of Kramers-Fokker-Planck type operators.
  \end{itemize}
A major difficulty when dealing with non-self-adjoint operators is the fact 
that, as opposed to self-adjoint operators, the norm of the resolvent can be very 
large even far away from the spectrum of the operator. As a consequence the 
spectrum can be highly unstable even under very small perturbations, see for 
example \cite{TrefEmbr,Da07} for a very good overview.
\par
Renewed activity in this field has been sparked by works in numerical analysis 
by L.N.Trefethen and M. Embree, see e.g. \cite{Tr97,TrefEmbr}, which provides a 
significant tool for studying spectral instability: the so-called 
$\varepsilon$-pseudo\-spectrum which consists of the regions where the resolvent is 
large and thus indicates how far the eigenvalues can spread under 
perturbations. Following \cite{TrefEmbr}, it can be defined by
  \begin{defn}
   Let $A$ be an closed linear operator on a Banach space $X$ and let $\varepsilon>0$ be arbitrary. Then, the 
   $\varepsilon$-$pseudospectrum$ $\sigma_{\varepsilon}(A)$ of $A$ is defined by 
    \begin{equation}\label{pssp1}
     \sigma_{\varepsilon}(A) : = \left\{z\in\mathds{C}:~\lVert (z-A)^{-1} \rVert > \frac{1}{\varepsilon} \right\}
			       \cup\sigma(A),
    \end{equation}
   or equivalently
    \begin{equation}\label{pssp}
     \sigma_{\varepsilon}(A)  = \bigcup_{\substack{B\in\mathcal{B}(X) \\ \lVert B\rVert < \varepsilon}}
			      \sigma(A+B),
    \end{equation}
    or equivalently 
    \begin{equation}\label{pssp2}
     z\in\sigma_{\varepsilon} \quad\Longleftrightarrow \quad 
	    z\in\sigma(A)~\text{or}~\exists u\in\mathrm{dom}(A),~\lVert u\rVert = 1~ \text{s.t.:}~
		    \lVert (z-A)u\rVert < \varepsilon.
    \end{equation}
  \end{defn}
The last condition also implicitly defines the so-called \textit{quasimodes} 
or $\varepsilon$-\textit{pseudo}\-\textit{eigenvectors}.
\par 
Highlighted by the works of L.N. Trefethen, M. Embree, E.B. Davies, M. Zworski, J. Sj\"ostrand, 
cf. \cite{TrefEmbr,Tr97,Da07,Da97,Da99,NSjZw,SjAX1002,ZwChrist10}, and many others, spectral 
instability of non-self-adjoint operators has become a popular and important subject. 
\par
In view of \eqref{pssp} it is natural to study the spectrum of such operators 
under small random perturbations. One direction of recent research interest 
has focused on the case of elliptic (pseudo)differential 
operators subject to random perturbations, see for example 
\cite{Ha06,BM,Ha06b,BoSj09,HaSj08,SjAX1002,Sj08,Sj09} and 
\cite{ZwChrist10,DaHa09,Sj13}.
\par
Following this line, we will consider a class of non-self-adjoint semiclassical 
differential operators, introduced by M. Hager \cite{Ha06}, subject to 
random perturbations and we will give a complete description of the density of eigenvalues.
\subsection{Hager's model}\label{sec:HM}
Let $0<h\ll 1$, we consider the semiclassical operator $P_h:L^2(S^1)\rightarrow L^2(S^1)$ as 
defined by Hager in \cite{Ha06} given by
  \begin{equation}\label{eqn:defnModelOperator}
    P_h := hD_x + g(x), \quad D_x := \frac{1}{i}\frac{d}{dx}, \quad S^1=\mathds{R}/2\pi\mathds{Z}
  \end{equation}
where $g\in\mathcal{C}^{\infty}(S^1:\mathds{C})$ is such that $\Ima g$ has exactly two critical 
points, one minimum and one maximum, say in $a$ and $b$, with 
$a < b< a+2\pi$ and $\Ima g(a) < \Ima g(b)$. Without loss of generality we may assume that 
$\Ima g(a) = 0$. The natural domain of $P_h$ is the semiclassical Sobolev space
  \begin{equation*}
   H^1(S^1) := \left\{u\in L^2(S^1):~ \left(\lVert u \rVert^2 + 
				\lVert hD_x u \rVert^2\right)^{\frac{1}{2}}< \infty \right\},
  \end{equation*}
where $\lVert\cdot\rVert$ denotes the $L^2$-norm on $S^1$ if nothing else is specified. We will use the 
standard scalar products on $L^2(S^1)$ and $\mathds{C}^N$ defined by 
  \begin{equation*}
   (f|g) := \int_{S^1} f(x) \overline{g}(x)dx,\quad f,g\in L^2(S^1),
  \end{equation*}
and 
  \begin{equation*}
   (X|Y) := \sum_{i=1}^N X_i \overline{Y}_i, \quad X,Y\in\mathds{C}^N.
  \end{equation*}
We denote the semiclassical principal symbol of $P_h$ by 
  \begin{equation}\label{eqn:SemClPrinSymModlOp}
   p(x,\xi) = \xi + g(x), \quad (x,\xi) \in T^*S^1.
  \end{equation}
The spectrum of $P_h$ is discrete with simple eigenvalues, given by 
  \begin{equation*}
   \sigma(P_h) = \{z\in\mathds{C}:~z=\langle g\rangle +kh,~k\in\mathds{Z}\},
  \end{equation*}
where $\langle g\rangle := (2\pi)^{-1} \int_{S^1} g(y)dy$. Next, consider the equation $z=p(x,\xi)$. It has precisely two 
solutions $\rho_{\pm}:=(x_{\pm},\xi_{\pm})$ where $x_{\pm}$ are given by $\Ima g(x_{\pm})  = \Ima z $, $\pm \Ima g'(x_{\pm}) < 0 $ 
and $\xi_{\pm} = \Rea z - \Rea g(x_{\pm})$. By the natural projection 
$\Pi:\mathds{R}\rightarrow S^1=\mathds{R}/2\pi\mathds{Z}$ and a slight abuse 
of notation we identify the points $x_{\pm}\in S^1$ with points $x_{\pm}\in\mathds{R}$ such that 
$x_- -2\pi < x_+ < x_-$. Furthermore, 
we will identify $S^1$ with the interval $[x_--2\pi,x_-[$. We recall that the 
Poisson bracket of $p$ and $\overline{p}$ is given by 
  \begin{equation*}
   \{p,\overline{p}\} = {p}'_{\xi}\cdot{\overline{p}}'_x - {p}'_x\cdot{\overline{p}}'_{\xi}.
  \end{equation*}
\subsection{Adding a random perturbation}
We are interested in the following random perturbation of $P_h$:
  \begin{equation}\label{eqn:DefPertP}
    P_h^{\delta} := P_h +\delta Q_{\omega} := hD_x + g(x) + \delta Q_{\omega},
  \end{equation}
where $\delta>0$ and $Q_{\omega}$ is an integral operator $L^2(S^1)\rightarrow L^2(S^1)$ of the form
\begin{equation}\label{defn:PertIntOp}
    Q_{\omega} u(x) := \sum\limits_{|j|,|k|\leq \left\lfloor\frac{C_1}{h}\right\rfloor} \alpha_{j,k} (u|e^k)e^j(x).
  \end{equation}
Here $\lfloor x\rfloor :=\max\{n\in\mathds{N}:~x\geq n\}$ for $x\in\mathds{R}$, $C_1>0$ is big enough, 
$e^k(x):=(2\pi)^{-1/2} \e^{ikx}$, $k\in\mathds{Z}$, and $\alpha_{j,k}$ are complex valued 
independent and identically distributed random variables with complex Gaussian distribution 
law $\mathcal{N}_{\mathds{C}}(0,1)$. The following result was observed by W. Bordeaux-Montrieux \cite{BM}.
\begin{prop}\label{lem:truncation}
 Let $h>0$ and let $\lVert  Q_{\omega}\rVert_{HS}$ denote the Hilbert-Schmidt norm 
 of $ Q_{\omega}$. If $C>0$ is large enough, then 
  \begin{equation*}
   \lVert Q_{\omega} \rVert_{HS} \leq \frac{C}{h} 
   \quad\text{with probability} \geq 1 - \e^{-\frac{1}{Ch^2}}.
  \end{equation*}
\end{prop}
Since $\lVert Q_{\omega} \rVert_{HS}^2 = \sum |\alpha_{j,k}(\omega)|^2$, we 
can also view the above bound as restricting the support of the joint 
probability distribution of the random vector $\alpha=(\alpha_{jk})_{j,k}$ 
to a ball of radius $C/h$. 
% Hence, to obtain a bounded perturbation 
% we will work from now on in this restricted probability space and we assume 
% that the coupling constant $\delta>0$ satisfies 
% %
%   \begin{equation}\label{c_1}
%    \delta \ll h^{5/2},
%   \end{equation}
% %
% which implies that $\lVert Q_{\omega} \rVert_{HS} \leq Ch^{3/2}$. 
Hence, to obtain a bounded perturbation we will 
work from now on in the restricted probability space:
\begin{hypo}[\bf Restriction of random variables]\label{hyp:H5} 
Define $N:=(2\lfloor C_1/h\rfloor + 1)^2$ where $C_1>0$ is as in 
\eqref{defn:PertIntOp}. We assume that for some constant $C>0$ 
  \begin{equation}\label{eq_i33}
    \alpha \in B(0,R)\subset \C^N, \quad R=\frac{C}{h}.
  \end{equation}
\end{hypo}
Furthermore, we assume that the coupling constant $\delta>0$ satisfies 
  \begin{equation}\label{c_1}
   \delta \ll h^{5/2},
  \end{equation}
which implies, for $\alpha\in B(0,R)$, that 
$\delta\lVert Q_{\omega} \rVert_{HS} \leq Ch^{3/2}$. Hence, for 
$\alpha\in B(0,R)$, the operator $Q_{\omega}$ is compact and 
the spectrum of $P_h^{\delta}$ is discrete. 
\\
\par
\noindent\textbf{Zone of spectral instability}
Since in the present work we are in the semiclassical setting, 
we follow \cite{NSjZw} and define
\begin{equation}\label{eq_i6}
     \Sigma:= \overline{p(T^{*}S^1)} \subset \mathds{C},
\end{equation}
where $p$ as in \eqref{eqn:SemClPrinSymModlOp}. 
\par
In the case of (\ref{eqn:defnModelOperator}) and (\ref{eqn:SemClPrinSymModlOp}) 
$p(T^{*}S^1)$ is already closed due to the ellipticity of $P_h$. It has been 
discussed in \cite{NSjZw} (in a more general context) that for any 
$z\in\mathring{\Sigma}$ such that there exists a $\rho(z)\in T^*S^1$ such that 
 \begin{equation}\label{eq_bc}
  \{\Rea p, \Ima p\}(\rho(z)) < 0,
 \end{equation}
there exists an $h^{\infty}$-quasimode, i.e. an $u_h\in L^2(S^1)$ such that 
   \begin{equation*}
      \lVert (P_h-z)u_h\rVert < \mathcal{O}\!\left(h^{\infty}\right)\lVert u_h \rVert.
   \end{equation*}
In addition, $u_h$ is localized to $\rho$, i.e. $\mathrm{WF}_h(u_h) = \{\rho(z)\}$. 
We recall that for $v=v(h)$, $\lVert v\rVert_{L^2}(S^1)= \mO(h^{-N})$, for some 
fixed $N$, the semiclassical wavefront set of $v$, $\mathrm{WF}_h(v)$, is defined by 
\begin{equation*}
 \complement
 \left\{(x,\xi)\in T^*S^1 : \exists 
 a\in\mathcal{S}(T^* S^1),~a(x,\xi)=1,~
 \lVert a^wv\rVert_{L^2}(S^1)=\mO(h^{\infty})
 \right\}
\end{equation*}
where $a^w$ denotes the Weyl quantization of $a$. In the case of \eqref{eqn:SemClPrinSymModlOp} 
we see that \eqref{eq_bc} holds for any $z\in\Omega\Subset\mathring{\Sigma}$. 
We will give more details on the construction of quasimodes for $P_h$ in Section \ref{subse:Quasmod}. 
\\
\par
For $z$ is close to the boundary of $\Sigma$ the situation is different as 
we have a good resolvent estimate on $\partial\Sigma$. Since 
$\{p,\{p,\overline{p}\}\}(\rho)\neq 0$ for all $z_0\in\partial\Sigma$ and all $\rho \in p^{-1}(z_0)$, 
Theorem 1.1 in \cite{Sj_4VM10} implies that there exists a constant $C_0>0$ 
such that for every constant $C_1>0$ there is a constant $C_2>0$ such that for 
$|z-z_0| < C_1 (h\ln \frac{1}{h})^{2/3}$, $h < \frac{1}{C_2}$, the resolvent 
$(P_h-z)^{-1}$ is well defined and satisfies
    \begin{equation}\label{eqn:TubeEmptyA}
     \lVert (P_h-z)^{-1}\rVert < C_0 h^{-\frac{2}{3}}\exp\left(\frac{C_0}{h} 
				 |z-z_0|^ {\frac{3}{2}}\right).
    \end{equation}
This implies for $\alpha$ as in \eqref{c_1} and $\delta=\mO(h^M)$, 
$M=M(C_1,C)>0$ large enough, that  
    \begin{equation}
     \label{eqn:TubeEmpty}
      \sigma(P_h+\delta Q_{\omega})\cap 
      D\left(z_0,C_1\left(h\ln\frac{1}{h}\right)^{2/3}\right) = \emptyset.
    \end{equation}
Thus, there exists a tube of radius $C_1\left(h\ln\frac{1}{h}\right)^{2/3}$ around 
$\partial\Sigma$ void of the spectrum of the perturbed operator $P_h^\delta$. 
Therefore, since we are interested in the eigenvalue distribution of $P_h^{\delta}$, 
we assume from now on implicitly that 
\begin{hypo}[\textbf{Restriction of $\Sigma$}]\label{hyp:H3}
 Let $\Sigma\subset \C$ be as in \eqref{eq_i6}. Then, we let
\begin{align}\label{hyp_Omega}
     &\Omega\Subset\Sigma ~\text{be open, relatively compact with}~
     \mathrm{dist\,}(\Omega,\partial\Sigma) > C\left(h\ln h^{-1} \right)^{2/3} \notag \\
     &\text{for some constant }C>0.
\end{align}
\end{hypo}
% 
% \begin{equation}
%      \label{hyp_Omega}
%      \Omega\Subset\Sigma ~\text{with}~
%      \mathrm{dist\,}(\Omega,\partial\Sigma) \gg Ch^{2/3} 
%      \text{ for some constant }C>0.
% \end{equation}
%
\subsection{Previous results and purpose of this work} 
The operator $P_h$ and small perturbations of it (deterministic and random) have first been 
studied by Hager \cite{Ha06} followed by works by B.-M. \cite{BM} and 
Sj\"ostrand \cite{SjAX1002}. In \cite{Ha06} Hager obtained that the eigenvalue 
distribution of random perturbations of $P_h$ in the interior of $\Sigma$ is given by a 
Weyl law with a probability close to one.
\par
In \cite{BM} Bordeaux-Montrieux extended Hager's result to strips at a distance 
$\gg\left(-h\ln{\delta h}\right)^{\frac{2}{3}}$ to the boundary of $\Sigma$. 
In both cases, the results concern only the interior of the pseudospectrum, thus missing 
an accumulation of eigenvalues effect close to the boundary of the spectrum where the Weyl 
law breaks down. This effect can be seen clearly in numerical simulations for the model 
$P_h$, see Figure \ref{fig_i2b} and \ref{fig_i2c}. It has furthermore been noted in 
numerical simulations for other models, e.g. in the case of Toeplitz quantization considered 
in \cite{ZwChrist10}.
\begin{figure}[ht]
\centering
 \begin{minipage}[b]{0.45\linewidth}
  \centering
  \includegraphics[width=\textwidth]{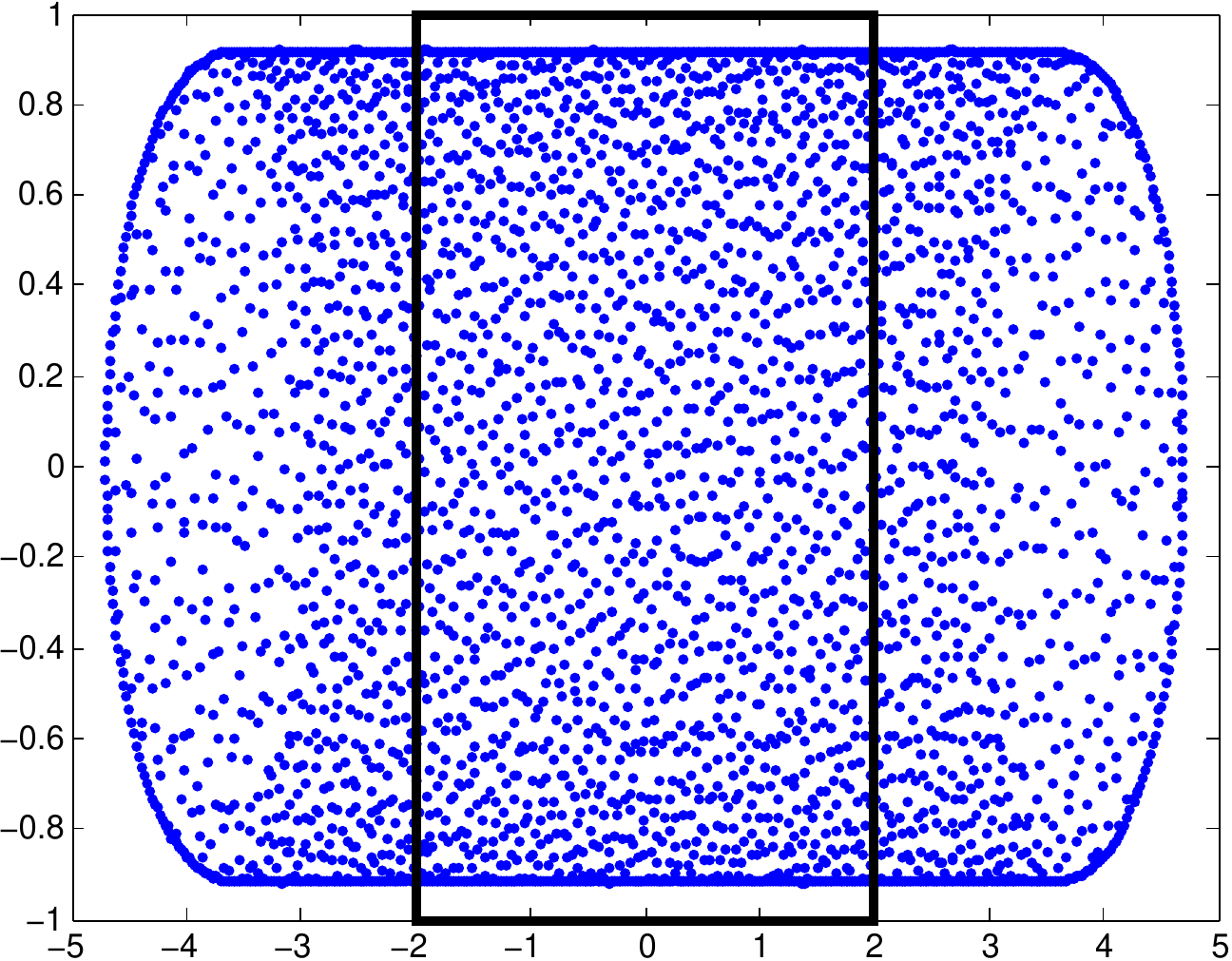}
 \end{minipage}
 \hspace{0.1cm}
 \begin{minipage}[b]{0.45\linewidth}
  \centering 
  \includegraphics[width=\textwidth]{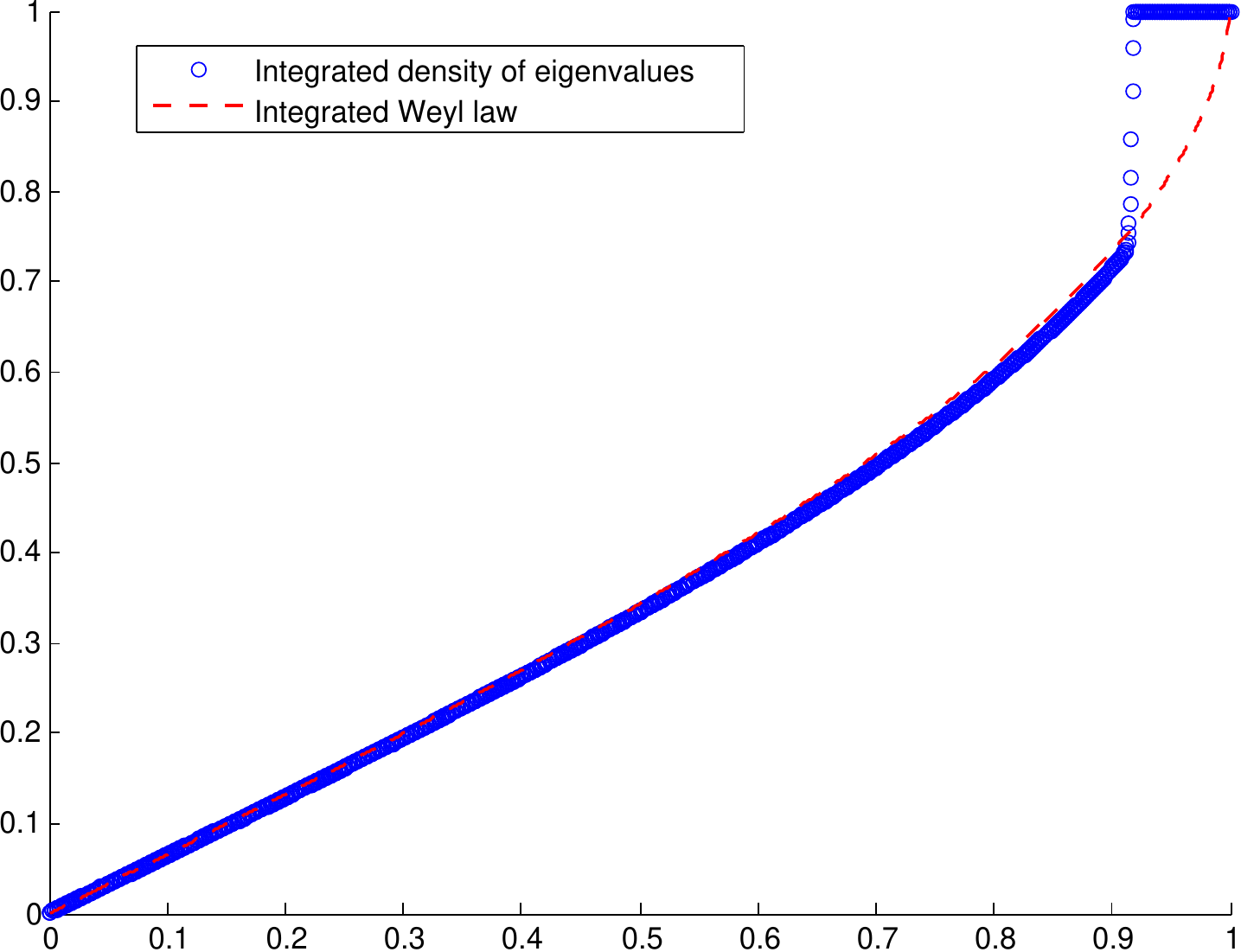}
 \end{minipage}
 \caption{On the left hand side we present the spectrum of the discretization of $hD + \exp(-ix)$ 
	  (approximated by a $3999\times 3999$-matrix) perturbed with a random Gaussian matrix $\delta R$ with 
	  $h=2\cdot 10^{-3}$ and $\delta=2\cdot 10^{-12}$. The black box indicates the region where we count 
	  the number of eigenvalues to obtain the image on the right hand side. There we show the integrated 
	  experimental density of eigenvalues, averaged over $400$ realizations of random Gaussian matrices, 
	  and the integrated Weyl law. We can see clearly a region close to the boundary of the pseudospectrum 
	  where Weyl asymptotics of the eigenvalues breaks down.}
 \label{fig_i2b}
\end{figure}
\begin{figure}[ht]
\centering
 \begin{minipage}[b]{0.45\linewidth}
  \centering
  \includegraphics[width=\textwidth]{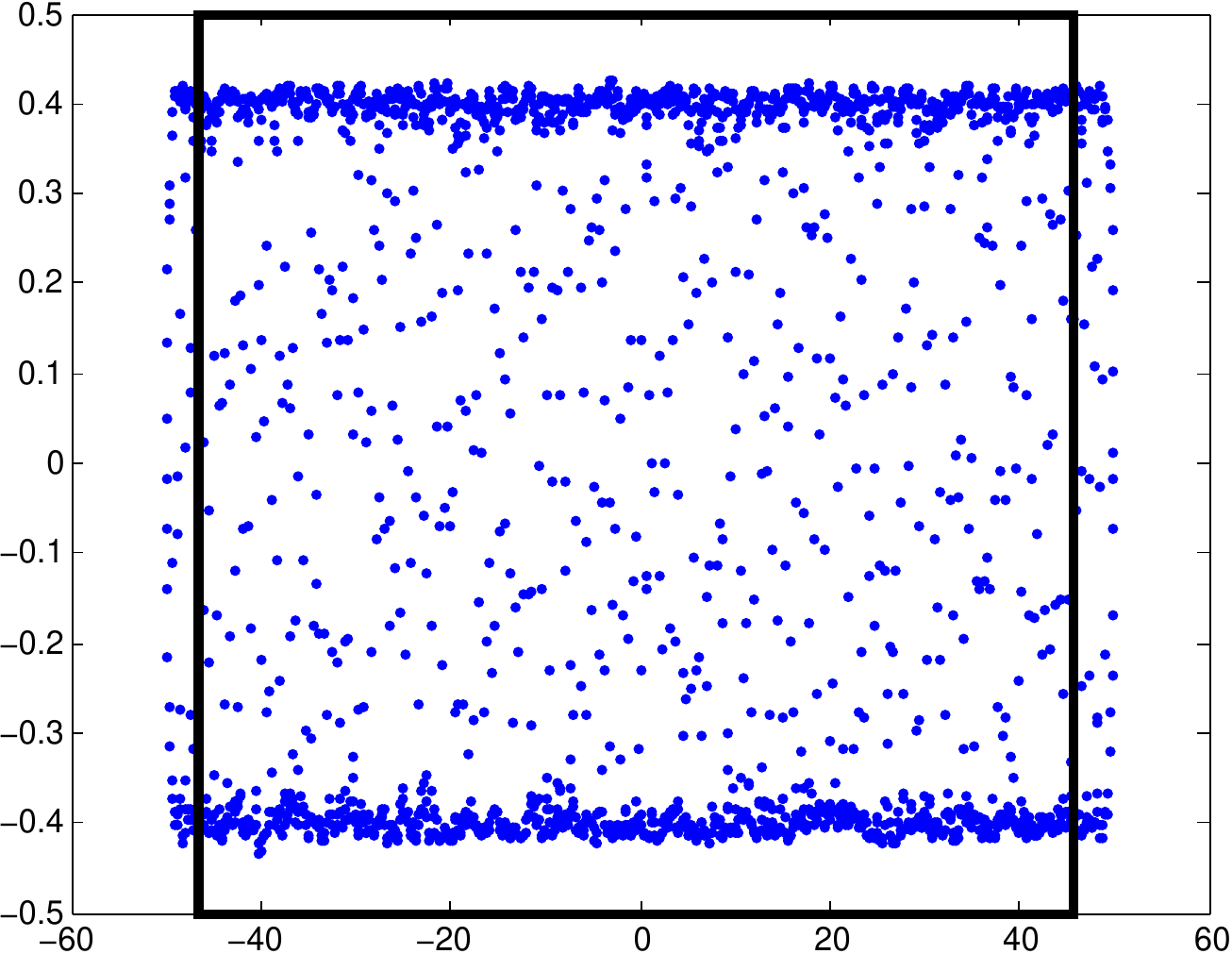}
 \end{minipage}
 \hspace{0.1cm}
 \begin{minipage}[b]{0.45\linewidth}
  \centering 
  \includegraphics[width=\textwidth]{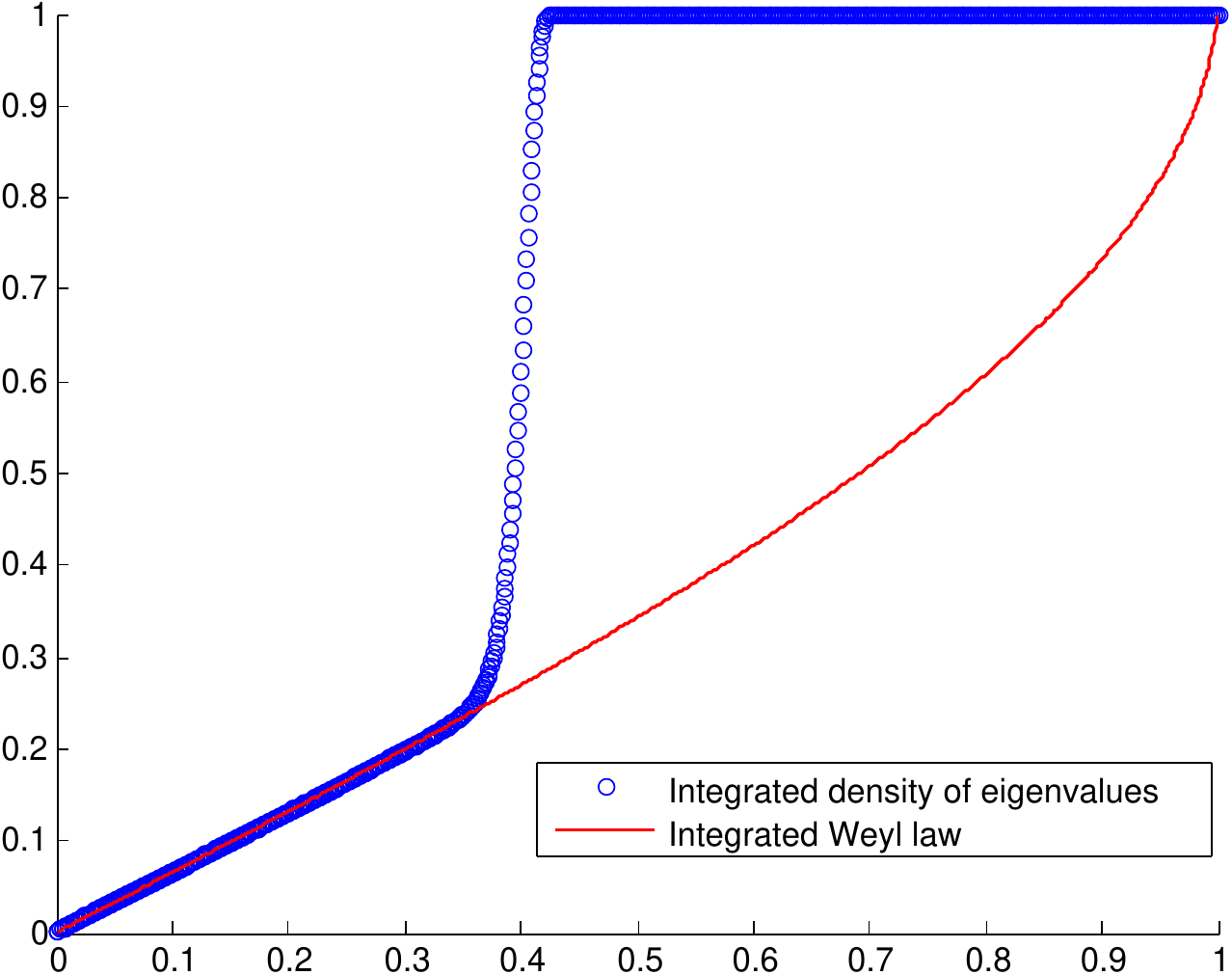}
 \end{minipage}
 \caption{As in Figure \ref{fig_i2b} (using a $1999\times 1999$-matrix for the approximation) with 
	  $h=5\cdot 10^{-2}$ and $\delta=\exp(-1/h)$. Here, the Weyl law breaks down even more 
	  dramatically than in Figure \ref{fig_i2b}.}
 \label{fig_i2c}
\end{figure}
\par
To the best of the author's knowledge, there has never been, until now, a precise
description of this phenomenon. 
This leads to the main question treated in this paper: We want to study the 
distribution of eigenvalues of a random perturbation of the operator $P_h$ 
in the whole of $\Sigma$. In particular this means studying regions where 
the norm of the resolvent of the unperturbed operator $P_h$ is 
much larger, of the same order of magnitude and much smaller than 
the coupling constant $\delta$. 
\\
\\
\noindent\textbf{Outline.} 
The principal aim of this work is to give a detailed description of the average density of 
eigenvalues of the randomly perturbed operator $P_h^{\delta}$.  \par
Section \ref{sec:MainRes} we shall present 
our main results: we shall state an $h$-asymp\-totic formula for the average density of eigenvalues and describe its 
properties. We will show that the spectrum of $P_h^{\delta}$ is be distributed, in average, in a 
band $\subset\Sigma$ whose breadth depends on the strength of the coupling constant. \par
In the interior of this band we will establish a Weyl law for the eigenvalues and show that they exhibit 
a strong accumulation property close to the boundary of this band. 
Outside of this band the average density of eigenvalues decays double exponentially.
\par 
Section \ref{subse:Quasmod} will give constructions of quasimodes 
for $z$ in the interior of $\Sigma$ and close to the boundary $\partial\Sigma$. 
In Section \ref{se:GrushProb} we treat the needed Grushin problems for the operator $P_h$. Section \ref{sec:EstEffectHamil} 
is dedicated to a Grushin problem for the perturbed operator $P_h^{\delta}$ and to its link with 
the symplectic volume of the phase space. 
Section \ref{sec:PrepLemmas} 
will state additional results to prove a formula for the first intensity measure of the random point 
process counting the eigenvalues of $P_h^{\delta}$ which then will be proven in Section \ref{sec:FIntMesForm}. 
Sections \ref{eqn:SecSymplVol} to \ref{sec:BehavDensity} will prove the main results.
\begin{rem}
Throughout this work we shall denote the Lebesgue measure 
on $\mathds{C^d}$ by $L(dz)$; denote $d(z):=\dist(z,\partial\Sigma)$; 
work with the convention that when we write $\mathcal{O}(h)^{-1}$ 
then we mean implicitly that $0 < \mathcal{O}(h) \leq C h$; denote by $f(x) \asymp g(x)$ that 
there exists a constant $C>0$ such that $C^{-1}g(x) \leq f(x) \leq Cg(x)$; write 
$\chi_1(x) \succ \chi_2(x)$, with $\chi_i \in \mathcal{C}_0^{\infty}$, 
if $\supp \chi_2 \subset \complement\supp ( 1- \chi_1)$. 
\end{rem}
%
%
%  \vskip.5cm 
  \noindent\textbf{Acknowledgments.} 
  I would like to thank very warmly my thesis advisor Johannes Sj\"ostrand for 
reading the first draft of this work and for his kind and enthusiastic manner 
in supporting me along the way. I would also like to thank sincerely my thesis 
advisor Fr\'ed\'eric Klopp for his kind and generous support.
%Chapter 2 %
%
\section{Main results}\label{sec:MainRes}
We begin by establishing how to choose the strength of the perturbation. For this 
purpose we discuss some estimates on the norm of the resolvent of $P_h$.
\subsection{The coupling $\delta$}
First, we give a description of the imaginary part of the action between 
$\rho_+(z)$ and $\rho_-(z)$. 
\begin{rem} 
 Much of the following is valid for $z\in\Omega\Subset\Sigma$ with 
\begin{equation}\label{hyp_Omega2}
     \Omega\Subset\Sigma ~\text{ open, relatively compact with}~
     \mathrm{dist\,}(\Omega,\partial\Sigma) \gg h^{2/3},
\end{equation}
instead of for $z\in\Omega$ as in Hypothesis \ref{hyp:H3}.
\end{rem}
\begin{defn}\label{defn:ZeroOrderDensity}
 Let $\Omega\Subset\Sigma$ as in \eqref{hyp_Omega2}, let $p$ denote the semiclassical 
 principal symbol of $P_h$ in (\ref{eqn:SemClPrinSymModlOp}) and let 
 $\rho_{\pm}(z)= (x_{\pm}(z),\xi_{\pm})$ be as above. Define
  \begin{equation*}
   S:=\min\left( \Ima \int_{x_+}^{x_-}(z-g(y))dy,\Ima \int_{x_+}^{x_--2\pi}(z-g(y))dy \right).
   \end{equation*}
\end{defn}
\begin{prop}\label{prop:SProp}
 Let $\Omega\Subset\Sigma$ as in \eqref{hyp_Omega2} and let $S(z)$ be as in Definition \ref{defn:ZeroOrderDensity}, 
 then $S(z)$ has the following properties for all $z\in\Omega$: 
  \begin{itemize}
   \item $S(z)$ depends only on $\Ima z$, is continuous and has the zeros $S(\Ima g(a)) = S(\Ima g(b)) =0$;
   \item $S(z)\geq 0$;
   \item for $\Ima z = \langle \Ima g\rangle$ the two integrals defining $S$ are equal; 
	 $S$ has its maximum at $\langle \Ima g\rangle$ and is strictly monotonously decreasing 
	 on the interval $[\langle \Ima g\rangle,\Ima g(b)]$ and stric. monotonously increasing on 
         $[\Ima g(a),\langle \Ima g\rangle]$;
   \item its derivative is piecewise of class 
   $\mathcal{C}^ {\infty}$ with the only discontinuity at $\Ima z = \langle \Ima g\rangle$. 
   Moreover, 
   \begin{align}\label{eqn:zRepS}
   &S(z) = \int_{\langle \Ima g\rangle}^{\Ima z} (\partial_{\Ima z}S)(t)dt+S(\langle \Ima g\rangle), \notag \\
   &(\partial_{\Ima z}S)(t) := \begin{cases}
                x_-(t) - x_+(t), ~\text{if } \Ima z \leq \langle \Ima g\rangle, \\
		 x_-(t)- 2\pi - x_+(t), ~\text{if }\Ima z > \langle \Ima g\rangle.
              \end{cases}
  \end{align}
   \item $S$ has the following asymptotic behavior for $z\in\Omega$ 
	  \begin{equation*}
	      S(z) \asymp d(z)^{\frac{3}{2}},
	 \end{equation*}
	 and 
	 \begin{equation*}
  	     |\partial_{\Ima z}S(z)| \asymp 
  	     d(z)^{\frac{1}{2}}.
 	 \end{equation*}
  \end{itemize}
\end{prop}
\begin{rem}
 Note that in (\ref{eqn:zRepS}) we chose to define 
 $\partial_{\Ima z}S(z) := x_-(z) - x_+(z)$ for $\Ima z = \langle \Ima g \rangle$. 
 We will keep this definition throughout this text.
\end{rem}
With the convention $\lVert (P_h - z)^{-1}\rVert = \infty$ for $z\in\sigma(P_h)$ we have the following estimate on the 
resolvent growth of $P_h$: 
\begin{prop}\label{prop:Resolvent}
Let $g(x)$ be as above. For $z\in\mathds{C}$ and $h>0$ define, 
 \begin{align*}
   \Phi(z,h) := 
    \begin{cases}
     -\frac{2\pi i}{h}(z -\langle g\rangle),~\text{if}~ \Ima z < \langle \Ima g\rangle , \\
      \frac{2\pi i}{h}(z -\langle g\rangle),~\text{if}~ \Ima z >\langle \Ima g\rangle,
     \end{cases}
  \end{align*}
 where $\Rea \Phi(z,h) \leq 0$. Then, under the assumptions of Definition 
 \ref{defn:ZeroOrderDensity} we have for $z\in\Omega\Subset\Sigma$ as in 
 \eqref{hyp_Omega2} that
 \begin{align}\label{res_id}
   \lVert (P_h - z)^{-1}\rVert &=  
   \frac{\sqrt{\pi}\left|1 - \e^{\Phi(z,h)}\right|^{-1}\e^{\frac{S(z)}{h}}}
	      {\sqrt{h}\left(\frac{i}{2}\{p,\overline{p}\}(\rho_+)\frac{i}{2}\{\overline{p},p\}(\rho_-)\right)^{\frac{1}{4}}}
	      (1 + \mO(h)) \\
        &\asymp
	     \frac{\e^{\frac{S(z)}{h}}}{\sqrt{h}\,d(z)^{1/4}}, ~ 
	     \text{for}~ |\Ima z - \langle\Ima g\rangle |>1/C, ~C\gg 1, \notag
  \end{align}
  where $\left|1 - \e^{\Phi(z,h)}\right| = 0$ if and only if 
  $z\in\sigma(P_h)$. Moreover, 
  \begin{equation*}
    \left|1 - \e^{\Phi(z,h)}\right| = 
    1  + \mO\!\left(\e^{-\frac{2\pi}{h}|\Ima z -\langle \Ima g\rangle|}\right) .
  \end{equation*}
\end{prop}
This proposition will be proven in Section \ref{susec:10.1}. 
The growth of the norm of the resolvent away from the line $\Ima z = \langle \Ima g\rangle$ 
is exponential and determined by the function $S(z)$. 
It will be very useful to write the coupling constant $\delta$ as follows:
\begin{defn}\label{defn:CoupConstDel}
 For $h>0$, define 
        \begin{equation*}
	   \delta := \delta(h) := \sqrt{h}\e^{-\frac{\epsilon_0(h)}{h}}
	\end{equation*}
        with $ \left(\kappa - \frac{1}{2}\right)h\ln (h^{-1}) +Ch \leq \epsilon_0(h) < S(\langle \Ima g\rangle)$ 
        for some $\kappa>0$ and $C>0$ large and where the last inequality is uniform in $h>0$. 
        This is equivalent to the bounds 
	\begin{equation*}
	 \sqrt{h}\e^{-\frac{S(\langle \Ima g\rangle)}{h}} <  \delta \ll h^{\kappa}.
	\end{equation*}
\end{defn}
\begin{rem}
 The upper bound on $\varepsilon_0(h)$ has been chosen in order to produce 
 eigenvalues sufficiently far away from the 
 line $\Ima z = \langle\Ima g\rangle$ where we find $\sigma(P_h)$. 
 The lower bound on $\varepsilon_0(h)$ is needed because we want to consider small random 
 perturbations with respect to $P_h$ (cf. \eqref{c_1}).
\end{rem}
\subsection{Auxiliary operator.}\label{sec:AuxOpe} To describe the elements of the average 
density of eigenvalues, it will be very useful to introduce the following operators 
which have already been used in the study of the spectrum of $P_h^{\delta}$ by 
Sj\"ostrand \cite{SjAX1002}. For the readers convenience, we will give a short 
overview: \par 
Let $z\in\C$ and we define the $z$-dependent elliptic self-adjoint operators 
$Q(z), \tilde{Q}(z):  L^2(S^1) \rightarrow L^2(S^1)$ by
  \begin{align}\label{eqn:defnSAOpQ}
   Q(z):=(P_h-z)^*(P_h-z),\quad \tilde{Q}(z):=(P_h-z)(P_h-z)^*
  \end{align}
with domains $\mathcal{D}(Q(z)),\mathcal{D}(\tilde{Q}(z))=H^2(S^1)$. Since $S^1$ is compact and these are 
elliptic, non-negative, self-adjoint operators their spectra are discrete and contained 
in the interval $[0,\infty[$. Since
  \begin{equation*}
   Q(z)u = 0 \Rightarrow (P_h-z)u=0
  \end{equation*}
it follows that $\mathcal{N}(Q(z)) = \mathcal{N}(P_h-z)$ and $\mathcal{N}(\tilde{Q}(z)) = \mathcal{N}((P_h-z)^*)$. 
Furthermore, if $\lambda\neq 0$ is an eigenvalue of $Q(z)$ with corresponding eigenvector $e_{\lambda}$ we see 
that $f_{\lambda}:=(P_h-z)e_{\lambda}$ is an eigenvector of $\tilde{Q}(z)$ with the eigenvalue $\lambda$. Similarly, 
every non-vanishing eigenvalue of $\tilde{Q}(z)$ is an eigenvalue of $Q(z)$ and moreover, since $P_h-z$, $(P_h-z)^*$ 
are Fredholm operators of index $0$ we see that 
  \begin{equation*}
   \text{dim}\,\mathcal{N}(P_h-z)=\text{dim}\,\mathcal{N}((P_h-z)^*).
  \end{equation*}
Hence the spectra of $Q(z)$ and $\tilde{Q}(z)$ are equal
  \begin{equation}\label{eqn:auxillaryOpEigVals}
   \sigma(Q(z))=\sigma(\tilde{Q}(z))=\{t_0^2,t_1^2,\dots\},~0\leq t_j\nearrow\infty.
  \end{equation}
We will show in Proposition \ref{prop:QDistFirstEig} that for $z\in\Omega\Subset\Sigma$ as in \eqref{hyp_Omega2}
   \begin{equation}\label{e_t0}
    t_0^2(z) \leq
      \mathcal{O}\!\left(d(z)^{\frac{1}{2}}h\e^{-\frac{2S}{h}}\right), 
      \quad
      t_1^2(z)\geq \frac{d(z)^{\frac{1}{2}}h}{\mO(1)}.
   \end{equation}
Now consider the orthonormal basis of $L^2(S^1)$ 
\begin{equation}\label{def:e0,e1,...}
 \{e_0,e_1,\dots\}
\end{equation}
consisting of the eigenfunctions of $Q(z)$.  By the previous observations we have 
 \begin{align*}
  (P_h-z)(P_h-z)^*(P_h-z)e_j &= t_j^2(P_h-z)e_j.
 \end{align*}
Thus defining $f_0$ to be the normalized eigenvector of $\widetilde{Q}$ corresponding to 
the eigenvalue $t_0^2$ and the vectors $f_j\in L^2(S^1)$, for $j\in\mathds{N}$, 
as the normalization of $(P_h-z)e_j$ such that 
  \begin{equation}\label{eqn:RelationEjFj}
   (P_h-z)e_j=\alpha_jf_j, \quad (P_h-z)^*f_j=\beta_je_j \quad\text{with} ~\alpha_j\beta_j = t_j^2,
  \end{equation}
yields an orthonormal basis of $L^2(S^1)$
\begin{equation}\label{def:f0,f1,...}
 \{f_0,f_1,\dots\}
\end{equation}
consisting of the eigenfunctions of $\tilde{Q}(z)$. 
Since 
  \begin{equation*}
   \alpha_j = ((P_h-z)e_j|f_j) = (e_j|(P_h-z)^*f_j) = \overline{\beta}_j
  \end{equation*}
we can conclude that $\alpha_j\overline{\alpha}_j=t_j^2$. \\
\par 
It is clear from \eqref{e_t0}, \eqref{eqn:RelationEjFj} that $e_0(z)$ (resp. $f_0(z)$) is and 
exponentially accurate quasimode for the $P_h-z$ (resp. $(P_h-z)^*$). We show in Section 
\ref{subse:Quasmod} that is localized to $\rho_+(z)$ (resp. $\rho_-(z)$). We will prove 
in the Sections \ref{susec:tunnel} and \ref{suse:EstEffHam} the 
following two formulas for tunneling effect: 
\begin{prop}\label{prop:TunnelEffect}
 Let $z\in\Omega \Subset\Sigma$ as in \eqref{hyp_Omega2} and let $e_0$ and $f_0$ be as in (\ref{def:e0,e1,...}) 
 and in (\ref{def:f0,f1,...}). Furthermore, let $S$ be as in Definition 
 \ref{defn:ZeroOrderDensity} and let $p$ and $\rho_{\pm}$ be as in Section 
 \ref{sec:Intro}. Let $h^{\frac{2}{3}} \ll d(z)$, then for all $z\in\Omega$ with 
 $|\Ima z - \langle\Ima g\rangle| > 1/C$, $C\gg 1$, 
 \begin{align*}
   |(e_0|f_0)| =
   \frac{\left(\frac{i}{2}\{p,\overline{p}\}(\rho_+)\frac{i}{2}\{\overline{p},p\}(\rho_-)\right)^{\frac{1}{4}}}
      {\sqrt{\pi h}}|\partial_{\Ima z}S|
      \left(1 + K(z;h)\right)
    \e^{-\frac{S}{h}},
  \end{align*}
  where $K(z;h)$ depends smoothly on $z$ and satisfies for all $\beta\in\mathds{N}^2$
  \begin{align*}
  \partial_{z\overline{z}}^{\beta} 
   K(z;h) = \mathcal{O}\!\left(d(z)^{\frac{|\beta|}{2}-\frac{3}{4}}h^{-|\beta|+\frac{1}{2}}\right).
  \end{align*}
\end{prop}

\begin{prop}\label{prop:TunnelEffect2}
 Under the same assumptions as in Proposition \ref{prop:TunnelEffect}, 
 let $\chi\in\mathcal{C}^{\infty}_0(S^1)$ with $\chi\equiv 1$ in a small open neighborhood of 
 $\overline{\{x_-(z):~z\in\Omega\}}$. Then, for $h^{\frac{2}{3}} \ll d(z)$, 
 \begin{align*}
  |([P_h,\chi]e_0|f_0)| = \sqrt{h}\left(\frac{\frac{i}{2}\{p,\overline{p}\}(\rho_+)
	  \frac{i}{2}\{\overline{p},p\}(\rho_-)}{\pi^2}\right)^{\frac{1}{4}}
   \left(1 + K(z;h)\right) \e^{-\frac{S}{h}},
    \end{align*}
  where $K(z;h)$ depends smoothly on $z$ and satisfies for all $\beta\in\mathds{N}^2$ 
  \begin{equation*}
  \partial_{z\overline{z}}^{\beta}K(z;h)= 
      \mathcal{O}\!\left(d(z)^{\frac{|\beta|-3}{2}}h^{1-(|\beta|)}\right).
  \end{equation*}
\end{prop}
\subsection{Average density of eigenvalues.}
Let $P_h^{\delta}$ be as in (\ref{eqn:DefPertP}), then we define the point process
  \begin{equation}\label{defn:PP}
   \Xi := \sum_{z\in\sigma(P_h^{\delta})}\delta_z,
  \end{equation}
where the zeros are counted according to their multiplicities and $\delta_z$ denotes the Dirac-measure in $z$. 
$\Xi$ is a well-defined random measure (cf. for example \cite{DaJo}) since, for $h>0$ small 
enough, $P_h^{\delta}$ is a random operator with discrete spectrum. 
To obtain an $h$-asymptotic formula for the average density of eigenvalues, we are interested in 
intensity measure of $\Xi$.
\begin{rem}
 Such an approach is more classical in the study of zeros of 
random polynomials and Gaussian analytic functions; we refer 
the reader to the works of B. Shiffman and S. Zelditch 
\cite{SZ03,ShZe08,ShZe99,Sh08}, M. Sodin \cite{So00} an the book 
\cite{FourPeop09} by J. Hough, M. Krishnapur, Y. Peres and 
B. Vir\'ag.
\end{rem}
The main result giving the average density of eigenvalues of $P_h^{\delta}$ is the following:
\begin{thm}\label{thm:ModelFirstIntensity}
Let $\Omega\Subset\Sigma$ be as in Hypothesis \ref{hyp:H3}. Let $C > 0$ be as in \eqref{eq_i33} 
and let $C_1>0$ as in (\ref{defn:PertIntOp}) such that $C-C_1>0$ is large enough. 
Let $\delta>0$ as in Definition \ref{defn:CoupConstDel} with $\kappa >0$ large enough.  
Define $N:=(2\lfloor C_1/h\rfloor + 1)^2$ and let $B(0,R)\subset\mathds{C}^N$ be the ball of radius 
$R:=Ch^{-1 }$ centered at zero. 
Then, there exists a $C_2>0$ such that for $h>0$ small enough and for 
all $\varphi\in\mathcal{C}_0(\Omega)$
 \begin{align}\label{eqn:thmModFirstIntens}
  \erw\big[~&\Xi(\varphi)\mathds{1}_{B(0,R)}~\big]  
   = 
  \int \varphi(z) D(z,h,\delta) L(dz)  + \mathcal{O}\!\left(\e^{-\frac{C_2}{h^2}}\right),
 \end{align}
with the density
 \begin{equation}\label{eqn:density}
   D(z,h,\delta)=\frac{1+ \mathcal{O}\!\left(\delta h^{-\frac{3}{2}}d(z)^{-1/4}\right)}{\pi} 
	     \Psi(z;h,\delta)\exp\{-\Theta(z;h,\delta)\},
  \end{equation}
which depends smoothly on $z$ and is independent of $\varphi$. Moreover, 
$\Psi(z;h,\delta) =\Psi_1(z;h) + \Psi_2(z;h,\delta)$ and 
for $z\in\Omega$ with $d(z) \gg (h\ln h^{-1})^{2/3}$ 
  \begin{align}\label{eqn:thm1_dens}
   &\Psi_1(z;h) = \frac{1}{h}\left\{\frac{i}{\{p,\overline{p}\}(\rho_+(z))} 
	+ \frac{i}{\{\overline{p},p\}(\rho_-(z))} \right\}
   + \mathcal{O}\!\left(d(z)^{-2}\right),
   \notag\\
   &\Psi_2(z;h,\delta) = \frac{\left| (e_0|f_0) \right|^2}{\delta^2}
   \left(1 + \mathcal{O}\!\left(d(z)^{-3/4}h^{1/2}\right)\right),
      \notag \\
    &\Theta(z;h,\delta) = \frac{\left|([P_h,\chi]e_0|f_0) 
	  +\mathcal{O}\!\left(d(z)^{-1/4}h^{-5/2}\delta^2\right) \right|^2}
	  {\delta^2(1+\mO(h^{\infty}))}
	  \left(1 +\mathcal{O}\!\left(\e^{-\frac{\asymp d(z)^{3/2}}{h}}\right)\right).
  \end{align}
Furthermore, in (\ref{eqn:thmModFirstIntens}), $\mathcal{O}\!\left(\e^{-\frac{C_2}{h^2}}\right)$ means 
$\langle T_h,\varphi\rangle$ where $T_h\in\mathcal{D}'(\mathds{C})$ such that 
  \begin{equation*}
   |\langle T_h,\varphi\rangle|\leq C \lVert\varphi\rVert_{\infty}\e^{-\frac{C_2}{h^2}}
  \end{equation*}
for all $\varphi\in\mathcal{C}_0(\Omega)$ where $C$ is independent 
of $h$, $\delta$, $\eta$ and $\varphi$. \par 
\end{thm}
Let us give some comments on this result. The dominant part of the 
density of eigenvalues $D$ consists of three parts: the first, $\Psi_1$, 
is up to a small error the Lebesgue density of $p_*(d\xi\wedge dx)$, 
where $d\xi\wedge dx$ is the symplectic form on $T^*S^1$ and $p$ as 
in (\ref{eqn:SemClPrinSymModlOp}). We prove in Proposition 
\ref{cor:1MomSymplVol} that 
  \begin{equation*}
   \frac{1}{h}\left\{\frac{i}{\{p,\overline{p}\}(\rho_+(z))} 
	+ \frac{i}{\{\overline{p},p\}(\rho_-(z))} \right\}L(dz)
	=\frac{1}{2h}p_*(d\xi\wedge dx).
  \end{equation*}
The second part, $\Psi_2$, is given by a tunneling effect. Inside the 
$(h\delta)$-pseudo\-spectrum its contribution vanishes in the error term 
of $\Psi_1$. However, close to the boundary of the $\delta$-pseudospectrum 
$\Psi_2$ becomes of order $h^{-2}$ and thus yields a higher density of 
eigenvalues. This can be seen by comparing the more explicit formula for 
$\Psi_2$ given in Proposition \ref{prop:DensExplicit} with the expression 
for the norm of the resolvent of $P_h$ given in Proposition \ref{prop:Resolvent}. 
More details on the form of $\Psi_2$ in this zone will be given in Proposition 
\ref{prop:WedgePoisson}. \par 
The third part, $\exp\{-\Theta\}$, is also given by a tunneling effect and it 
plays the role of a cut-off function which exhibits double exponential decay 
outside the $\delta$-pseudospectrum and is close to $1$ inside. This will be 
made more precise in Section \ref{suse:MainResPropDens}.
\par
We have the following explicit formulas for these functions and their growth 
properties:
\begin{prop}\label{prop:DensExplicit}
Under the assumptions of Definition \ref{defn:ZeroOrderDensity} 
and Theorem \ref{thm:ModelFirstIntensity}, define for $h>0$ and $\delta>0$ the functions
  \begin{align*}
    \Theta^0(z;h,\delta) :=
     \frac{h\left(\frac{i}{2}\{p,\overline{p}\}(\rho_+)
	  \frac{i}{2}\{\overline{p},p\}(\rho_-)\right)^{\frac{1}{2}}}{\pi}\frac{\e^{-\frac{2S}{h}}}{\delta^2}.
  \end{align*}
 Then, for $|\Ima z - \langle\Ima g\rangle| > 1/C$, $C\gg 1$, 
 \begin{align}\label{eqn:thmModFirstIntensDensity}
   &\Psi_2(z;h,\delta)=
     \frac{\left(\frac{i}{2}\{p,\overline{p}\}(\rho_+)\frac{i}{2}\{\overline{p},p\}(\rho_-)\right)^{\frac{1}{2}}}
	{\pi h\delta^2\exp\{\frac{2S}{h}\}}|\partial_{\Ima z}S(z)|^2
      \left(1 + \mathcal{O}\!\left(\frac{h^{1/2}}{d(z)^{\frac{3}{4}}}\right)\right)
    \notag \\
   &\Theta(z;h,\delta) = \Theta^0(z;h,\delta)
   \left(1 +\mathcal{O}\!\left(\frac{h^{\frac{3}{2}}}{d(z)^{\frac{1}{4}}}\right)\right)
   + \mathcal{O}\!\left(\frac{d(z)^{\frac{1}{4}}\delta }{h^{2}}+ 
   \frac{\delta^2 }{d(z)^{\frac{1}{2}}h^{5}}\right).
  \end{align}
  The estimates in (\ref{eqn:thmModFirstIntensDensity}) are stable under 
application of $d(z)^{-\frac{|\beta|}{2}}h^{|\beta|}\partial_{z\overline{z}}^{\beta}$, 
for $\beta\in\N^2$.
\end{prop}
\begin{prop}\label{prop:GrowPropDens}
 Under the assumptions of Definition \ref{defn:ZeroOrderDensity} and 
 Theorem \ref{thm:ModelFirstIntensity} we have that
 \begin{align*}
   \frac{i}{2}\{p,\overline{p}\}(\rho_+)
	  \frac{i}{2}\{\overline{p},p\}(\rho_-)
	  \asymp  d(z), ~
   \frac{i}{\{p,\overline{p}\}(\rho_+(z))} 
	+ \frac{i}{\{\overline{p},p\}(\rho_-(z))}  \asymp \frac{1}{\sqrt{d(z)}}
  \end{align*}
and
  \begin{equation*}
   \Psi_2(z;h,\delta) \asymp \frac{(d(z))^{3/2}\e^{-\frac{2S}{h}}}{h\delta^2}, \quad 
   \Theta^0(z;h,\delta)\asymp h\sqrt{d(z)}\left|1 - \e^{\Phi(z,h)}\right|\frac{\e^{-\frac{2S}{h}}}{\delta^2}.
  \end{equation*}
\end{prop}
$\phantom{[.}$\\
In the next Subsection we will explain the asymptotic properties of the density appearing 
in (\ref{eqn:thmModFirstIntens}).
\subsection{Properties of the average density of eigenvalues and its integral with 
respect to $\Ima z$}\label{suse:MainResPropDens} 
It will be sufficient for our purposes to consider rectangular subsets of $\Sigma$: 
for $c<d$ define
  \begin{equation}\label{def:stripSigma}
   \Sigma_{c,d}:=\left\{z\in\overline{\Sigma}~\Big|~\min\limits_{x\in S^1}\Ima g(x)\leq 
	      \Ima z \leq \max\limits_{x\in S^1}\Ima g(x), ~ c <\Rea z< d\right\}.
  \end{equation}
Roughly speaking, there exist three regions in $\Sigma$: 
 \begin{align*}
  & (1)\quad z\in\Sigma_w \subset\Sigma ~\Longleftrightarrow~ \lVert (P_h-z)^{-1} \rVert \gg (h\delta)^{-1}, \notag \\
  & (2)\quad z\in\Sigma_r \subset\Sigma ~\Longleftrightarrow~ \lVert (P_h-z)^{-1} \rVert \asymp \delta^{-1}, \notag \\
  & (3)\quad z\in\Sigma_v \subset\Sigma ~\Longleftrightarrow~ \lVert (P_h-z)^{-1} \rVert \ll \delta^{-1},
 \end{align*} 
which depend on the strength of the coupling constant $\delta>0$. In $\Sigma_w$, the average density is of order $h^{-1}$ 
and is governed by the symplectic volume and thus yielding a Weyl law. 
In $\Sigma_r$, the average density spikes and is of order $h^{-2}$ and 
is equal to the symplectic volume plus the function $\Psi_2$ yielding in total a Poisson-type distribution. 
In $\Sigma_v$, the average density is rapidly decaying and is void of eigenvalues with a high probability, since 
\begin{equation*}
   \Theta \asymp \lVert (P_h-z)^{-1} \rVert^{-2}\delta^{-2}
  \end{equation*}
which follows from Proposition \ref{prop:TunnelEffect2} and Proposition
\ref{prop:Resolvent}.
\par
\begin{figure}[ht]
\includegraphics[width=0.7\textwidth]{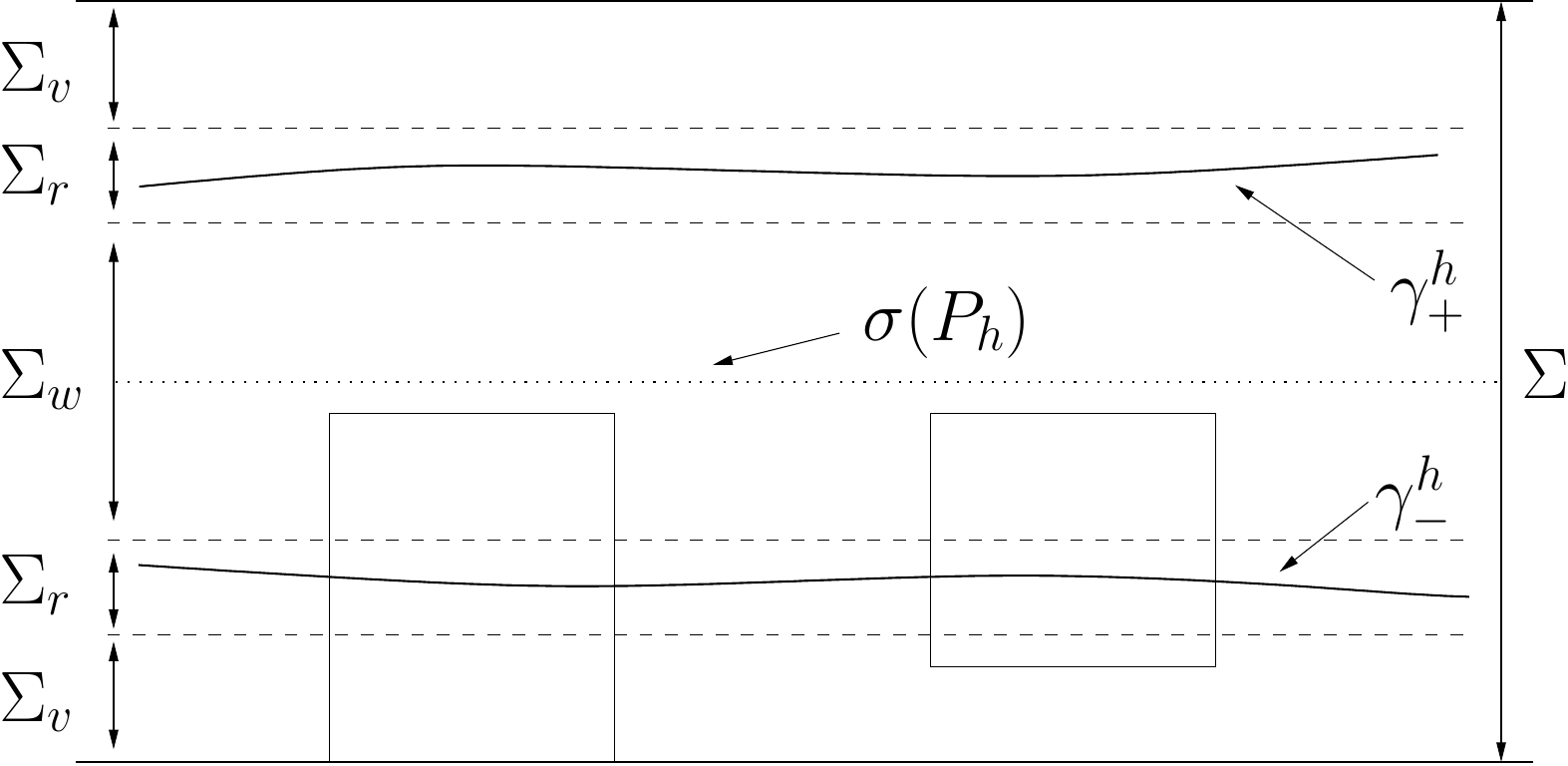}
\caption{The three zones in $\Sigma$ with a schematic representation of $\gamma_{\pm}^h$. 
The two boxes indicate zones where the integrated densities are equal up to a small error.}\label{zones}
\end{figure}
% We prove that there exist two smooth curves, $\Gamma_{\pm}^h$, along which the average density 
% of eigenvalues obtains its local maxima. 
We will prove that there exist two smooth curves, $\Gamma_{\pm}^h$, close to the boundary of 
the $\delta$-pseudo\-spectrum of $P_h^{\delta}$, along which the average density of eigenvalues obtains 
its local maxima. Note that this is still inside the $(Ch^{-1}\delta)$-pseudo\-spectrum of $P_h^{\delta}$ 
(cf Hypothesis \ref{hyp:H5}) since pseudospectra are nested (meaning that 
$\sigma_{\varepsilon_1}(P_h^{\delta})\subset \sigma_{\varepsilon_2}(P_h^{\delta})$ for 
$\varepsilon_1 < \varepsilon_2$).
\begin{prop}\label{prop:the2curves}
 Let $z\in\Omega\Subset\Sigma_{c,d}$ as in Hypothesis \ref{hyp:H3} with $\Sigma_{c,d}$ as 
 in (\ref{def:stripSigma}), let $S(z)$ be as in Definition \ref{defn:ZeroOrderDensity} and 
 let $t_0^2(z)$ be as in (\ref{eqn:auxillaryOpEigVals}). Let $\delta>0$ and $\varepsilon_0(h)$ 
 be as in Definition \ref{defn:CoupConstDel} with $\kappa > 0$ large enough. Moreover, let 
 $D(z,h,\delta)$ be the average density of eigenvalues of the operator of $P_h^{\delta}$ 
 given in Theorem \ref{thm:ModelFirstIntensity}. Then,
  \begin{enumerate}
   \item for $0<h\ll 1$, there exist numbers $y_{\pm}(h)$ such that $\varepsilon_0(h)=S(y_{\pm}(h))$ with 
      \begin{align*}
	\frac{1}{C}(h\ln h^{-1})^{\frac{2}{3}} &\ll y_{-}(h) < \langle \Ima g\rangle - ch\ln h^{-1} \notag \\
	&<\langle \Ima g\rangle + ch\ln h^{-1} < y_{+}(h) \ll \Ima g(b) - \frac{1}{C}(h\ln h^{-1})^{\frac{2}{3}},
      \end{align*}
     for $c>1$. Furthermore, 
       \begin{align*}
        y_{-}(h), (\Ima g(b) - y_+(h)) \asymp (\varepsilon_0(h))^{2/3};
       \end{align*}
   \item there exists $h_0>0$ and a family of smooth curves, indexed by $h\in]h_0,0[$,
	\begin{equation*}
	 \gamma_{\pm}^h: ~ ]c,d[ \longrightarrow \mathds{C} 
	 ~\text{with}~ \Rea \gamma_{\pm}^h(t) =t
	\end{equation*}
	such that
	\begin{equation*}
	  |t_0(\gamma_{\pm}^h(t))| = \delta.
	\end{equation*}
	Moreover,
	\begin{equation*}
	  \lVert (P_h -\gamma_{\pm}^h(t))^{-1}\rVert 
	 = \delta^{-1},
	\end{equation*}
	and 
	\begin{equation*}
	  \Ima \gamma^h_{\pm}(t) = y_{\pm}(\varepsilon_0(h))
	  \left(1 + \mathcal{O}\left(\frac{h}{\varepsilon_0(h)}\right)\right).
	\end{equation*}
	Furthermore, there exists a constant $C>0$ such that 
      \begin{equation*}
	  \frac{d\Ima \gamma_{\pm}^h}{d t}(t) 
	  = \mathcal{O}\!\left(\exp\left[-\frac{\varepsilon_0(h)}{Ch}\right]\right).
	\end{equation*}
  \item there  exists $h_0>0$ and a family of smooth curves, indexed by $h\in]h_0,0[$,
       \begin{equation*}
	 \Gamma_{\pm}^h: ~ ]c,d[ \longrightarrow \mathds{C},~ \Rea \Gamma_{\pm}^h(t) = t,
	\end{equation*}
	with $\Gamma_-\subset\{\Ima z <  \langle \Ima g\rangle\}$ and 
	$\Gamma_+\subset\{\Ima z >  \langle \Ima g\rangle\}$, along which 
	$\Ima z \mapsto D(z,h)$ takes its local 
	maxima on the vertical line $\Rea z = \mathrm{const.}$ and 
	  \begin{align*}
	   \frac{d}{dt}\Ima\Gamma_{\pm}^h(t) = 
	   \mO\!\left(\frac{h^{4}}{\varepsilon_0(h)^{4}}\right).
	  \end{align*}
	Moreover, for all $c < t < d$
	\begin{equation*}
	  |\Gamma_{\pm}^h(t) - \gamma_{\pm}^h(t)|
	  \leq \mathcal{O}\left(\frac{h^{5}}{\varepsilon_0(h)^{13/3}}\right).
	  \end{equation*}	
  \end{enumerate}
\end{prop}
With respect to the above described curves we prove the following properties of 
the average density of eigenvalues:
\begin{prop}\label{prop:DensityProperties}
 Let $ d\xi\wedge dx$ be the symplectic form on $T^*S^1$ and $p$ as in (\ref{eqn:SemClPrinSymModlOp}). 
 Let $\varepsilon_0=\varepsilon_0(h)$ be as in Definition \ref{defn:CoupConstDel}. Then, under 
 the assumptions of Theorem \ref{thm:ModelFirstIntensity} there exist $\alpha, \beta > 0$ such that 
  \begin{enumerate} 
   \item for $z\in\Sigma_{c,d}$ with 
	  \begin{align*}
		  \Ima\gamma_{-}(\Rea z)
		  + \alpha \frac{h}{\varepsilon_0^{1/3}}\ln\frac{\varepsilon_0^{1/3}}{h}
 		  \leq\Ima z \leq \Ima\gamma_{+}(\Rea z)-
 		  \alpha \frac{h}{\varepsilon_0^{1/3}}\ln\frac{\varepsilon_0^{1/3}}{h}
 		  \end{align*}
         we have that 
 		\begin{align*}
 		 D(z;h,\delta)L(dz)
 		 = \frac{1}{2\pi h} p_*(d\xi\wedge dx) + \mO\left(d(z)^{-2}\right)L(dz),
 		\end{align*}
	where $D(z;h,\delta)$ is the average density of eigenvalues of the operator $P_h^{\delta}$ given in 
	Theorem \ref{thm:ModelFirstIntensity}.
  \item for \begin{align*}
		  \Omega_1(\beta) := 
		  \bigg\{z\in\Sigma_{c,d}~\Big| ~ &\Ima\gamma_{-}(\Rea z)-\frac{h}{\varepsilon_0^{1/3}}
		  \ln \left( \beta\ln\frac{\varepsilon_0^{1/3}}{h}\right) \notag \\
 		  &\leq\Ima z \leq \Ima\gamma_{+}(\Rea z)+\frac{h}{\varepsilon_0^{1/3}}
		  \ln \left( \beta\ln\frac{\varepsilon_0^{1/3}}{h}\right)\bigg\}.
 		  \end{align*}
 	we have that 	  
 	\begin{align*}
 	\int\limits_{z\in\Omega_1(\beta)}
 	D(z;h,\delta) L(dz)
 		 = \int\limits_{\Sigma_{c,d}}\frac{p_*(d\xi\wedge dx)}{2\pi h} 
 		   + \mO\left(\varepsilon_0^{-\frac{2}{3}}\right) 		 
 	\end{align*}  
  \item for all $\varepsilon>0$ and all 
	  $\Omega(\varepsilon)\Subset\Sigma_{c,d}\backslash\Omega_2(\beta,\varepsilon)$ satisfying 
	  Hypothesis \ref{hyp:H3}, where 
	  \begin{align*}
		  \Omega_2(\beta,\varepsilon) := 
		  \bigg\{ z\in&\Sigma_{c,d}~\Big| ~\Ima\gamma_{-}(\Rea z)-\frac{h}{\varepsilon_0^{1/3}}
		  \ln  \left(\beta\ln\frac{\varepsilon_0^{1/3}}{h}\right)
		  - \varepsilon \notag \\
 		  &\leq\Ima z \leq \Ima\gamma_{+}(\Rea z)+\frac{h}{\varepsilon_0^{1/3}}
		  \ln  \left(\beta\ln\frac{\varepsilon_0^{1/3}}{h}\right)+\varepsilon\bigg\},
 	\end{align*}
 	we have that 
 		  \begin{equation*}
 		   \int_{\Omega(\varepsilon) }
		   D(z;h,\delta) L(dz)
		   =\mathcal{O}\!\left(\exp\left\{-\e^{\frac{\varepsilon}{Ch}}\right\}\right).
 		  \end{equation*}
  \end{enumerate}
\end{prop}
Proposition \ref{prop:DensityProperties} makes more precise the rough description of the behavior of the average density 
of eigenvalues, given at the beginning of this section: Point $1.$ tells us that 
in the interior of the $\delta$-pseudospectrum, 
up to a distance of order $h\ln\frac{1}{h}$ to the curves $\gamma_{\pm}^h$ (see Figure \ref{zones}), 
the density is given by a Weyl law. Assertion $2.$ tells us that the eigenvalues accumulate strongly 
in the close vicinity of these curves such that when integrating the density in the box 
$\Omega_1\Subset\Sigma_{c,d}$ the number of eigenvalues is given (up to small error) by the 
integrated Weyl density in all of $\Sigma_{c,d}$ (cf Figure \ref{zones}). This augmented density 
can be seen as the accumulated eigenvalues which would have been given by a Weyl law in the region 
from $\gamma_{\pm}^h$ up to the boundary $\partial\Sigma$ (see also Figures \ref{fig_i31} and 
\ref{fig_i3} for an example).
\par
The last point of the proposition tells us that outside of a strip of the form of $\Omega_1$ 
the density decays double-exponentially. 
\subsubsection{The density in the zone of spectral accumulation}
We give a finer description of the density of eigenvalues close to its local maxima 
at $\Gamma_{\pm}^h$: 
\begin{prop}\label{prop:WedgePoisson}
 Assume the hypothesis of Theorem \ref{thm:ModelFirstIntensity}. Let $S(z)$ be 
 as in Definition \ref{defn:ZeroOrderDensity} and let $\Psi_{2}(z,h,\delta)$ and 
 $\Theta(z;h,\delta)$ be as in Theorem \ref{thm:ModelFirstIntensity}. 
 Then for $|\Ima z - \langle\Ima g\rangle| > 1/C$ with $C \gg 1$ large enough,
  \begin{equation*}
   \Psi_2\e^{-\Theta}
   = 
   \left[\frac{|\partial_{\Ima z}S|^2}{h^2}\Theta
   \left(1+\mO\left(d(z)^{-3/4}h^{1/2}\right)\right)
   +
   \mathcal{O}\!\left(d(z)^{5/4}\right)\right]
   \e^{-\Theta}.
  \end{equation*}
\end{prop}
Let us give some remarks on this result. First, we see that we can approximate 
the second part of the density of eigenvalues by Poisson distribution 
scaled by the monotone function $\partial_{\Ima z} S(z)$. Second, since 
$\Theta \asymp \lVert (P_h-z)^{-2} \rVert^{-1}\delta^{-2}$, we see that 
the effects of the second part of the density vanish in the error term 
of $\Psi_1$ as long as $\lVert (P_h-z)^{-1} \rVert \gg \delta^{-1}$.
However, for $\lVert (P_h-z)^{-1} \rVert \asymp \delta^{-1}$ it is of order 
$\mO(d(z)h^{-2})$ and dominates the Weyl term.
\subsection{Example: Numerical simulations for $hD+\e^{-ix}$}\label{suse:Example1}
To illustrate our results we look at the discretization of $P_h=hD+\e^{-ix}$ in Fourier 
space which is approximated by the $(2N+1)\times(2N+1)$-matrix $H=hD + E$, 
$N\in\mathds{N}$, where $D$ and $E$ are defined by 
  \begin{equation*}
    D_{j,k} := \begin{cases}
		  j \quad \text{if } j=k, \\
		  0 \quad \text{else}
		 \end{cases}
  \quad \text{and} \quad
    E_{j,k} := \begin{cases}
		  1 \quad \text{if } k=j+1, \\
		  0 \quad \text{else},
		 \end{cases}
  \end{equation*}
where $j,k\in\{-N,-N+1,\dots,N\}$. Let $R$ be a $(2N+1)\times(2N+1)$ random matrix, 
where the entries $R_{j,k}$ are independent and identically distributed complex Gaussian 
random variables, $R_{j,k}\sim\mathcal{N}_{\mathds{C}}(0,1)$. For $h>0$ and $\delta>0$ as 
in Theorem \ref{thm:ModelFirstIntensity}, we let MATLAB calculate the spectrum $\sigma(H+\delta R)$. Since 
here $g(x)=\e^{-ix}$ (cf. \eqref{eqn:defnModelOperator}), it follows that in this case $\Sigma$ is given 
by $\{z\in\C;~|\Ima z|\leq 1\}$ (cf. \eqref{eq_i6}).
We are going to perform our numerical experiments for the following two cases: \\ 
\paragraph{Polynomially small (in $h$) coupling $\delta$} We set the above parameters to be $h=2\cdot 10^{-3}$, 
       $\delta=2\cdot 10^{-12}\approx 0.1\cdot h^4$ and $N=1999$. Figure \ref{fig_i31} shows 
       the spectrum of $H+\delta R$ computed by MATLAB.
       \begin{figure}[ht]
         \centering
         \begin{minipage}[b]{0.49\linewidth}
         \centering
         \includegraphics[width=\textwidth]{fig10.pdf}
        \end{minipage}
        \hspace{0cm}
        \begin{minipage}[b]{0.49\linewidth}
         \centering 
         \includegraphics[width=\textwidth]{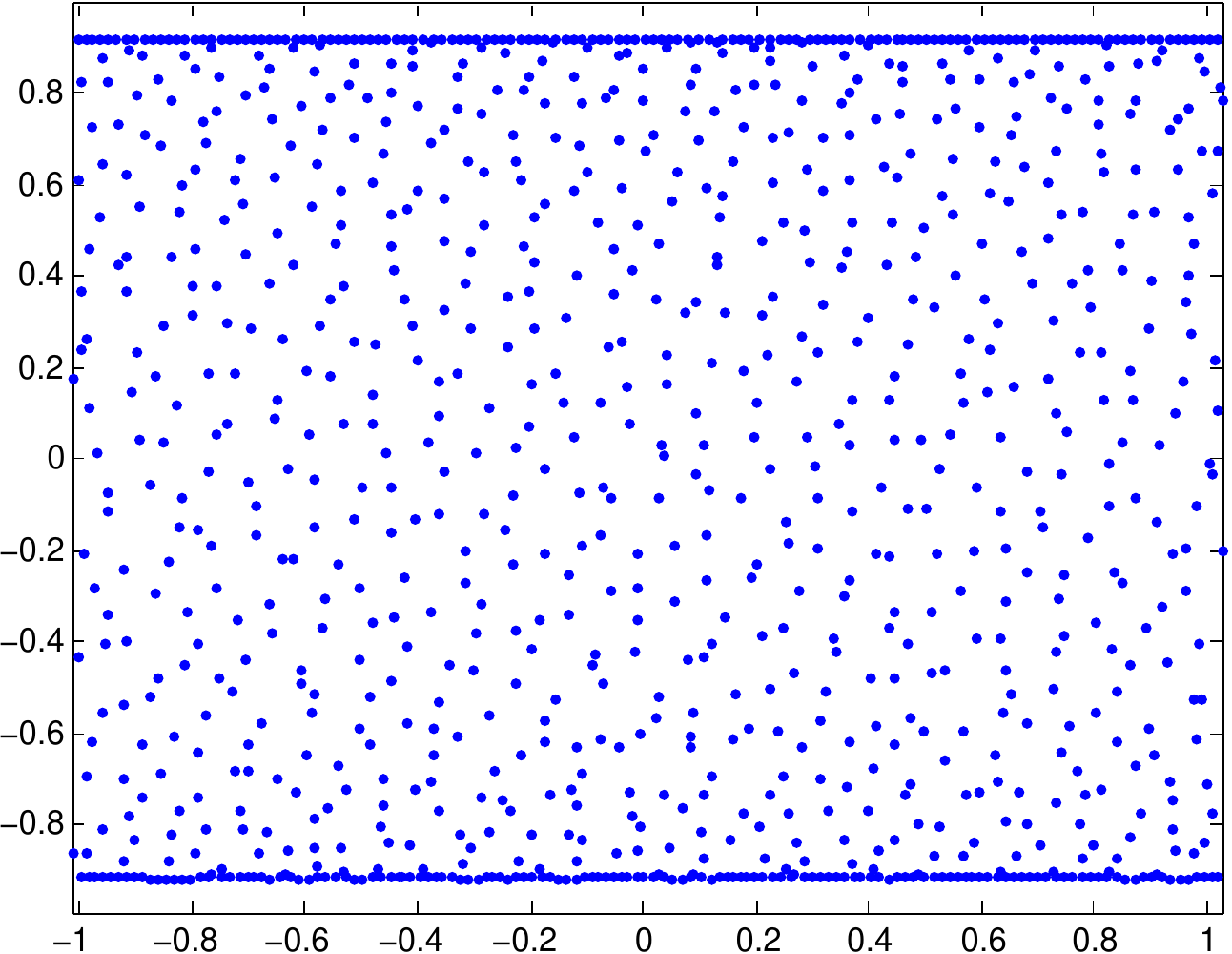}
        \end{minipage}
        \caption{On the left hand side we present the spectrum of the discretization of $hD + \exp(-ix)$ 
       	  (approximated by a $3999\times 3999$-matrix) perturbed with a random Gaussian matrix $\delta R$ with 
       	  $h=2\cdot 10^{-3}$ and $\delta=2\cdot 10^{-12}$. The black box indicates the region where we count 
       	  the number of eigenvalues to obtain Figure\ref{fig_i3}. The right hand side is a magnification of the 
       	  central part of the spectrum depicted on the left hand side.}
        \label{fig_i31}
       \end{figure}
\begin{figure}[ht]
 \begin{minipage}[b]{0.49\linewidth}
  \centering
  \includegraphics[width=\textwidth]{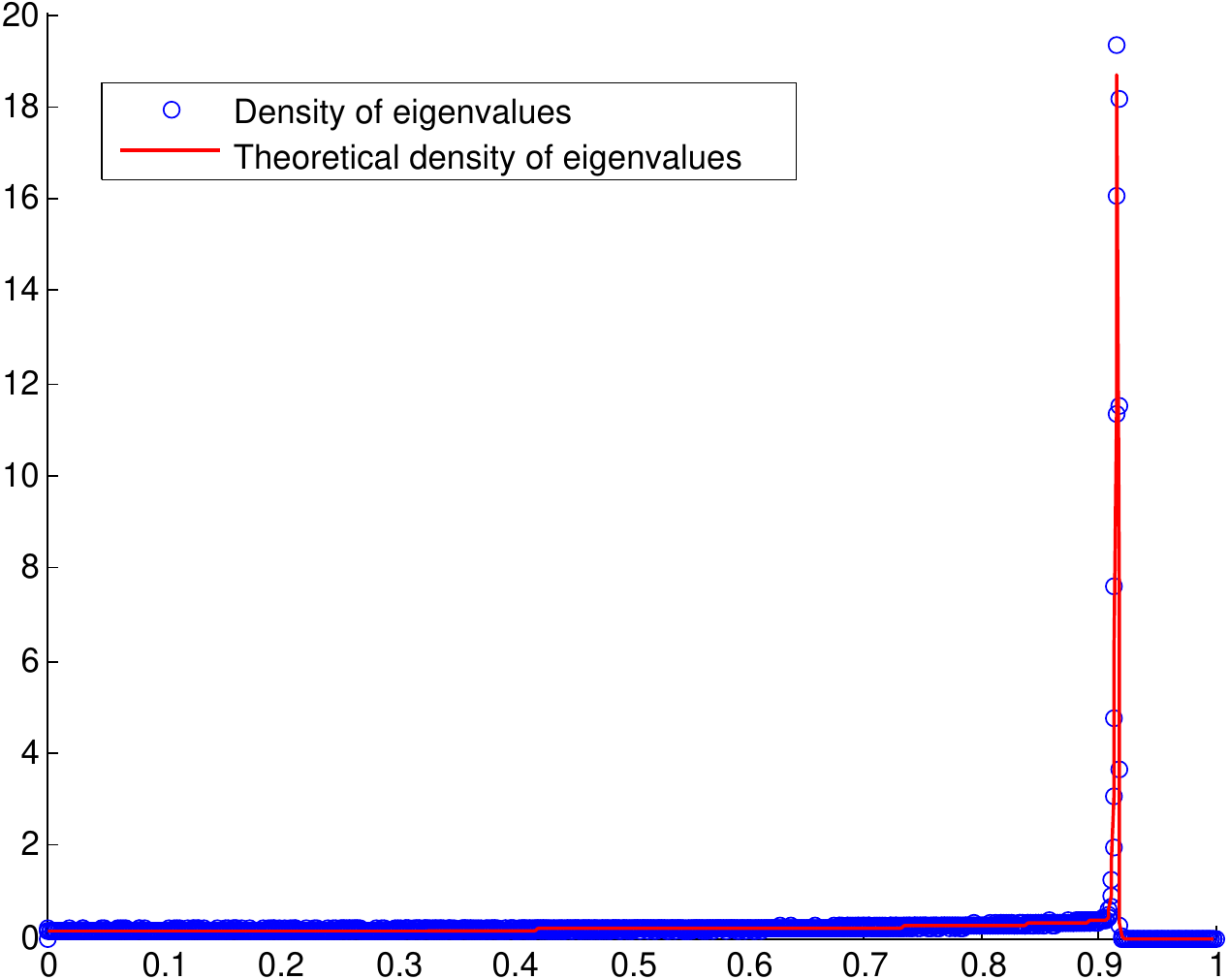}
 \end{minipage}
 \hspace{0cm}
 \begin{minipage}[b]{0.49\linewidth}
  \centering 
  \includegraphics[width=\textwidth]{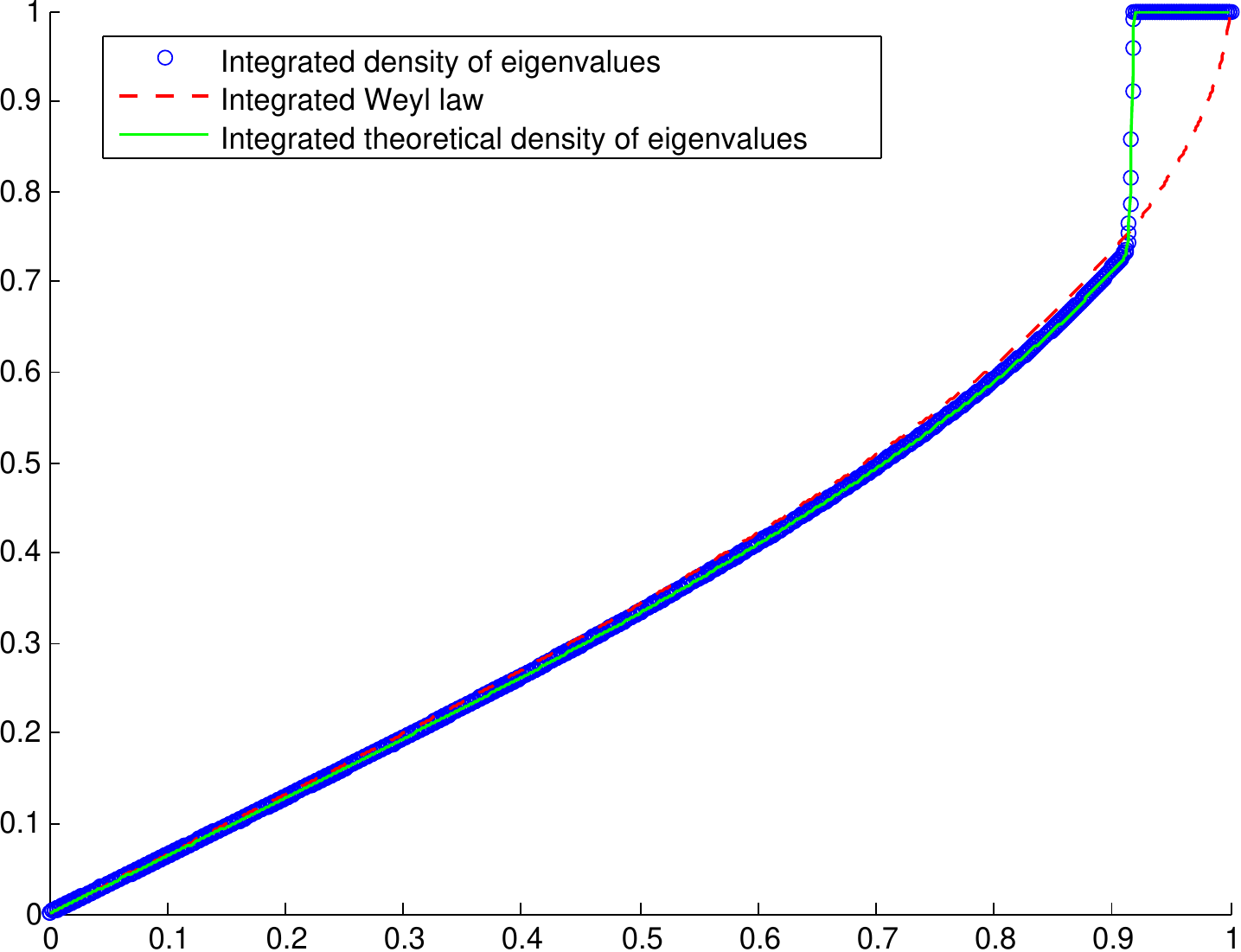}
 \end{minipage}
 \caption{On the left hand side we compare the experimental and the theoretical 
	  (cf. Theorem \ref{thm:ModelFirstIntensity}) density of eigenvalues. On the right hand side 
	  we compare the experimental and the theoretical integrated density of eigenvalues 
	  with the integrated Weyl law. Here $h=2\cdot 10^{-3}$ and $\delta =2\cdot 10^{-12}$.}
  \label{fig_i3}%\label{fig:expvspredic}
\end{figure}
       The black box indicates the region where we count the number of eigenvalues to obtain the density of eigenvalues 
       presented in Figure \ref{fig_i3}. 
       Outside this box the influence from the boundary effects from our $N$-dimensional matrix are too 
       strong. Figure \ref{fig_i3} compares the experimental (given by counting the number of eigenvalues 
       in the black box restricted to $\Ima z\geq 0$ and averaging over $400$ realizations of random Gaussian 
       matrices) and the theoretical (cf Theorem \ref{thm:ModelFirstIntensity}) density and integrated density of eigenvalues.  
\\
%\pagebreak
\paragraph{Exponentially small (in $h$) coupling $\delta$} We set the above parameters to be $h=5\cdot 10^{-2}$, 
       $\delta=\exp(-1/h)$ and $N=1000$. Figure \ref{fig_i5} shows 
       the spectrum of $H+\delta R$ computed by MATLAB. 
       \begin{figure}[ht]
        \begin{minipage}[b]{0.49\linewidth}
         \centering
         \includegraphics[width=\textwidth]{fig13.pdf}
        \end{minipage}
        \hspace{0cm}
        \begin{minipage}[b]{0.49\linewidth}
         \centering 
         \includegraphics[width=\textwidth]{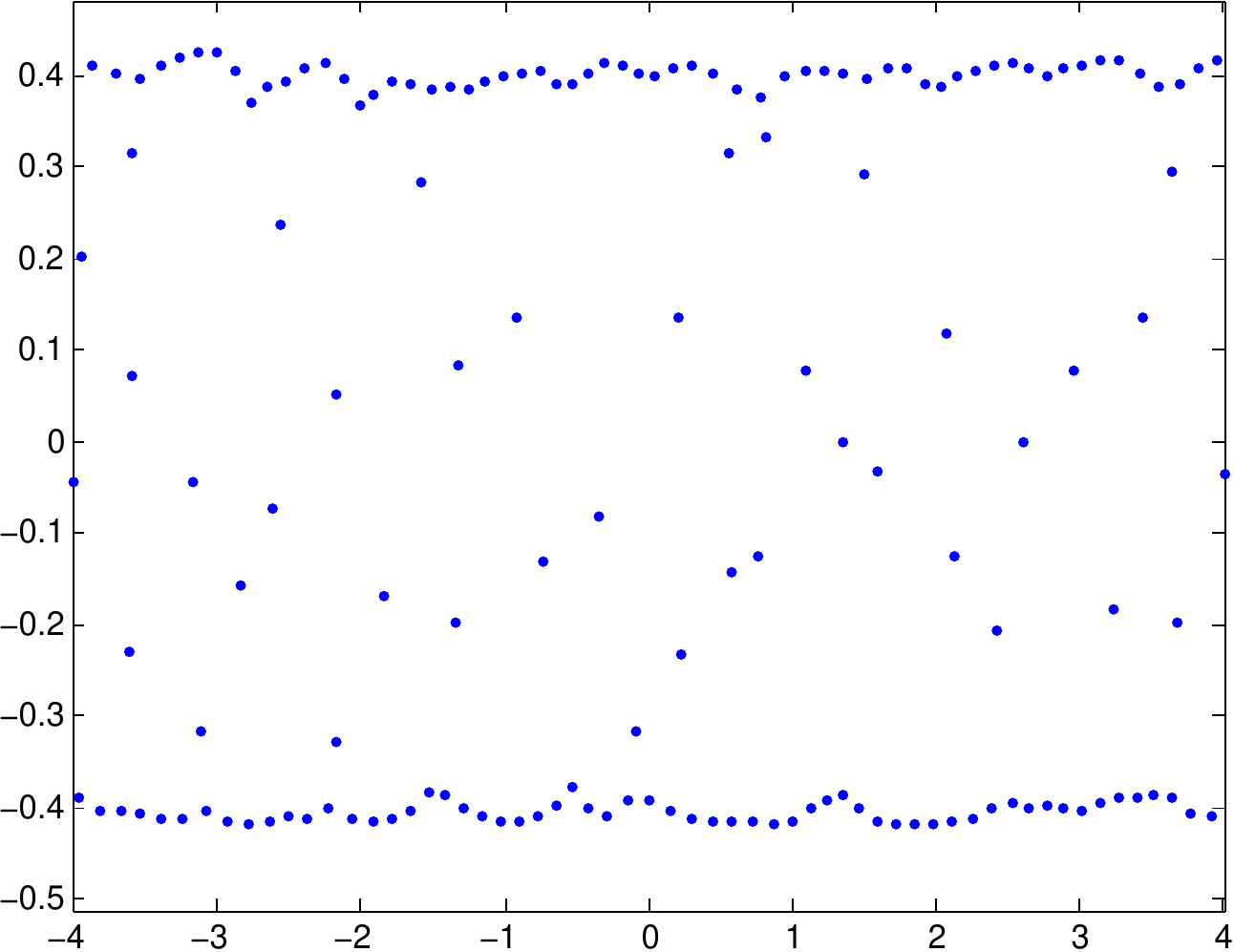}
        \end{minipage}
        \caption{On the left hand side we present the spectrum of the discretization of $hD + \exp(-ix)$ 
       	  (approximated by a $1999\times 1999$-matrix) perturbed with a random Gaussian matrix $\delta R$ with 
       	  $h=5\cdot 10^{-2}$ and $\delta=\exp(-1/h)$. The black box indicates the region where we count 
       	  the number of eigenvalues to obtain Figure \ref{fig_i4}. The right hand side is a magnification of the 
       	  central part of the spectrum depicted on the left hand side.}
        \label{fig_i5}%label{fig:PlotDensityAndPoints}
       \end{figure} 
       \begin{figure}[ht]
 \begin{minipage}[b]{0.49\linewidth}
  \centering
  \includegraphics[width=\textwidth]{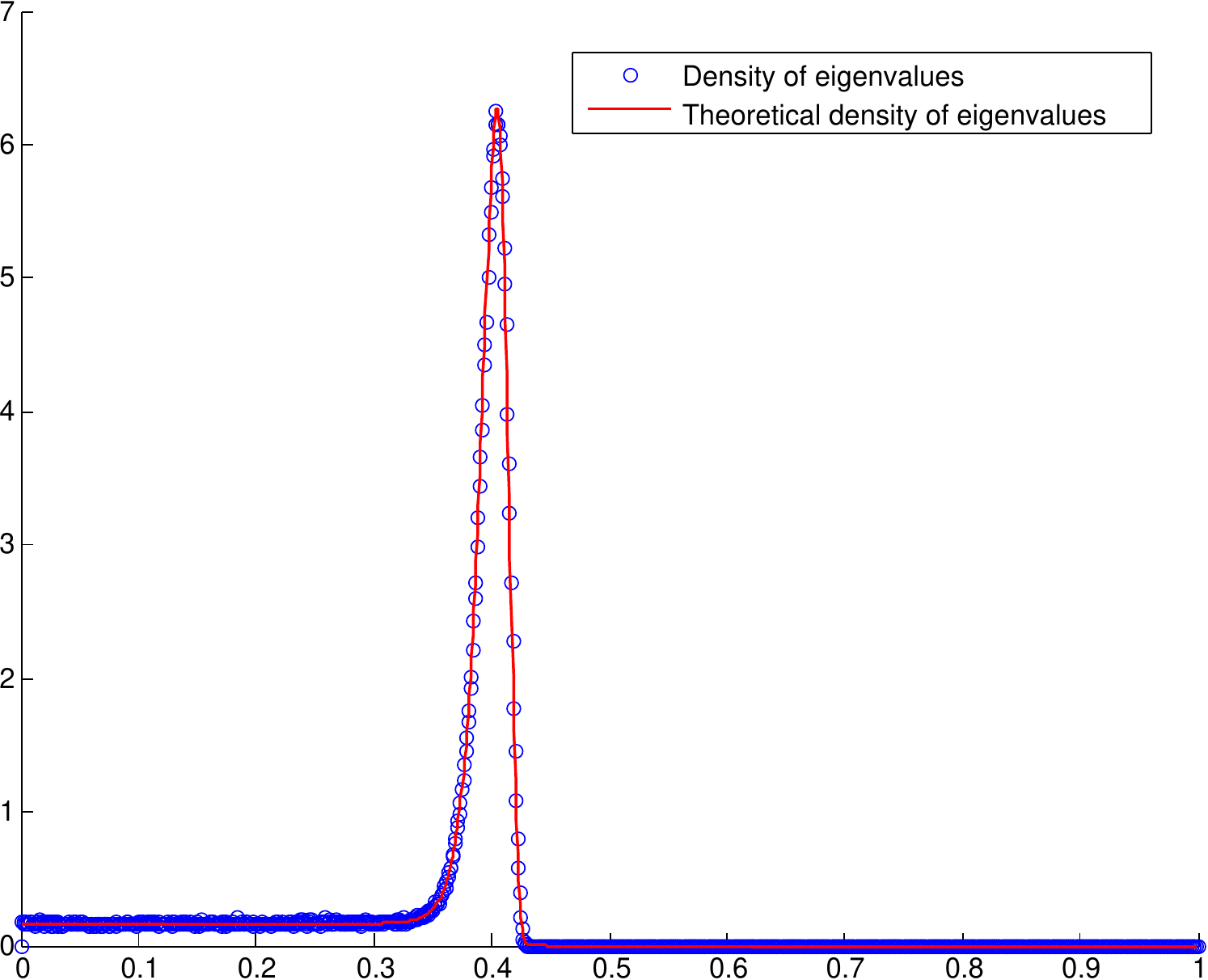}
 \end{minipage}
 \hspace{0cm}
 \begin{minipage}[b]{0.49\linewidth}
  \centering 
  \includegraphics[width=\textwidth]{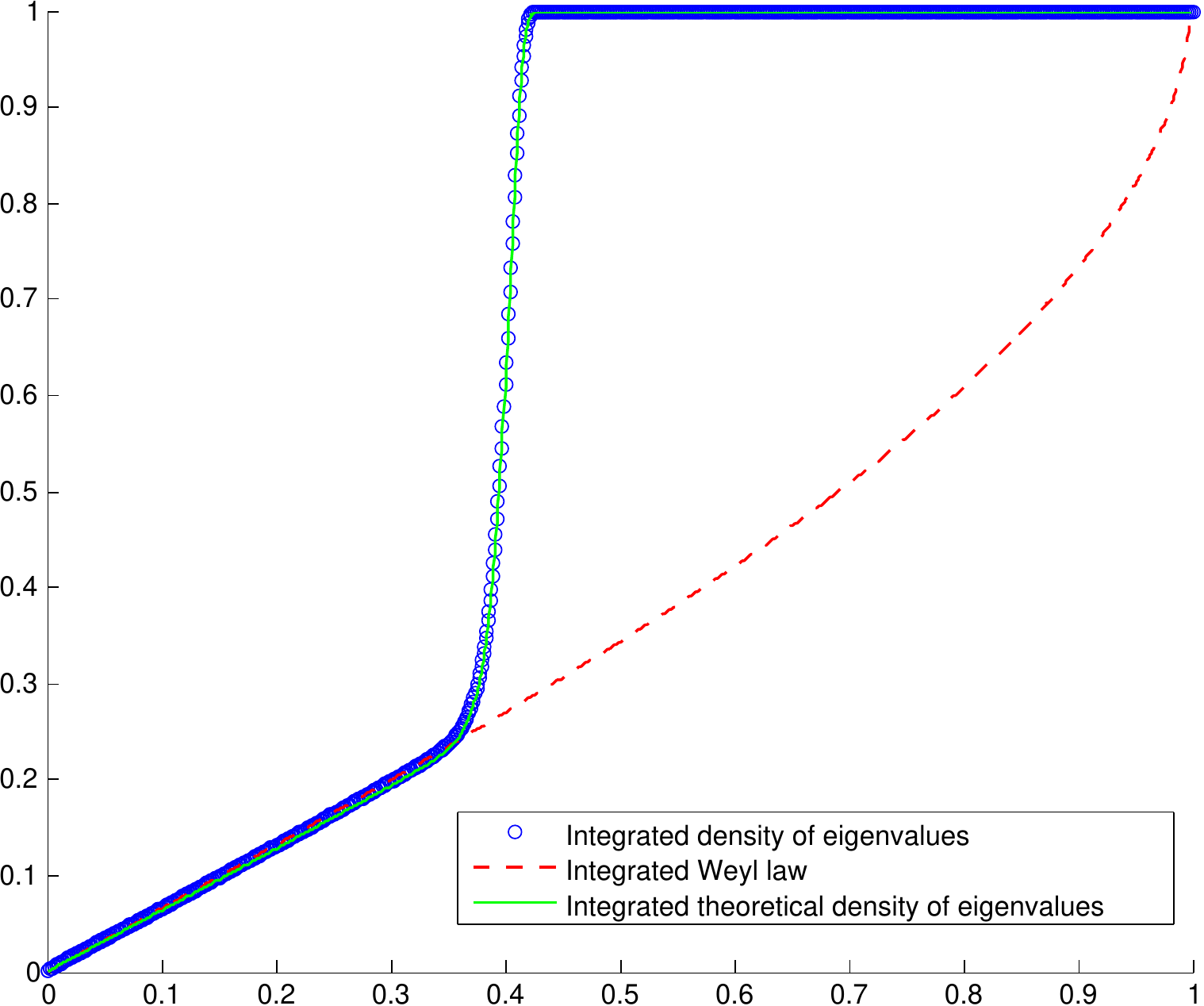}
 \end{minipage}
 \caption{Experimental (each point represents the mean, over $1000$ realizations, number 
	  of eigenvalues in a small box) vs predicted eigenvalue density (i.e. the principal 
	  terms of the average eigenvalue density given in Theorem \ref{thm:ModelFirstIntensity}) for 
	  $h=2\cdot 10^{-3}$ and $\delta =2\cdot 10^{-12}$.}
  \label{fig_i4}%\label{fig:expvspredic}
\end{figure}
       Similar to the above, the black box indicates the region where we count the 
       number of eigenvalues to obtain the density of eigenvalues presented in Figure
       \ref{fig_i4}. This figure compares the experimental (given by counting the number of eigenvalues 
       in the black box restricted to $\Ima z\geq 0$ and averaging over $400$ 
       realizations of random Gaussian matrices) and the theoretical 
       (cf Theorem \ref{thm:ModelFirstIntensity}) density and integrated density of eigenvalues.
\\
\par
The Figures \ref {fig_i31}, \ref{fig_i3}, \ref{fig_i5} and \ref{fig_i4} confirm the 
theoretical result presented in Theorem \ref{thm:ModelFirstIntensity} since the green lines, 
representing the plotted average density of eigenvalues given by Theorem 
\ref{thm:ModelFirstIntensity}, match perfectly the experimentally obtained density of eigenvalues. 
Furthermore, these figures show the three zones described in Section \ref{suse:MainResPropDens} 
(see also Proposition \ref{prop:DensityProperties}): 
\\
\par
The \textit{first zone}, is in the middle of the spectrum (cf. Figures \ref{fig_i31}, \ref{fig_i5}) corresponding 
to the zone where $\lVert (P_h-z)^{-1} \rVert \gg (h\delta)^{-1}$. There we see 
roughly an aequidistribution of points at distance $~\sqrt{h}$. The right hand side of Figures 
\ref{fig_i3} and \ref{fig_i4} shows that the number of eigenvalues in this zone is given 
by a \textit{Weyl law}, as predicted by Proposition \ref{prop:DensityProperties}.
\par
When comparing Figure \ref{fig_i3} and \ref{fig_i4} we can see clearly that the Weyl law breaks down earlier 
when the coupling constant $\delta$ gets smaller. Indeed, when $\delta>0$ is exponentially 
small in $h>0$, the break down happens well in the interior of $\Sigma$, precisely as predicted by Proposition 
\ref{prop:DensityProperties}.
\par
Another important property of this zone is that there is an increase in the density of the spectral 
points as we approach the boundary of $\Sigma$, see Figure \ref{fig_i3}. This is due to the 
fact that the density given by the Weyl law becomes more and more singular as we approach 
$\partial\Sigma$ (cf. Proposition \ref{prop:GrowPropDens}). 
\\
\par
We will find the \textit{second zone} by moving closer to the \textquotedblleft edge{\textquotedblright} 
of the spectrum, see Figure \ref{fig_i31} and \ref{fig_i5}. It can be characterized as the zone where 
$\lVert (P_h-z)^{-1} \rVert \asymp \delta^{-1}$. Figures \ref{fig_i3} and \ref{fig_i4} show that there 
is a strong accumulation of the spectrum close to the boundary of the pseudospectrum. Furthermore, we 
see in the image on the right hand side of Figure \ref{fig_i31} and of Figure \ref{fig_i5} that the zone 
of accumulation of eigenvalues is in a small tube around roughly a straight line. 
This is exactly as predicted by Proposition \ref{prop:the2curves} and Proposition \ref{prop:WedgePoisson}. 
Finally, let us remark that when looking at the Figures \ref{fig_i31} and \ref{fig_i5}, we note that in this zone the 
average distance between eigenvalues is much closer than in the first zone.
\\
\par
The \textit{third zone} is between the spectral edge and the boundary of $\Sigma$ where we find no spectrum at all. It can 
be characterized as the zone where $\lVert (P_h-z)^{-1} \rVert \ll \delta^{-1}$, a \textit{void region} as described in 
Proposition \ref{prop:DensityProperties} (cf. Figures \ref{fig_i3} and \ref{fig_i4}).
\par
Let us stress again that as $\delta$ gets smaller the zone of accumulation moves further into 
the interior of $\Sigma$, thus diminishing the zone determined by the Weyl law and increasing the zone 
void of eigenvalues. This effect is most drastic in the case of $\delta$ being exponentially small in $h$, 
see Figure \ref{fig_i4}.
\section{Quasimodes}\label{subse:Quasmod}
The purpose of this section is to construct quasimodes for 
$P_h - z$ for $z\in\Omega\Subset\Sigma$ with 
\begin{align}\label{hyp_Omega_a}
     &\Omega\Subset\Sigma ~\text{is open, relatively compact with}~
     \mathrm{dist\,}(\Omega,\partial\Sigma) > Ch^{2/3} \notag \\
     &\text{for some constant }C>0.
\end{align}
We will in particular always assume that this assumption on 
$\Omega\Subset\Sigma$ is satisfied, if nothing else is specified. 
\par
We make the distinction between the following two cases:
  \begin{description}
   \item[Quasimodes in the interior of $\Sigma$] We consider $z$ being in the \textit{interior} of $\Sigma$, i.e. 
      $z\in\Omega_i\Subset\mathring{\Sigma}$ such that there exists a 
   constant $C_{\Omega_i}>0$ such that 
	\begin{equation*}
	    \mathrm{dist\,}(\Omega_i,\partial\Sigma) > \frac{1}{C_{\Omega_i}}.
	\end{equation*}
       In this case, following the approach of Hager \cite{Ha06}, we can 
       find quasimodes by a WKB construction for the operator $(P_h - z)$; 
  \item[Quasimodes close to the boundary $\Sigma$] We consider $z$ being \textit{close} to the boundary of $\Sigma$, i.e. 
        $z\in\Omega\cap(\Omega_{\eta}^{a}\cup\Omega_{\eta}^{b})$ where, following the notation used in \cite{BM}, we define 
        for some constant $C>0$
	\begin{align}\label{def:etaOmega}
	  &\Omega_{\eta}^{a} := \left\{z\in\mathds{C}:~ \frac{\eta}{C} \leq \Ima z \leq C\eta \right\}, \notag \\
	  &\Omega_{\eta}^{b} := \left\{z\in\mathds{C}:~ \frac{\eta}{C} \leq (\Ima g(b) - \Ima z) \leq C\eta \right\},  
	\end{align}
    with $ h^{2/3} \ll \eta \leq \mathrm{const.}$ (recall from Section \ref{sec:HM} that $\Ima g(a) =0$). 
    The precise value of the above constant $C>0$ is not important 
    for the obtained asymptotic results. We will only consider the case $z\in\Omega_{\eta}^a$ since 
    $z\in\Omega_{\eta}^b$ can be treated the same way. We may follow the approach 
    of Bordeaux-Montrieux \cite{BM} and find quasimodes by a WKB construction for the rescaled operator 
    \begin{equation}\label{eqn:RescaledOp}
     \widetilde{P}_{\widetilde{h}} - \widetilde{z} 
     := \frac{h}{\eta^{3/2}}D_{\widetilde{x}} +\frac{g(\sqrt{\eta}\widetilde{x})}{\eta} - \frac{z}{\eta}
     :=\widetilde{h}D_{\widetilde{x}} +\widetilde{g}(\widetilde{x}) - \widetilde{z},
    \end{equation}
  with the rescaling
    \begin{equation*}
     S^1 \ni x = \sqrt{\eta}\widetilde{x} \quad \text{and} \quad \widetilde{h} := \frac{h}{\eta^{3/2}}.
    \end{equation*}
  Note that in this case demanding $\widetilde{h} \ll 1 $ implies the condition $ h^{2/3} \ll \eta$. 
  The rescaling is motivated by analyzing the Taylor expansion of $\Ima g(x)$ around the critical 
  point $a$ yielding that for $\Ima z \rightarrow 0$
    \begin{equation}\label{eqn:HTscale}
     |x_{\pm}(z) - a| \asymp \sqrt{\eta},
    \end{equation}
  where $x_{\pm}(z)$ are as in Section \ref{sec:Intro}. This shows that the rescaling shifts the problem 
  of constructing quasimodes for $z$ close to the boundary of $\Sigma$ to constructing quasimodes for $z$ 
  well in the interior of the range of the semiclassical principal symbol of the new operator 
  $\widetilde{P}_{\widetilde{h}}$. \par
  \begin{rem}
   Throughout this text we shall work with the convention that when writing an estimate, e.g. 
   $\mO\!\left(\delta^q \eta^{r} h^{s} \right)$ or $ A \asymp \eta^{r} h^{s}$, we implicitly set $\eta=1$ when 
   $\dist(z,\partial\Sigma) > 1/C$ but keep $\eta$ when $z\in\Omega_{\eta}^{a}$.  
   \end{rem}
   Let us note, that by Taylor expansion we may deduce that $S=S(z)$, as defined in Definition 
  \ref{defn:ZeroOrderDensity}, satisfies
    \begin{equation}\label{eqn:SEstINBox}
     S(z) \asymp \eta^{3/2}
    \end{equation}
  \end{description}
\subsection{Quasimodes for the interior of $\Sigma$}\label{suse:QuasmodInt}
\begin{defn}\label{def:Quasimodes}
Let $z\in\Omega_i \Subset \mathring{\Sigma}$ and let $x_-,x_+$ be as in the introduction. Let 
$\psi\in\mathcal{C}^{\infty}_0(\mathds{R})$ with $\supp\psi \subset ]0,1[$ and $\int\psi(x)dx = 1$.
Define $\chi_e\in\mathcal{C}^{\infty}_0(]x_--2\pi,x_-[)$ and $\chi_f\in\mathcal{C}^{\infty}_0(]x_+,x_++2\pi[)$ by
  \begin{align}\label{eqn:defCutOffWKB}
   \chi_e(x,z;h) &:=  \int_{-\infty}^xh^{-\frac{1}{2}}\left\{\psi\left(\frac{y-x_-+2\pi}{\sqrt{h}}\right)-\psi\left(\frac{x_--y}{\sqrt{h}}\right)\right\}dy, \notag \\
   \chi_f(x,z;h) &:=  \int_{-\infty}^xh^{-\frac{1}{2}}\left\{\psi\left(\frac{y-x_+}{\sqrt{h}}\right)-\psi\left(\frac{x_++2\pi-y}{\sqrt{h}}\right)\right\}dy.
  \end{align}
Furthermore, define for $x\in ]x_--2\pi,x_-[$
  \begin{equation*}
   \phi_+(x;z) := \int_{x_+}^x \left(z - g(y)\right) dy,
  \end{equation*}
and for $x\in ]x_+,x_++2\pi[$
  \begin{equation*}
    \phi_-(x;z) := \int_{x_-}^x \overline{\left(z - g(y)\right)} dy.
  \end{equation*}
\end{defn} 
Consider the $L^2(S^1)$-normalized quasimodes
  \begin{equation}\label{eqn:DefEWKB}
   e_{wkb}(x,z;h) := h^{-\frac{1}{4}}a(z;h)\chi_e(x,z;h)e^{\frac{i}{h}\phi_+(x;z)} 
   \in \mathcal{C}^{\infty}_0(]x_- -2\pi,x_-[)
  \end{equation}
and 
  \begin{equation}\label{eqn:DefFWKB}
   f_{wkb}(x,z;h) := h^{-\frac{1}{4}}b(z;h)\chi_f(x,z;h)e^{\frac{i}{h}\phi_-(x;z)} 
   \in  \mathcal{C}^{\infty}_0(]x_+,x_++2\pi[)
  \end{equation}
where $a(h;z)$ and $b(h;z)$ are normalization factors obtained by the stationary phase 
method. Thus, $a(h;z) \sim a_0(z) + ha_1(z) + \cdots \neq 0$ and $b(h;z) \sim b_0(z) + hb_1(z) + \cdots \neq 0$
depend smoothly on $z$ such that all derivatives with respect to $z$ and $\overline{z}$ are bounded when $h\rightarrow 0$. \par
The quasimodes $e_{wkb}$ and $f_{wkb}$ are WKB approximate null solutions to $(P_h-z)$ and $(P_h-z)^*$ where 
$a(z;h)$ and $b(z;h)$ are the asymptotic expansions of the normalization coefficients and it is 
easy to see that for all $\beta\in\N^2$
  \begin{equation}\label{d_coef1}
    \partial_{z\overline{z}}^{\beta}a(z,h), 
    \partial_{z\overline{z}}^{\beta}b(z,h)
    =
    \mO(h^{-|\beta|}).
  \end{equation}
\begin{lem}\label{lem:AssympA0B0BySPM}
   \begin{equation}\label{eqn:AssympA0B0}
   a_0 = \left(\frac{-\Ima g'(x_+)}{\pi}\right)^{\frac{1}{4}},~\text{and} ~ b_0 = \left(\frac{\Ima g'(x_-)}{\pi}\right)^{\frac{1}{4}}.
  \end{equation}
\end{lem}
\begin{proof}
We will show the proof only for $a_0^i$ since the statement for $b_0^i$ 
can be achieved by analogous steps. We are interested in the integral 
\begin{equation*}
   I_h:=h^{-\frac{1}{2}}\int \chi_e(x,z;h)^2e^{\frac{-\Phi(x;z)}{h}} dx,
  \end{equation*}
where
  \begin{equation*}
    i\phi_+(x;z) - i \overline{\phi_+(x;z)} = -2\Ima \int_{x_+(z)}^x (z-g(y))dy = -\Phi(x;z).
  \end{equation*}
On the support of $\chi_e$ the phase $\Phi(x;z)$ has the unique critical point $x=x_+(z)$ which is non-degenerate since 
$\partial^2_{xx}\Phi(x_+(z);z) = -2\Ima g'(x_+(z))> 0 $. Thus, the stationary phase method yields that 
    \begin{equation*}
    I_h = \left(\frac{\pi}{-\Ima g'(x_+(z))}\right)^{\frac{1}{2}}+   \mathcal{O}(h). \qedhere
  \end{equation*}
\end{proof} 
By the natural projection $\Pi:\mathds{R} \rightarrow S^1$ as in Section \ref{sec:Intro} 
we can identify 
  \begin{equation*}
   \mathcal{C}^{\infty}_0(]x_+,x_++2\pi[) = \{u\in\mathcal{C}^{\infty}(S^1):~ x_+\notin \supp u\}
  \end{equation*}
and 
 \begin{equation*}
   \mathcal{C}^{\infty}_0(]x_- -2\pi,x_-[) = \{u\in\mathcal{C}^{\infty}(S^1):~ x_-\notin \supp u\},
  \end{equation*}
with the slight abuse of notation that on the right hand side $x_{\pm}\in\mathds{R}$ and 
on the left hand side $x_{\pm}\in S^1$. This identification permits us to define 
$e_{wkb}(x,z;h), f_{wkb}(x,z;h)$ on $\mathcal{C}^{\infty}(S^1)$.

\subsection{Quasimodes close to the boundary of $\Sigma$}\label{suse:QuasmodCB}
Now let $z\in\Omega_{\eta}^{a}$. Following \cite{BM}, we 
shall construct quasimodes for the operator $P_h-z$ by looking at the rescaled 
operator $\widetilde{P}_{\widetilde{h}} - \widetilde{z}$ as defined in (\ref{eqn:RescaledOp}). \par 

Let us first note that $\frac{i}{h}\phi_+(x;z)$ and $\frac{i}{h}\phi_-(x;z)$ have the following 
behavior under the rescaling described at the beginning of this section: 
  \begin{equation}\label{eqn:phaseScaling}
   \frac{i}{h}\phi_+(x;z) = \frac{i}{h}\int_{x_+}^x \left(z - g(y)\right) dy 
   = \frac{i}{ \widetilde{h}}\int_{\widetilde{x}_+}^{\widetilde{x}} \left(\widetilde{z} - \widetilde{g}(\widetilde{y})\right) d\widetilde{y}
   =:\frac{i}{\widetilde{h}}\widetilde{\phi}_+(\widetilde{x};\widetilde{z})
  \end{equation}
and analogously for $\frac{i}{h}\phi_-(x;z)$. Taylor expansion shows us that the 
rescaled phase functions $\widetilde{\phi}_{\pm}(\widetilde{x};\widetilde{z})$ have for 
$z\in\Omega_{\eta}^{a}$ a non-degenerate critical point $\widetilde{x}_{\pm}(\widetilde{z})$ 
which satisfy the relation 
\begin{equation}\label{eq:relCrtiP}
 x_{\pm}(z) = \sqrt{\eta}\widetilde{x}_{\pm}(\widetilde{z}).
\end{equation}
It is easy to see that locally
   \begin{equation*}
    (\widetilde{P}_{\widetilde{h}} - \widetilde{z})\e^{\frac{i}{\widetilde{h}}\widetilde{\phi}_+(\widetilde{x};z)} = 0, 
  \end{equation*}
Thus, the natural choice of quasimodes for $z\in\Omega\cap\Omega_{\eta}^{a}$ in the rescaled 
variables is 
  \begin{prop}\label{prop:QuasmodCB}
 Let $\Omega\Subset\Sigma$, $z\in\Omega\cap\Omega_{\eta}^{a}$ and set $\widetilde{h}:=\frac{h}{\eta^{3/2}}$. Then 
 there exist functions
  \begin{equation*}
   a^{\eta}(\widetilde{h};\widetilde{z}) \sim a_0^{\eta}(\widetilde{z}) +\widetilde{h}a_1^{\eta}(\widetilde{z}) + \cdots \neq 0, \quad 
   b^{\eta}(\widetilde{h};\widetilde{z}) \sim b_0^{\eta}(\widetilde{z}) +\widetilde{h}b_1^{\eta}(\widetilde{z}) + \cdots \neq 0,
  \end{equation*}
depending smoothly on $\widetilde{z}$ such that all $\widetilde{z}$- and $\overline{\widetilde{z}}$-derivatives 
remain bounded as $h\rightarrow 0$ and $h^{\frac{2}{3}}< \eta \rightarrow 0$, such that 
  \begin{align*}
   &e_{wkb}^{\eta}(\widetilde{x},\widetilde{z};\widetilde{h}) 
   := \left(\widetilde{h}\eta\right)^{-\frac{1}{4}}a^{\eta}(\widetilde{z};\widetilde{h})
   \chi_e(\widetilde{x},\widetilde{z};\widetilde{h})e^{\frac{i}{\widetilde{h}}\widetilde{\phi}_+(\widetilde{x};\widetilde{z})} ~\text{and}~
   \notag \\
   &f_{wkb}^{\eta}(\widetilde{x},\widetilde{z};\widetilde{h}) 
   := \left(\widetilde{h}\eta\right)^{-\frac{1}{4}}b^{\eta}(\widetilde{z};\widetilde{h})
   \chi_f(\widetilde{x},\widetilde{z};\widetilde{h})e^{\frac{i}{\widetilde{h}}\widetilde{\phi}_-(\widetilde{x};\widetilde{z})},
  \end{align*}
where $\chi_{e,f}$ are as in Definition \ref{def:Quasimodes}, 
are $L^2(S^1/\sqrt{\eta},\sqrt{\eta}d\widetilde{x})$-normalized. Furthermore, 
  \begin{align*}
   &a_0^{\eta}(\widetilde{z})= \left(\frac{|\Ima g''(a)(\widetilde{x}_+(\widetilde{z})-a/\sqrt{\eta})(1+o(1))|}{\pi}\right)^{\frac{1}{4}}, \quad 
    z\in\Omega_{\eta}^{a}, \notag \\
   &b_0^{\eta}(\widetilde{z}) = 
    \left(\frac{|\Ima g''(a)(\widetilde{x}_-(\widetilde{z})-a/\sqrt{\eta})(1+o(1))|}{\pi}\right)^{\frac{1}{4}}, 
      \quad z\in\Omega_{\eta}^{a}.
  \end{align*}
  
\end{prop}
\begin{rem}
 In Proposition \ref{prop:QuasmodCB}, we stated the 
 Taylor expansion of the first order terms of $a^{\eta}(z;h)$ 
 and $b^{\eta}(z;h)$. However, note that we have 
  \begin{equation*}
   a_0(z)=\left( \frac{-\Ima g'(x_+(z))}{\pi}\right)^{\frac{1}{4}} 
   = 
   \eta^{\frac{1}{8}}a_0^{\eta}(\widetilde{z}),
  \end{equation*}
 where $a_0$ is the first order term of the normalization coefficient 
 $a$ of the quasimode $e_{wkb}$; see Lemma \ref{lem:AssympA0B0BySPM}.
 Similar for $b_0^{\eta}$. 
\end{rem}
\begin{proof}
We shall consider the proof only for the case of $e_{wkb}^{\eta}$ since the case of $f_{wkb}^{\eta}$ is the same. \par
By \eqref{eq:relCrtiP}, \eqref{eqn:defCutOffWKB} one computes that
  \begin{equation*}
   \chi_e(\sqrt{\eta}\widetilde{x},z;h/\eta^{1/2}) = 
   \chi_e(\widetilde{x},\widetilde{z};\widetilde{h})
  \end{equation*}
Consider $\lVert \chi_e(\cdot,z;h/\eta^{1/2})e^{\frac{i}{h}\phi_+(\cdot,z)}\rVert^2_{L^2(S^1)}$ and 
perform the change of variables $x =\sqrt{\eta}\widetilde{x}$. Hence, 
  \begin{equation}\label{BCquasmodeAsympExp}
   \int \chi_e\left(x,z;\frac{h}{\eta^{\frac{1}{2}}}\right)^2 e^{-\frac{2}{h}\Ima \phi_+(x,z)}dx = 
   \eta^{\frac{1}{2}}\int \chi_e\left(\widetilde{x},\widetilde{z};\widetilde{h}\right)^2 
   e^{-\frac{2}{\widetilde{h}}\Ima \int_{\widetilde{x}_{+}}^{\widetilde{x}}(\widetilde{z}-\widetilde{g}(\widetilde{y}))d\widetilde{y} }d\widetilde{x}.
  \end{equation}
The stationary phase method yields that 
(\ref{BCquasmodeAsympExp})$\sim \sqrt{\eta}\widetilde{h}^{\frac{1}{2}}(\widetilde{c}_0(\widetilde{z}) + \widetilde{h}\widetilde{c}_1(\widetilde{z}) + \dots)$, 
where the $\widetilde{c}_j(\widetilde{z})$ depend smoothly on $\widetilde{z}$ such that all $\widetilde{z}$- and $\overline{\widetilde{z}}$-derivatives 
remain bounded as $h\rightarrow 0$ and $h^{\frac{2}{3}}< \eta \rightarrow 0$. \par
On the other hand, the stationary phase method applied to $\lVert \chi_e e^{\frac{i}{h}\phi_+}\rVert^2$ 
(compare with Section \ref{suse:QuasmodInt}) yields that 
  \begin{equation*}
   \lVert \chi_e(\cdot,z;h)e^{\frac{i}{h}\phi_+(\cdot;z)}\rVert^2_{L^2(S^1)}\sim h^{\frac{1}{2}}(c_0(z) + hc_1(z) + \dots)
  \end{equation*}
with 
  \begin{equation*}
   c_0(z) = \left(\frac{\pi}{-\Ima g'(x_+(z))}\right)^{\frac{1}{2}}.
  \end{equation*}
Since $ \chi_e(x,z;h) \equiv \chi_e(x,z;h/\eta^{1/2})$ locally around $x_+(z)$, we may conclude that 
for all $k\in\mathds{N}_0$
  \begin{equation*}
   \widetilde{c}_k(\widetilde{z}) = \eta^{\frac{3k}{2}+\frac{1}{4}}c_j(z).
  \end{equation*}
In particular, the Taylor expansion around the critical point $a$ yields that 
  \begin{equation*}
   \widetilde{c}_0(\widetilde{z}) =\left(\frac{\pi}{|\Ima g''(a)(\widetilde{x}_+(\widetilde{z})-a/\sqrt{\eta})(1+o(1))|}\right)^{\frac{1}{2}}, \quad 
     ~ z\in\Omega_{\eta}^{a}.
   \end{equation*}
Thus, we conclude the statement of the proposition.
\end{proof}
Considering the above describe quasimodes in the original variable $x\in S^{1}$ leads to the following 
\begin{defn}\label{def:quasmodCB}
 Let $\Omega\Subset\Sigma$, $z\in\Omega\cap\Omega_{\eta}^{a}$ and set $\widetilde{h}:=\frac{h}{\eta^{3/2}}$. 
 Then define 
 \begin{align*}
   &e_{wkb}^{\eta}(x,z;h) 
   := \left(\frac{h}{\eta^{1/2}}\right)^{-\frac{1}{4}}a^{\eta}(\widetilde{z};\widetilde{h})
   \chi_e^{\eta}(x,z;h/\eta^{1/2})e^{\frac{i}{h}\phi_+(x;z)} ~\text{and}~
   \notag \\
   &f_{wkb}^{\eta}(x,z;h) 
   := \left(\frac{h}{\eta^{1/2}}\right)^{-\frac{1}{4}}b^{\eta}(\widetilde{z};\widetilde{h})
   \chi_f^{\eta}(x,z;h/\eta^{1/2})e^{\frac{i}{h}\phi_-(x;z)},
  \end{align*}
where $\chi_{e,f}^{\eta}(x,z;h/\eta^{1/2}) = \chi_{e,f}(x,z;h/\eta^{1/2})$. We choose this notation 
to make the distinctions between the two cases $z\in\Omega_i$ and $z\in\Omega_{\eta}^{a}$ more 
apparent. 
\end{defn}
\subsection{Approximation of the eigenfunctions of $Q(z)$ and $\widetilde{Q}(z)$} 
Recall $Q$ and $\widetilde{Q}$ given in Section \ref{sec:AuxOpe}. We will use the above 
defined quasimodes to prove estimates on the lowest eigenvalue of $Q$, $t_0^2$. Furthermore, 
we will give estimates on the approximation of the eigenfunctions $e_0$ and $f_0$ by 
the quasimodes $e_{wkb}$ and $f_{wkb}$.
We will prove an extended version of a result in \cite[Sec. 7.2 and 7.4]{SjAX1002}.
\begin{prop}\label{prop:QDistFirstEig}
 Let $z\in\Omega\Subset\Sigma$ and let $S=S(z)$ be defined as in Definition \ref{defn:ZeroOrderDensity}. 
 Then, for $h^{\frac{2}{3}} \ll \eta \leq C $
	  \begin{equation*}
	   t_0^2(z)\leq\mathcal{O}\!\left(\eta^{\frac{1}{2}}h\e^{-\frac{2S}{h}}\right).
	  \end{equation*}
	 Furthermore, there exists a constant $C>0$, uniform in $z$, such that 
	  \begin{equation*}
	   t_1^2-t_0^2\geq \frac{\eta^{\frac{1}{2}}h}{C}
	  \end{equation*}
	 for $h>0$ small enough.
\end{prop}
\begin{rem}
 The case $\dist(z,\partial\Sigma)>1/C$ has been proven in \cite[Sec. 7.1]{SjAX1002}. Since it will be useful
 further on we shall give a proof of the statement and indicate how to deduce the statement in 
 the case of $z\in\Omega\cap\Omega_{\eta}^{a}$.
\end{rem}
\begin{proof} 
Let us first suppose that $z\in\Omega_i$ (cf. Section \ref{subse:Quasmod}). 
Recall the definition of the self-adjoint 
operator $Q(z)$ given in (\ref{eqn:defnSAOpQ}) and define
  \begin{equation}\label{eqn:DefR}
    r:=r(x,z;h):=Q(z) e_{wkb}(x;z).
  \end{equation}
Recall, by (\ref{eqn:DefEWKB}), that $e_{wkb}(x;z)= h^{-\frac{1}{4}}a(z;h)\chi_e(x,z;h)e^{\frac{i}{h}\phi_+(x;z)}$. 
Since $x_-(z)$ is smooth in $z$ and all its $z$- and $\overline{z}$-derivatives are independent of 
$h$, it follows from (\ref{eqn:defCutOffWKB}) that for all $\alpha\in\mathds{N}^3\backslash\{0\}$ 
  \begin{equation}\label{eqn:higherzzbarderchi_e}
    \partial_{z\overline{z}x}^{\alpha} \chi_e(x,z;h) = \mathcal{O}\!\left(h^{-\frac{|\alpha|}{2}}\right),
  \end{equation}
with support in $X_-:=]x_- - 2\pi, x_- -2\pi+ h^{1/2}[\cup ]x_- - h^{1/2}, x_-[$.
One computes that 
  \begin{align}\label{eqn:QuasModRFullRepres}
   &(P_h-z)^*(P_h-z)e_{wkb}(x;z)=  \\
   &a(z;h) \frac{h^{\frac{3}{4}}}{i}
	    \left\{
		   \frac{h}{i}\partial^2_{xx}\chi_e(x,z;h)
		    + \partial_{x}\chi_e(x,z;h)\left(\partial_x\phi_+ +\overline{g(x)-z}\right)
	    \right\}e^{\frac{i}{h} \phi_+}.\notag
  \end{align}
where $\phi_+=\phi_+(x;z)$. Since for $x\in X_-$
  \begin{equation}\label{eqn:4_5_1}
   \partial_x\phi_+(x;z)+\overline{g(x)-z} = z - g(x) + \overline{g(x)-z} = -2i\Ima \left(g(x)-z\right)
    = \mathcal{O}\!\left(h^{\frac{1}{2}}\right),
  \end{equation}
it follows from (\ref{eqn:higherzzbarderchi_e}), (\ref{eqn:QuasModRFullRepres}) that
  \begin{equation}\label{eqn:QuasModProofrEstimate}
   r=Q(z) e_{wkb}(x;z) = \mathcal{O}\!\left(h^{\frac{3}{4}}\right)e^{\frac{i}{h} \phi_+(x;z)},
  \end{equation}
which has its support in $X_-$. Thus, one computes that 
 \begin{equation}\label{eqn:NormEstimateR}
   \lVert r \rVert ^2 = \mathcal{O}\!\left(h^{2}\e^{-\frac{2S}{h}}\right),
  \end{equation}
and, since $Q$ is self-adjoint, it follows that 
$t_0^2(z) =\mathcal{O}\!\left(h\e^{-\frac{2S(z)}{h}}\right)$.
\par
The proof of the desired statement about $t_1^2(z) - t_0^2(z)$ for $z\in\Omega_i$ 
can be found in the proof of Proposition 7.2 in \cite[Sec. 7.1]{SjAX1002}.\\ 
\par
Suppose now that $z\in\Omega\cap\Omega_{\eta}^{a}$. The desired statement follows by 
a rescaling argument. Recall (\ref{eqn:RescaledOp}) and, using the quasimodes $e_{wkb}^{\eta}(x;z)$, 
note that 
  \begin{equation*}
   t_0^2(Q(z)) = t_0^2(\eta^2 (\widetilde{P}_{\widetilde{h}} - \widetilde{z} )^*(\widetilde{P}_{\widetilde{h}} - \widetilde{z})) 
   = \mathcal{O}\left(\eta^{2}\widetilde{h} e^{-\frac{2\widetilde{S}(\widetilde{z})}{\widetilde{h}}}\right),
  \end{equation*}
where $\widetilde{S}$ is defined in the obvious way via $\widetilde{\phi}_+$ and 
  \begin{equation}\label{eqn:Sinvar}
   \frac{\widetilde{S}(\widetilde{z})}{\widetilde{h}} = \frac{S(z)}{h}.
  \end{equation}  
Hence, 
  \begin{equation}\label{eqn:lemQuasModrNorm_CB}
   t_0^2(z) = \mathcal{O}\left(h\eta^{1/2} e^{-\frac{2S(z)}{h}}\right).
  \end{equation}
The estimate on $t_1^2(z) - t_0^2(z)$ in the case $z\in\Omega\cap\Omega_{\eta}^{a}$ can be deduced 
as well by a rescaling argument: note that 
$t_1^2(Q(z)) = t_1^2(\eta^2 (\widetilde{P}_{\widetilde{h}} - \widetilde{z} )^*(\widetilde{P}_{\widetilde{h}} - \widetilde{z}))$. 
The statement then follows by performing the same steps of the proof of Proposition 7.2 in \cite[Sec. 7.1]{SjAX1002} in 
the rescaled space $L^2(S^1/\sqrt(\eta),\sqrt{\eta}d\widetilde{x})$ and using the quasimode $e_{wkb}^{\eta}(x;z)$ together 
with the estimate given in Proposition 4.3.5 in \cite{BM}.
\end{proof}
\begin{prop}\label{prop:E0T0Smooth}
Let $z\in\Omega \Subset\Sigma$. Then the eigenvalue $t_0^2(z)$ is a smooth function of $z$ and 
the eigenfunctions $e_0(z)$ and $f_0(z)$ can be chosen to have the same property. 
\end{prop}
\begin{proof}
Let us suppose first that $z\in\Omega_i$. The operator $Q(z)$ is bounded in 
$H^2(S^1)\rightarrow L^2(S^1)$ and in norm real-analytic in $z$ since for $z_0\in\Omega$ 
  \begin{equation}\label{eqn:expansionQinZ}
   Q(z) = Q(z_0) -(P-z_0)^*( z-z_0 ) - (P-z_0)(\overline{z-z_0}) + |z-z_0|^2.
  \end{equation}
Let $\zeta$ be in the resolvent set $\rho(Q(z))$ of $Q(z)$ and consider the resolvent
  \begin{equation*}
   R(\zeta,Q(z)):=(\zeta - Q(z))^{-1}.
  \end{equation*}
By \cite[II - \S 1.3]{Kato} we know that the resolvent depends locally analytically on the 
variables $\zeta$ and $z$. More precisely if $\zeta_0\notin\sigma(Q(z_0))$ for $z_0\in\Omega$ then 
$R(\zeta,Q(z))$ is holomorphic in $\zeta$ and real-analytic in 
$z$ in a small neighborhood of $\zeta_0$ and in a small neighborhood of $z_0$. 
\begin{rem}
 The proof in \cite[II - \S 1.3]{Kato} is given in the case of finite dimensional spaces. However, it can be extended directly 
 to bounded operators on Banach spaces. 
\end{rem}
By \cite[IV - \S 3.5]{Kato} we know that the simple eigenvalue $t_0^2(z)$ depends continuously on $Q(z)$. 
Thus, by Proposition \ref{prop:QDistFirstEig} and the continuity of $t_0^2(z)$ there exists, for $h>0$ small enough, 
a constant $D>0$ such that for all $z$ in a neighborhood of a point $z_0\in\Omega$ 
  \begin{equation*}
   t_1^2(z) > \frac{h}{D}.
  \end{equation*}
Define $\gamma$ to be the positively oriented circle of radius $h/(2D)$ centered at 
$0$ and consider the spectral projection of $Q(z)$ onto the eigenspace 
associated with $t_0^2(z)$
  \begin{equation*}
   \Pi_{t_0^2}(z) = \frac{1}{2\pi i} \int_{\gamma} R(\zeta,Q(z)) d\zeta.
  \end{equation*}
Since the resolvent $R(\zeta,Q(z))$ is smooth in $z$ it follows that $\Pi_{t_0^2}(z)$ is smooth in $z$. Now 
set $e(x,z)$ to be a smooth quasimode for $P_h-z$ for $z\in\Omega_i$ as in Section \ref{subse:Quasmod} which 
depends smoothly on $z$. Thus, by setting 
  \begin{equation*}
   e_0(x,z,h) = \frac{\Pi_{t_0^2}(z)e_{wkb}(x,z,h)}{\lVert \Pi_{t_0^2}(z)e_{wkb}(-,z,h)\rVert},
  \end{equation*}
we deduce that also $e_0(x,z)$ depends smoothly on $z$. 
The statement for $f_0(z)$ follows by performing the same argument for $\widetilde{Q}(z)$ instead of $Q(z)$ and 
with the quasimode $f_{wkb}$. \par 
Using that $\Pi_{t_0^2}(z)$ and $Q(z)$ are smooth and that the operator $\Pi_{t_0^2}Q\Pi_{t_0^2}$ has finite rank 
we see by 
  \begin{equation*}
   t_0^2(z) = \tr\left(\Pi_{t_0^2}(z)Q(z)\Pi_{t_0^2}(z)\right)
  \end{equation*}
that $t_0^2(z)$ is smooth. \\ \par
In the case of $z\in\Omega\cap\Omega_{\eta}^{a}$ for $h^{2/3} < \eta < \text{const.}$ 
we follow the exact same steps as above, mutandi mutandis. We take the estimate 
$t_1^2(z) > \frac{h\sqrt{\eta}}{D}$ for $z$ in a neighborhood of a fixed $z_0\in\Omega\cap\Omega_{\eta}^{a}$ (following from 
Proposition \ref{prop:QDistFirstEig}) and thus we pick, as above, $\widetilde{\gamma}$ to be the positively 
oriented circle of radius $h\sqrt{\eta}/(2D)$ centered at $0$. Hence, for $z\in\Omega\cap\Omega_{\eta}^{a,b}$
  \begin{equation*}
    \Pi_{t_0^2}(z) = \frac{1}{2\pi i} \int_{\widetilde{\gamma}} R(\zeta,Q(z)) d\zeta, \quad 
     e_0(x,z,h) = \frac{\Pi_{t_0^2}(z)e_{wkb}^{\eta}(x,z,h)}{\lVert \Pi_{t_0^2}(z)e_{wkb}^{\eta}(-,z,h)\rVert}.
  \end{equation*}
Following the same arguments as above we conclude the statement of the proposition also 
in the case of $z\in\Omega\cap\Omega_{\eta}^{a}$.
\end{proof}

%%%%%%% Estimates on the quasimodes 
\begin{prop}\label{quasimodes}
Let $z\in\Omega \Subset\Sigma$ and let $e_0$ and $f_0$ be the eigenfunctions of the operators $Q$ and $\widetilde{Q}$ with 
respect to their smallest eigenvalue (as in Section \ref{suse:GrushForUnPert}). Let $S=S(z)$ be defined as in 
Definition \ref{defn:ZeroOrderDensity}. Then 
  \begin{itemize}
   \item for $z\in\Omega$ with $\dist(\Omega,\partial\Sigma)>1/C$ 
	 and for all $\beta\in\mathds{N}^2$ %{lem:QuasModEstimate1}
	  \begin{equation}\label{lem:QuasModEstimate2}
	      \lVert\partial_{z\overline{z}}^{\beta}( e_0 - e_{wkb}) \rVert, ~
 	      \lVert \partial_{z\overline{z}}^{\beta}( f_0 - f_{wkb}) \rVert 
 	      = \mathcal{O}\!\left(h^{-|\beta|}e^{-\frac{S}{h}}\right).
 	    \end{equation}
 	  Furthermore, the various $z$- and $\overline{z}$-derivatives of 
	  $e_0$, $f_0$, $e_{wkb}$ and $f_{wkb}$ have at most temperate growth in $1/h$, more precisely 
	  for all $\beta\in\mathds{N}^2$
	    \begin{equation}\label{lem:QuasModEstimate3}
	         \lVert \partial_{z\overline{z}}^{\beta} e_{wkb} \rVert, ~
	         \lVert \partial_{z\overline{z}}^{\beta} f_{wkb} \rVert, ~
	         \lVert \partial_{z\overline{z}}^{\beta} e_0 \rVert, ~ \lVert \partial_{z\overline{z}}^{\beta} f_0 \rVert 
	  	 =\mathcal{O}\!\left(h^{-|\beta|}\right);
  	    \end{equation}
   \item for $h^{2/3} \ll \eta < const.$, $z\in\Omega\cap\Omega_{\eta}^{a}$ and for all $\beta\in\mathds{N}^2$
	  \begin{equation}\label{lem:QuasModEstimate2_CB} 
	      \lVert \partial_{z\overline{z}}^{\beta} (e_0 -  e_{wkb}^{\eta}) \rVert, ~
 	      \lVert \partial_{z\overline{z}}^{\beta} (f_0 -  f_{wkb}^{\eta}) \rVert 
		  = \mathcal{O}\!\left(\eta^{\frac{|\beta|}{2}}h^{-|\beta|}e^{-\frac{S}{h}}\right).
 	    \end{equation}
 	  Furthermore, the various $z$- and $\overline{z}$-derivatives of 
	  $e_0$, $f_0$, $e_{wkb}^{\eta}$ and $f_{wkb}^{\eta}$ have at most temperate growth in $\sqrt{\eta}/h$, more precisely
	    \begin{equation}\label{lem:QuasModEstimate3_CB}
	         \lVert \partial_{z\overline{z}}^{\beta} e_{wkb}^{\eta} \rVert, ~
	         \lVert \partial_{z\overline{z}}^{\beta} f_{wkb}^{\eta}  \rVert, ~
	         \lVert \partial_{z\overline{z}}^{\beta} e_0 \rVert, ~ \lVert \partial_{z\overline{z}}^{\beta} f_0 \rVert 
	  	 =\mathcal{O}\!\left(\eta^{\frac{|\beta|}{2}}h^{-|\beta|}\right)
  	    \end{equation}
	  for all $\beta\in\mathds{N}^2$.
\end{itemize}
\end{prop}
\begin{rem}
Let us recall that
  \begin{itemize}
   \item for $z\in\Omega \Subset \mathring{\Sigma}$, in the case where $\Omega$ is independent of $h>0$ and has 
         a positive distance to the boundary of $\Sigma$ we have $1/C \leq S \leq C$ for some constant $C>0$. 
         Thus, we may formulate the corresponding estimates of Proposition \ref{quasimodes} uniformly in $z$;
  \item for $h^{2/3} \ll \eta < const.$ and $z\in\Omega\cap\Omega_{\eta}^{a}$ (\ref{eqn:SEstINBox}) 
        implies estimates uniform in $z$ but $\eta$ dependent. 
  \end{itemize}
\end{rem}
\noindent This implies the following
\begin{cor}\label{QMEstimUnifZ}
 Under the assumptions of Proposition \ref{quasimodes}, 
\begin{itemize}
   \item for $z\in\Omega_i$ there exists a constant $C>0$ such that for all $\beta\in\mathds{N}^2$ 
	      \begin{equation}\label{cor:QuasModEstimate2}
	    \lVert \partial_{z\overline{z}}^{\beta}( e_0 - e_{wkb}) \rVert, ~
 	      \lVert \partial_{z\overline{z}}^{\beta}( f_0 - f_{wkb}) \rVert 
 	      = \mathcal{O}\!\left(h^{-|\beta|}\e^{-\frac{1}{Ch}}\right);
	    \end{equation}
   \item for $h^{2/3}\ll \eta < const.$, $z\in\Omega\cap\Omega_{\eta}^{a,b}$ and for all $\beta\in\mathds{N}^2$ 
	      \begin{equation}\label{cor:QuasModEstimate2_CB}
	     \lVert \partial_{z\overline{z}}^{\beta} (e_0 -  e_{wkb}^{\eta}) \rVert, ~
 	      \lVert \partial_{z\overline{z}}^{\beta} (f_0 -  f_{wkb}^{\eta}) \rVert 
		  = \mathcal{O}\!\left(\eta^{\frac{|\beta|}{2}}h^{-|\beta|}\e^{-\frac{\asymp\eta^{3/2}}{h}}\right).
	    \end{equation}
\end{itemize}
\end{cor}
\begin{rem}
 The proof of Proposition \ref{quasimodes} is unfortunately somewhat long 
 and technical and we have split it into several lemmas. Furthermore, we 
 will only be discussing the results for $e_{wkb}(z)$, $e_{wkb}^{\eta}(z)$ and $e_0(z)$, 
 since the others can be obtained similarly.
\end{rem}
\begin{lem}\label{lem:QuasmodProof1}
 Let $\Omega\Subset\Sigma$ such that $\dist(\Omega,\partial\Sigma)>1/C$. For $z\in\Omega$ define 
 $r:=r(x,z;h):=Q(z) e_{wkb}(x;z)$ as in (\ref{eqn:DefR}). Then, for all $\beta\in\mathds{N}^2$, 
 $\supp \partial_{z\overline{z}}^{\beta} r \subset ]x_- - 2\pi, x_- -2\pi+ h^{1/2}[\cup ]x_- - h^{1/2}, x_-[$ 
 and
  \begin{equation*}
  \lVert \partial_{z\overline{z}}^{\beta}r \rVert  
		      = \mathcal{O}\!\left(h^{1-|\beta|}\e^{-\frac{S}{h}}\right).
  \end{equation*}
\end{lem}
\begin{proof}
Using (\ref{eqn:higherzzbarderchi_e}), (\ref{eqn:QuasModRFullRepres}) we conclude 
by the Leibniz rule that for $\beta\in\mathds{N}^2$
  \begin{equation*}
   \partial_{z\overline{z}}^{\beta}r = 
   \mathcal{O}\!\left(h^{\frac{3}{4}-|\beta|}\right)e^{\frac{i}{h} \phi_+(x;z)}
  \end{equation*}
which is supported in $]x_- - 2\pi, x_- -2\pi+ h^{1/2}[\cup ]x_- - h^{1/2}, x_-[$ and 
one computes that $\lVert \partial_{z\overline{z}}^{\beta}r \rVert ^2 = 
\mathcal{O}\!\left(h^{2-2|\beta|}\e^{-\frac{2S}{h}}\right)$.
\end{proof}
\begin{lem}\label{lem:QuasmodProof2}
Let $\Omega\Subset\Sigma$ such that $\dist(\Omega,\partial\Sigma)>1/C$ and let $z\in\Omega$. Moreover, 
let $\Pi_{t_0^2}:~ L^2(S^1)\rightarrow \mathds{C} e_0$ denote 
the spectral projection of $Q(z)$ onto the eigenspace associated with $t_0^2$. Then, 
\begin{equation*}
   \lVert \partial_{z\overline{z}}^{\beta} \Pi_{t_0^2}(z) \rVert_{L^2\rightarrow H_{sc}^2 }
   = \mathcal{O}\!\left({h }^{-\frac{|\beta|}{2}}\right).
  \end{equation*}
\end{lem}
\begin{proof}
By virtue of Proposition \ref{prop:QDistFirstEig} and the continuity of $t_0^2(z)$ there exists for $h>0$ small enough a 
constant $D>0$ such that for all $z$ in a neighborhood of a point $z_0\in\Omega$ 
  \begin{equation*}
   t_1^2(z) > \frac{h}{D}.
  \end{equation*}
Let $\gamma$ be the positively oriented circle of radius $h/(2D)$ centered at $0$. Note 
that $\gamma$ is locally independent of $z$. Thus, we gain a path such that 
$0,t_1^2(z)\notin \gamma$ and which has length $|\gamma| = h\pi/D $. For $\lambda\in\gamma$ we have that
  \begin{equation}\label{eqn:ResEstimateSAQ}
   \lVert (\lambda - Q(z))^{-1} \rVert= \frac{1}{\mathrm{dist\,}(\lambda,\sigma(Q(z)))} = \mathcal{O}({|\gamma| }^{-1}).
  \end{equation}
By (\ref{eqn:expansionQinZ}) and the resolvent identity we see that the 
derivatives $\partial_{z}^n\partial_{\overline{z}}^m (\lambda -Q(z))^{-1}$, 
for $(n,m)\in\mathds{N}^2\backslash\{0\}$, are finite linear combinations of terms of the form 
  \begin{equation}\label{eq_dq1}
   (\lambda -Q(z))^{-1}\partial_{z\overline{z}}^{\alpha_1}(Q(z))(\lambda -Q(z))^{-1} \cdots 
   \partial_{z\overline{z}}^{\alpha_k}(Q(z))(\lambda -Q(z))^{-1}
  \end{equation}
with $\alpha_j =(1,0),(0,1),(1,1)$ and $\alpha_1 + \dots + \alpha_k = (n,m)$. Thus it is sufficient to estimate 
the terms of the form $(P_h-z)(Q(z)-\lambda)^{-1}$ and $(P_h-z)^*(Q(z)-\lambda)^{-1}$. 
Since $Q(z)=(P_h-z)^*(P_h-z)$, it follows that
  \begin{equation}\label{eqn:EstNormQ}
   \lVert(P_h-z)u\rVert^2 - {|\gamma| }\lVert u\rVert^2 \leq \left|((Q(z)-\lambda)u|u)\right| 
   \leq \lVert (Q(z)-\lambda)u \rVert \lVert u \rVert.
  \end{equation}
Since $Q(z)>0$ is self-adjoint and since $\dist (\lambda, \sigma(Q(z)))\asymp {|\gamma| }$ we have the a priori estimate
  \begin{equation*}
   \lVert (Q(z) -\lambda)u\rVert \geq C{|\gamma| }\lVert u\rVert
  \end{equation*}
for all $u\in H_{sc}^2(S^1)$, where $C>0$ is a constant locally uniform in $z$. This implies 
  \begin{align*}
   \lVert(P_h-z)u\rVert^2 &\leq \left(\lVert (Q(z)-\lambda)u \rVert + {|\gamma| }\lVert u\rVert \right)\lVert u \rVert \notag \\
    & \leq \widetilde{C} \lVert (Q(z)-\lambda)u \rVert \lVert u \rVert 
    \leq \frac{C}{|\gamma| }\lVert (Q(z)-\lambda)u \rVert^2,
  \end{align*}
where $C>0$ is a constant uniform in $z$. Hence	
  \begin{equation*}
   \lVert (P_h-z)(Q(z)-\lambda)^{-1}\rVert_{L^2\rightarrow L^2} = \mathcal{O}\!\left({|\gamma| }^{-\frac{1}{2}}\right).
  \end{equation*}
Finally, note that since $[P_h^*,P_h]=\mathcal{O}_{H^2_{sc}\rightarrow L^2}(h)$ we can replace 
$P_h$ by it's adjoint in (\ref{eqn:EstNormQ}) and gain the estimate
   \begin{equation*}
   \lVert (P_h-z)^*(Q(z)-\lambda)^{-1}\rVert_{L^2\rightarrow L^2} = \mathcal{O}\!\left({|\gamma| }^{-\frac{1}{2}}\right).
  \end{equation*}
Using \eqref{eq_dq1} and the fact that $|\gamma|=h\pi/D$ 
we have that for all $\beta\in\mathds{N}^2\backslash\{0\}$ 
  \begin{equation}\label{eqn:EstDerResQ}
   \lVert \partial_{z\overline{z}}^{\beta} (\lambda -Q(z))^{-1}\rVert_{L^2\rightarrow H_{sc}^2 }
   = \mathcal{O}\!\left({h }^{-\frac{|\beta|+2}{2}}\right).
  \end{equation}
Since for $u\in L^2(S^1)$
 \begin{equation*}
   \frac{1}{2\pi i}\int_{\gamma}(\lambda - Q(z))^{-1}ud\lambda  = \Pi_{t_0^2}u,
  \end{equation*}
(\ref{eqn:EstDerResQ}) implies 
  \begin{equation*}
   \lVert \partial_{z\overline{z}}^{\beta} \Pi_{t_0^2}(z) \rVert_{L^2\rightarrow H_{sc}^2 }
   =  \mathcal{O}\!\left({h }^{-\frac{|\beta|}{2}}\right). \qedhere
  \end{equation*}
\end{proof}
%Lemma 3 for the long proof %
\begin{lem}\label{lem:QuasmodProof3}
 Under the assumptions of Lemma \ref{lem:QuasmodProof2} we have 
  \begin{equation*}
   \lVert \partial_{z\overline{z}}^{\beta} e_{wkb}(\cdot;z)\rVert, 
   \lVert \partial_{z\overline{z}}^{\beta} \Pi_{t_0^2}e_{wkb}(\cdot;z)\rVert
   = \mathcal{O}\!\left(h^{-|\beta|}\right).
  \end{equation*}
\end{lem}
\begin{proof}
 Using \eqref{eqn:DefEWKB} and the triangular inequality, we get
    \begin{align*}
     \lVert \partial_z e_{wkb}(\cdot;z)\rVert &\leq  h^{-\frac{1}{4}} 
  	    \lVert \partial_z\chi_e(\cdot;z)a^i(z;h)\e^{\frac{i}{h}\phi_+(\cdot;z)}\rVert  \notag \\
	    & + h^{-\frac{1}{4}}\lVert\chi_e(\cdot;z)\partial_za^i(z;h)\e^{\frac{i}{h}\phi_+(\cdot;z)}\rVert \notag \\
   	    & +h^{-\frac{1}{4}} \lVert \chi_e(\cdot;z)a^i(z;h)ih^{-1}
	       \partial_z\phi_+(\cdot;z)\e^{\frac{i}{h}\phi_+(\cdot;z)}\rVert.
    \end{align*}
Recalling from (\ref{eqn:higherzzbarderchi_e}) that $\partial_z \chi_e(x,z;h) = \mathcal{O}(h^{-1/2})$ is supported in  
$]x_- - 2\pi, x_- -2\pi+ h^{1/2}[\cup ]x_- - h^{1/2}, x_-[$, one computes 
  \begin{equation*}
   h^{-\frac{1}{4}} 
  	    \lVert \partial_z\chi_e(\cdot;z)a^i(z;h)\e^{\frac{i}{h}\phi_+(\cdot;z)}\rVert = \mathcal{O}\!\left(h^{-\frac{1}{2}}
  	    \e^{-\frac{S}{h}}\right).
  \end{equation*}
Using \eqref{d_coef1}, the stationary phase method implies
  \begin{equation*}
    h^{-\frac{1}{4}}\lVert\chi_e(\cdot;z)\partial_za^i(z;h)\e^{\frac{i}{h}\phi_+(\cdot;z)}\rVert = \mathcal{O}(h^{-1}). 
  \end{equation*}
Furthermore, since 
  \begin{equation}\label{eqn:DerzOfPhasePhi+}
    \partial_z \phi_+(x;z) = \int_{x_+(z)}^x dy - \xi_+(z)\partial_z x_+(z)
  \end{equation}
it follows by the stationary phase method that 
  \begin{equation*}
   h^{-\frac{1}{4}} \lVert \chi_e(\cdot;z)a^i(z;h)\frac{i}{h}\partial_z
		    \phi_+(\cdot;z)\e^{\frac{i}{h}\phi_+(\cdot;z)}\rVert = 
    \frac{1}{h}\left|\xi_+(z)\partial_z x_+(z)\right| + \mathcal{O}(1).
  \end{equation*}
Hence, by putting all of the above together
  \begin{equation*}
   \lVert \partial_z e_{wkb}(\cdot;z)\rVert = \mathcal{O}\!\left(h^{-1}\right).
  \end{equation*}
Similarly, using \eqref{d_coef1}, (\ref{eqn:higherzzbarderchi_e}), 
the stationary phase method implies 
  \begin{equation*}
   \lVert \partial_{z\overline{z}}^{\beta} e_{wkb}(\cdot;z)\rVert 
   = \mathcal{O}\!\left(h^{-|\beta|}\right).
  \end{equation*}
Lemma \ref{lem:QuasmodProof2} then implies by the Leibniz rule that 
  \begin{equation*}
     \lVert \partial_{z\overline{z}}^{\beta}\Pi_{t_0^2} e_{wkb} \rVert 
     = \mathcal{O}\!\left(h^{-|\beta|}\right).\qedhere
  \end{equation*}
\end{proof}
\begin{proof}[Proof of Proposition \ref{quasimodes}] 
\textit{Part I }- First, suppose that 
$z\in\Omega_i$. Let $r$ be as in Lemma \ref{lem:QuasmodProof1} 
and consider for $\lambda\in\mathds{C}$
  \begin{equation*}
   (\lambda - Q(z))e_{wkb}= \lambda e_{wkb}- r. 
  \end{equation*}
If $\lambda \notin \sigma(Q(z))\cup\{0\}$ we have 
  \begin{equation*}
   (\lambda - Q(z))^{-1}e_{wkb}= \frac{1}{\lambda} e_{wkb}+ \frac{1}{\lambda}(\lambda - Q(z))^{-1}r. 
  \end{equation*}
As in the proof of Lemma \ref{lem:QuasmodProof1}, define $\gamma$ to be the positively 
oriented circle of radius $h/(2D)$ centered at $0$. $\gamma$ is locally independent of $z$. 
Thus, we gain a path such that $0,t_1^2(z)\notin \gamma$ and which has length $|\gamma| = h\pi/D $. 
Hence
  \begin{equation}\label{eqn:QuasModIntProjRep}
   \frac{1}{2\pi i}\int_{\gamma}(\lambda - Q(z))^{-1}e_{wkb}d\lambda  
   = e_{wkb}
   + \frac{1}{2\pi i}\int_{\gamma}\frac{1}{\lambda}(\lambda - Q(z))^{-1}r d\lambda.
  \end{equation} 
By Lemma \ref{lem:QuasmodProof1}, (\ref{eqn:lemQuasModrNorm_CB}) and (\ref{eqn:ResEstimateSAQ})
  \begin{equation*}
  \left\lVert\frac{1}{2\pi i}\int_{\gamma}\frac{1}{\lambda}(\lambda - Q(z))^{-1}r d\lambda \right\rVert
   = \mathcal{O}\!\left(e^{-\frac{S}{h}}\right)
  \end{equation*}
By (\ref{eqn:QuasModIntProjRep})
  \begin{equation}\label{eqn:ProjSolWKBOnE0}
  \lVert \Pi_{t_0^2}e_{wkb}-  e_{wkb}\rVert = \mathcal{O}\!\left(e^{-\frac{S}{h}}\right).  
  \end{equation}
Recall that $e_{wkb}$ is normalized. Pythagoras' theorem then implies
  \begin{equation}\label{eqn:QuasModPythagoras}
   \lVert \Pi_{t_0^2}e_{wkb}\rVert^2 = \lVert e_{wkb}\rVert^2 - \lVert e_{wkb} - \Pi_{t_0^2}e_{wkb}\rVert^2 
   = 1 - \mathcal{O}\!\left(e^{-\frac{2S}{h}}\right)
  \end{equation}
which yields
  \begin{equation}\label{eqn:QuasModE0Repres}
   e_0 = \frac{1}{\lVert \Pi_{t_0^2}e_{wkb}\rVert}\Pi_{t_0^2} e_{wkb}
       = \left(1 + \mathcal{O}\!\left(e^{-\frac{2S}{h}}\right)\right)\Pi_{t_0^2}e_{wkb}.
  \end{equation}
Let us now turn to the $z$- and $\overline{z}$-derivatives of $e_0 - e_{wkb}$. 
By (\ref{eqn:QuasModE0Repres})
  \begin{align*}
   \left\lVert \partial_{z\overline{z}}^{\beta} (e_0(z) - e_{wkb}(z)) \right\rVert=
   \left\lVert \partial_{z\overline{z}}^{\beta}  
   \left(\frac{(\Pi_{t_0^2}-1)e_{wkb}
	  +(1-\lVert \Pi_{t_0^2}e_{wkb}\rVert)e_{wkb}}
	  {\lVert \Pi_{t_0^2}e_{wkb}(z)\rVert}\right) \right\rVert.
  \end{align*}
First, note that Lemma \ref{lem:QuasModEstimate3} together with (\ref{eqn:QuasModPythagoras}) implies 
  \begin{equation*}
    \partial_{z\overline{z}}^{\beta}\lVert  \Pi_{t_0^2} e_{wkb} \rVert 
    = \mathcal{O}\!\left(h^{-|\beta|}\right).
  \end{equation*}
Using this result and (\ref{eqn:QuasModPythagoras}) implies by the Leibniz rule 
applied to (\ref{eqn:QuasModE0Repres}) that 
  \begin{equation*}
    \lVert \partial_{z\overline{z}}^{\beta} e_{0} \rVert 
    = \mathcal{O}\!\left(h^{-|\beta|}\right).
  \end{equation*}  
Next, applying Lemma \ref{lem:QuasmodProof1} and (\ref {eqn:EstDerResQ}) to (\ref{eqn:QuasModIntProjRep}) 
yields
  \begin{equation*}
  \left\lVert\partial_{z\overline{z}}^{\beta}(\Pi_{t_0^2}-1)e_{wkb}\right\rVert
   =
  \left\lVert\partial_{z\overline{z}}^{\beta}\frac{1}{2\pi i}
	  \int_{\gamma}\frac{1}{\lambda}(\lambda - Q(z))^{-1}rd\lambda \right\rVert
   =  
   \mathcal{O}\!\left(h^{-|\beta|}e^{-\frac{S}{h}}\right).
  \end{equation*}
Thus, Lemma \ref{lem:QuasModEstimate3} and (\ref{eqn:QuasModPythagoras}) together with the Leibniz rule then imply 
  \begin{equation*}
  \left\lVert \partial_{z\overline{z}}^{\beta} (e_0(z) - e_{wkb}(z)) \right\rVert
   = \mathcal{O}\!\left(h^{-|\beta|}e^{-\frac{S}{h}}\right).
   \end{equation*}
\textit{Part II }- Now, let $z\in\Omega_{\eta}^a$ with $h^{\frac{2}{3}} \ll \eta < \mathrm{const.}$ The 
statements of the proposition follow from a simple rescaling argument. For the rescaling we 
use the same notation as in the beginning of Section \ref{subse:Quasmod}. Let $\widetilde{e}_0(\widetilde{z})$ be 
the $L^2(S^1/\sqrt{\eta},d\widetilde{x})$-normalized eigenfunction of the operator 
$\widetilde{Q}(\widetilde{z})=(\widetilde{P}_{\widetilde{h}} - \widetilde{z} )^*(\widetilde{P}_{\widetilde{h}} 
- \widetilde{z})$ and note that $\eta^{\frac{1}{4}}e_{wkb}^{\eta}$ is 
$L^2(S^1/\sqrt{\eta},d\widetilde{x})$-normalized. Thus, 
  \begin{equation*}
  \left\lVert \partial_{z\overline{z}}^{\beta}(\widetilde{e}_0(\widetilde{z}) 
      - e_{wkb}^{\eta}(\cdot,\widetilde{z},\widetilde{h}) \right\rVert_{L^2(S^1/\sqrt{\eta},d\widetilde{x})}
   = \mathcal{O}\!\left(\widetilde{h}^{-|\beta|}e^{-\frac{\widetilde{S}}{\widetilde{h}}}\right),
   \end{equation*}
where $\widetilde{S}$ is as in (\ref{eqn:Sinvar}). Since $e_0(z) = \eta^{-1/4}\widetilde{e}_0(\widetilde{z})$, 
it follows by rescaling that 
  \begin{equation*}
  \left\lVert \partial_{z\overline{z}}^{\beta} (e_0(z) - e_{wkb}^{\eta}(z)) \right\rVert_{L^2(S^1,dx)}
   = \mathcal{O}\!\left(\eta^{\frac{|\beta|}{2}}h^{-|\beta|}e^{-\frac{S}{h}}\right).
   \end{equation*}
The results on $\lVert \partial_{z\overline{z}}^{\beta} e_{wkb}^{\eta}\rVert$ and on 
$ \lVert \partial_{z\overline{z}}^{\beta} e_0 \rVert$ can be proven by the same 
rescaling argument.
\end{proof}
%
% Chapter 4
%
\section{Grushin problem for the unperturbed operator $P_h$}\label{se:GrushProb}
To start with we give a short refresher on Grushin problems since 
they have become an essential tool in Microlocal Analysis and it is a key 
method to the present work. As reviewed in \cite{SjZwor07}, the central idea 
is to set up an auxiliary problem of the form 
\begin{equation*}
 \begin{pmatrix}
  P(z) & R_- \\ 
  R_+ & 0 \\
 \end{pmatrix}
 :
 \mathcal{H}_1\oplus \mathcal{H}_- 
 \longrightarrow \mathcal{H}_2\oplus \mathcal{H}_+,
\end{equation*}
where $P(z)$ is the operator of interest and $R_{\pm}$ are 
suitably chosen. We say that the Grushin problem is well-posed 
if this matrix of operators is bijective. If 
$\dim\mathcal{H}_-  = \dim\mathcal{H}_+ < \infty$, one usually  
writes 
\begin{equation*}
 \begin{pmatrix}
  P(z) & R_- \\ 
  R_+ & 0 \\
 \end{pmatrix}^{-1}
 =
 \begin{pmatrix}
  E(z) & E_{+}(z) \\ 
  E_{-}(z) & E_{-+}(z) \\
 \end{pmatrix}.
\end{equation*}
The key observation, going back to the Shur complement formula or 
equivalently the Lyapunov-Schmidt bifurcation method, 
is that the operator $P(z): \mathcal{H}_1 \rightarrow \mathcal{H}_2$ 
is invertible if and only if the finite dimensional matrix 
$E_{-+}(z)$ is invertible and when $E_{-+}(z)$ is invertible, 
we have 
\begin{equation*}
  P^{-1}(z) = E(z) - E_{+}(z) E_{-+}^{-1}(z) E_{-}(z).
\end{equation*}
$E_{-+}(z)$ is sometimes called effective Hamiltonian. 
\\
\par
The principal aim of this section is to introduce the three different 
Grushin Problems needed to study $P_h^{\delta}$: one valid in all 
of $\Sigma$ which is however less explicit (here we will follow the 
construction given in \cite[Sec. 7.2 and 7.4]{SjAX1002}), and two very explicit 
Grushin Problems, one valid in the interior of $\Sigma$ and one valid 
close to $\partial\Sigma$ (here we will recall the construction given by Hager 
in \cite{Ha06} respectively Bordeaux-Montrieux in \cite{BM}).
\subsection{Grushin problem valid in all of $\Sigma$}\label{suse:GrushForUnPert}
Following the ideas of \cite{SjAX1002}, we will use the eigenfunctions $e_0$ and $f_0$ 
to set up the Grushin problem 
\begin{prop}\label{propGrushUnpertOp}
Let $\Omega \Subset\Sigma$ be open and relatively compact and let $\alpha_0$ 
be as in (\ref{eqn:RelationEjFj}). Define for $z\in\Omega$
  \begin{align}\label{eqn:defnR+R-}
    R_+:~ &H^1(S^1)\longrightarrow\mathds{C}:~ u  \longmapsto (u|e_0) \notag \\
    R_-:~&\mathds{C}\longrightarrow L^2(S^1):~ u_-  \longmapsto u_-f_0.  
  \end{align}
Then 
  \begin{equation*}
   \mathcal{P}(z):=\begin{pmatrix}
                    P_h-z & R_- \\ R_+ & 0 \\
                   \end{pmatrix}
   :~ H^1(S^1)\times \mathds{C}\longrightarrow L^2(S^1)\times \mathds{C}
  \end{equation*}
is bijective with the bounded inverse 
  \begin{equation*}
   \mathcal{E}(z) = \begin{pmatrix}
                     E(z) & E_+(z) \\ E_-(z) & E_{-+}(z) \\
                    \end{pmatrix}
  \end{equation*}
where $E_-(z)v = (v|f_0)$, $E_+(z)v_+ = v_+e_0$ 
and $E(z)=(P_h-z)^{-1}|_{(f_0)^{\perp}\rightarrow (e_0)^{\perp}}$ and $E_{-+}(z)v_+ = -\alpha_0 v_+$. 
Furthermore, we have the estimates 
  \begin{itemize}
   \item for $z\in\Omega$ with $\dist(\Omega,\partial\Sigma)>1/C$
    \begin{align}\label{prop:GrushUnPert}
      & \lVert E_-(z)\rVert_{L^2\rightarrow \mathds{C}}, \lVert E_+(z)\rVert_{\mathds{C}\rightarrow H^1 } 
      = \mO(1), \notag\\
      &\lVert E(z)\rVert_{L^2\rightarrow H^1} = \mathcal{O}(h^{-1/2}), \notag \\
      &|E_{-+}(z)| = \mathcal{O}\!\left(\sqrt{h}\e^{-\frac{S}{h}}\right)
                   =\mathcal{O}\!\left(\e^{-\frac{1}{Ch}}\right);
    \end{align}
   \item for $z\in\Omega\cap\Omega_{\eta}^{a}$ with $h^{\frac{2}{3}} \ll \eta < \mathrm{const.}$
    \begin{align}\label{prop:GrushUnPert2}
      & \lVert E_-(z)\rVert_{L^2\rightarrow \mathds{C}}, \lVert E_+(z)\rVert_{\mathds{C}\rightarrow H^1 } 
      = \mO(1), \notag\\
      &\lVert E(z)\rVert_{L^2\rightarrow H^1} = \mathcal{O}((h\sqrt{\eta})^{-1/2}), \notag \\
      &|E_{-+}(z)| = \mathcal{O}\!\left(\sqrt{h}\eta^{\frac{1}{4}}\e^{-\frac{S}{h}}\right)
                   =\mathcal{O}\!\left(\e^{-\frac{\asymp \eta^{3/2}}{h}}\right).
    \end{align}
  \end{itemize}
\end{prop}
\begin{proof}
 For a proof of the existence of the bounded inverse as well as the estimate for 
 $\lVert E(z)\rVert_{L^2\rightarrow H^1}$ in the case of $\dist(\Omega,\partial\Sigma)>1/C$
 see \cite[Section 7.2]{SjAX1002}. \par 
 The other estimate for $\lVert E(z)\rVert_{L^2\rightarrow H^1}$ can be proven 
 by performing the same steps as in the case of $\dist(\Omega,\partial\Sigma)>1/C$, mutandi mutandis, 
 together with the estimate given by Bordeaux-Montrieux in \cite[Proposition 4.3.5]{BM}.
 The estimates for $|E_{-+}(z)|$ follow from Proposition \ref{prop:QDistFirstEig}, whereas 
 the estimates on $\lVert E_-(z)\rVert_{L^2\rightarrow\mathds{C}}$ and 
 $\lVert E_+(z)\rVert_{\mathds{C}\rightarrow H^1}$ come from the fact that $e_0$ and $f_0$ are 
 normalized.\par
 Alternatively, one can conclude the result in the case of $z\in\Omega\cap\Omega_{\eta}^{a}$ by 
 a rescaling argument similar to the one in the proof of Proposition \ref{quasimodes}.
\end{proof}

\subsection{Tunneling}\label{susec:tunnel}
We prove now the following formula for a tunnel effect from which we conclude 
Proposition \ref{prop:TunnelEffect}.
\begin{prop}\label{prop:AbsSPe0f0}
 Let $z\in\Omega \Subset\Sigma$ and let $e_0$ and $f_0$ be as in (\ref{def:e0,e1,...}) 
 and in (\ref{def:f0,f1,...}). Furthermore, let $\Phi(z,h)$ be as in Proposition \ref{prop:Resolvent}, 
 let $S$ be as in Definition \ref{defn:ZeroOrderDensity} and let $p$ and $\rho_{\pm}$ be as in Section 
 \ref{sec:Intro}. Let $h^{\frac{2}{3}} \ll \eta < \mathrm{const.}$ Then, 
 for all $z\in\Omega$ with $|\Ima z - \langle \Ima g\rangle|> 1/C $, $C\gg 1$, 
\begin{align*}
   |(e_0|f_0)| 
   = \frac{\left(\frac{i}{2}\{p,\overline{p}\}(\rho_+)
	\frac{i}{2}\{\overline{p},p\}(\rho_-)\right)^{\frac{1}{4}}}{\sqrt{\pi h}}
    |\partial_{\Ima z}S(z)|\left(1 + \mathcal{O}\!\left(\eta^{-\frac{3}{4}}h^{\frac{1}{2}}\right) \right)
    \e^{-\frac{S}{h}}
  \end{align*}
  where for all $\beta\in\mathds{N}^2$
  \begin{align*}
   \partial_{z\overline{z}}^{\beta} 
   \mathcal{O}\!\left(\eta^{-3/4}h^{\frac{1}{2}}\right)
    = \mathcal{O}\!\left(\eta^{\frac{|\beta|}{2}-\frac{3}{4}}h^{-|\beta|+\frac{1}{2}}\right).
  \end{align*}
\end{prop}
This implies Proposition \ref{prop:TunnelEffect}. Furthermore, Proposition \ref{prop:AbsSPe0f0} 
implies by direct calculation the following result:
\begin{prop}\label{prop:EstimAndFormulaDerzE-+_pt2}
 Under the assumptions of Proposition \ref{prop:AbsSPe0f0} we have 
 for $h^{\frac{2}{3}} \ll \eta < const.$ 
  \begin{align*}
    \partial_{\Ima z}|(e_0|f_0)|^2 
   & = \frac{2\left(\frac{i}{2}\{p,\overline{p}\}(\rho_+)\frac{i}{2}\{\overline{p},p\}(\rho_-)\right)^{\frac{1}{2}}}{\pi h^2}
   |\partial_{\Ima z}S(z)|^2(-\partial_{\Ima z}S(z))
   \e^{-\frac{2S}{h}} \notag \\
   &+\mathcal{O}\!\left(\eta^{5/4}h^{-\frac{3}{2}}\e^{-\frac{2S}{h}}\right),
  \end{align*}
  \begin{align*}
    \partial_{\Rea z}|(e_0|f_0)|^2 , \partial_{\Rea z}\partial_{\Ima z}|(e_0|f_0)|^2
    = \mO\!\left(\e^{-\frac{1}{Ch}}\e^{-\frac{2S}{h}}\right).
  \end{align*}
\end{prop}
\begin{rem}
 \label{rem:MoreExplSP13}
 Let us point out that we can find an even more detailed 
 formula for $|(e_0|f_0)|$ (cf. (\ref{eqn:formSPe0f013})) 
 valid even for $|\Ima z - \langle \Ima g\rangle| \geq 1/C $:
  \begin{align*}
   |(e_0|f_0)| =
   &\frac{\left(\frac{i}{2}\{p,\overline{p}\}(\rho_+)
	       \frac{i}{2}\{\overline{p},p\}(\rho_-)\right)^{\frac{1}{4}}}{\sqrt{\pi h}}
   \e^{-\frac{S}{h}}|\partial_{\Ima z}S|
   \left(1 + \frac{2\pi - |\partial_{\Ima z}S|}{|\partial_{\Ima z}S|}\e^{\Rea\Phi}\right) \notag \\
   & + \mathcal{O}\!\left(\e^{-\frac{S}{h}}\right)
   + \mathcal{O}\!\left(\eta^{3/4}h^{-\frac{1}{2}}\e^{-\frac{S}{h}+\Rea\Phi}\right)
  \end{align*}
\end{rem}
\begin{proof}[Proof of Proposition \ref{prop:AbsSPe0f0}]
First, suppose that $z\in\Omega$ with $\dist(\Omega,\partial\Sigma)> 1/C$. 
Then, by Proposition \ref{quasimodes}
  \begin{align}\label{eqn:PfSPe0f0Estim}
   (e_0|f_0) = (e_0|f_{wkb}) + \mathcal{O}\!\left(\e^{-\frac{S}{h}}\right) 
	     = (e_{wkb}|f_{wkb}) + \mathcal{O}\!\left(\e^{-\frac{S}{h}}\right).
  \end{align}
Recall the definition of the quasimodes $e_{wkb}$ and $f_{wkb}$ from Section \ref{subse:Quasmod}. 
Moreover, recall from Section \ref{sec:Intro} that by the natural projection 
$\Pi:\mathds{R}\rightarrow S^1$ we identify $S^1$ with the interval $[x_-(z) -2\pi, x_-(z)[$. 
This choice leads to the fact that $\phi_+$ is given by 
  \begin{equation*}
     \phi_+(x) =  \int_{x_+(z)}^{x}(z-g(y))dy
  \end{equation*}
on this interval, whereas $\phi_-$ is given by 
  \begin{equation*}
   \phi_-(x) = \begin{cases}
       \displaystyle{
       \int_{x_-(z)}^{x}\overline{(z-g(y))}dy, \quad \text{for }x\in [x_+(z),x_-(z)[}, \\ 
        \displaystyle{
        \int_{x_-(z)-2\pi}^{x}\overline{(z-g(y))}dy, \quad \text{for }x\in [x_-(z)-2\pi,x_+(z)[.}
      \end{cases}
  \end{equation*}
Define 
   \begin{equation}\label{defn:R}
   R:= \frac{a\overline{b}}{\sqrt{h}}
                = \frac{\left(\frac{i}{2}\{p,\overline{p}\}(\rho_+)\frac{i}{2}\{\overline{p},p\}(\rho_-)\right)^{\frac{1}{4}}}{\sqrt{\pi}}
		  + \mathcal{O}(\sqrt{h}),
  \end{equation}
where we used Lemma \ref{lem:AssympA0B0BySPM}, Proposition \ref{prop:QuasmodCB} and (\ref{eqn:TaylorPoissonBracket}) 
to gain the equality. 
A straight forward computation yields that 
  \begin{align}\label{eqn:SPewkbfwkb}
  (e_{wkb}|f_{wkb}) 
   &=R\e^{\frac{i}{h}\int_{x_+(z)}^{x_-(z)-2\pi}(z-g(y))dy}
		  \int_{x_-(z)-2\pi}^{x_+(z)}\chi_e(x)\chi_f(x) dx \notag \\
   &\phantom{=.}+R\e^{\frac{i}{h}\int_{x_+(z)}^{x_-(z)}(z-g(y))dy} 
   \int_{x_+(z)}^{x_-(z)}\chi_e(x)\chi_f(x)dx.
  \end{align}
Using (\ref{eqn:defCutOffWKB}) and Definition \ref{def:quasmodCB}, we have that 
  \begin{align}\label{eqn:DeltaLength}
   \int_{x_-(z)-2\pi}^{x_+(z)}\chi_e(x)\chi_f(x) dx &= x_+(z) - (x_-(z)-2\pi) \notag \\
      &\hspace{-1cm}-\int_{x_--2\pi}^{x_--2\pi+\sqrt{h}}(1-\chi_e(x))dx - \int_{x_++2\pi-\sqrt{h}}^{x_++2\pi}(1-\chi_f(x))dx \notag \\
		&= x_+(z) - (x_-(z)-2\pi) +\mathcal{O}\!\left(\sqrt{h}\right), \notag \\
   \end{align}
and similarly 
   \begin{align}\label{eqn:DeltaLength_2}
   \int_{x_+(z)}^{x_-(z)}\chi_e(x)\chi_f(x)dx = x_-(z) - x_+(z) + \mathcal{O}\!\left(\sqrt{h}\right).
  \end{align}
Now let us assume that we are below the spectral line of $P_h$, i.e. $\Ima z \leq \langle \Ima g\rangle$. 
There, we see that 
  \begin{align*}
   |(e_{wkb}|f_{wkb})|
  &= R\e^{-\frac{1}{h}\Ima\int_{x_+(z)}^{x_-(z)}(z-g(y))dy}
      \Big|(x_-(z) -x_+(z)) +\mathcal{O}\!\left(\sqrt{h}\right) 
      \notag \\
     &+ \left(x_+(z) - (x_-(z)-2\pi)+\mathcal{O}\!\left(\sqrt{h}\right) \right)
       \e^{-\frac{2\pi i}{h}( z - \langle g\rangle)}\Big|.
  \end{align*}
Analogously, if we are above the spectral line, i.e. $\Ima z \geq \langle \Ima g\rangle$, 
  \begin{align*}
   |(e_{wkb}|f_{wkb})|
  &= R\e^{-\frac{1}{h}\Ima\int_{x_+(z)}^{x_-(z)-2\pi}(z-g(y))dy}
      \Big|(x_+(z) - (x_+(z)-2\pi)) 
      \notag \\
     &+\mathcal{O}\!\left(\sqrt{h}\right) + \left(x_-(z) -x_+(z)+\mathcal{O}\!\left(\sqrt{h}\right) \right)
       \e^{\frac{2\pi i}{h}( z - \langle g\rangle)}\Big|.
  \end{align*}
Together with (\ref{eqn:PfSPe0f0Estim}), we conclude that 
  \begin{align}\label{eqn:formSPe0f013}
   |(e_0|f_0)| =
   &\frac{\left(\frac{i}{2}\{p,\overline{p}\}(\rho_+)
	       \frac{i}{2}\{\overline{p},p\}(\rho_-)\right)^{\frac{1}{4}}}{\sqrt{\pi h}}
   \e^{-\frac{S}{h}}|\partial_{\Ima z}S|
   \left(1 + \frac{2\pi - |\partial_{\Ima z}S|}{|\partial_{\Ima z}S|}\e^{\Rea\Phi}\right) \notag \\
   & + \mathcal{O}\!\left(\e^{-\frac{S}{h}}\right)
   + \mathcal{O}\!\left(\eta^{3/4}h^{-\frac{1}{2}}\e^{-\frac{S}{h}+\Rea\Phi}\right)
  \end{align}
where $\Phi=\Phi(z,h)$ is as in Proposition \ref{prop:Resolvent}.  
Note that $\exp\left\{\Phi(z,h)\right\}$ is exponentially 
small for $|\Ima z - \langle \Ima g \rangle| > 1/C$. Thus, 
  \begin{align}\label{eqn:formSPe0f0}
   |(e_0|f_0)| =
   \frac{\left(\frac{i}{2}\{p,\overline{p}\}(\rho_+)
	       \frac{i}{2}\{\overline{p},p\}(\rho_-)\right)^{\frac{1}{4}}}{\sqrt{\pi h}}
   \e^{-\frac{S}{h}}|\partial_{\Ima z}S(z)|
   \left(1 + \mathcal{O}\!\left(\eta^{-3/4}h^{\frac{1}{2}}\right)\right).
  \end{align}
Now let us discuss the $\partial_{z\overline{z}}^{\beta}$-derivatives of the 
errors. First let us treat the error term $\mathcal{O}\!\left(\sqrt{h}\right)$ from 
the definition of $R$ which is given as a product of the normalization coefficients of 
the quasimodes $e_{wkb}$ and $f_{wkb}$. Thus, it is easy to see that 
  \begin{align}\label{eqn:estimErrVDer}
   \partial_{z\overline{z}}^{\beta} \mathcal{O}\!\left(\sqrt{h}\right)
    = \mathcal{O}\!\left(h^{-(|\beta|-1/2)}\right).
  \end{align}
The $\partial_{z\overline{z}}^{\beta}$-derivatives of the error term in (\ref{eqn:DeltaLength}), 
\eqref{eqn:DeltaLength_2} can be treated as follows: note that  
  \begin{align*}
   &\partial_z\int_{x_--2\pi}^{x_--2\pi+\sqrt{h}}(1-\chi_e(x;z))dx  = \notag \\
   &\left(\chi_e(x_--2\pi;z) - \chi_e(x_--2\pi+\sqrt{h};z)\right)\partial_zx_-
   -\int_{x_--2\pi}^{x_--2\pi+\sqrt{h}}\partial_z\chi_e(x;z)dx.
  \end{align*}
By (\ref{eqn:higherzzbarderchi_e})
  \begin{align*}
   \int_{x_--2\pi}^{x_--2\pi+\sqrt{h}}\partial_z\chi_e(x;z)dx &= 
   - \int_{x_--2\pi}^{x_--2\pi+\sqrt{h}}\psi\left(\frac{x-x_-+2\pi}{\sqrt{h}}\right)\partial_z x_-(z)dx 
   \notag \\
   &= - \partial_z x_-(z).
  \end{align*}
Since $\chi_e(x_--2\pi;z)=0$ and $\chi_e(x_--2\pi+\sqrt{h};z)=1$,
\begin{equation*}
   \partial_z\int_{x_--2\pi}^{x_--2\pi+\sqrt{h}}(1-\chi_e(x;z))dx =0.
  \end{equation*}
\eqref{eqn:DeltaLength_2} as well as the respective $\overline{z}$-derivatives can be treated 
analogously, and we conclude that $\partial_{z\overline{z}}^{\beta} \mathcal{O}(\sqrt{h}) = 0$ 
for all $\beta\in\mathds{N}^2\backslash\{0\}$. Hence, we have
\begin{align*}
   \partial_z^n\partial_{\overline{z}}^m 
   \mathcal{O}\!\left(\eta^{-3/4}h^{\frac{1}{2}}\right)
    = \mathcal{O}\!\left(\eta^{\frac{|\beta|}{2}-\frac{3}{4}}h^{-|\beta|+\frac{1}{2}}\right).
  \end{align*}
Finally, in the case where $z\in\Omega\cap\Omega_{\eta}^{a}$ we can conclude the statement by a 
rescaling argument similar as in the proof of Proposition \ref{quasimodes}.
\end{proof}
\begin{rem}\label{rem:derTunnel}
 It is a direct consequence of (\ref{eqn:SPewkbfwkb}), 
 (\ref{eqn:PfSPe0f0Estim}) and Proposition \ref{quasimodes} 
 that 
  \begin{equation*}
   \partial_{z\overline{z}}^{\beta}(e_0|f_0) 
    =  \mO\left(\eta^{\frac{|\beta|+3/2}{2}}h^{-(|\beta|+1/2)}\e^{-\frac{S}{h}}\right),
  \end{equation*}
 where we conclude the case where $z\in\Omega\cap\Omega_{\eta}^{a}$ by a rescaling argument 
 similar  as in the proof of Proposition \ref{quasimodes}.
\end{rem}
\begin{proof}[Proof of Proposition \ref{prop:EstimAndFormulaDerzE-+_pt2}]
 The first statement follows directly from Proposition \ref{prop:AbsSPe0f0}. 
 The statements regarding the derivatives can be derived by a 
 direct calculation from Proposition \ref{prop:AbsSPe0f0} together with 
 the fact that the $z$- respectively the $\overline{z}$-derivative of 
 the error term increases its growth at most by a term of order $\eta^{1/2}h^{-1}$. 
 Moreover, we use that $\e^{\Phi}$ is exponentially small in $h$ due to $|\Ima z - \langle\Ima g\rangle| >1/C$. 
 Furthermore, we use that the prefactor 
 $\left(\frac{i}{2}\{p,\overline{p}\}(\rho_+)\frac{i}{2}\{\overline{p},p\}(\rho_-)\right)^{\frac{1}{4}}$ 
 is the first order term of $R$ (cf. (\ref{defn:R})). Recall that $R$ is defined via 
 the normalization coefficients of the quasimodes $e_{wkb}$ and $f_{wkb}$. It is thus independent of $\Rea z$ 
 and its $\partial_{\Ima z}$ derivative is of order $\mO(\eta^{-1/4})$ which can be seen by 
 the stationary phase method and a rescaling argument similar to the one in the proof of Proposition 
 \ref{quasimodes}.
\end{proof}
Now let us give estimates on the derivatives of the effective Hamiltonian 
$E_{-+}(z)$.
\begin{prop}\label{prop:EstimDerzE-+}
 Let $z\in\Omega\Subset\Sigma$ and let $E_{-+}(z)$ be as in Proposition \ref{propGrushUnpertOp}. 
 Then there exists a $C>0$ such that for $h>0$ small enough and all $\beta\in\mathds{N}^2$
 \begin{equation*}
  |\partial_{z\overline{z}}^{\beta} E_{-+}(z)| = 
      \mathcal{O}\!\left(\eta^{\frac{|\beta|+1/2}{2}}h^{-|\beta|+1/2}\e^{-\frac{S}{h}}\right).
 \end{equation*}
\end{prop}
\begin{proof}
Take the $\partial_{\overline{z}}$ derivative and the $\partial_z$ derivative of the first 
equation in (\ref{eqn:RelationEjFj}) to gain 
  \begin{align*}
   (P_h -z)\partial_{\overline{z}}e_0 = (\partial_{\overline{z}}\alpha_0)f_0 + \alpha_0 \partial_{\overline{z}}f_0, ~
   (P_h -z)\partial_{z}e_0 - e_0 = (\partial_{z}\alpha_0)f_0 + \alpha_0\partial_{z}f_0.
  \end{align*}
Now consider the scalar product of these equations with $f_0$ and 
recall from Proposition \ref{propGrushUnpertOp} that $E_{-+}(z) = -\alpha_0(z)$ to conclude 
   \begin{align}\label{eqn:derzbarzE-+}
   &\partial_{\overline{z}}E_{-+}(z)= E_{-+}(z)\left\{(\partial_{\overline{z}}e_0|e_0)
    -(\partial_{\overline{z}}f_0|f_0)\right\} 
    ~\text{and}  \notag \\
   & \partial_{z}E_{-+}(z)= E_{-+}(z)\left\{(\partial_{z}e_0|e_0)-(\partial_{z}f_0|f_0)\right\} + (e_0|f_0) .
  \end{align}
The statement of the Proposition then follows by repeated differentiation of 
(\ref{eqn:derzbarzE-+}) and induction using Remark \ref{rem:derTunnel}, the 
estimate $|E_{-+}(z)| = \mO(\eta^{\frac{1}{4}}h^{\frac{1}{2}}\e^{-\frac{S}{h}})$ given in 
(\ref{prop:GrushUnPert}) and (\ref{prop:GrushUnPert2}) and the estimates 
given in Proposition \ref{quasimodes}.
\end{proof}
Finally, Proposition \ref{prop:AbsSPe0f0} permits us to prove the following 
extension of Proposition \ref{quasimodes}: 
\begin{prop}\label{quasimodes2}
Let $z\in\Omega \Subset\Sigma$ and let $e_0$ and $f_0$ be the eigenfunctions of the operators $Q$ and $\tilde{Q}$ with 
respect to their smallest eigenvalue (as in Section \ref{suse:GrushForUnPert}). Let $S=S(z)$ be defined as in 
Definition \ref{defn:ZeroOrderDensity}. Then 
  \begin{itemize}
   \item for $z\in\Omega$ with $\dist(\Omega,\partial\Sigma)>1/C$ 
	 and for all $\alpha\in\mathds{N}^3$ 
	  \begin{equation*}
	      \lVert \partial_{z\overline{z}x}^{\alpha}( e_0 - e_{wkb}) \rVert, ~
 	      \lVert \partial_{z\overline{z}x}^{\alpha}( f_0 - f_{wkb}) \rVert 
 	      = \mathcal{O}\!\left(h^{-|\alpha|}e^{-\frac{S}{h}}\right).
 	    \end{equation*}
 	  Here, we set $\partial_{z\overline{z}x}^{\alpha}= 
 	  \partial_z^{\alpha_1}\partial_{\overline{z}}^{\alpha_2}\partial_x^{\alpha_3}$. 
 	  Furthermore, the various $z$-, $\overline{z}$- and $x$-derivatives of 
	  $e_0$, $f_0$, $e_{wkb}$ and $f_{wkb}$ have at most temperate growth in $1/h$, more precisely
	    \begin{equation*}
	         \lVert \partial_{z\overline{z}x}^{\alpha}e_{wkb} \rVert, ~
	         \lVert \partial_{z\overline{z}x}^{\alpha} f_{wkb} \rVert, ~
	         \lVert \partial_{z\overline{z}x}^{\alpha} e_0 \rVert, 
	         ~ \lVert \partial_{z\overline{z}x}^{\alpha}f_0 \rVert 
	  	 =\mathcal{O}\!\left(h^{-|\alpha|}\right)
  	    \end{equation*}
	  for all $\alpha\in\mathds{N}^3$;
   \item for $h^{2/3} \ll \eta < const.$, $z\in\Omega\cap\Omega_{\eta}^{a}$ and for all $\alpha\in\mathds{N}^3$
	  \begin{equation*}
	      \lVert \partial_{z\overline{z}x}^{\alpha}(e_0 -  e_{wkb}^{\eta}) \rVert, ~
 	      \lVert \partial_{z\overline{z}x}^{\alpha}(f_0 -  f_{wkb}^{\eta}) \rVert 
		  = \mathcal{O}\!\left(\eta^{\frac{\alpha_1+\alpha_2}{2}+\alpha_3}h^{-|\alpha|}e^{-\frac{S}{h}}\right).
 	    \end{equation*}
 	  Furthermore, the various  $z$-, $\overline{z}$- and $x$-derivatives of 
	  $e_0$, $f_0$, $e_{wkb}^{\eta}$ and $f_{wkb}^{\eta}$ have at most temperate growth in $\sqrt{\eta}/h$, more precisely
	    \begin{equation*}
	         \lVert \partial_{z\overline{z}x}^{\alpha} e_{wkb}^{\eta} \rVert, ~
	         \lVert \partial_{z\overline{z}x}^{\alpha}f_{wkb}^{\eta}  \rVert, ~
	         \lVert \partial_{z\overline{z}x}^{\alpha}e_0 \rVert, 
	         ~ \lVert \partial_{z\overline{z}x}^{\alpha} f_0 \rVert 
	  	 =\mathcal{O}\!\left(\eta^{\frac{\alpha_1+\alpha_2}{2}+\alpha_3}h^{-|\alpha|}\right)
  	    \end{equation*}
	  for all $\alpha\in\mathds{N}^3$.
\end{itemize}
\end{prop}
\begin{proof}
Will show the proof in the case of $e_0(z)$ since the case of $f_0(z)$ is similar. 
Suppose first that $z\in\Omega$ with $\dist(\Omega,\partial\Sigma)>1/C$. 
Recall from (\ref{eqn:RelationEjFj}) that 
  \begin{equation}\label{eqn:ReGrStateFormula}
   (P_h-z)e_0 = \alpha_0f_0 \quad \text{and} \quad (P_h-z)^*f_0 = \overline{\alpha}_0e_0
  \end{equation}
  First consider the $\partial_z^n \partial_{\overline{z}}^m$ derivatives of (\ref{eqn:ReGrStateFormula}):
  \begin{align}\label{eqn:DiffofReGrStateFormula}  
   (P_h-z)\partial_z^n \partial_{\overline{z}}^m e_0(z) = 
   n\partial_z^{n-1} \partial_{\overline{z}}^{m}e_0(z) + \hspace{-0.2cm}
   \sum\limits_{\substack{|\alpha_1 + \beta_1| = n \\|\alpha_2 + \beta_2| = m } }
   \hspace{-0.2cm}
   \binom{\eta +\beta }{\beta} (\partial^{\eta}\alpha_0(z) )(\partial^{\beta}f_0(z) )
   \end{align}
   and
   \begin{align*}
   (P_h-z)^*\partial_z^n \partial_{\overline{z}}^m f_0(z) = 
   m\partial_z^{n} \partial_{\overline{z}}^{m-1}f_0(z) + \hspace{-0.25	cm}
    \sum\limits_{\substack{|\alpha_1 + \beta_1| = n \\|\alpha_2 + \beta_2| = m } }  
    \hspace{-0.2cm}\binom{\eta +\beta }{\beta}
    (\partial^{\eta}\overline{\alpha}_0(z) ) \partial^{\beta}e_0(z)
  \end{align*}
  and thus 
  \begin{align*}
    h \lVert D_x \partial_z^n \partial_{\overline{z}}^m  e_0(z) \rVert 
     \leq & n \lVert\partial_z^{n-1} \partial_{\overline{z}}^{m}e_0(z)\rVert + 
     \hspace{-0.2cm}
    \sum\limits_{\substack{|\alpha_1 + \beta_1| = n \\|\alpha_2 + \beta_2| = m } } 
    \hspace{-0.2cm}
    \binom{\eta +\beta }{\beta} \lVert \partial^{\eta}\alpha_0(z) \rVert
   \lVert \partial^{\beta}f_0(z) \rVert
      \notag \\
     & + \lVert g -z \rVert_{L^{\infty}(S^1)} \cdot\lVert \partial_z^n \partial_{\overline{z}}^m e_0(z) \rVert 
  \end{align*}
  and
    \begin{align*}
    h \lVert D_x \partial_z^n \partial_{\overline{z}}^m  f_0(z) \rVert 
     \leq & m \lVert\partial_z^{n} \partial_{\overline{z}}^{m-1}f_0(z)\rVert + 
     \hspace{-0.2cm}
    \sum\limits_{\substack{|\alpha_1 + \beta_1| = n \\|\alpha_2 + \beta_2| = m } }
    \hspace{-0.2cm}
     \binom{\eta +\beta }{\beta} \lVert \partial^{\eta}\overline{\alpha}_0(z) \rVert
   \lVert \partial^{\beta}e_0(z) \rVert
      \notag \\
     & + \lVert g -z \rVert_{L^{\infty}(S^1)} \cdot\lVert \partial_z^n \partial_{\overline{z}}^m f_0(z) \rVert .
  \end{align*}
  By Proposition \ref{prop:EstimDerzE-+}, there exists a constant $C>0$ such that 
    \begin{equation}\label{d_alph1}
 |\partial_z^k \partial_{\overline{z}}^j  \alpha_0(z)|=|\partial_z^k \partial_{\overline{z}}^j  E_{-+}(z)| =
   \mathcal{O}\!\left(h^{-(k+j)}\e^{-\frac{S}{h}}\right).
 \end{equation}
By (\ref{lem:QuasModEstimate3}) we conclude 
    \begin{equation*}
     \lVert D_x  \partial_z^n \partial_{\overline{z}}^m e_0(z) \rVert , 
     ~\lVert D_x  \partial_z^n \partial_{\overline{z}}^m f_0(z) \rVert = \mathcal{O}\!\left(h^{-(n+m+1)}\right).
    \end{equation*}
  Repeated differentiation of (\ref{eqn:DiffofReGrStateFormula}) and induction then yield that for all $l\in\mathds{N}$
  \begin{equation*}
   \lVert D_x^l  \partial_z^n \partial_{\overline{z}}^me_0(z) \rVert, 
   \lVert D_x^l  \partial_z^n \partial_{\overline{z}}^mf_0(z) \rVert  =\mathcal{O}\!\left(h^{-(l+n+m)}\right).
  \end{equation*}
 The estimate 
  \begin{equation*}
   \lVert D_x^l  \partial_z^n \partial_{\overline{z}}^me_{wkb} \rVert, 
   \lVert D_x^l  \partial_z^n \partial_{\overline{z}}^mf_{wkb} \rVert  =\mathcal{O}\!\left(h^{-(l+n+m)}\right)
  \end{equation*}
  follows directly by the stationary phase method together with \eqref{d_coef1}, 
  \eqref{eqn:higherzzbarderchi_e}. Finally, using \eqref{eqn:RelationEjFj}, \eqref{eqn:DefEWKB}, consider 
    \begin{equation*}
     (P_h-z)(e_0-e_{wkb}) = \alpha_0f_0 - 
     h^{-\frac{1}{4}}a(z)\frac{h}{i} \partial_{x}\chi_e \e^{\frac{i}{h} \phi_+(x)}
    \end{equation*}
  which implies for $k\geq 1$ that $(hD_x)^k\partial_z^n\partial_{\overline{z}}^{m}(e_0-e_{wkb})$ 
  is equal to 
    \begin{align*}
    (hD_x)^{(k-1)}\partial_z^n\partial_{\overline{z}}^{m}(\alpha_0f_0) &- 
    (hD_x)^{(k-1)}\partial_z^n\partial_{\overline{z}}^{m}
    \left(h^{-\frac{1}{4}}a(z)\frac{h}{i} \partial_{x}\chi_e \e^{\frac{i}{h} \phi_+(x)}\right)
    \notag \\ 
    & 
    +(hD_x)^{(k-1)}\partial_z^n\partial_{\overline{z}}^{m}(g(x)-z)(e_0-e_{wkb}).
    \end{align*}
 By induction over $k$ together with Proposition \ref{quasimodes} and \eqref{d_alph1}, 
 \eqref{eqn:higherzzbarderchi_e}, we conclude the first point of the Proposition. The 
 results in the case where $z\in\Omega\cap\Omega_{\eta}^a$ follow by a rescaling argument 
 similar as in the proof of Proposition \ref{quasimodes}.
\end{proof}

\subsection{Alternative Grushin problems for the unperturbed operator $P_h$}
In \cite{Ha06} Hager set up a different Grushin problem for $P_h$ and $z\in\Omega_i$ 
which results in a more explicit effective Hamiltonian $E^H_{-+}(z)$. 
To avert confusion, we will mark the elements of Hager's Grushin problem with an additional $``H``$. \par
Bordeaux-Montrieux in \cite{BM} then extended Hager's Grushin problem to $z\in\Omega\cap\Omega_{\eta}^{a}$.
It is very useful for the further discussion to have an explicit effective Hamiltonian. Thus we will briefly 
introduce Hager's Grushin problem $\mathcal{P}^H$ and show that $E_{-+}(z)$ and $E^H_{-+}(z)$ differ 
only by an exponentially small error. \\ 

\begin{prop}[\cite{Ha06,BM}]\label{porp:QuasmodHager}
   For $z\in\Omega\Subset\Sigma$, let $x_{\pm}(z)\in\mathds{R}$ be as in Section \ref{sec:Intro}.  
   \begin{itemize}
    \item for $z\in\Omega$ with $\dist(\Omega,\partial\Sigma)>1/C$: 
    let $I_{\pm}$ be open intervals, independent of $z$ such that 
   \begin{align*}
    x_{\pm}(z)\in I_{\pm}, ~ x_{\mp}(z)\notin\overline{I_{\pm}} \quad \text{for all }z\in\overline{\Omega}.
   \end{align*}
   Let $\phi_{\pm}(x;z)$ be as in Definition \ref{def:Quasimodes}. Then, there exist smooth functions $c_{\pm}(z;h)>0$ such 
   that 
    \begin{equation*}
     c_{\pm}(z;h) \sim h^{-\frac{1}{4}}\left(c_{\pm}^0(z)+hc_{\pm}^1(z)+\dots \right)
    \end{equation*}
   and, for $e_+(z;h):=c_+(z;h)\exp(\frac{i\phi_+(x;z)}{h})\in H^1(I_+)$ and 
   $e_-(z;h):=c_-(z;h)\exp(\frac{i\overline{\phi_-(x;z)}}{h})\in H^1(I_+)$,
    \begin{equation*}
     \lVert e_+ \rVert_{L^2(I_+)} =1 = \lVert e_- \rVert_{L^2(I_-)}. 
    \end{equation*}
    Furthermore, we have 
      \begin{equation*}
       c_+^0(z) = \left(\frac{-\Ima g'(x_+(z))}{\pi}\right)^{\frac{1}{4}},
       ~\text{and} ~ c_-^0(z)  = \left(\frac{\Ima g'(x_-(z))}{\pi}\right)^{\frac{1}{4}}.
      \end{equation*}
      
    \item for $z\in\Omega\cap\Omega_{\eta}^{a}$ with $h^{2/3} \ll \eta < const.$: 
    let $J_{\pm}$ be open intervals, such that 
	  \begin{equation*}
	   x_{\pm}(\Omega_{\eta}^{a})\in J_{\pm}, \quad \dist (J_+,J_-) > \frac{1}{C} \eta^{1/2}.
	  \end{equation*}   
    Define $\tilde{I}_{\pm} := S^1 \backslash \overline{J_{\mp}}$.
   Let $\phi_{\pm}(x;z)$ be as in Definition \ref{def:Quasimodes} and set $\tilde{h}:=h/\eta^{3/2}$. 
   Then, there exist smooth functions $c_{\pm}(z;\tilde{h})>0$ such 
   that 
    \begin{equation*}
     c_{\pm}^{\eta}(z;\tilde{h}) \sim \tilde{h}^{-\frac{1}{4}}\eta^{-1/4}\left(c_{\pm}^{0,\eta}(z)
		    +\tilde{h}c_{\pm}^{1,\eta}(z)+\dots \right)
    \end{equation*}
   and, for $e_+^{\eta}(z;h):=c_+^{\eta}(z;\tilde{h})\exp(\frac{i\phi_+(x;z)}{h})\in H^1(\tilde{I}_+)$ and 
   $e_-^{\eta}(z;h):=c_-^{\eta}(z;\tilde{h})\exp(\frac{i\overline{\phi_-(x;z)}}{h})\in H^1(\tilde{I}_+)$,
    \begin{equation*}
     \lVert e_+^{\eta} \rVert_{L^2(\tilde{I}_+)} =1 = \lVert e_-^{\eta} \rVert_{L^2(\tilde{I}_-)}. 
    \end{equation*}
    Furthermore, we have 
     \begin{align*}
   &c_+^{0,\eta}(z)= \left(\frac{|\Ima g''(a)(\tilde{x}_+(\tilde{z})-a/\sqrt{\eta})(1+o(1))|}{\pi}\right)^{\frac{1}{4}}, \quad 
    z\in\Omega_{\eta}^{a}, \notag \\
   &c_-^{0,\eta}(z) = 
    \left(\frac{|\Ima g''(a)(\tilde{x}_-(\tilde{z})-a/\sqrt{\eta})(1+o(1))|}{\pi}\right)^{\frac{1}{4}}, 
      \quad z\in\Omega_{\eta}^{a}.
  \end{align*}
   \end{itemize}
  \end{prop}
  \begin{proof}
  For a proof of the first statement see \cite{Ha06}. The second statement has been proven in \cite{BM} 
  with the exception of the representation of $c_{\pm}^{0,\eta}(z)$ which can be achieved by an analogous argument to 
  the one used in the proof of Proposition \ref{prop:QuasmodCB}. 
  \end{proof}
Note that $(P_h-z)e_+^{\bullet}(x;z)=0$ on $I_+$ and that $(P_h-z)^*e_-^{\bullet}(x;z)=0$ on $I_-$. 
With these quasimodes Hager and then Bordeaux-Montrieux set up a Grushin problem
$\mathcal{P}^H$ and proved the existence of an inverse $\mathcal{E}^H$. 
  \begin{prop}[\cite{Ha06}]\label{prop:GrushHager}
   For $z\in\Omega_i\Subset\mathring{\Sigma}$ and $x_{\pm}(z)$ as in Section \ref{sec:Intro}. 
   Let $g\in\mathcal{C}^{\infty}(S^1:\mathds{C})$ be as in (\ref{eqn:defnModelOperator}) and 
   let $a < b< a+2\pi$ where $a$ denotes the minimum and $b$ the maximum 
   of $\Ima g$. Let $J_{+}\subset(b,a+2\pi)$ and $J_{-}\subset(a,b)$ such that 
   $\overline{\{x_{\pm}(z):~z\in\Omega\}}\subset J_{\pm}$. Let $\chi_{\pm}\in\mathcal{C}^{\infty}_c(I_{\pm})$ be such that 
   $\chi_{\pm}\equiv 1$ on $\overline{J_{\pm}}$ and $\supp(\chi_+) \cap \supp(\chi_-) = \emptyset$. Define
  \begin{align*}
    R_+^H:~ &H^1(S^1)\longrightarrow\mathds{C}:~ u  \longmapsto (u|\chi_+e_+) \notag \\
    R_-^H:~&\mathds{C}\longrightarrow L^2(S^1):~ u_-  \longmapsto u_-\chi_-e_-.  
  \end{align*}
Then 
  \begin{equation*}
   \mathcal{P}^H(z):=\begin{pmatrix}
                    P_h-z & R_-^H \\ R_+^H & 0 \\
                   \end{pmatrix}
   :~ H^1(S^1)\times \mathds{C}\longrightarrow L^2(S^1)\times \mathds{C}
  \end{equation*}
is bijective with the bounded inverse 
  \begin{equation*}
   \mathcal{E}^H(z) = \begin{pmatrix}
                     E^H(z) & E_+^H(z) \\ E_-^H(z) & E_{-+}^H(z) \\
                    \end{pmatrix}
  \end{equation*}
where 
  \begin{align}\label{eqn:DetailGrushHag}
   &\lVert E^H(z)\rVert_{L^2\rightarrow H^1} = \mathcal{O}(h^{-1/2}),
   &\lVert E_-^H(z)\rVert_{L^2\rightarrow\mathds{C}} = \mathcal{O}(1), \notag \\
   & \lVert E_+^H(z)\rVert_{\mathds{C}\rightarrow H^1} = \mathcal{O}(1), 
   &|E_{-+}^H(z)| = \mathcal{O}\!\left(\e^{-\frac{1}{Ch}}\right). 
  \end{align}
Furthermore,
   \begin{align}\label{eqn:EffectHamHager}
    E_{-+}^H(z) = 
    &\left(\left(\frac{i}{2}\{p,\overline{p}\}(\rho_+)
	  \frac{i}{2}\{\overline{p},p\}(\rho_-)\right)^{\frac{1}{4}}\left(\frac{h}{\pi}\right)^{\frac{1}{2}}+
	  \mathcal{O}\!\left(h^{\frac{3}{2}}\right)\right)\cdot \notag \\
    &\left(\e^{\frac{i}{h}\int_{x_+}^{x_-}(z-g(y))dy} - \e^{\frac{i}{h}\int_{x_+}^{x_- +2\pi}(z-g(y))dy}\right),
   \end{align}
where the prefactor of the exponentials depends only on $\Ima z$ and has bounded derivatives of order 
$\mathcal{O}(\sqrt{h})$. 
  \end{prop}
  \begin{proof}
   See \cite{Ha06}.
  \end{proof}
  
 \begin{prop}[\cite{BM}]\label{prop:GrushBM}
   Let $\Omega\Subset\Sigma$. For $z\in\Omega\cap\Omega_{\eta}^{a,b}$ and $x_{\pm}(z)$ as in Section \ref{sec:Intro}. 
   Let $g\in\mathcal{C}^{\infty}(S^1)$ be as in (\ref{eqn:defnModelOperator}). Let $J_{\pm}$ and $I_{\pm}$ 
   be as in the second point of Proposition \ref{porp:QuasmodHager}. 
   Let $\chi_{\pm}^{ \eta}\in\mathcal{C}^{\infty}_c(I_{\pm})$ such that 
   $\chi_{\pm}^{ \eta}\equiv 1$ on $\overline{J_{\pm}}$ and $\supp(\chi_+^{ \eta}) \cap \supp(\chi_-^{ \eta}) = \emptyset$. Define
  \begin{align*}
    R_+^{\eta}:~ &H^1(S^1)\longrightarrow\mathds{C}:~ u  \longmapsto (u|\chi_+e_+^{ \eta}) \notag \\
    R_-^{\eta}:~&\mathds{C}\longrightarrow L^2(S^1):~ u_-  \longmapsto u_-\chi_-e_-^{ \eta}.  
  \end{align*}
Then 
  \begin{equation*}
   \mathcal{P}^{\eta}(z):=\begin{pmatrix}
                    P_h-z & R_-^{\eta} \\ R_+^{\eta} & 0 \\
                   \end{pmatrix}
   :~ H^1(S^1)\times \mathds{C}\longrightarrow L^2(S^1)\times \mathds{C}
  \end{equation*}
is bijective with the bounded inverse 
  \begin{equation*}
   \mathcal{E}^{\eta}(z) = \begin{pmatrix}
                     E^{\eta}(z) & E_+^{\eta}(z) \\ E_-^{\eta}(z) & E_{-+}^{\eta}(z) \\
                    \end{pmatrix}
  \end{equation*}
where 
  \begin{align}\label{eqn:DetailGrushBM}
   &\lVert E^{\eta}(z)\rVert_{L^2\rightarrow H^1} = \mathcal{O}((\sqrt{\eta}h)^{-1/2}),
   &\lVert E_-^{\eta}(z)\rVert_{L^2\rightarrow\mathds{C}} = \mathcal{O}(1), \notag \\
   & \lVert E_+^{\eta}(z)\rVert_{\mathds{C}\rightarrow H^1} = \mathcal{O}(1),
   &|E_{-+}^{\eta}(z)| = \mathcal{O}\!\left(\eta^{1/4}h^{1/2}\e^{-\frac{\asymp\eta^{3/2}}{h}}\right). 
  \end{align}
Furthermore,
   \begin{align}\label{eqn:EffectHamBM}
    E_{-+}^{\eta}(z) = &
    \left(c_+^{0,\eta}(z)c_-^{0,\eta}(z)\left(h\sqrt{\eta}\right)^{\frac{1}{2}}+
	  \mathcal{O}\!\left(h^{\frac{3}{2}}\eta^{-5/4}\right)\right)
    \notag\\
    &\cdot
    \left(\e^{\frac{i}{h}\int_{x_+}^{x_-}(z-g(y))dy} - \e^{\frac{i}{h}\int_{x_+}^{x_- +2\pi}(z-g(y))dy}\right),
   \end{align}
where the prefactor of the exponentials depends only on $\Ima z$ and has bounded derivatives of order 
$\mathcal{O}(\sqrt{h\sqrt{\eta}})$. 
  \end{prop}
  \begin{proof}
   See \cite{BM}. (\ref{eqn:EffectHamBM}) has not been stated in this form on \cite{BM}. However, 
   it can easily be deduce from the results in \cite{BM} together with Proposition \ref{porp:QuasmodHager}.
  \end{proof}  
\begin{rem}\label{rem:CutOff_CB}
 The cut-off function $\chi_{\pm}^{ \eta}$ in the above proposition can be chosen 
 similarly to $\chi_{e,f}^{\eta}$ in Definition \ref{def:quasmodCB} (compare also 
 with Definition \ref{def:Quasimodes}).
\end{rem}

\subsection{Estimates on the effective Hamiltonians}
\label{suse:EstEffHam}
In \cite{Ha06} Hager chose to represent $S^1$ as an interval between 
two of the periodically appearing minima of $\Ima g$ and thus 
chose the notation for $x_{\pm}$ accordingly (this notation was 
used in (\ref{eqn:EffectHamHager})). In our case however, we chose 
to represent $S^1$ as an interval between two of the periodically 
appearing maxima of $\Ima g$. This results in the following 
difference between notations: 
  \begin{equation*}
   x_+(z) = x_+^H(z) -2\pi \quad \text{and} \quad x_-(z) = x_-^H(z).
  \end{equation*}
Thus, in our notation, we have for $\bullet = H,\eta$
  \begin{align}\label{eqn:EstimEffectHamHager}
    E_{-+}^{\bullet}(z) = V^{\bullet}(z,h)
	\left(\e^{\frac{i}{h}\int_{x_+}^{x_--2\pi}(z-g(y))dy} 
      - \e^{\frac{i}{h}\int_{x_+}^{x_-}(z-g(y))dy}\right),
  \end{align}
where $V^{\bullet}=V^{\bullet}(z,h)$ satisfies 
 \begin{align}\label{eqn:Prefac} 
    V^{\bullet} = 
    \begin{cases}
    \left(\frac{i}{2}\{p,\overline{p}\}(\rho_+)
	  \frac{i}{2}\{\overline{p},p\}(\rho_-)\right)^{\frac{1}{4}}\left(\frac{h}{\pi}\right)^{\frac{1}{2}}
	  \left(1 + \mathcal{O}\!\left(h\right)\right),\text{if}~\bullet = H,~ z\in\Omega_i \\ 
          c_+^{0,\eta}(z)c_-^{0,\eta}(z)\left(h\sqrt{\eta}\right)^{\frac{1}{2}}
	  \left(1 + \mathcal{O}\!\left(\eta^{-\frac{3}{2}}h\right)\right), 
	   \text{if}~\bullet = \eta,~ z\in\Omega_{\eta}^{a}. \\
    \end{cases}
  \end{align}
Note that Taylor expansion around the point $a$ yields 
 \begin{align}\label{eqn:TaylorPoissonBracket}
   \{p,\overline{p}\}(\rho_{\pm}) &= -2i\Ima g'(x_{\pm}) \\
   &=
      2i\sqrt{\eta}\left|\Ima g''(a)(\tilde{x}_{\pm}(\tilde{z})-a/\sqrt{\eta})(1+o_{\sqrt{\eta}}(1))\right|, 
      ~\text{for} ~ z\in\Omega_{\eta}^{a}.\notag 
  \end{align}
Therefore, we may write for all $z\in\Omega\Subset\Sigma$ 
  \begin{align}\label{eqn:Prefac2}
   V(z,h):= V^{\bullet}(z,h) = 
    \left(\frac{i}{2}\{p,\overline{p}\}(\rho_+)
	  \frac{i}{2}\{\overline{p},p\}(\rho_-)\right)^{\frac{1}{4}}\left(\frac{h}{\pi}\right)^{\frac{1}{2}}
	  \left(1 + \mathcal{O}\!\left(\eta^{-\frac{3}{2}}h\right)\right)
  \end{align}
where the first order term is $\eta^{1/4}$ for $z\in\Omega\cap\Omega_{\eta}^{a}$. Note that
 \begin{align}\label{eqn:EstimEffectHamHagerPr1}
   \Big|\e^{\frac{i}{h}\int_{x_+}^{x_--2\pi}(z-g(y))dy} &- \e^{\frac{i}{h}\int_{x_+}^{x_-}(z-g(y))dy}\Big|
    = \e^{-\frac{S}{h}}
       \left|1 - \e^{ \Phi(z,h)}\right|,
  \end{align}
where $ \Phi(z,h)$ is defined already in Proposition \ref{prop:Resolvent}. For the 
readers convenience:
\begin{align*}
   \Phi(z,h)= 
    \begin{cases}
     -\frac{2\pi i}{h}(z -\langle g\rangle),~\text{if}~ \Ima z < \langle \Ima g\rangle , \\
      \frac{2\pi i}{h}(z -\langle g\rangle),~\text{if}~ \Ima z >\langle \Ima g\rangle,
     \end{cases}
  \end{align*}
Hence 
  \begin{align}\label{eqn:EstimModEffectHamHager}
    |E_{-+}^{\bullet}(z)| = V(z,h)\e^{-\frac{S}{h}}
    \left|1 - \e^{ \Phi(z,h)}\right|.
  \end{align}
The aim of this section is to prove the following proposition. 
\begin{prop}\label{cor:E-+Repres}
 Let $\Omega\Subset\Sigma$, let $\Phi(z,h)$ be as in Proposition \ref{prop:Resolvent} and let 
 $E_{-+}(z)$ the effective Hamiltonian given in Proposition \ref{propGrushUnpertOp}. 
 Then, for $h>0$ small enough, there exists a constant $C>0$ such that for 
 $h^{\frac{2}{3}} \ll \eta \leq \mathrm{const.}$
    \begin{align*}
     |E_{-+}(z)| = V(z,h)&\e^{-\frac{S(z)}{h}}
    \left|1 - \e^{\Phi(z,h)}\right|
      \left(1 + \mathcal{O}\!\left(\e^{-\frac{\asymp\eta^{3/2}}{h}}\right)\right).
    \end{align*}
 Furthermore, for all $\beta\in\mathds{N}^2$ the $\partial_{z\overline{z}}^{\beta}$ derivatives 
of the error terms are bounded and of order 
  \begin{equation*}
      \mathcal{O}\!\left(\eta^{\frac{|\beta|}{2}}h^{-|\beta|}\e^{-\frac{\asymp\eta^{\frac{3}{2}}}{h}}\right).
  \end{equation*}
\end{prop}
\begin{proof}[Proof of Proposition \ref{prop:TunnelEffect2}]
 Recall that $(P_h-z)e_0=\alpha_0f_0$ (cf. (\ref{eqn:RelationEjFj})). Suppose first that 
 $z\in\Omega$ with $\dist(\Omega,\partial\Sigma)>1/C$. By Proposition \ref{quasimodes} 
 we find
  \begin{align*}
   \left((1-\chi)(P_h-z)e_0|f_0\right) 
      &= \alpha_0(f_0|(1-\chi)f_0) \notag\\ 
      &= \alpha_0\left((f_{wkb}|(1-\chi)f_{wkb}) + \mO\left(\e^{-\frac{S}{h}}\right)\right).
  \end{align*}
 Since the phase of $f_{wkb}$ has no critical point on the support of $\chi$, it follows 
 that there exists a constant $C>0$, depending on $\chi$ but uniform in $z\in\Omega$, 
 such that 
  \begin{equation*}
   \left((1-\chi)(P_h-z)e_0|f_0\right) 
      = \mO\left(\alpha_0\e^{-\frac{1}{Ch}}\right).
  \end{equation*}
 By a similar argument we find that 
 \begin{equation*}
   \left((P_h-z)\chi e_0|f_0\right) 
      = \alpha_0\left(\chi e_0|e_0\right) 
      = \mO\left(\alpha_0\e^{-\frac{1}{Ch}}\right).
  \end{equation*}
 In the case where $z\in\Omega\cap\Omega_{\eta}^a$, we perform a rescaling argument 
 similar to the one in the proof of Proposition \ref{quasimodes}. Thus, 
  \begin{align*}
   \left((1-\chi)(P_h-z)e_0|f_0\right),
   \left((P_h-z)\chi e_0|f_0\right) 
      = \mO\left(\alpha_0\exp\left\{-\frac{\eta^{\frac{3}{2}}}{Ch}\right\}\right).
  \end{align*}
 Note that Proposition \ref{quasimodes} implies that each $z$- and $\overline{z}$-
 derivative of the exponentially small error term increases its order of growth 
 at most by factor of order $\mO(\eta^{1/2}h^{-1})$.
 Thus, using (\ref{eqn:RelationEjFj}) yields 
  \begin{equation}\label{eqn:alphaEFFTUNN}
   \alpha_0 =  \left((1-\chi + \chi)(P_h-z)e_0|f_0\right) 
	    =  \left([\chi,P_h]e_0|f_0\right) 
	    + \mO\left(\alpha_0\exp\left\{-\frac{\eta^{\frac{3}{2}}}{Ch}\right\}\right)
  \end{equation}
  The statement of the Proposition then follows by the fact that $|\alpha_0|=|E_{-+}(z)|$ 
  (cf. Proposition \ref{propGrushUnpertOp}) together with Proposition \ref{cor:E-+Repres}.
\end{proof}
We give some estimates on the elements of the Grushin problems introduced 
in Section \ref{se:GrushProb}.
\begin{prop}\label{prop:EstimOnGrushElements}
 Let $\Omega\Subset\Sigma$, let $E_{-+}^{\bullet}, E_{\pm}^{\bullet},R_{\pm}^{\bullet}, E^{\bullet}$ 
 be as in the Propositions \ref{propGrushUnpertOp}, \ref{prop:GrushHager} and \ref{prop:GrushBM}, 
 where $\bullet = -,H,\eta$ with ``$-$'' symbolizing no index. Furthermore, let 
 $S(z)$ as in Definition \ref{defn:ZeroOrderDensity}. 
 Then we have the following estimates 
  \begin{enumerate}
   \item for $\bullet = -,H$ and for $z\in\Omega_i\subset\Omega$ 
    \begin{align*}
     &\lVert\partial_{z\overline{z}}^{\beta}R_{\pm}^{\bullet} \rVert, 
     \lVert\partial_{z\overline{z}}^{\beta}E_{\pm}^{\bullet}\rVert
     = \mathcal{O}\left( h^{-|\beta|}\right), \notag \\
     &|\partial_{z\overline{z}}^{\beta}E_{-+}^{H}|
     = \mathcal{O}\!\left(h^{-(|\beta|-\frac{1}{2})}\e^{-\frac{S(z)}{h}}\right) , \quad
     \lVert\partial_{z\overline{z}}^{\beta}E^{\bullet}\rVert
     = \mathcal{O}\left( h^{-(|\beta|+1/2)}\right). 
    \end{align*}
   \item for $\bullet = -,\eta$ and for $z\in\Omega_{\eta}^{a,b}\subset\Omega$ 
    \begin{align*}
     &\lVert\partial_{z\overline{z}}^{\beta}R_{\pm}^{\bullet} \rVert, 
     \lVert\partial_{z\overline{z}}^{\beta}E_{\pm}^{\bullet}\rVert
     = \mathcal{O}\left(\eta^{\frac{|\beta|}{2}}h^{-|\beta|}\right), \notag \\
     &|\partial_{z\overline{z}}^{\beta}E_{-+}^{\eta}|= 
     \mathcal{O}\!\left(\eta^{\frac{|\beta|+1/2}{2}}h^{-(|\beta|-\frac{1}{2})}\e^{-\frac{\asymp\eta^{3/2}}{h}}\right) ,
      \notag \\
    & \lVert\partial_{z\overline{z}}^{\beta}E^{\bullet}\rVert
     = \mathcal{O}\left( \eta^{\frac{|\beta|-1/2}{2}}h^{-(|\beta|+1/2)}\right) .
    \end{align*}
  \end{enumerate}
\end{prop}
\begin{proof}
 Recall the definition of $R_{\pm}$ and $E_{\pm}$ given in Proposition \ref{propGrushUnpertOp}. By the estimates 
 on the $z$- and $\overline{z}$- derivatives of $e_0$ and $f_0$ given in Proposition \ref{quasimodes}, 
 we may conclude for $z\in\Omega$ that
  \begin{align}\label{eqn:estimDzzEpmRpm}
   &\lVert \partial_{z\overline{z}}^{\beta} E_+ \rVert_{\mathds{C} \rightarrow L^2}, ~
   \lVert \partial_{z\overline{z}}^{\beta} R_+ \rVert_{H^1\rightarrow \mathds{C}} 
   \leq \lVert\partial_{z\overline{z}}^{\beta}e_0\rVert_{L^2}  
   = \mathcal{O}\!\left(\eta^{\frac{|\beta|}{2}}h^{-|\beta|}\right),
     \notag \\
   &\lVert \partial_{z\overline{z}}^{\beta} E_- \rVert_{L^2 \rightarrow \mathds{C}}, ~
   \lVert \partial_{z\overline{z}}^{\beta} R_- \rVert_{\mathds{C}\rightarrow L^2} 
   \leq \lVert \partial_{z\overline{z}}^{\beta} f_0\rVert_{L^2}
   =  \mathcal{O}\!\left(\eta^{\frac{|\beta|}{2}}h^{-|\beta|}\right),
  \end{align}
 and thus prove the corresponding ``-''-cases in the Proposition. The estimates for the other 
 cases of $R_{\pm}^{\bullet}$ and $E_{\pm}^{\bullet}$ then follow from 
 (\ref{eqn:estimDzzEpmRpm}), (\ref{eqn:PfComGrushHag3}) and (\ref{eqn:PfComGrushHag3_R-}). 
 
 Recall from Proposition 
 \ref{propGrushUnpertOp} that $\mathcal{E}(z)\mathcal{P}(z)=1$. Thus, note that
  \begin{align*}
   &\partial_z\mathcal{E}(z) + \mathcal{E}(z)(\partial_{z}\mathcal{P}(z))\mathcal{E}(z) =0, \notag \\
   &\partial_{\overline{z}}\mathcal{E}(z) + \mathcal{E}(z)(\partial_{\overline{z}}\mathcal{P}(z))\mathcal{E}(z) = 0,
  \end{align*}
 which implies 
  \begin{align*}
   \partial_z E &= - E(\partial_z(P_h-z)) E - E_+(\partial_zR_+) E - E(\partial_zR_-) E_- \notag \\
		 &= E^2 - E_+(\partial_zR_+) E - E(\partial_zR_-) E_-
  \end{align*}
and
  \begin{align*}
   \partial_{\overline{z}}E(z) = - E_+(z)(\partial_{\overline{z}}R_+)E(z) - E(z)(\partial_{\overline{z}}R_-)E_-(z).
  \end{align*}
Thus, by induction we conclude from this, from (\ref{eqn:estimDzzEpmRpm}) and from Proposition 
\ref{propGrushUnpertOp} that for $z\in\Omega$
  \begin{equation*}
   \lVert\partial_{z\overline{z}}^{\beta}E(z)\rVert 
   = \mathcal{O}\!\left(\eta^{\frac{|\beta|-1/2}{2}}h^{-(|\beta|+\frac{1}{2})}\right).
  \end{equation*}
The estimates on $\lVert\partial_{z\overline{z}}^{\beta}E^{\bullet}(z)\rVert$, 
for $\bullet = \eta,H$, can 
be conclude by following the same steps and by using the corresponding estimates on 
$R_{\pm}^{\bullet}$ and $E_{\pm}^{\bullet}$ and the Propositions \ref{prop:GrushHager} 
and \ref{prop:GrushBM}.\par 
It remains to prove the estimates on $|\partial_{z\overline{z}}^{\beta}E_{-+}^{\eta}(z)|$ and 
$|\partial_{z\overline{z}}^{\beta}E_{-+}^{H}(z)|$: let us first consider the case where 
$z\in\Omega_i\subset\Omega$. Recall (\ref{eqn:EstimEffectHamHager}) and recall from 
Proposition \ref{prop:GrushHager} that the prefactor $V^H(z)$ has bounded $z$- and $\overline{z}$-derivatives 
of order $\mathcal{O}(\sqrt{h})$. Thus, the statement follows immediately. \\
\par
In the case where $z\in\Omega_{\eta}^{a,b}\subset\Omega$, recall (\ref{eqn:EstimEffectHamHager}) and 
from Proposition \ref{prop:GrushBM} that the prefactor $V^{\eta}(z)$ has bounded $z$- and $\overline{z}$-derivatives 
of order $\mathcal{O}(\sqrt{h\sqrt{\eta}})$. Using that 
  \begin{align*}
   \e^{\frac{i}{h}\int_{x_+}^{x_--2\pi}(z-g(y))dy} &- \e^{\frac{i}{h}\int_{x_+}^{x_-}(z-g(y))dy} 
   \notag \\
   &= 
   \e^{\frac{i}{\tilde{h}}\int_{\tilde{x}_+}^{\tilde{x}_--2\pi/\sqrt{\eta}}(\tilde{z}-\tilde{g}(\tilde{y}))d\tilde{y} }
   - \e^{\frac{i}{\tilde{h}}\int_{\tilde{x}_+}^{\tilde{x}_-}(\tilde{z}-\tilde{g}(\tilde{y}))d\tilde{y}},
  \end{align*}
(\ref{eqn:SEstINBox}) implies 
  \begin{equation*}
   |\partial_{z\overline{z}}^{\beta}E_{-+}^{\eta}(z)| = 
   \eta^{-|\beta|}|\partial_{\tilde{z}\overline{\tilde{z}}}^{\beta}E_{-+}^{\eta}(z)| 
   = \mathcal{O}\!\left(\eta^{\frac{|\beta|+1/2}{2}}h^{-(|\beta|-\frac{1}{2})}\e^{\frac{\asymp\eta^{3/2}}{h}}\right).
   \qedhere
  \end{equation*}
\end{proof}
\begin{proof}[Proof of Proposition \ref{cor:E-+Repres}]
Let $\bullet= H,\eta$ denote the quasimodes and elements of the Grushin problems 
corresponding to the different zones of $z$. \par
Since $\mathcal{P}^{\bullet}\mathcal{E}^{\bullet}:~L^2(S^1)\times \mathds{C}\longrightarrow L^2(S^1)\times \mathds{C}$ let us introduce the 
following norm for an operator-valued matrix $A:~L^2(S^1)\times \mathds{C}\longrightarrow L^2(S^1)\times \mathds{C}$:
  \begin{equation*}
   \lVert A\rVert_{\infty} := \max\limits_{1\leq i\leq 2} \sum_{j=1}^{2} \lVert A_{ij}\rVert, 
  \end{equation*}
where $\lVert A_{ij}\rVert$ denotes the respective operator norm for $A_{ij}$. Next, note that 
 \begin{equation*}
   \mathcal{P}\mathcal{E}^{\bullet} = \left(\mathcal{P}^{\bullet} 
   + (\mathcal{P} - \mathcal{P}^{\bullet})\right)\mathcal{E}^{\bullet} = 
   1 + (\mathcal{P} - \mathcal{P}^{\bullet})\mathcal{E}^{\bullet}.
  \end{equation*}
\\
\textbf{Estimates for ($\mathcal{P} - \mathcal{P}^{\bullet}$) } 
 Recall the definition of $\mathcal{P}$ and of $\mathcal{P}^{\bullet}$ from the Propositions \ref{propGrushUnpertOp}, 
 \ref{prop:GrushHager} and \ref{prop:GrushBM} and note that 
  \begin{equation*}
   \mathcal{P} - \mathcal{P}^{\bullet} = \begin{pmatrix}
                                  0 & R_- - R_-^{\bullet} \\
                                  R_+ - R_+^{\bullet} & 0 \\
                                 \end{pmatrix}.
  \end{equation*}
We will now prove that for all $(n,m)\in\mathds{N}^2$
  \begin{align}\label{eqn:PfComGrushHag3}
   \lVert\partial_{z\overline{z}}^{\beta}( R_+ &- R_+^{\bullet} )\rVert_{H^1(S^1)\rightarrow\mathds{C}} \leq 
   \lVert\partial_{z\overline{z}}^{\beta}( e_0 - \chi_+^{\bullet} e_+^{\bullet} ) \rVert \notag \\
   &= 
   \begin{cases}
    \mathcal{O}\!\left( h^{-|\beta|}\e^{-\frac{1}{Ch}}\right), ~\text{for} ~ z\in\Omega,~\dist(\Omega,\partial\Sigma)>1/C,\\
    \mathcal{O}\!\left(\eta^{\frac{|\beta|}{2}}h^{-|\beta|}\e^{-\frac{\asymp\eta^{\frac{3}{2}}}{h}}\right), 
      ~\text{for} ~ z\in\Omega_{\eta}^{a},
   \end{cases}
  \end{align}
where the first estimate follows from the Cauchy-Schwartz inequality. Note that
  \begin{align}\label{eqn:PfComGrushHag2}
   \lVert\partial_{z\overline{z}}^{\beta}( e_0 - \chi_+^{\bullet}e_+^{\bullet} )\rVert 
   \leq 
   \lVert \partial_{z\overline{z}}^{\beta}(e_{wkb}^{\bullet} - \chi_+^{\bullet}e_+^{\bullet} )\rVert 
   + 
   \lVert\partial_{z\overline{z}}^{\beta}( e_0 - e_{wkb}^{\bullet} )\rVert.
  \end{align}
By Proposition \ref{quasimodes} it remains to prove the desired estimate on 
$\lVert \partial_{z\overline{z}}^{\beta}(e_{wkb}^{\bullet} - \chi_+^{\bullet}e_+ ^{\bullet})\rVert$.
Recall the definition of the quasimodes $e_{wkb}^{\bullet}$ and $e_+^{\bullet}$ from 
Section \ref{subse:Quasmod} and from Proposition \ref{porp:QuasmodHager}. \\ 
\par
Let us first consider the case of $z\in\Omega$ with $\dist(\Omega,\partial\Sigma)>1/C$: 
recall from Proposition \ref{prop:GrushHager} that all $z$- and $\overline{z}$-derivatives of $\chi_+$ 
are bounded independently of $h>0$, whereas for the derivatives of $\chi_e$
we have (\ref{eqn:higherzzbarderchi_e}). Thus 
	  \begin{equation*}
	   \partial_{z\overline{z}}^{\beta}\chi_+,
	   \partial_{z\overline{z}}^{\beta}\chi_e= \mathcal{O}(h^{-|\beta|/2}).
	  \end{equation*}
Thus, since $\chi_e(-;z) \succ\chi_+$ for all $z\in\overline{\Omega}_i$, which implies 
that $x_+(z)\notin\supp (\chi_e(-;z)-\chi_+)$ for all $z\in\overline{\Omega}_i$, the
Leibniz rule then implies
  \begin{align}\label{eqn:PfComGrushHag1}
   \left\lVert\partial_{z\overline{z}}^{\beta}\left(\left(\chi_e(\cdot;z)
   - \chi_+ \right)\e^{\frac{i}{h}\phi_+(\cdot;z)}\right)\right\rVert
%   & =\left(\int \left|\partial_{z\overline{z}}^{\beta}\left(\left(\chi_e(-;z)
%    - \chi_+ \right)\e^{\frac{i}{h}\phi_+(-;z)}\right)\right|^2 dx\right)^{\frac{1}{2}} \notag \\
   \leq
    \mathcal{O}\!\left(h^{-|\beta|} \e^{-\frac{F}{h}}\right).
  \end{align}
where $F>0$ is given by the infimum of $\Ima\phi(x;z)$ over all $z\in\overline{\Omega}$ and all 
\begin{equation*}
 x\in\left(\bigcup_{z\in\overline{\Omega}}\supp(\chi_e(\cdot;z))\right)\backslash\{x\in I_+:~\chi_+\equiv 1\}. 
\end{equation*}
Note that 
$F>0$ is strictly positive because $x_-(z)\notin \overline{I_+}$ for all $z\in\overline{\Omega}$ and 
$\chi_+\in\mathcal{C}^{\infty}_0(I_+)$ (cf. Propositions \ref{prop:GrushHager} and \ref{porp:QuasmodHager}).\\ 
\par
Recall that $h^{-1/4}a(z;h)$ and $c_+(z;h)$ are the normalization factors 
of $e_{wkb}$ and $e_+$ (cf. (\ref{eqn:DefEWKB}) and Proposition \ref{porp:QuasmodHager}). 
Hence, for $z\in\Omega_i$, 
  \begin{align*}
   h^{-\frac{1}{4}}\partial_{z\overline{z}}^{\beta}a(z;h),
      \partial_{z\overline{z}}^{\beta}c_+(z;h)= 
   \mathcal{O}\!\left(h^{-(|\beta|+1/2)}\right).
  \end{align*}
Thus the Leibniz rule implies 
  \begin{align*}
    | \partial_{z\overline{z}}^{\beta}c_+(z;h) &-  \partial_{z\overline{z}}^{\beta}h^{-1/4}a(z;h)| 
   \notag\\
   &= 
   \left|\partial_{z\overline{z}}^{\beta}\frac{\left\lVert\left(\chi_e(\cdot;z)\e^{\frac{i}{h}\phi_+(\cdot;z)}\right)\right\rVert 
   - \left\lVert\left( \chi_+\e^{\frac{i}{h}\phi_+(\cdot;z)}\right) \right\rVert}
   {\left\lVert\left(\chi_e(\cdot;z)\e^{\frac{i}{h}\phi_+(\cdot;z)}\right)\right\rVert 
    \left\lVert\left( \chi_+\e^{\frac{i}{h}\phi_+(\cdot;z)}\right) \right\rVert}\right| \notag \\
   &= \mathcal{O}\!\left(h^{-(|\beta|+1/2)} \e^{-\frac{F}{h}}\right).
  \end{align*}
Since $h^{-\frac{1}{4}}a(z;h),c_+(z;h)= \mathcal{O}(h^{-\frac{1}{4}})$, 
the Leibniz rule and the above imply that for $z\in\Omega_i$
  \begin{align*}
   \lVert\partial_{z\overline{z}}^{\beta}\left( e_{wkb} - \chi_+e_+ \right)\rVert
    \leq \mathcal{O}\!\left(h^{-(|\beta|+1/2)} \e^{-\frac{F}{h}}\right).
  \end{align*} 
Thus there exists a constant $C>0$, for $h>0$ small enough, such that for $z\in\Omega_i$ 
  \begin{align}\label{eqn:PfEffHamCompEstimFF}
   \lVert\partial_{z\overline{z}}^{\beta}\left( e_{wkb} - \chi_+e_+ \right)\rVert =
     \mathcal{O}\!\left(h^{-|\beta|} \e^{-\frac{1}{Ch}}\right).
  \end{align}
Now let us consider the case $z\in\Omega\cap\Omega_{\eta}^{a,b}$: recall the quasimodes 
$e_{wkb}^{\eta}$ and $e_+^{\eta}$ as given in Definition \ref{def:quasmodCB} and 
Proposition \ref{porp:QuasmodHager}. A rescaling argument similar to the one in 
the proof of Proposition \ref{quasimodes} then implies 
  \begin{align*}
   \lVert\partial_{z\overline{z}}^{\beta}\left( e_{wkb}^{\eta} - \chi_+^{\eta}e_+^{\eta} \right)\rVert 
   = \mathcal{O}\!\left(\eta^{\frac{|\beta|+3/2}{2}}h^{-(|\beta|+1/2)}\e^{-\frac{\asymp\eta^{\frac{3}{2}}}{h}}\right).
  \end{align*}
Absorbing the factor $\eta^{3/4}h^{-1/2}$ into $ \e^{-\frac{\asymp\eta^{\frac{3}{2}}}{h}}$ then yields 
the desired estimate. 
\par
It is possible to achieve an analogous estimate for $R_- - R_-^{\bullet}$, namely that for all $z\in\Omega$ 
and for all $(n,m)\in\mathds{N}^2$
  \begin{align}\label{eqn:PfComGrushHag3_R-}
   &\lVert\partial_{z\overline{z}}^{\beta}( R_- - R_-^{\bullet} )\rVert_{\mathds{C}\rightarrow H^1(S^1)} =
   \lVert\partial_{z\overline{z}}^{\beta}( f_0 - \chi_-^{\bullet} e_-^{\bullet} ) \rVert \notag \\
   &= 
   \begin{cases}
    \mathcal{O}\!\left( h^{-|\beta|}\e^{-\frac{1}{Ch}}\right), ~\text{for} ~z\in\Omega,~\dist(\Omega,\partial\Sigma)>1/C,\\
    \mathcal{O}\!\left(\eta^{\frac{|\beta|}{2}}h^{-|\beta|}\e^{-\frac{\asymp\eta^{\frac{3}{2}}}{h}}\right), 
      ~\text{for} ~ z\in\Omega_{\eta}^{a},
   \end{cases}
  \end{align}
This can be achieved by analogous reasoning as for the estimate on $R_+ - R_+^{\bullet}$. \\
\\
\textbf{A formula for $E_{-+}$ } 
It is easy to see, that for $h>0$ small enough 
  \begin{equation*}
    \lVert (\mathcal{P} - \mathcal{P}^{\bullet})\mathcal{E}^{\bullet} \rVert_{\infty} \ll 1 .
  \end{equation*}
Thus, $1 + (\mathcal{P} - \mathcal{P}^{\bullet})\mathcal{E}^{\bullet}$ is invertible by the Neumann series, wherefore 
  \begin{equation*}
   \mathcal{P}\mathcal{E}^{\bullet}\left[1 + (\mathcal{P} - \mathcal{P}^{\bullet})\mathcal{E}^{\bullet}\right]^{-1} = 1.
  \end{equation*}
We conclude that 
  \begin{equation*} 
   \mathcal{E}= \mathcal{E}^{\bullet}\sum_{n\geq 0} (-1)^n\left[(\mathcal{P} - \mathcal{P}^{\bullet})\mathcal{E}^{\bullet}\right]^n.
  \end{equation*}
Define $g_-:=R_- - R_-^{\bullet} $ and $g_+:=R_+ - R_+^{\bullet}$. Hence, by Propositions \ref{prop:GrushHager} and 
\ref{prop:GrushBM} as well as by (\ref{eqn:PfComGrushHag2}) 
and (\ref{eqn:PfComGrushHag3}), there exists a constant $C>0$ such that 
  \begin{equation*} 
   (\mathcal{P} - \mathcal{P}^{\bullet})\mathcal{E}^{\bullet} =  \begin{pmatrix}
                                                  g_-E_-^{\bullet} & g_-E_{-+}^{\bullet} \\ 
                                                  g_+E^{\bullet} & g_+E_+^{\bullet} \\
                                                 \end{pmatrix}
                      =  \begin{pmatrix}
                         \mathcal{O}\!\left( \e^{-\frac{1}{Ch}}\right) & E_{-+}^{\bullet}\mathcal{O}\!\left( \e^{-\frac{1}{Ch}}\right) \\ 
                           \mathcal{O}\!\left( \e^{-\frac{1}{Ch}}\right) & \mathcal{O}\!\left( \e^{-\frac{1}{Ch}}\right) \\
                         \end{pmatrix}.
  \end{equation*}
By induction it follows that for $n\in\mathds{N}$
    \begin{equation*} 
   \left[(\mathcal{P} - \mathcal{P}^{\bullet})\mathcal{E}^{\bullet}\right]^n 
                      =  \begin{pmatrix}
                         \mathcal{O}\!\left( \e^{-\frac{n}{Ch}}\right) & E_{-+}^{\bullet}\mathcal{O}\!\left( \e^{-\frac{n}{Ch}}\right) \\ 
                           \mathcal{O}\!\left( \e^{-\frac{n}{Ch}}\right) & \mathcal{O}\!\left( \e^{-\frac{n}{Ch}}\right) \\
                         \end{pmatrix}.
  \end{equation*}
We conclude that 
  \begin{equation*}
   E_{-+}(z) = E_{-+}^{\bullet}\left(1 + \sum_{n\geq 1}\mathcal{O}\!\left( \e^{-\frac{n}{Ch}}\right)\right) 
	     = E_{-+}^{\bullet}\left(1 + \mathcal{O}\!\left( \e^{-\frac{1}{Ch}}\right)\right).
  \end{equation*}
Finally, by the estimates on $g_+$ and $g_+$ obtained above and by the 
estimates given in Proposition \ref{prop:EstimOnGrushElements} we conclude 
the desired estimates on the $z$- and $\overline{z}$-derivatives of the error term.
\end{proof}
%
% Chapter 5 
%
\section{Grushin problem for the perturbed operator $P_h^{\delta}$}\label{sec:EstEffectHamil}
For $\delta>0$ small enough, we can use the Grushin problem for the 
unperturbed operator $P_h$ to gain a well-posed Grushin 
problem for the perturbed operator $P_h^{\delta}$. 
\begin{prop}[\cite{SjAX1002}]\label{propGrushPertOp}
Let $z\in\Omega \Subset\Sigma$, let $h^{2/3} \ll \eta \leq \mathrm{const.}$ 
and let $R_-,R_+$ be as in Proposition \ref{propGrushUnpertOp}. Then
  \begin{equation*}
   \mathcal{P}_{\delta}(z):=\begin{pmatrix}
                    P_h^{\delta}-z & R_- \\ R_+ & 0 \\
                   \end{pmatrix}
   :~ H^1(S^1)\times \mathds{C}\longrightarrow L^2(S^1)\times \mathds{C}
  \end{equation*}
is bijective with the bounded inverse 
  \begin{equation*}
   \mathcal{E}_{\delta}(z) = \begin{pmatrix}
                     E^{\delta}(z) & E^{\delta}_+(z) \\ E^{\delta}_-(z) & E^{\delta}_{-+}(z) \\
                    \end{pmatrix}
  \end{equation*}
where 
  \begin{align*}
   &E^{\delta}(z)= E(z) + \mathcal{O}_{\eta^{-1/2}}\!\left(\delta h^{-2}\right)=\mathcal{O}(\eta^{-1/4}h^{-1/2})  \notag \\
   &E^{\delta}_-(z)=   E_-(z) + \mathcal{O}\!\left(\delta\eta^{-1/4} h^{-3/2}\right)=\mathcal{O}(1)  \notag \\
   &E^{\delta}_+(z)= E_+(z) + \mathcal{O}\!\left(\delta\eta^{-1/4} h^{-3/2}\right)=\mathcal{O}(1)  
  \end{align*}
and 
\begin{align}\label{eqn:estimateE-+}
    E_{-+}^{\delta}(z) &= E_{-+}(z) - \delta \left(E_-Q_{\omega}E_+ + 
			  \sum_{n=1}^{\infty}(-\delta)^n E_-Q_{\omega}(EQ_{\omega})^nE_+\right) \notag \\
	  &=  E_{-+}(z) - \delta \left(E_-Q_{\omega}E_+ + \mathcal{O}(\delta\eta^{-1/4} h^{-5/2}) \right)
\end{align} 
\end{prop}
\begin{proof}
The statement follows immediately from Proposition \ref{propGrushUnpertOp} by use of the Neumann series. 
\end{proof}
By (\ref{prop:GrushUnPert}) we get
  \begin{equation*}
    E_-Q_{\omega}E_+ = \sum\limits_{|j|,|k|\leq \left\lfloor\frac{C_1}{h}\right\rfloor} \alpha_{j,k} (e_0|e^k)\cdot (e^j|f_0)
		     = \sum\limits_{|j|,|k|\leq \left\lfloor\frac{C_1}{h}\right\rfloor} 
			\alpha_{j,k}~\widehat{e_0}(k) \overline{\widehat{f_0}(j)}.
  \end{equation*}
Recall from Proposition \ref{lem:truncation} that the random variables 
satisfy $\alpha\in B(0,C/h)$. For a more convenient notation we make 
the following definition:
\begin{defn}\label{def:firstordertermnota} 
For $x\in\mathds{R}$ we shall denote the Gauss brackets by 
$\lfloor x\rfloor:=\max\{k\in\mathds{Z}:~k\leq x\}$. Let $C_1>0$ be big enough as above and define $N:=(2\lfloor\frac{C_1}{h}\rfloor +1)^2$. 
For $z\in\Omega\Subset\mathring{\Sigma}$ let $X(z)=(X_{j,k}(z))_{|j|,|k|\leq \lfloor\frac{C_1}{h}\rfloor }\in\mathds{C}^{N}$ 
be given by 
  \begin{equation*}
	X_{j,k}(z) =  \widehat{e_0}(z;k) \overline{\widehat{f_0}(z;j)}, \quad\text{for }|j|,|k| \leq \left\lfloor\frac{C_1}{h}\right\rfloor.
  \end{equation*}
\end{defn}
Thus, for $z\in\Omega \Subset \Sigma$ and $\alpha\in B(0,C/h)\subset\mathds{C}^{N}$
  \begin{equation}\label{eqn:ShortCutE-+d}
   E_{-+}^{\delta}(z) = E_{-+}(z) - \delta\left[ X(z)\cdot\alpha + T(z;\alpha) \right],
  \end{equation}
where the dot-product $X(z)\cdot\alpha$ is the bilinear one, and 
  \begin{equation}\label{eqn:PowerSeriesT}
   T(z;\alpha) : = \sum_{n=1}^{\infty}(-\delta)^n E_-Q_{\omega}(EQ_{\omega})^nE_+ 
    =  \mathcal{O}(\delta\eta^{-1/4} h^{-5/2}),
  \end{equation}
where the estimate comes from Proposition \ref{propGrushPertOp}. Note that $T(z;\alpha)$ is $\mathcal{C}^{\infty}$ in $z$ and 
holomorphic in $\alpha$ in a ball of radius $C/h$, $B(0,C/h)\subset\mathds{C}^{N}$, 
by Proposition \ref{lem:truncation}. 
\begin{prop}\label{prop:EstimFourrierCoeffE0F0}
 Let $z\in\Omega\Subset\Sigma$, let $X(z)$ be as in Definition \ref{def:firstordertermnota}. 
 Let $h|k|\geq C$ for $C>0$ large enough, then the Fourier coefficients satisfy
  \begin{equation*}
   \widehat{e_0}(z;k), ~\widehat{f_0}(z;k) 
   =  \mathcal{O}\!\left(|k|^{-M}\dist(\Omega,\partial\Sigma)^{-\frac{M}{2}}\right), \quad 
   \dist(\Omega,\partial\Sigma) \gg h^{\frac{2}{3}}
  \end{equation*}
 for all $M\in\mathds{N}$. In particular
  \begin{equation*} 
    \lVert X(z)\rVert = 1 + \mathcal{O}\!\left(h^{\infty}\right).
  \end{equation*}
\end{prop}
\begin{proof} Will show the proof in the case of $e_0(z)$ since the case of $f_0(z)$ is similar. 
 Let us first suppose that $z\in\Omega$ with $\dist(\Omega,\partial\Sigma)>1/C$. 
 Recall the definition of the quasimode $e_{wkb}$ given in (\ref{eqn:DefEWKB}).
 By Proposition \ref{quasimodes2} 
  \begin{align*}
   \widehat{e_0}(z;k) = \int \left(e_{wkb}(z;x)
	  + \mO_{C^{\infty}}\!\left(\e^{-\frac{S}{2h}}\right)\right) \e^{-ikx} dx.
  \end{align*}
 For $k\in\mathds{Z}\backslash\{0\}$, repeated integration by parts using the operator 
  \begin{equation*}
   ^t L := \frac{i}{k}\frac{d}{dx}
  \end{equation*}
 applied to the error term yields by Proposition \ref{quasimodes2} 
 that for all $n\in\mathds{N}$
  \begin{align*}
   \widehat{e_0}(z;k) = \int e_{wkb}(z;x)\e^{-ikx} dx
	  + \mO\!\left(|k|^{-n} h^{\infty}\right).
  \end{align*}
 Define the phase function $\Phi(x,z) := (\phi_+(x,z)h^{-1} - kx)$. Since $h|k|\geq C$ is large 
 enough and since $\Omega$ is relatively compact, it follows that 
  \begin{equation*}
   |\partial_x\Phi(x,z)| = |\partial_x\phi_+(x,z)h^{-1}  - k|\geq  C_1|k| > 0.
  \end{equation*}
Repeated integration by parts using the operator
  \begin{equation*}
   ^t L' := \frac{1}{\partial_x\Phi(x,z)}D_x
  \end{equation*}
yields that for all $n\in\mathds{N}$ 
 \begin{align*}
    \int e_{wkb}(z;x) \e^{-ikx} dx = \mO\!\left(|k|^{-n}\right).
  \end{align*}
Thus, for all $n\in\mathds{N}$ 
  \begin{align*}
   \widehat{e_0}(z;k) = \mO\!\left(|k|^{-n}\right).
  \end{align*}
 For $z\in\Omega\cap\Omega_{\eta}^a$ one performs a similar rescaling argument 
 as in the proof of Proposition \ref{quasimodes}. Since in the rescaled coordinates 
 $\tilde{k} = \sqrt{\eta}k$, we conclude that for all $n\in\mathds{N}$
  \begin{equation*}
   |\widehat{e_0}(k)|\leq \mO\!\left(\eta^{-\frac{n}{2}}|k|^{-n}\right).
  \end{equation*}
 Finally, by definition \ref{def:firstordertermnota}, Parseval identity and the 
 estimates on the Fourier coefficients above, it follows that
  \begin{align*}
   \lVert X(z)\rVert^2 = \hspace{-0.1cm}\sum_{|j|,|k|\leq N} \hspace{-0.1cm}
   |\widehat{e_0}(z;j)|^2 |\widehat{f_0}(z;k)|^2 
   =(e_0(z)|e_0(z))(f_0(z)|f_0(z))  + \mathcal{O}\!\left(h^{\infty}\right).
  \end{align*}
Since $\lVert e_0\rVert,\lVert f_0\rVert=1$, we conclude the second statement of the Proposition.
\end{proof}
The following is an extension of Proposition \ref{prop:EstimFourrierCoeffE0F0}.
\begin{prop}\label{cor:EstimFourrierCoeffE0F0}
 Let $z\in\Omega\Subset\Sigma$, let $X(z)$ be as in Definition \ref{def:firstordertermnota}. 
 Let $h|k|\geq C$ for $C>0$ large enough, then for $\dist(\Omega,\partial\Sigma) \gg h^{\frac{2}{3}}$ and 
 for all $n,m\in\mathds{N}_0$ 
  \begin{equation*}
   \partial_z^n \partial_{\overline{z}}^m \widehat{e_0}(z;k),\partial_z^n \partial_{\overline{z}}^m\widehat{f_0}(z;k)  
   =  \left(|k|^{-M}\dist(\Omega,\partial\Sigma)^{-\frac{M}{2}}\right).
  \end{equation*}
 Furthermore, 
  \begin{equation*}
   \lVert \partial_z^n \partial_{\overline{z}}^mX(z) \rVert = 
    \mathcal{O}\!\left(\dist(\Omega,\partial\Sigma)^{\frac{n+m}{2}}h^{-(n+m)}\right).
  \end{equation*}
\end{prop}
\begin{proof}
 Since 
  \begin{align*}
   \partial_z^n \partial_{\overline{z}}^m\widehat{e_0}(z;k) &=
   \int \partial_z^n \partial_{\overline{z}}^me_0(z;x)\e^{-ikx} dx. 
  \end{align*}
  We then conclude similar to the proof of Proposition \ref{prop:EstimFourrierCoeffE0F0} that 
  for all $N\in\mathds{N}$
  \begin{equation*}
   |\partial_z^n \partial_{\overline{z}}^m\widehat{e_0}(z;k)| = \mathcal{O}\!\left(\eta^{-\frac{N}{2}}|k|^{-N}\right).
  \end{equation*}
The second statement of the Proposition is a direct consequence of Parseval's identity and 
Proposition \ref{quasimodes}.
\end{proof}
%
% Chapter 6
%
\section{Connections with symplectic volume and tunneling effects}\label{eqn:SecSymplVol}
The first two terms of the effective Hamiltonian $E_{-+}^{\delta}$ for the perturbed 
operator $P_h^{\delta}$ (cf. \eqref{eqn:ShortCutE-+d}) have a relation to the symplectic 
volume form on $T^*S^1$ and to the tunneling effects described in Section \ref{susec:tunnel}.
\subsection{Link with the symplectic volume}
\begin{prop}\label{lem:1MomSymplVol}
Let $z\in\Omega\Subset\Sigma$ and let $p$ and $\rho_{\pm}$ be as in Section \ref{sec:Intro}. 
Let $X(z)$ be as in Definition \ref{def:firstordertermnota}. 
Then we have for $h>0$ small enough and $h^{2/3}\ll \eta \leq \mathrm{const.}$
  \begin{equation*}
   (\partial_z X | \partial_z X) - \frac{\left|(\partial_z X | X) \right|^2}{\lVert X \rVert^2} = 
    \frac{1}{h}\left(\frac{i}{\{p,\overline{p}\}(\rho_{+}(z))}-\frac{i}{\{p,\overline{p}\}(\rho_{-}(z))}\right) 
    +\mathcal{O}(\eta^{-2}),
  \end{equation*}
where
  \begin{equation*}
   |\{p,\overline{p}\}(\rho_{\pm})| \asymp \sqrt{\eta}.
  \end{equation*}
 The $\partial_{z\overline{z}}^{\beta}$ derivatives of the error term 
 $\mathcal{O}(\eta^{-2})$ are of order 
 $\mathcal{O}\left(\eta^{\frac{|\beta|}{2}-2}h^{-\frac{|\beta|}{2}}\right)$.
\end{prop}
\begin{prop}\label{cor:1MomSymplVol}
 Let $z\in\Omega\Subset\Sigma$, let $p$ and $\rho_{\pm}$ be as in Section \ref{sec:Intro} and 
 let $d\xi \wedge dx$ be the symplectic form on $T^*S^1$. Then,
  \begin{align*}
   \frac{1}{h}\left(\frac{i}{\{p,\overline{p}\}(\rho_{+}(z))}-\frac{i}{\{p,\overline{p}\}(\rho_{-}(z))}\right)L(dz) 
%    &= 
%    \frac{1}{2h}\left(d\xi_{-}\wedge dx_{-} - d\xi_{+}\wedge dx_{+}\right) \notag \\
   = \frac{1}{2h}p_*(d\xi\wedge  dx) 
  \end{align*}
\end{prop}
\begin{proof}
 See \cite{Ha06} or \cite[Prop. 7.4]{SjAX1002}.
\end{proof}

To prove Proposition \ref{lem:1MomSymplVol} we first prove the 
following result. 
\begin{lem}\label{lem:proofSymplVol}
 Let $\Omega\Subset\Sigma$ such that $\dist(\Omega,\partial\Sigma)>1/C$ and let $g\in\mathcal{C}^{\infty}(\mathds{C})$ 
 and $\rho_{\pm}$ be as in Section \ref{sec:Intro}. Let $e_{wkb}$ and $f_{wkb}$ 
 be as in (\ref{eqn:DefEWKB}) and (\ref{eqn:DefFWKB}). Let $\Pi_{e_{wkb}}:L^2(S^1)\rightarrow L^2(S^1)$ 
 and $\Pi_{f_{wkb}}:L^2(S^1)\rightarrow L^2(S^1)$ denote the orthogonal projections onto the 
 subspaces spanned by $e_{wkb}$ and $f_{wkb}$ respectively. Then, 
 \begin{align*}
   &\lVert (1-\Pi_{e_{wkb}})\partial_ze_{wkb}(\cdot;z) \rVert^2   = \frac{-1}{2h\Ima g'(x_+(z))} + \mathcal{O}(1), \notag \\
  & \lVert (1-\Pi_{f_{wkb}})\partial_{\overline{z}}f_{wkb}(\cdot;z) \rVert^2   = \frac{1}{2h\Ima g'(x_-(z))} + \mathcal{O}(1).
  \end{align*}
\end{lem}
\begin{rem}
 In the following, we shall regard $z$ as a fixed parameter. Hence, by the support of functions 
 depending on both $x$ and $z$ we mean the support with respect to the variable $x$.
\end{rem}
\begin{proof}
 We will consider only the case of $e_{wkb}$ since the case of $f_{wkb}$ is similar. One calculates
  \begin{multline}\label{eqn:zDiffEWKB}
   \partial_z e_{wkb}(x;z) = h^{-\frac{1}{4}} 
	    \left\{\phantom{\frac{i}{h}}\partial_z\chi_e(x;z)a(z;h) + \chi_e(x;z)\partial_za(z;h) \right. \\
	    \left.+ \chi_e(x;z)a(z;h)\frac{i}{h}\partial_z\phi_+(x;z)\right\}\e^{\frac{i}{h}\phi_+(x;z)}.
  \end{multline}
Thus 
  \begin{multline}\label{eqn:lemWeylProjInt}
    (\partial_z e_{wkb}|e_{wkb})  = h^{-\frac{1}{2}}
    \int\bigg(\left(\partial_z\chi_e(x;z)\right)|a(z;h)|^2 + \left(\partial_za(z;h)\right)\overline{a(z;h)}\chi_e(x;z)   \\
    \left. +|a(z;h)|^2\chi_e(x;z)\frac{i}{h}\partial_z\phi_+(x;z)\right)\chi_e(x;z)\e^{-\frac{\Phi(x;z)}{h}} dx,
  \end{multline}
where
  \begin{equation}\label{eqn:lemPhaseSTPM}
   \Phi(x;z):=- i(\phi_+(x;z)-\overline{\phi_+(x;z)}) = 2\Ima \int_{x_+(z)}^x (z - g(y)) dy.
  \end{equation}
First, we will compute 
\begin{equation}\label{eqn:lemPfIntForStPM}
   h^{-\frac{1}{2}}\int\left(\partial_z\chi_e(x;z)\right) \chi_e(x;z)|a(z;h)|^2 \e^{-\frac{\Phi(x;z)}{h} } dx.
  \end{equation}
Using \eqref{eqn:higherzzbarderchi_e} and the fact that $\partial_z \chi_e(z,\cdot)$ 
has support in $]x_- - 2\pi, x_- -2\pi+ h^{1/2}[\cup ]x_- - h^{1/2}, x_-[$, Taylor 
expansion of $\Phi(\cdot,z)$ at $x_-$ and $x_- - 2\pi$ yields that
  \begin{equation*}
   \e^{-\frac{\Phi(x;z)}{h}} \leq \mathcal{O}\!\left(\e^{-\frac{2S}{h}}\right),
  \end{equation*}
uniformly in $]x_- - 2\pi, x_- -2\pi+ h^{1/2}[\cup ]x_- - h^{1/2}, x_-[$. Here 
$S$ is as in Definition \ref{defn:ZeroOrderDensity}. Now, applying this and 
\eqref{eqn:higherzzbarderchi_e} to (\ref{eqn:lemPfIntForStPM}), yields
  \begin{equation}\label{eqn:lemPfSttPhasRes1int}
   h^{-\frac{1}{2}}|a(z;h)|^2 \int\partial_z\chi_e(x;z)\chi_e(x;z)\e^{-\frac{\Phi(x;z)}{h} } dx 
   = \mathcal{O}\!\left(h^{-\frac{1}{2}}\e^{-\frac{2S}{h}}\right).
  \end{equation}
Next, we will treat the other two contributions to (\ref{eqn:lemWeylProjInt}). 
First, consider
\begin{equation*}
    h^{-\frac{1}{2}}\left(\partial_za(z;h)\right)\overline{a(z;h)}\int\chi_e(x;z)^2\e^{-\frac{\Phi(x;z)}{h}} dx.
  \end{equation*}
Since $h^{-\frac{1}{2}}|a(z;h)|^2$ is the normalization factor of $\lVert e_{wkb}\rVert^2$ we see that 
    \begin{equation}\label{eqn:lemPfSttPhasRes2int}
    h^{-\frac{1}{2}}\partial_za(z;h)\overline{a(z;h)}\int\chi_e(x;z)^2\e^{-\frac{\Phi(x;z)}{h}} dx 
    = \frac{\partial_za(z;h)}{a(z;h)}.
  \end{equation}
Let us now turn to the third contribution to (\ref{eqn:lemWeylProjInt})
  \begin{equation*}
   I_h:=h^{-\frac{1}{2}}|a(z;h)|^2\int\frac{i}{h}\partial_z\phi_+(x;z)\chi_e(x;z)^2\e^{-\frac{\Phi(x;z)}{h}} dx .
  \end{equation*}
The stationary phase method implies together with (\ref{eqn:AssympA0B0}) that
    \begin{equation}\label{eqn:lemPfSttPhasResint3}
    I_h = \frac{i}{h}\partial_z\phi_+(x_+(z);z)+   \mathcal{O}(1).
  \end{equation}
Thus, by combining (\ref{eqn:lemPfSttPhasRes1int}), (\ref{eqn:lemPfSttPhasRes2int}) and (\ref{eqn:lemPfSttPhasResint3}) 
  \begin{equation*}
    (\partial_z e_{wkb}|e_{wkb}) = 
		 \frac{i}{h}\partial_z\phi_+(x_+(z);z) + \mathcal{O}(1)
  \end{equation*}
and thus
  \begin{multline}\label{eqn:ProjewkbonzDiffeWKB}
   (\partial_z e_{wkb}|e_{wkb})e_{wkb}(x;z) \\ 
     = h^{-\frac{1}{4}}\left\{a(z;h)\frac{i}{h}\partial_z\phi_+(x_+(z);z)
	+\mathcal{O}(1)\right\}\chi_e(x;z) \e^{\frac{i}{h}\phi_+(x;z)} .
  \end{multline}
Subtract (\ref{eqn:ProjewkbonzDiffeWKB}) from (\ref{eqn:zDiffEWKB}) and note that the term 
$a(z;h)\partial_z\chi_e(x;z)\e^{\frac{i}{h}\phi_+(x;z)}$ is exponentially small in $h$ like in (\ref{eqn:lemPfSttPhasRes1int}). 
Thus 
  \begin{align}\label{eqn:FMFDpfTRI}
   &(1-\Pi_{e_{wkb}})\partial_ze_{wkb}(x;z)\notag \\ 
   &=\frac{\e^{\frac{i}{h}\phi_+(x;z)}}{h^{1/4}}
   \left\{ a(z;h)\chi_e(x;z)\frac{i}{h}\left(\partial_z\phi_+(x;z)-\partial_z\phi_+(x_+(z);z)\right) 
    \right\}  +\mathcal{O}_{L^2}(1).
  \end{align}
It remains to treat 
  \begin{multline}\label{eqn:STPInt4Weyl}
   I_h:=\left\lVert a(z;h)\chi_e(x;z)\frac{i}{h^{\frac{5}{4}}}\left(\partial_z\phi_+(x;z)-\partial_z\phi_+(x_+(z);z)\right)\e^{\frac{i}{h}\phi_+(x;z)}\right\rVert^2 \\
   =h^{-\frac{1}{2}} \int \chi_e(x;z)^2|a(z;h)|^2 \left| \frac{i}{h}\left(\partial_z\phi_+(x;z)-\partial_z\phi_+(x_+(z);z)\right) \right|^2 \e^{-\frac{\Phi(x;z)}{h}} dx,
  \end{multline}
where $\Phi(x;z)$ is given in \eqref{eqn:lemPhaseSTPM}. This can be done by the stationary phase method, 
as in the proof of Lemma \ref{lem:AssympA0B0BySPM}. Thus 
  \begin{equation*}
    I_h =  \sqrt{2\pi}\sum_{n=0}^{N}\frac{1}{n!}
    \left(\frac{h}{2}\right)^n(\Delta^n_y u)(0) + \mathcal{O}(h^{N+1}),
  \end{equation*}
where 
  \begin{align*}
   u(y) = \chi_e(\kappa^{-1}(y);z)^2\frac{|a(z;h)|^2}{ |\kappa'(\kappa^{-1}(y))|}
	      \left| \frac{i}{h}\left(\partial_z\phi_+(\kappa^{-1}(y);z) 
	      -\partial_z\phi_+(x_+(z);z)\right) \right|^2
  \end{align*}
and $\kappa:~V\rightarrow U$ is a local $\mathcal{C}^{\infty}$ diffeomorphism 
from $V \subset\mathds{R}$, a neighborhood of $x_+(z)$, to $U\subset\mathds{R}$, 
a neighborhood of $0$, such that 
  \begin{equation*}
   \Phi(\kappa^{-1}(x);z) = \Phi(x_+(z);z) + \frac{x^2}{2},
  \end{equation*}
$\kappa^{-1}(0) = x_+(z)$ and 
  \begin{equation}\label{eqn:SPMCOV}
      \frac{d\kappa}{dx}(x_+(z)) = |\partial^2_{xx}\Phi(x_+(z);z)|^{\frac{1}{2}} = \sqrt{-2\Ima g'(x_+(z))} \neq 0.
  \end{equation}
This implies that $u(0)=0$ and thus we have to calculate the second order term 
in the above asymptotics, i.e. $\Delta_y u(y)$ is equal to 
  \begin{align*}
   &  \left(\Delta_y\chi_e(\kappa^{-1}(y);z)^2\frac{|a(z;h)|^2}{ |\kappa'(\kappa^{-1}(y))|}\right)
	      \left| \frac{i}{h}\left(\partial_z\phi_+(\kappa^{-1}(y);z) 
	      -\partial_z\phi_+(x_+;z)\right) \right|^2 \notag \\
	      &+2\frac{d}{dy}\left(\chi_e(\kappa^{-1}(y);z)^2\frac{|a(z;h)|^2}{ |\kappa'(\kappa^{-1}(y))|}\right)
	      \frac{d}{dy} \frac{1}{h^2}\left| \partial_z\phi_+(\kappa^{-1}(y);z) 
	      -\partial_z\phi_+(x_+;z)\right|^2\notag \\
	      &+\chi_e(\kappa^{-1}(y);z)^2\frac{|a(z;h)|^2}{ |\kappa'(\kappa^{-1}(y))|}
	      \Delta_y \left(\left| \frac{i}{h}\left(\partial_z\phi_+(\kappa^{-1}(y);z) 
	      -\partial_z\phi_+(x_+;z)\right) \right|^2\right).
  \end{align*}
Note that at $y=0$ the first and the second term of the right hand side vanish. By (\ref{eqn:DerzOfPhasePhi+})
  \begin{equation*}
   \Delta_y \left.\left(\left| \frac{i}{h}\left(\partial_z\phi_+(\kappa^{-1}(y);z) 
	      -\partial_z\phi_+(x_+(z);z)\right) \right|^2\right)\right|_{y=0}
	      = 2 h^{-2}\left|\frac{d}{dy}\kappa^{-1}(0)\right|^2.
  \end{equation*}
Thus, since $\chi_e(\kappa^{-1}(0);z) = \chi_e(x_+(z);z) =1$ (cf. Definition \ref{def:Quasimodes}),
  \begin{equation*}
   (\Delta_y u)(0) = \frac{2|a(z;h)|^2}{h^2|\kappa'(x_+(z))|^3}.
  \end{equation*}
Using (\ref{eqn:SPMCOV})and (\ref{eqn:AssympA0B0}), we have that 
  \begin{equation*}
   (\Delta_y u)(0) = \frac{1}{\sqrt{2\pi}h^2}\left(-\Ima g'(x_+(z))\right)^{-1} + \mathcal{O}(h^{-1})
  \end{equation*}
which yields
  \begin{equation*}
  I_h = \frac{-1}{2h\Ima g'(x_+(z))} + \mathcal{O}(1).
  \end{equation*}
This, together with (\ref{eqn:FMFDpfTRI}), yields
  \begin{equation*}
   \lVert (1-\Pi_{e_{wkb}})\partial_ze_{wkb}(-;z) \rVert^2   = \frac{-1}{2h\Ima g'(x_+(z))} + \mathcal{O}(1). \qedhere
  \end{equation*}
\end{proof}
\begin{proof}[Proof of Proposition \ref{lem:1MomSymplVol}]
Recall that $e_0(z)$ (respectively $f_0(z)$) denotes an 
eigenfunction of the $z$-dependent operator $Q(z)$ 
(respectively $\tilde{Q}(z)$). 
Using Definition \ref{def:firstordertermnota}, Proposition \ref{prop:EstimFourrierCoeffE0F0}, 
Corollary \ref{cor:EstimFourrierCoeffE0F0} and the Parseval identity one 
computes that 
  \begin{align*}
    &(\partial_z X | \partial_z X) - \frac{\left|(\partial_z X | X) \right|^2 }{\lVert X \rVert^2} = \notag \\ 
    & = (\partial_ze_0|\partial_ze_0) -|(\partial_ze_0|e_0)|^2 
      +(\partial_{\overline{z}}f_0|\partial_{\overline{z}}f_0) - |(f_0|\partial_{\overline{z}}f_0)|^2
     +\mathcal{O}\!\left(h^{\infty}\right).
  \end{align*}
Suppose that $z\in\Omega$ with $\dist(\Omega,\partial\Sigma)>1/C$. 
By Corollary \ref{QMEstimUnifZ} it then follows that 
$(\partial_ze_0|\partial_ze_0) -|(\partial_ze_0|e_0)|^2$ is equal to 
  \begin{align*}
    (\partial_ze_{wkb}|\partial_ze_{wkb}) 
   -|(\partial_ze_{wkb}|e_{wkb})|^2
   +    \mathcal{O}\!\left(h^{-1}\e^{-\frac{1}{Ch}}\right).
  \end{align*}
Let $\Pi_{e_{wkb}}$ and $\Pi_{f_{wkb}}$ be as in Lemma \ref{lem:proofSymplVol} and note that  
  \begin{align}\label{eqn:PfLemSymVolProjID}
   &\lVert (1-\Pi_{e_{wkb}})\partial_ze_{wkb}\rVert^2 = \lVert \partial_ze_{wkb}\rVert^2  
   -|(\partial_ze_{wkb}|e_{wkb})|^2 \quad \text{and} \notag \\
   &\lVert (1-\Pi_{f_{wkb}})\partial_zf_{wkb}\rVert^2 = \lVert \partial_zf_{wkb}\rVert^2  
   -|(\partial_zf_{wkb}|f_{wkb})|^2.
  \end{align}
Hence
  \begin{align}\label{eqn:ProofLemmaSymplecVol1Mom}
   (\partial_z X | \partial_z X) - \frac{\left|(\partial_z X | X) \right|^2 }{\lVert X \rVert^2} = 
   &\lVert (1-\Pi_{e_{wkb}})\partial_ze_{wkb}\rVert^2 
     + \lVert (1-\Pi_{f _{wkb}})\partial_{\overline{z}}f_{wkb}\rVert^2 \notag \\
     &+\mathcal{O}\!\left(h^{-1}\e^{-\frac{1}{Ch}}+h^{\infty}\right).
  \end{align}
Since $\{p,\overline{p}\}(\rho_{\pm}) = -2i\Ima g'(x_{\pm})$, 
it follows by Lemma \ref{lem:proofSymplVol} and (\ref{eqn:ProofLemmaSymplecVol1Mom}) that
  \begin{equation}\label{eqn:PfLemWeylEndRes}
   (\partial_z X | \partial_z X) - \frac{\left|(\partial_z X | X) \right|^2 }{\lVert X \rVert^2}= 
   \frac{1}{h}\left(\frac{i}{\{p,\overline{p}\}(\rho_{+}(z))}-\frac{i}{\{p,\overline{p}\}(\rho_{-}(z))}\right) + \mathcal{O}(1)
  \end{equation}
\\
Now let us consider the case where $z\in\Omega\cap\Omega_{\eta}^{a}$. Similar to 
Lemma \ref{lem:proofSymplVol} we get that 
  \begin{align*}
   &\lVert (1-\Pi_{e^{\eta}_{wkb}})\partial_{\tilde{z}}e^{\eta}_{wkb}(\cdot;\tilde{z}) 
   \rVert_{L^2(S^1/\sqrt{\eta},\sqrt{\eta}d\tilde{x})}^2   
   = \frac{-1}{2\tilde{h}\Ima g'(x_+(z))} + \mathcal{O}(1), \notag \\
   &\lVert (1-\Pi_{f^{\eta}_{wkb}})\partial_{\overline{\tilde{z}}}f^{\eta}_{wkb}(\cdot;\tilde{z})
   \rVert_{L^2(S^1/\sqrt{\eta},\sqrt{\eta}d\tilde{x})}^2   
   = \frac{1}{2\tilde{h}\Ima g'(x_-(z))} + \mathcal{O}(1),
  \end{align*}
where $|\Ima g'(x_{\pm}(z))|\asymp\sqrt{\eta}$. A rescaling argument 
similar to the one in the proof of Proposition \ref{quasimodes} and Corollary \ref{QMEstimUnifZ} then 
imply 
 \begin{align*}
   (\partial_ze_0|\partial_ze_0) -|(\partial_ze_0|e_0)|^2  
   =  \frac{-1}{2\tilde{h}\Ima g'(x_+(z))} + \mathcal{O}(\eta^{-2})
  \end{align*}
and similar for $(\partial_{\overline{z}}f_0|\partial_{\overline{z}}f_0) - |(f_0|\partial_{\overline{z}}f_0)|^2$. 
Hence,
  \begin{equation*}
   (\partial_z X | \partial_z X) - \frac{\left|(\partial_z X | X) \right|^2 }{\lVert X \rVert^2}= 
   \frac{1}{h}\left(\frac{i}{\{p,\overline{p}\}(\rho_{+}(z))}-\frac{i}{\{p,\overline{p}\}(\rho_{-}(z))}\right) 
   + \mathcal{O}(\eta^{-2})
  \end{equation*}
with $|\{p,\overline{p}\}(\rho_{\pm}(z))|\asymp\sqrt{\eta}$. The statement on the derivatives of the error estimates 
follow by the Stationary phase method and the usual rescaling argument. 
  \end{proof}
\subsection{Link with the tunneling effects}
We will prove the following result in the light of Proposition \ref{prop:AbsSPe0f0}.
\begin{prop}\label{prop:EstimAndFormulaDerzE-+}
 Let $z\in\Omega\Subset\Sigma$, let $X(z)$ be as in Definition \ref{def:firstordertermnota} and let 
 $E_{-+}(z)$ be as in Proposition \ref{propGrushUnpertOp}. Let $S$ be as in Definition 
 \ref{defn:ZeroOrderDensity}. Then, 
  \begin{equation*}
   \left| \partial_z E_{-+}(z)- E_{-+}(z)\frac{(\partial_z X(z)|X(z))}{\lVert X(z)\rVert^2} - (e_0|f_0) \right| 
   \leq \mathcal{O}\left(h^{\infty}\e^{-\frac{S}{h}}\right).
  \end{equation*}
\end{prop}
\begin{proof}[Proof of Proposition \ref{prop:EstimAndFormulaDerzE-+}]
Apply the $\partial_z$ derivative to the first equation in (\ref{eqn:RelationEjFj}),
  \begin{equation*}
   (P_h-z)\partial_z e_0 - e_0 = \partial_z \alpha_0 \cdot f_0 + \alpha_0 \partial_z f_0.
  \end{equation*}
Taking the scalar product with $f_0$ (which is $L^2$-normalized) then yields
  \begin{equation*}
   (\partial_ze_0|(P_h-z)^*f_0) - (e_0|f_0) = \partial_z \alpha_0 + \alpha_0 (\partial_zf_0|f_0).
  \end{equation*}
Recall from Proposition \ref{propGrushUnpertOp} that $E_{-+}(z) = -\alpha_0(z)$ and use the second 
equation in (\ref{eqn:RelationEjFj}) to see 
  \begin{equation}\label{eqn:ExactDerE}
   \partial_z E_{-+}(z) - E_{-+}(z)( (\partial_ze_0|e_0) - (\partial_zf_0|f_0) )  - (e_0|f_0) = 0.
  \end{equation}
By Definition \ref{def:firstordertermnota} we have the following identity 
  \begin{align*}
    (\partial_z X | X)= \hspace{-0.2cm}\sum_{|j|,|k|<\frac{C_1}{h}}\hspace{-0.2cm}
	    \left(\partial_z\widehat{e}_0(z;j)\overline{\widehat{f}_0(z;k)}+\widehat{e}_0(z;j)\overline{\partial_{\overline{z}}\widehat{f}_0(z;k)}\right)
	    \left(\overline{\widehat{e}_0(z;j)}\widehat{f}_0(z;k)\right).
  \end{align*}
Proposition \ref{prop:EstimFourrierCoeffE0F0}, Corollary \ref{cor:EstimFourrierCoeffE0F0} and the Parseval identity then imply 
  \begin{align}\label{eqn:SPDerXX}
    \frac{(\partial_z X | X)}{\lVert X \rVert^2} =
     (\partial_ze_0|e_0) + (f_0|\partial_{\overline{z}}f_0) +\mathcal{O}\!\left(h^{\infty}\right).
  \end{align}
Note that in the above we also used that $e_0$ and $f_0$ are normalized. 
Since $(f_0|\partial_{\overline{z}} f_0) = - (\partial_zf_0|f_0)$ we conclude by the triangular inequality 
   \begin{equation*}
   \left| \partial_z E_{-+}(z)- E_{-+}(z)\frac{(\partial_z X(z)|X(z))}{\lVert X(z)\rVert^2} - (e_0|f_0) \right| 
   \leq \mathcal{O}(h^{\infty})|E_{-+}(z)|.
  \end{equation*}
The statement of the proposition then follows by the estimate 
$|E_{-+}(z)| = \mathcal{O}\!\left(\eta^{\frac{1}{4}}h^{\frac{1}{2}}\e^{-\frac{S}{h}}\right)$ 
given in Proposition \ref{prop:EstimDerzE-+}.
\end{proof}
%
% Chapter 7 
%
\section{Preparations for the distribution of eigenvalues of $P_h^{\delta}$}\label{sec:PrepLemmas}
To calculate the intensity measure of $\Xi$ we make use of the following observations: 
\subsection{Counting zeros}
\begin{lem}\label{lem:approxdelta}
 Let $\Omega\subset\mathds{C}$ be open and convex and let $g,F:\Omega \longrightarrow \mathds{C}$ be ${C}^{\infty}$ such that 
 $g\not\equiv0$ and 
  \begin{equation}\label{lem:dbarequ}
   \partial_{\overline{z}}g(z) + \partial_{\overline{z}}F(z)\cdot g(z)=0
  \end{equation}
holds for all $z\in\Omega$. The zeros of $g$ form a discrete set of locally finite multiplicity.The notion of multiplicity here is 
the same as for holomorphic functions, more details can be found in the proof. Furthermore, 
for all $\varphi\in\mathcal{C}_0(\Omega)$
  \begin{equation*}
    \left\langle \chi\left(\frac{g}{\varepsilon}\right)\frac{1}{\varepsilon^2}\left|\partial_z g \right|^2,\varphi \right\rangle 
    \longrightarrow \sum_{z\in g^{-1}(0)} \varphi(z), \quad\varepsilon\rightarrow 0,
  \end{equation*}
where $\chi\in\mathcal{C}^{\infty}_0(\C)$ such that $\chi\geq 0$ and $\int\chi(w) L(dw) = 1 $ and the zeros are counted 
according to their multiplicities.
\end{lem}
\begin{proof}
(\ref{lem:dbarequ}) implies that 
  \begin{equation}\label{eqn:pflemholfunc}
   \e^{F(z)}g(z)
  \end{equation}
is holomorphic in $\Omega$. $g$ has the same zeros as the holomorphic function (\ref{eqn:pflemholfunc}). 
Thus, the zeros of $g$ in $\Omega$ form a discrete set and the notion of the multiplicity of the zeros of $g$ is well-defined since 
we can view the zeros as those of a holomorphic function.\par
Let $z_0\in g^{-1}(0)$ have multiplicity $n$. There exists a neighborhood $W\subset\Omega$ of $z_0$ such that 
$W\cap g^{-1}(0) = {z_0}$. Since $\e^{F(z)}g(z)$ is holomorphic, there exists a neighborhood $U\subset\Omega$ of $z_0$ 
and a holomorphic function $f:U\rightarrow\mathds{C}$ such that for all $z\in U$ 
 \begin{equation*}
   f(z)\neq 0, \quad\text{and}\quad \e^{F(z)}g(z) = f(z)(z-z_0)^n.
 \end{equation*}
Choose a $\lambda>0$ such that $|\e^{-F(z)}f(z)-\e^{-F(z_0)}f(z_0)| < |\e^{-F(z_0)}f(z_0)|$ for $|z-z_0|<\lambda$. In this disk we 
can define a single-valued branch of $\sqrt[n]{\e^{-F(z)}f(z)}$.\par 
We take a test function $\varphi\in\mathcal{C}_0(\Omega)$ with 
  \begin{equation}\label{eqn:lemApproxDeltaLocalTestfctRestr}
   \supp\varphi \subset (U\cap W \cap\{z:~|z-z_0|<\lambda\})=:N
  \end{equation}
and consider for $\varepsilon>0$
  \begin{align*}
   \left\langle \chi\left(\frac{g}{\varepsilon}\right)\frac{1}{\varepsilon^2}\left|\partial_z g \right|^2,\varphi \right\rangle  
   = \frac{1}{\varepsilon^2}\int\limits_N \chi\left(\frac{g(z)}{\varepsilon}\right)|\partial_z g(z)|^2\varphi(z) L(dz).
  \end{align*}
Let us perform a change of variables. Define 
  \begin{equation}\label{eqn:lemapproxdelInvert}
   w := g(z) = (z-z_0)^n\e^{-F(z)}f(z),
  \end{equation}
Since 
  \begin{align*}
   &\partial_z w(z) = (z-z_0)^{n-1}\e^{-F(z)}\left(nf(z) +(z-z_0)(\partial_z f(z) - \partial_z F(z)f(z))\right),\notag \\
   &\partial_{\overline{z}} w(z_0) = 0,
  \end{align*}
the implicit function theorem implies that we can invert equation (\ref{eqn:lemapproxdelInvert}) for $z$ 
in a small neighborhood of $z_0$ without $\{z_0\}$, say the disk $D(z_0,r)\backslash\{z_0\}$ for some radius $r>0$, 
and $w$ in the $n$-fold covering surface of $w(D(z_0,r)\backslash\{z_0\})$. 
Thus, if we denote the domain on each leaf of the covering by $B_k$, 
for $k=1,\dots,n$, as a subset of $\mathds{C}$, and the respective branch of $g$ by $g_k$ 
we get for $\varepsilon>0$ small enough
 \begin{equation*}
  \left\langle \chi\left(\frac{g}{\varepsilon}\right)\frac{1}{\varepsilon^2}\left|\partial_z g \right|^2,\varphi \right\rangle  =
   \sum_{k=1}^n \frac{1}{\varepsilon^2} 
	\int_{B_k} \varphi(g_k^{-1}(w))\chi\left(\frac{w}{\varepsilon}\right)(1+\mathcal{O}(w^2)) L(dw),
 \end{equation*}
with $g_k^{-1}(0)=z_0$. In the above we used that 
  \begin{equation*}
   L(dw) = \left(|\partial_z g(z)|^2 - |\partial_{\overline{z}} g(z)|^2\right)Ld(z)
  \end{equation*}
and the $\overline{\partial}$-equation (\ref{lem:dbarequ}) which implies 
  \begin{equation*}
   |\partial_{\overline{z}} g(z)|^2 = |\partial_{\overline{z}} F(z)g(z)|^2 \asymp w^2.
  \end{equation*}
Thus we can conclude 
  \begin{equation}\label{eqn:lemdelapprox_Local}
   \left\langle \chi\left(\frac{g}{\varepsilon}\right)\frac{1}{\varepsilon^2}\left|\partial_z g \right|^2,\varphi \right\rangle  
   \longrightarrow   \sum_{k=1}^n \varphi(z(0))= n\varphi(z_0), \quad\text{for }\varepsilon
      \rightarrow 0.
  \end{equation}
Since $g$ has at most countably many zeros in $\Omega$, there exists some index set $I\subset\mathds{N}$ such that we can denote the set 
of zeros of $g$ in $\Omega$ by $\{z_i\}_{i\in I}:= g^{-1}(0)\cap\Omega$. Furthermore, let $m(i)$ for all $i\in I$ denote the multiplicity 
of the respective zero $z_i$. \\
For each zero $z_i$ we can construct a neighborhood $N_i$, as above, such that for a test function with support in $N_i$ we have 
the convergence as in (\ref{eqn:lemdelapprox_Local}). By potentially shrinking the $N_i$ we can gain $N_i\cap N_j = \emptyset$ for 
$i\neq j$. Consider the following locally finite open covering of $\Omega$
  \begin{equation*}
   \Omega=\left(\bigcup_{i\in I} N_i \right)\cup\left(\Omega\backslash\{z_i:~i\in I\}\right).
  \end{equation*}
Let $\{\chi_i\}_{i\in I\cup \{0\}}$ be a partition of unity subordinate to this open covering such that 
  \begin{equation*}
   1 = \sum_{i\in I} \chi_i + \chi_0.
  \end{equation*}
Here $\chi_i \in \mathcal{C}^{\infty}_0(N_i)$ and $\chi_i \equiv 1$ in a neighborhood of $z_i$ for all $i\in I$. Furthermore, 
$\chi_0 \in \mathcal{C}^{\infty}(\Omega)$ and $z_i \notin \supp\chi_0$ for all $i\in I$. 
Let $\varphi\in\mathcal{C}_0(\Omega)$ be an arbitrary test function. By (\ref{eqn:lemdelapprox_Local}) we have for $\varepsilon\rightarrow 0$
  \begin{align*}
   \left\langle \chi\left(\frac{g}{\varepsilon}\right)\frac{1}{\varepsilon^2}\left|\partial_z g \right|^2,\varphi \right\rangle &=
   \sum_{i\in I}\left\langle \chi\left(\frac{g}{\varepsilon}\right)
    \frac{1}{\varepsilon^2}\left|\partial_z g \right|^2,\chi_i \varphi \right\rangle 
    \rightarrow \sum_{i\in I} m(i) \chi_i(z_i) \varphi(z_i).
  \end{align*}
Since $g(z)\neq 0$ for all $z\in\supp\chi_0$ we have for $\varepsilon>0$ small enough 
  \begin{equation*}
   \left\langle \chi\left(\frac{g}{\varepsilon}\right)\frac{1}{\varepsilon^2}\left|\partial_z g \right|^2,\chi_0\varphi \right\rangle = 0
  \end{equation*}
and we can conclude the statement of the Lemma.
\end{proof}
\subsection{An implicit function theorem}
\begin{lem}\label{lem:globalInvFunThm}
Let $R>0$ and $a>c\geq 0$ be constants. Let $D(0,R)\subset\mathds{C}$ be the open disk of radius 
$R$ centered at $0$ and let ${g,f:D(0,R)\longrightarrow \mathds{C}}$ be holomorphic such that 
  \begin{equation}\label{lem:globalInvFunThm:assump}
   \lVert g\rVert_{\infty}\leq c, \quad\text{and for all}~ z\in D(0,R):~ \partial_{z} f(z) = a + g(z).
  \end{equation}
Assume that
  \begin{equation*}
    \xi\in D\!\left(f(0),(a-c)R\right)\subset\mathds{C}.
  \end{equation*}
Then the equation
  \begin{equation*}
   f(z) = \xi
  \end{equation*}
has exactly one solution $z=z(\xi)\in D(0,R)$ and it depends holomorphically on $\xi$.
\end{lem}
\begin{proof} For $z\in D(0,R)$
  \begin{equation*}
   f(z) = \int_{0}^{z} \left(a + g(w)\right)dw + f(0) = az + f(0) + G(z),
  \end{equation*}
where $G(z) := \int_{0}^{z}g(w)dw$. Now let us consider the equation 
  \begin{equation*}
   az + f(0) - \xi = 0.
  \end{equation*}
It has a unique solution in the disk $D(0,R)$ since 
  \begin{equation*}
   \frac{|\xi - f(0)|}{a} < \frac{|a-c|}{a}R < R. 
  \end{equation*}
Now consider for $\varepsilon>0$ and for $z\in D(0,R-\varepsilon)$ the equation 
  \begin{equation*}
   f(z) - \xi = az + f(0) - \xi + G(z) = 0.
  \end{equation*}
Recall that $\xi\in D\!\left(f(0),(a-c)R\right)$ which implies that there exists a 
$\varepsilon(\xi)>0$ such that $|\xi - f(0)| \leq (a-c)(R-\varepsilon(\xi))$. 
Thus for all $\varepsilon <\varepsilon(\xi)$ 
 \begin{equation*}
  |az + f(0) - \xi| \geq  |az| - |f(0) - \xi| > a|z| 
	    -(a-c)(R-\varepsilon)
  \end{equation*}
and, using that $  |G(z)| \leq c |z| $, we may conclude that for $|z|= R-\varepsilon$
 \begin{equation*}
  |G(z)| < |az + f(0) - \xi|.
 \end{equation*}
By Rouch\'e's theorem we have that $az + f(0) - \xi$ and $f(z)-\xi$ have the same number of zeros in the disk 
$D\left(0,R-\varepsilon\right)$. We also see that $f(z)-\xi$ has no zero in 
$D\left(0,R\right)\backslash D\left(0,R-\varepsilon\right)$ and the result follows. 
\end{proof}
\begin{prop}\label{cor:globalInvFunThm}
Let $a>c\geq 0$ be constants, $n\in\mathds{N}$, let $\Omega\subset\mathds{C}^n$ be open, bounded and of the form 
  \begin{equation*}
   \Omega = \{z =(z',z_n)\in\mathds{C}^n~:~z'\in\Omega',~|z_n| < R_{z'}\} 
  \end{equation*}
where $R_{z'}>0$ is continuous in $z'$. Furthermore, assume that 
  \begin{itemize}
   \item ${g,F:\Omega\longrightarrow \mathds{C}}$ are holomorphic such that 
  \begin{equation}\label{cor:globalInvFunThm:assump}
   \lVert g\rVert_{\infty}\leq c, \quad\text{and for all}~ z\in\Omega:~ \partial_{z_n} F(z) = a + g(z),
  \end{equation}  
  \item $\Gamma\Subset\Omega'$ is open so that $\inf\limits_{z'\in\Gamma}R_{z'} \geq \mathrm{const.}>0$,
  \item $\xi\in \bigcap_{z'\in\Gamma}D\!\left(F(z',0),(a-c)R_{z'}\right)\subset\mathds{C}$.
  \end{itemize}
Then, when $z'\in\Gamma$, the equation
  \begin{equation*}
   F(z',z_n) = \xi
  \end{equation*}
has exactly one solution $z_n(z',\xi)\in D(0,R_{z'})$ 
and it depends holomorphically on $\xi$ and on $z'\in\Gamma$.
\end{prop}
\begin{proof} 
Lemma \ref{lem:globalInvFunThm} implies the existence an uniqueness of the solutions 
\\
$z_n(z',\xi)$ in each disk $D(0,R_{z'})$.
By (\ref{cor:globalInvFunThm:assump}) it follows that 
  \begin{equation*}
   \frac{\partial F}{\partial z_n}(z',z_n(z',\xi)) \neq 0 
  \end{equation*}
for all $z' \in\Gamma$ and all 
$\xi\in  D\!\left(F(z',0),(a-c)(R_{z'}-\lambda)\right)$. Hence, the implicit function theorem 
implies that $z_n(z',\xi)$ depends holomorphically on $\xi$ and $z'$.
\end{proof}
%
% Chapter 8
%
\section{A formula for the intensity measure of the point process of eigenvalues of $P_h^{\delta}$}\label{sec:FIntMesForm}
We prove the following formula for the intensity measure of $\Xi$:

\begin{prop}\label{prop:formulaFor1Intensity}
Let $h^{2/3}\ll \eta < \mathrm{const.}$ and let $\Omega:=\Omega_{\eta}^a\Subset\Sigma$. 
Let $C > 0$ and let $C_1>0$ be as in 
(\ref{defn:PertIntOp}) such that $C-C_1>0$ is large enough. Let 
$\delta$ be as in Definition \ref{defn:CoupConstDel} with $\kappa >4$, define 
$N:=(2\lfloor C_1/h\rfloor + 1)^2$ and let $B(0,R)\subset\mathds{C}^N$ 
be the ball of radius $R:=Ch^{-1 }$ centered at zero. For $z\in\Omega$ 
let $X(z)$ be as in Definition \ref{def:firstordertermnota}, let 
$E_{-+}(z)$ be as in Proposition \ref{propGrushUnpertOp} and let 
$e_0$ and $f_0$ be as in (\ref{def:e0,e1,...}) and 
(\ref{def:f0,f1,...}). There exist functions
  \begin{align}
  \label{eqn:Psi11}
   &\Psi(z;h,\delta)= (\partial_z X | \partial_z X) 
      - \frac{1}{\lVert X \rVert^2}\left|(\partial_z X | X) \right|^2 \notag \\
     & \phantom{\Psi(z;h,\delta)=  }
      + \delta^{-2}\left|(e_0|f_0)(1+\mO\left(h^{\infty}\right))
                          +\mathcal{O}\!\left(\eta^{1/4}\delta^2 h^{-7/2}\right)
                          \right|^2 + \mathcal{O}\!\left(\delta^3 h^{-3}\right), \\
    \label{eqn:Psi12}                     
   &\Theta(z;h,\delta)= \frac{|E_{-+}(z) + 
   \mathcal{O}\!\left(\delta^2 \eta^{-1/4}h^{-5/2}\right)|^2}{\delta^2 \lVert X(z) \rVert^2 },
  \end{align}
and $D>0$ and $\tilde{C}>0$ such that 
for all $\varphi\in\mathcal{C}_0(\Omega)$ and for $h>0$ small 
enough
 \begin{align*}
  \erw\big[~\Xi(\varphi)\mathds{1}_{B(0,R)}~\big] = 
  &\int \varphi(z) 
  \frac{1+ \mathcal{O}\!\left(\delta \eta^{-1/4}h^{-3/2}\right)}{\pi \lVert X \rVert^2} \Psi(z;h,\delta) 
       \e^{-\Theta(z;h,\delta)} L(dz)  \notag \\
   & + \mathcal{O}\!\left(\e^{-\frac{D}{h^2}}\right).
 \end{align*}
Here, $\mathcal{O}\!\left(\eta^{-1/4}\delta h^{-3/2}\right)$ is 
independent of $\varphi$ and 
$\mathcal{O}\!\left(\e^{-\frac{D}{h^2}}\right)$ means 
$\langle T_h,\varphi\rangle$ where $T_h\in\mathcal{D}'(\mathds{C})$ such that 
$|\langle T_h,\varphi\rangle|\leq C \lVert\varphi\rVert_{\infty}\e^{-\frac{D}{h^2}}$ for all 
$\varphi\in\mathcal{C}_0(\Omega)$ where $C$ and $D$ is independent of $h$, $\delta$, $\eta$ and $\varphi$. 
Moreover, the estimates in (\ref{eqn:Psi11}) and (\ref{eqn:Psi12}) are stable under 
application of $\eta^{-\frac{n+m}{2}}h^{n+m}\partial_z^n\partial_{\overline{z}}^m$.
\end{prop}

\begin{proof}
\textbf{Step I } Recall form of Section \ref{suse:GrushForUnPert} that 
$\sigma(P_h^{\delta}) = (E_{-+}^{\delta})^{-1}(0)$, thus $\Xi$ (cf. Definition 
\ref{defn:PP}) satisfies
  \begin{equation*}
   \Xi = \sum_{z\in(E_{-+}^{\delta})^{-1}(0)}\delta_z.
  \end{equation*}
It has been proven in \cite{SjAX1002}, that $E_{-+}^{\delta}(z)$ 
satisfies \eqref{lem:dbarequ}. Let $\chi$ be as in Lemma 
\ref{lem:approxdelta} then by Lemma \ref{lem:approxdelta}, Fubini's 
theorem and the dominated convergence theorem we have 
  \begin{align}\label{eqn:bigpart2}
   &\erw\Bigg[\sum_{z\in(E_{-+}^{\delta})^{-1}(0)} \varphi(z)~\mathds{1}_{B(0,R)}\Bigg] 
      = \lim\limits_{\varepsilon\rightarrow 0} \int \varphi(z) \left(\int\limits_{B(0,R)}D(z,\alpha)L(d\alpha)\right)L(dz), \notag \\ 
    &\text{where} ~D(z,\alpha) = \pi^{-N} \chi\left(\frac{E_{-+}^{\delta}(z;\alpha)}{\varepsilon}\right)
     \frac{1}{\varepsilon^2}\left|\partial_z E_{-+}^{\delta}(z;\alpha) \right|^2 \e^{-\alpha\overline{\alpha}}.
  \end{align}
\textbf{Step II } Next we give an estimate on $\partial_z E_{-+}^{\delta}(z)$. By (\ref{eqn:ShortCutE-+d})
\begin{equation}\label{eqn:dE_+estimate1}
    \partial_zE_{-+}^{\delta}(z) = \partial_zE_{-+}(z) - \delta \left( \partial_zX(z)\cdot\alpha + 
			  \partial_zT(z;\alpha)\right),
\end{equation}
where the derivative $\partial_z$ acts on $X(z)$ component wise and the dot-product $\partial_zX(z)\cdot\alpha$ is bilinear. 
To estimate $\partial_z T(z;\alpha)$, recall (\ref{eqn:PowerSeriesT}) and consider the derivative
  \begin{align*}
   &\partial_zE_-Q_{\omega}(EQ_{\omega})^nE_+ = 
	    (\partial_zE_{-})Q_{\omega}(EQ_{\omega})^nE_+ \notag \\
	&+ E_-Q_{\omega}\left[ 
	\sum_{j=1}^n \left(EQ_{\omega}\right)^{j-1}(\partial_z E)Q_{\omega} \left(EQ_{\omega}\right)^{n-j}\right]E_+ 
	       + E_-Q_{\omega}(EQ_{\omega})^n(\partial_zE_+),
  \end{align*}
with the convention $(EQ_{\omega})^0=1$. Recall the Grushin problem from Proposition 
\ref{propGrushUnpertOp} and take the derivative with respect to $z$ of the 
relation $\mathcal{E}(z)\mathcal{P}(z)= 1$ to obtain
  \begin{equation*}
   \partial_z\mathcal{E}(z) + \mathcal{E}(z)(\partial_{z}\mathcal{P}(z))\mathcal{E}(z) =0.
  \end{equation*}
A direct calculation yields
  \begin{align*}
   \partial_z E &= - E(\partial_z(P_h-z)) E - E_+(\partial_zR_+) E - E(\partial_zR_-) E_- \notag \\
		&= E^2 - E_+(\partial_zR_+) E - E(\partial_zR_-) E_-.
  \end{align*}
Recall the definition of $R_+$ and $R_-$ given in (\ref{eqn:defnR+R-}). By the estimates on the $z$- and $\overline{z}$-
derivatives of $e_0$ and $f_0$ given in Lemma \ref{quasimodes}, we conclude that
  \begin{align*}
   \lVert \partial_z R_+ \rVert_{H^1\rightarrow \mathds{C}}, 
   \lVert \partial_z R_- \rVert_{\mathds{C}\rightarrow L^2}  
   = \mathcal{O}\!\left(\eta^{1/2}h^{-1}\right).
  \end{align*}
Similarly, we have the same estimates on $\lVert \partial_z E_+ \rVert_{\mathds{C}\rightarrow L^2}$ 
and $\lVert \partial_z E_- \rVert_{H^1\rightarrow \mathds{C}}$. 
Thus, since $\lVert E(z)\rVert_{L^2\rightarrow H^1} = \mathcal{O}((h\sqrt{\eta})^{-1/2})$ 
and $\lVert E_{\pm}\rVert =\mathcal{O}(1)$, we have
  \begin{equation*}
   \lVert \partial_z E\rVert_{L^2\rightarrow H^1} = 
    \mathcal{O}(\eta^{1/4}h^{-3/2}).
  \end{equation*}
Putting all of this together, we get that the series of $\partial_z T(z;\alpha)$ converges again geometrically and we gain the estimate
  \begin{align}\label{eqn:PfFirstmomentErrorEstimpartialZT}
   \partial_{z}T(z;\alpha) = 
    \mathcal{O}\!\left(\eta^{1/4}\delta h^{-7/2} \right).
    \end{align}
Analogously, we conclude for all $\beta\in\N^2$
  \begin{align}\label{eqn:PfFirstmomentErrorEstimpartialZbarT}
   \eta^{-\frac{|\beta|}{2}}h^{|\beta|}\partial_{z\overline{z}}^{\beta}T(z;\alpha) =
    \mathcal{O}\!\left(\eta^{-1/4}\delta h^{-5/2} \right).
  \end{align}
Thus,
  \begin{equation*}
   \partial_zE_{-+}^{\delta}(z) = \partial_zE_{-+}(z) - \delta \partial_zX(z)\cdot\alpha +
    \mathcal{O}\!\left(\eta^{1/4}\delta h^{-7/2} \right).
  \end{equation*}
\\
\textbf{Step III } Consider the integral (\ref{eqn:bigpart2}) and choose vectors $e_1,e_2,\dots\in\mathds{C}^N$ 
as a basis of the $\alpha$-space such that $e_1=\overline{X}/\lVert\overline{X}\rVert$ and such that $e_1,e_2$ and 
$\overline{X}/\lVert\overline{X}\rVert, \partial_z\overline{X}$ span the same space:  
Therefore, we perform a unitary transformation in the $\alpha$-space such that with a slight 
abuse of notion 
  \begin{equation}\label{a1}
   \alpha= \alpha_1\frac{
			  \overline{X(z)}
			 }
			 {
			  \lVert X(z)\rVert
			 }
			  + 
		    \alpha_2 b\left(
			  \frac{ \overline{\partial_zX(z)} }{ \lVert \partial_zX(z) \rVert } - 
			  \frac{ 
				\overline{\left(\partial_zX(z)|X(z)\right)}\overline{X(z)} 
			       }
			       { \lVert \partial_zX(z) \rVert ~ \lVert X(z)\rVert^2 }
		    \right)
		    + \alpha^{\perp},
  \end{equation}
where $\alpha_1,\alpha_2\in\mathds{C}$ and $\alpha^{\perp}\in\mathds{C}^{N-3}$ and $b>0$ is a factor of normalization, 
  \begin{equation}\label{b1}
   b = \frac{
	      \lVert \partial_zX(z) \rVert ~ \lVert X(z)\rVert
	    }
	    {
	      \sqrt{ 
		    \lVert \partial_zX(z) \rVert^2 ~ \lVert X(z)\rVert^2 - 
		    \left|\left(\partial_zX(z)|X(z)\right)\right|^2
		   }
	    }.
  \end{equation}
This change of variables is well defined by Lemma \ref{lem:1MomSymplVol}. 
In the following we will also use the notation $(\alpha_1,\alpha_2,\alpha^{\perp})=(\alpha_1,\alpha')$.
This choice of basis yields by (\ref{eqn:PowerSeriesT}) and (\ref{eqn:ShortCutE-+d})
  \begin{equation}\label{eqn:E-+AfterChangeInAlpha}
   E_{-+}^{\delta}(z) = E_{-+}(z) - \delta \lVert X(z)\rVert\alpha_1 +\mathcal{O}\!\left(\eta^{-1/4}\delta^2 h^{-5/2} \right)
  \end{equation}
and by (\ref{eqn:dE_+estimate1}), \eqref{a1}, \eqref{b1}
  \begin{align}\label{derzEfHamPert}
   \partial_z E_{-+}^{\delta}(z) 
    &= \partial_z E_{-+}(z) - \delta \frac{\left(\partial_zX(z)|X(z)\right)}{\lVert X(z)\rVert}\alpha_1 \notag \\
    &\phantom{=}-\delta\left(
		   \lVert \partial_zX(z) \rVert^2  - 
			  \frac{ 
				\left|\left(\partial_zX(z)|X(z)\right)\right|^2
			       }
			       {  \lVert X(z)\rVert^2 }
	    \right)^{\frac{1}{2}}
		    \alpha_2+\mathcal{O}\!\left(\eta^{1/4}\delta^2 h^{-7/2} \right).
  \end{align}
Now let us split the ball $B(0,R)$, $R=Ch^{-1}$, into two pieces: 
pick $C_0>0$ such that $0<C_1<C_0<C$ and define $R_0:=C_0h^{-1}$. 
Then we shall consider one piece such that
$\lVert\alpha'\rVert_{\mathds{C}^{N-1}} < R_0$ and the other such that $\lVert\alpha'\rVert_{\mathds{C}^{N-1}} > R_0$. Hence, 
(\ref{eqn:bigpart2}) is equal to 
  \begin{align}\label{eqn:pfSplitBigInt}
    &\lim\limits_{\varepsilon\rightarrow 0} \int \varphi(z) 
    \hspace{-0.5cm}\int\limits_{\substack{B(0,R) \\ \lVert\alpha'\rVert_{\mathds{C}^{N-1}} < R_0} } 
	\hspace{-0.5cm}   D(z,\alpha)L(d\alpha) L(dz) 
     + \lim\limits_{\varepsilon\rightarrow 0} \int \varphi(z) 
    \hspace{-0.5cm}\int\limits_{\substack{B(0,R) \\ \lVert\alpha'\rVert_{\mathds{C}^{N-1}} > R_0} }
     \hspace{-0.5cm} D(z,\alpha)L(d\alpha) L(dz) \notag \\
     & =: I_1(\varphi) + I_2(\varphi).
  \end{align}
\\
\textbf{Step IV } In this step we will calculate $I_1(\varphi)$ of (\ref{eqn:pfSplitBigInt}). 
There we perform a change of variables such that $\beta:=E_{-+}^{\delta}(z;\alpha)$ is one of them. Due to 
(\ref{eqn:E-+AfterChangeInAlpha}) it is natural to express $\alpha_1$ as a function of $\beta$ and $\alpha'$. To this 
purpose we will apply Proposition \ref{cor:globalInvFunThm} to the function $E_{-+}^{\delta}(z;\alpha)$: \\ \par
  $E_{-+}^{\delta}(z;\alpha_1,\alpha')$ is holomorphic in $\alpha$ in ball of radius $R=Ch^{-1}$ centered at $0$. 
  Here, $\alpha$ plays the role of $z$ in the Proposition, in particular $\alpha_1$ plays the role of $z_n$.
  Recall (\ref{eqn:ShortCutE-+d}) and note that since 
$T(z;\alpha)=\mathcal{O}\!\left(\eta^{-1/4}\delta h^{-5/2}\right)$ (cf. (\ref{eqn:PowerSeriesT})) 
we can conclude by the Cauchy inequalities that 
  \begin{equation*}
   \partial_{\alpha_1}\delta T(z;\alpha) = \mathcal{O}\!\left(\eta^{-1/4}\delta^2 h^{-3/2}\right)
  \end{equation*}
which implies 
  \begin{equation}\label{eqn:betaderest2}
   \partial_{\alpha_1}E_{-+}^{\delta}(z;\alpha_1,\alpha') 
   = -\delta \lVert X(z)\rVert + \mathcal{O}\!\left(\eta^{-1/4}\delta^2 h^{-3/2}\right).
  \end{equation}
By Proposition \ref{prop:EstimFourrierCoeffE0F0} we have that 
$\lVert X(z) \rVert = 1 + \mathcal{O}\!\left(h^{\infty}\right)$ which implies that 
    \begin{equation*}
   \partial_{\alpha_1}E_{-+}^{\delta}(z;\alpha_1,\alpha') 
    = -\delta\left(1  + \mathcal{O}\!\left(h^{\infty}+\eta^{-1/4}\delta h^{-3/2}\right)\right).
  \end{equation*}
Hence, $E_{-+}^{\delta}(z;\alpha)$ satisfies the assumptions of Proposition \ref{cor:globalInvFunThm}. 
Since we restricted $\alpha'$ to ${\lVert\alpha'\rVert_{\mathds{C}^{N-1}} < R_0}$ and 
since
  \begin{equation*}
   |\alpha_1| < R^2 -\lVert\alpha'\rVert_{\mathds{C}^{N-1}}=:R_{\alpha'},
  \end{equation*}
it follows by Proposition \ref{cor:globalInvFunThm} that for 
  \begin{equation}\label{eqn:DomainBeta}
   \beta \in \bigcap_{\lVert\alpha'\rVert_{\mathds{C}^{N-1}} < R_0} 
   D\!\left(E_{-+}^{\delta}(z;0,\alpha'), r_{\alpha'} \right)
  \end{equation}
with 
  \begin{equation}\label{eqn:radiusBeta}
   r_{\alpha'} \geq \delta\left(1  + \mathcal{O}\!\left(h^{\infty}+\eta^{-1/4}\delta h^{-3/2}\right)\right)
	   \frac{\sqrt{C^2 - C_0^2}}{h}
   \geq \frac{\delta h^{-1}}{\mO\!\left(1\right)} > 0.
  \end{equation}
and $h>0$ small enough, $\beta=E_{-+}^{\delta}(z;\alpha_1,\alpha')$ 
has exactly one solution $\alpha_1(\beta,\alpha')$ in the disk 
$D(0,R_{\alpha'})$ and it depends holomorphically on $\beta$ and $\alpha'$. 
More precisely,
  \begin{equation}\label{eqn:PfFormulaFirstMomentAlpha}
   \alpha_1(\beta,\alpha') = \frac{-\beta + E_{-+}(z) + \mathcal{O}\!\left(\eta^{-1/4}\delta^2 h^{-5/2}\right)}{\delta \lVert X(z)\rVert}.
  \end{equation}
Furthermore, 
  \begin{equation*}
   L(d\alpha) = |\partial_{\alpha_1}E_{-+}^{\delta}|^{-2}L(d\beta)L(d\alpha').
  \end{equation*}
Since the support of $\chi$ is compact, we can restrict our attention to $\beta$ and 
$E_{-+}^{\delta}(z;0,\alpha')$ in a small disk of radius $\varepsilon>0$ centered at $0$. By choosing 
$\varepsilon<\delta h^{-1}/C$, $C>0$ large enough, as in (\ref{eqn:radiusBeta}) we see that 
$\beta, E_{-+}^{\delta}(z;0,\alpha') \in D(0,\varepsilon)$ implies (\ref{eqn:DomainBeta}). By performing this 
change of variables and by picking $\varepsilon>0$ small enough as above, we get
  \begin{align}\label{eqn:firstIntensityPfSplitInt1}
     I_1(\varphi) = \lim\limits_{\varepsilon\rightarrow 0} \int \varphi(z) \left\{
      \int\limits_{\mathds{C}} \chi\left(\frac{\beta}{\varepsilon}\right)\frac{1}{\varepsilon^2}
      \varLambda(\beta;z)L(d\beta)\right\}  L(dz),
  \end{align}
where $\varLambda(\beta;z)$ depends smoothly on $z$ and on $\beta$ and, using (\ref{derzEfHamPert}), 
is given by
\begin{align}\label{eqn:densitybig}
 \varLambda(\beta,z):=& 
      \pi^{-N}\int\limits_{\lVert\alpha'\rVert_{\mathds{C}^{N-1}} < R_0}\mathds{1}_{B(0,R)}(\alpha_1,\alpha') 
      \left|\partial_{\alpha_1}E_{-+}^{\delta}(\alpha_1,\alpha';z)\right|^{-2}\notag \\
     &\cdot\left|A(\alpha,z) - \beta\frac{(\partial_z X(z)|X(z))}{\lVert X(z)\rVert^2}
     - B(z)\alpha_2+\mathcal{O}\!\left(\eta^{1/4}\delta^2 h^{-7/2} \right)\right|^2 \notag \\
      &\cdot\exp\left\{-\alpha'\overline{\alpha'}
	-\left|\frac{-\beta + E_{-+}(z)
	+ \mathcal{O}\!\left(\eta^{-1/4}\delta^2 h^{-5/2}\right)}{\delta \lVert X(z)\rVert}\right|^2\right\} L(d\alpha'), 
\end{align}
where where $\alpha_1=\alpha_1(\beta,\alpha',z)$ and $A(\alpha,z),B(z)$ are defined as follows: 
     \begin{align}\label{eqn:EstA}
    A(\alpha,z)  :&= \partial_z E_{-+}(z) - \frac{(\partial_z X(z)|X(z))}{\lVert X(z)\rVert^2}
      \left(E_{-+}(z) + \mathcal{O}\!\left(\eta^{-1/4}\delta^2 h^{-5/2}\right)\right)
         \notag \\
	  &\phantom{=}+\mathcal{O}\!\left(\eta^{1/4}\delta^2 h^{-7/2} \right) \notag \\
      &= (e_0|f_0)(1+\mO\left(h^{\infty}\right))+ \mathcal{O}\!\left(\eta^{1/4}\delta^2 h^{-7/2} \right) \notag \\
      &= \mathcal{O}\!\left(\eta^{3/4}h^{-\frac{1}{2}}\e^{-\frac{\asymp\eta^{3/2}}{h}}\right)
       + \mathcal{O}\!\left(\eta^{1/4}\delta^2 h^{-7/2} \right).
    \end{align}
The second identity for $A$ is due to Proposition \ref{prop:EstimAndFormulaDerzE-+} 
and the following estimate
 \begin{align*}
   \left|\frac{(\partial_z X(z)|X(z))}{\lVert X(z)\rVert^2}\right|  \leq \frac{\lVert \partial_z X(z) \rVert}{\lVert X(z) \rVert}
	  = \left( 1+\mathcal{O}\!\left(h^{\infty}\right)\right)\mathcal{O}\!\left(\eta^{1/2}h^{-1}\right) 
	  = \mathcal{O}\!\left(\eta^{1/2}h^{-1}\right)
  \end{align*}
which follows from Propositions \ref{prop:EstimFourrierCoeffE0F0} and \ref{cor:EstimFourrierCoeffE0F0}. 
In the last line we used Proposition \ref{prop:EstimDerzE-+} together with (\ref{eqn:SEstINBox}). 
Furthermore, recall by Step II and Step III that $A(\alpha,z)$ is holomorphic in $\alpha$.\par
Similarly, we define
    \begin{align}\label{eqn:EstB}
    B(z) := \delta \left(
		   \lVert \partial_zX(z) \rVert^2   - 
			  \frac{ 
				\left|\left(\partial_zX(z)|X(z)\right)\right|^2
			       }
			       {  \lVert X(z)\rVert^2 }
	    \right)^{\frac{1}{2}} 
	 =\mathcal{O}\!\left(\eta^{-1/4}\delta h^{-\frac{1}{2}}\right).
  \end{align}
The estimate in (\ref{eqn:EstB}) follows from Proposition  \ref{lem:1MomSymplVol}. 
\begin{rem}
 It follows from Proposition \ref{prop:EstimAndFormulaDerzE-+}, Proposition \ref{prop:EstimDerzE-+}, 
 Proposition  \ref{cor:EstimFourrierCoeffE0F0} and from (\ref{eqn:PfFirstmomentErrorEstimpartialZbarT})
 that 
  \begin{align}\label{eqn:EstimDerAB}
   &\eta^{-\frac{n+m}{2}}h^{n+m}\partial_z^n\partial_{\overline{z}}^m A(z) =
       \mathcal{O}\!\left(\eta^{3/4}h^{-\frac{1}{2}}\e^{-\frac{\asymp\eta^{3/2}}{h}}\right) 
       + \mathcal{O}\!\left(\eta^{1/4}\delta^2 h^{-7/2} \right),
     \notag \\
   &\eta^{-\frac{n+m}{2}}h^{n+m}\partial_z^n\partial_{\overline{z}}^m B(z) =
   \mathcal{O}\!\left(\eta^{-1/4}\delta h^{-\frac{1}{2}}\right).
  \end{align}
\end{rem}
Since $\varLambda(\beta,z)$ is continuous in $\beta$, the dominated convergence theorem shows that
  \begin{equation*}
    I_1(\varphi)=\int \varphi(z) \varLambda(0,z) L(dz).
  \end{equation*}
Next, let us look at the indicator function $\mathds{1}_{B(0,R)}(\alpha_1(\beta,\alpha';z),\alpha')$ for 
$\lVert\alpha'\rVert < R_0$: By (\ref{eqn:PfFormulaFirstMomentAlpha}) we have 
  \begin{equation*}
   |\alpha_1(0,\alpha')| = \frac{\left|E_{-+}(z) + \mathcal{O}\!\left(\delta^2h^{-5/2}\right)\right|}{\delta \lVert X (z)\rVert}.		 
  \end{equation*}
Thus, $\mathds{1}_{B(0,R)}(\alpha_1(0,\alpha';z),\alpha')= 1$ if 
$|\alpha_1(0,\alpha')|^2 \leq R^2 - R_0^2 = \frac{\tilde{C}^2}{h^2}, ~\lVert\alpha'\rVert < R_0^2$ 
and if $R^2 - R_0^2 <|\alpha_1(0,\alpha')|^2 < R^2, ~\lVert\alpha'\rVert < R_0^2 - |\alpha_1(0,\alpha')|^2$, 
and $\mathds{1}_{B(0,R)}(\alpha_1(0,\alpha';z),\alpha')= 0$ if 
$R^2 \leq |\alpha_1(0,\alpha')|^2$, with $\tilde{C}^2 := C^2 - C_0^2$. Hence, we split $\Lambda(0,z)$ into 
  \begin{align}\label{Lsp}
    \Lambda(0,z) &= \Lambda(0,z)\left(
	  \mathds{1}_{\{\sqrt{\Theta(z;h,\delta)} \leq \frac{\tilde{C}}{h}\}}(z) 
	 +  \mathds{1}_{\left\{\frac{\tilde{C}}{h} < \sqrt{\Theta(z;h,\delta)}  < R\right\}}(z) \right) \notag \\
	    &=: \Lambda_1(0,z) +\Lambda_2(0,z),
  \end{align}
where 
  \begin{equation*}
   \Theta(z;h,\delta)
   := \frac{|E_{-+}(z) + \mathcal{O}\!\left(
		\delta^2 \eta^{-1/4}h^{-5/2}\right)|^2}
           {\delta^2 \lVert X(z) \rVert^2 }.
  \end{equation*}
We start by treating $\Lambda_1$. Note that the function 
\begin{equation*}
 \{\lVert\alpha'\rVert_{\mathds{C}^{N-1}} < R_0\}\ni\alpha' \longmapsto 
 \exp\left\{-\left|\alpha_1(0,\alpha';z)\right|^2\right\} \in [0,1]
\end{equation*}
is continuous, bounded and recall that (\ref{eqn:PfFormulaFirstMomentAlpha}) holds 
for all $\alpha'\in\{\lVert\alpha'\rVert_{\mathds{C}^{N-1}} < R_0\}$. Furthermore, note that all factors in the integral 
(\ref{eqn:densitybig}) are positive. Since the ball $\{\lVert\alpha'\rVert_{\mathds{C}^{N-1}} < R_0\}$ is 
simply connected the intermediate value theorem yields
\begin{align}\label{eqn:densitybigContinue}
 \varLambda_1(0,z)= &\pi^{-N}\mathds{1}_{\{ \sqrt{\Theta(z;h,\delta)} \leq \frac{\tilde{C}}{h}\}}(z) 
 \left|\delta \lVert X(z)\rVert + \mathcal{O}\!\left(\eta^{-1/4}\delta^2 h^{-3/2}\right)\right|^{-2}
	\notag \\
       &\cdot
      \exp\{- \Theta(z;h,\delta)\}
      \int\limits_{\lVert\alpha'\rVert_{\mathds{C}^{N-1}} < R_0}
      \left|A(\alpha,z)-\delta B(z)\alpha_2\right|^2 \e^{-\alpha'\overline{\alpha'}}  L(d\alpha').
\end{align}
Here we also applied (\ref{eqn:betaderest2}). Before we can further simplify 
(\ref{eqn:densitybigContinue}), let us consider the following technical Lemma:
\begin{lem}\label{lem:FirstMomentHelpLemmaGaussError}
Let $h>0$, let  $C_0,C_1>0$ and let  $N:=(2\lfloor\frac{C_1}{h}\rfloor +1)^2$. 
Let $n\in\mathds{N}^{N-1},m\in\mathds{N}^{N-1}$, let $R_0=C_0/h$ and let $\alpha\in\mathds{C}^N$. 
If $C_0>C_1>0$ are large enough and such that 
  \begin{equation*}
   \ln \left(2 + \frac{\e R_0^2}{N-2} \right) < \frac{R_0^2}{2(N-2)},
  \end{equation*}
then, for $h>0$ small enough, there exists a constant $D_{n,m}=:D>0$ such that 
  \begin{equation*}
   \left|\pi^{1-N} \int_{\lVert \alpha' \rVert_{\mathds{C}^{N-1}}\geq R_0} 
	    \alpha'^n\overline{\alpha'}^m\e^{-\alpha'\overline{\alpha'}} L(d\alpha')\right| 
    = \mathcal{O}\!\left(\e^{-\frac{D}{h^2}}\right).
  \end{equation*}
\end{lem}
\begin{proof}
Define 
  \begin{equation*}
   2u:=\begin{cases}
        |n|+|m|, ~\text{if it is even} \\
        |n|+|m|+1, ~\text{else}
       \end{cases}
  \end{equation*}
and notice
  \begin{align*}
   &\left| \pi^{1-N}\int_{\lVert \alpha' \rVert_{\mathds{C}^{N-1}}\geq R_0} 
	  \alpha'^n\overline{\alpha'}^m\e^{-\alpha'\overline{\alpha'}} L(d\alpha')\right| 
	  \notag \\
   & \leq \pi^{1-N}\left|S^{2N-3}\right|\int_{R_0}^{\infty} r^{2u+2N-3}\e^{-r^2}dr 
     = \frac{2}{(N-2)!} \int_{R_0^2}^{\infty} \tau^{u+N-2}\e^{-\tau}d\tau.
  \end{align*}
Repeated partial integration then yields
  \begin{equation}\label{eqn:FirstMomProofGaussRestLemm1}
   \frac{2}{(N-2)!}\e^{-R_0^2} \sum_{i=0}^{u+N-2}\begin{pmatrix} u+N-2 \\ i \\ \end{pmatrix} (u+N-2 -i)! R_0^{2i}.
  \end{equation}
Using Stirling's formula one gets that \eqref{eqn:FirstMomProofGaussRestLemm1}$\leq$
  \begin{align*}
   & \frac{\e\sqrt{(u+N-2)}}{(N-2)!}\e^{-R_0^2} \sum_{i=0}^{u+N-2}\begin{pmatrix} u+N-2 \\ i \\ \end{pmatrix} \left(\frac{u+N-2}{\e}\right)^{u+N-2 -i} R_0^{2i} \notag \\ 
   & \leq \frac{\e\sqrt{(u+N-2)}}{\sqrt{2\pi (N-2)}}\e^{-R_0^2} \left(\frac{\e}{N-2}\right)^{N-2} \left( R_0^{2} + \frac{u+N-2}{\e}\right)^{u+N-2} \notag\\
   & = \e^{-R_0^2}\frac{\e}{\sqrt{2\pi}}\sqrt{1+\frac{u}{N-2}} \left(\frac{R_0^{2}\e}{N-2} + 1+ \frac{u}{N-2}\right)^{N-2} \left( R_0^{2} + \frac{u+N-2}{\e}\right)^{u}.
  \end{align*}
Since $u/(N-2)$ is bounded for $h>0$ small, it remains to consider
  \begin{equation}\label{eqn:FirstMomProofGaussRestLemm2}
   \exp\left\{-R_0^2 +(N-2)\ln\left(\frac{R_0^{2}\e}{N-2} + 1+ \frac{u}{N-2}\right) + u\ln\left( R_0^{2} + \frac{u+N-2}{\e}\right) \right\}.
  \end{equation}
However, there exists a $1> \kappa >0$ such that 
  \begin{equation*}
   -R_0^2 +(N-2)\ln\left(\frac{R_0^{2}\e}{N-2} + 1+ \frac{u}{N-2}\right) \leq - R_0^2\kappa = - \frac{C_0^2}{h^2},
  \end{equation*}
which implies that (\ref{eqn:FirstMomProofGaussRestLemm2}) is dominated by 
  \begin{equation*}
   \exp\left\{- \frac{C_0^2}{h^2}\left(\kappa - \frac{h^2}{\mathcal{O}(1)}\ln(h)\right) \right\},
  \end{equation*}
and we conclude the statement of the Lemma for $h>0$ small enough.
\end{proof}
Let us return to (\ref{eqn:densitybigContinue}): We are interested in the integral
  \begin{align}\label{eqn:IntToclaculate}
   \pi^{-N}\int\limits_{\lVert\alpha'\rVert_{\mathds{C}^{N-1}} < R_0}
	|A - B\alpha_2|^2 \exp\left\{-\alpha'\overline{\alpha'}\right\} L(d\alpha'). 
  \end{align}
We will investigate each term of \eqref{eqn:IntToclaculate} separately. 
Since $B$ is constant in $\alpha$ and since $\int |\alpha_2|^2\exp(-\alpha'\overline{\alpha'}) L(d\alpha') = \pi^{N-1}$, 
we conclude, by Lemma \ref{lem:FirstMomentHelpLemmaGaussError} for $C_0>C_1>0$ large enough and $h>0$ small enough, 
that there exists a constant $D>0$ such that 
  \begin{equation*}
   \pi^{-N}\int\limits_{\lVert\alpha'\rVert_{\mathds{C}^{N-1}} < R_0} 
   |B\alpha_2|^2\e^{-\alpha'\overline{\alpha'}} L(d\alpha') = 
   \pi^{-1}|B|^2 + \mathcal{O}\!\left(\eta^{-\frac{1}{2}}\delta^2 h^{-1}\e^{-\frac{D}{h^2}}\right).
  \end{equation*}
The mean value theorem, (\ref{eqn:EstA}) and Lemma \ref{lem:FirstMomentHelpLemmaGaussError} imply that there 
exists a constant $D>0$ (not necessarily the same as above) such that 
  \begin{equation*}
   \pi^{-N}\int\limits_{\lVert\alpha'\rVert_{\mathds{C}^{N-1}} < R_0} |A|^2 \exp\left\{-\alpha'\overline{\alpha'}\right\} L(d\alpha') = 
   \pi^{-1}|A|^2 + \mathcal{O}\!\left(\e^{-\frac{D}{h^2}}\right).
  \end{equation*}
Note that after the equality sign we have $A=A(\tilde{\alpha}',z)$ for an $\tilde{\alpha}'\in B(0,R_0)$ 
given by the mean value theorem. Next, since \eqref{eqn:EstB} is independent of $\alpha$, 
  \begin{align*}
   \pi^{-N}\int\limits_{\lVert\alpha'\rVert_{\mathds{C}^{N-1}} < R_0}
      \overline{A}B\alpha_2 \e^{-\alpha'\overline{\alpha'}} L(d\alpha')
      =\pi^{-N}B\int\limits_{\lVert\alpha'\rVert_{\mathds{C}^{N-1}} < R_0}
	 \overline{A}\alpha_2  \e^{-\alpha'\overline{\alpha'}} L(d\alpha').
  \end{align*}
Since $A(\alpha,z)$ is holomorphic in $\alpha$ we gain from (\ref{eqn:EstA}) by the Cauchy inequalities 
  \begin{equation}\label{d_A}
   |\partial_{\alpha_2}A| = \mathcal{O}\!\left(\eta^{1/4}\delta^2 h^{-5/2} \right).
  \end{equation}
Here we used that the first term in (\ref{eqn:EstA}) is independent of $\alpha$. 
Extend $A$ to a function on $\mathds{C}^{N-1}$ such that the above estimate still holds. 
Then, by Lemma \ref{lem:FirstMomentHelpLemmaGaussError} there exists a constant $D>0$ such that
   \begin{align*}
   \pi^{-N}B\int\limits_{\lVert\alpha'\rVert_{\mathds{C}^{N-1}} \geq R_0}
      \overline{A}\alpha_2  \e^{-\alpha'\overline{\alpha'}} L(d\alpha')
      = \mathcal{O}\!\left(\eta^{1/2}h^{-1}\delta\e^{-\frac{\asymp\eta^{3/2}}{h}} + \delta^3 h^{-4} \right)
       \e^{-\frac{D}{h^2}}.
  \end{align*}
Here we used (\ref{eqn:EstA}) and (\ref{eqn:EstB}). Stokes' theorem and \eqref{d_A} imply 
  \begin{align*}
  \pi^{-N}B\int_{\mathds{C}^{N-1}}
	  \overline{A}\alpha_2\e^{-\alpha'\overline{\alpha'}} L(d\alpha') &= 
   \pi^{-N}B\int_{\mathds{C}^{N-1}}
	  \left(\partial_{\overline{\alpha}_2}\overline{A}\right)\e^{-\alpha'\overline{\alpha'}} L(d\alpha')
	  \notag \\
	 & \leq  
       \mathcal{O}\!\left(\delta^3 h^{-3} \right).
  \end{align*}
Plugging the above into (\ref{eqn:IntToclaculate}), we gather that there exist a constant $D>0$ such that 
  \begin{align}\label{eqn:alltogehterInt}
   \pi^{-N}\int\limits_{\lVert\alpha'\rVert_{\mathds{C}^{N-1}} < R_0} &
	|A - B\alpha_2|^2 \exp\left\{-\alpha'\overline{\alpha'}\right\} L(d\alpha')\notag \\
	 &= \pi^{-1}\left(|A(z)|^2 + |B(z)|^2\right)
	+\mathcal{O}\!\left(\delta^3 h^{-3}  + \e^{-\frac{D}{h^2}}\right) \notag \\
	& =:\delta^{2}\Psi(z,h,\delta).
  \end{align}
By \eqref{eqn:EstA} and \eqref{eqn:EstB}, we see that 
$\pi^{-1}(|A(z)|^2 + |B(z)|^2)$ is equal to 
\begin{align*}
	 \frac{\delta^2}{\pi}\bigg(&(\partial_z X | \partial_z X) 
                                   - \frac{1}{\lVert X \rVert^2}\left|(\partial_z X | X) \right|^2
                                   \notag \\
                                   &+
          \delta^{-2}\left|(e_0|f_0)(1+\mO\left(h^{\infty}\right))
                          +\mathcal{O}\!\left(\eta^{1/4}\delta^2 h^{-7/2} \right)\right|^2\bigg).
  \end{align*}
The above, \eqref{eqn:alltogehterInt}, (\ref{eqn:densitybigContinue}) and
  \begin{equation*}
   \left|\delta \lVert X(z)\rVert + \mathcal{O}\!\left(\eta^{-1/4}\delta^2 h^{-3/2}\right)\right|^{-2}= 
    \frac{\left(1 + \mathcal{O}\!\left(\eta^{-1/4}\delta h^{-3/2}\right) \right)}{\delta^{2}\pi\lVert X(z)\rVert^2},
  \end{equation*}
imply that for $h>0$ small enough, there exists a constant $D>0$ such that
  \begin{align}\label{eqn:densitybigContinue2}
  \varLambda_1(0,z):= \frac{\left(1 + \mathcal{O}\!\left(\eta^{-1/4}\delta h^{-3/2}\right) \right)}{\pi\lVert X(z)\rVert^2} 
                   \mathds{1}_{\{ \sqrt{\Theta(z;h,\delta)}  \leq \frac{\tilde{C}}{h}\}}(z) 
		  \Psi(z,h,\delta)
		  \exp^{- \Theta(z;h,\delta)}.
  \end{align}
Finally, let us estimate $\Lambda_2$ from \eqref{Lsp}: applying (\ref{eqn:EstA}), (\ref{eqn:EstB}) 
and Lemma \ref{lem:FirstMomentHelpLemmaGaussError} to (\ref{eqn:densitybig}) yields
  \begin{align*}
    \Lambda_2(0,z) &\leq \e^{-\frac{\widetilde{C}}{h^2}}
	    \mathcal{O}\!\left(\delta^4 \eta^{1/2}h^{-7} 
		    + \eta^{1/2}h^{-1}\delta\e^{-\frac{\asymp\eta^{3/2}}{h}}\right)
    = \mathcal{O}\!\left(\e^{-\frac{D}{h^2}}\right),
  \end{align*}
for some $D>0$. Thus, up to an error of order $ \mathcal{O}\!(\e^{-\frac{D}{h^2}})$, we can 
substitute $ \mathds{1}_{\{ \sqrt{\Theta(z;h,\delta)}  \leq \frac{\tilde{C}}{h}\}}(z)$ with $1$ in 
(\ref{eqn:densitybigContinue2}).
\\
\\
\textbf{Step V } In this step we will estimate $I_2(\varphi)$ of (\ref{eqn:pfSplitBigInt}). 
Therefore, we increase the space of integration 
  \begin{align*}
   &\int\limits_{\substack{B(0,R) \\ \lVert\alpha'\rVert_{\mathds{C}^{N-1}} > R_0} }
      \chi\left(\frac{E_{-+}^{\delta}(z;\alpha)}{\varepsilon}\right)
     \frac{1}{\varepsilon^2}\left|\partial_z E_{-+}^{\delta}(z;\alpha) \right|^2 \e^{-\alpha\overline{\alpha}}L(d\alpha) \notag \\
   &\leq 
   \int\limits_{\substack{B(0,2R) \\ R_0 < \lVert\alpha'\rVert_{\mathds{C}^{N-1}} < 2R_0} }
      \chi\left(\frac{E_{-+}^{\delta}(z;\alpha)}{\varepsilon}\right)
     \frac{1}{\varepsilon^2}\left|\partial_z E_{-+}^{\delta}(z;\alpha) \right|^2 \e^{-\alpha\overline{\alpha}}L(d\alpha) =: W_{\varepsilon}.
  \end{align*}
It is easy to see that Lemma \ref{lem:globalInvFunThm} holds true for the set 
$B(0,2R) \cap \{ R_0 < \lVert\alpha'\rVert_{\mathds{C}^{N-1}} < 2R_0\}$, 
potentially by choosing a larger $C>0$ in Proposition \ref{lem:truncation} larger. We can proceed as in Step IV: 
perform the same change of variables and the limit of $\varepsilon\rightarrow 0$. This yields
  \begin{align*}
   \lim\limits_{\varepsilon\rightarrow 0} \,W_{\varepsilon}  = 
     & \pi^{-N}\int\limits_{R_0 < \lVert\alpha'\rVert_{\mathds{C}^{N-1}} < 2R_0}\mathds{1}_{B(0,2R)}(\alpha_1(0,\alpha';z),\alpha') 
      \left|\partial_{\alpha_1}\beta(\alpha_1,\alpha';z)\right|^{-2}\notag \\
     &\cdot\left|A(\alpha;z) - B(z)\alpha_2\right|^2
      \exp\left\{-\alpha'\overline{\alpha'}
	-\Lambda(z,h,\delta)^2\right\} L(d\alpha').
\end{align*}
By (\ref{eqn:EstA}), (\ref{eqn:EstB}) and Lemma \ref{lem:FirstMomentHelpLemmaGaussError} we see that there 
exists a constant $D>0$ such that 
  \begin{align*}
  &\pi^{-N}\int\limits_{R_0 < \lVert \alpha' \rVert_{\mathds{C}^{N-1}} < 2R_0} 
  |A - B\alpha_2|^2 \e^{-\alpha'\overline{\alpha'}}L(d\alpha') \\
  &\leq \e^{-\frac{\tilde{D}}{h^2}}
	    \mathcal{O}\!\left(\delta^4 \eta^{1/2}h^{-7} 
		    + \eta^{1/2}h^{-1}\delta\e^{-\frac{\asymp\eta^{3/2}}{h}}\right)
   = \mathcal{O}\!\left(\e^{-\frac{D}{h^2}}\right).
  \end{align*}
The statement about the derivatives of the error terms follows 
from (\ref{eqn:EstimDerAB}) and (\ref{eqn:PfFirstmomentErrorEstimpartialZbarT}).
\end{proof}
%
% Chapter 9
%
\section{Average Density of Eigenvalues}
First, we will give the proof the main result of this work:
\begin{proof}[Proof of Theorem \ref{thm:ModelFirstIntensity}]
Due to \eqref{eqn:TubeEmptyA} and Hypothesis \ref{hyp:H5} we have that, for $\kappa >0$ 
(as in Definition \ref{defn:CoupConstDel}) large enough, that \eqref{eqn:TubeEmpty} holds. Therefore, 
we assume that $\left(h\ln\frac{1}{h}\right)^{2/3} \ll \eta \leq C$, 
where $C>0$ is a constant. 
\par 
In particular, we now strengthen assumption \eqref{hyp_Omega_a} and assume 
from now on that $\Omega\Subset\Sigma$ satisfies Hypothesis \ref{hyp:H3} if 
nothing else is specified, i.e. we assume that 
\begin{equation*}
     \Omega\Subset\Sigma ~\text{is open, relatively compact with}~
     \mathrm{dist\,}(\Omega,\partial\Sigma) \gg \left(h\ln h^{-1} \right)^{2/3}.
\end{equation*}
Recall the definition of $\Omega_{\eta}^{a}$ given in (\ref{def:etaOmega}): 
  \begin{equation*}
   \Omega_{\eta}^{a} = 
   \left\{z\in \Omega:~ \frac{\eta}{C} \leq \Ima z \leq C\eta \right\}
  \end{equation*}
for some constant $C>0$. Define
\begin{equation*}
 \widetilde{\Omega}_{\eta}^a := 
     \left\{z\in \Omega:~ \frac{\eta}{2C} \leq \Ima z \leq 2C\eta \right\}.
\end{equation*}
Define $\eta_j:=C^{-j}$, $j\in\mathds{N}_0$, 
and consider the open covering of $\Omega$ 
  \begin{equation*}
   \Omega \subset \bigcup_{j\in\mathds{N}_0}\widetilde{\Omega}_{\eta_j}^{a} \cup 
      \left(\Omega \backslash \bigcup_{j\in\mathds{N}_0}\overline{\Omega_{\eta_j}^{a}}\right),
  \end{equation*}
where 
$\dist(\Omega \backslash \bigcup_{j\in\mathds{N}_0}\overline{\Omega_{\eta_j}^{a}},\partial\Sigma) > 1/C$, 
thus, conforming with the previous notation, we may define 
  \begin{equation*}
   \Omega_i :=\Omega \backslash \bigcup_{j\in\mathds{N}_0}\overline{\Omega_{\eta_j}^{a}}.
  \end{equation*}
Let $\{\chi_{\eta_j}\}_{j\in \mathds{N}_0}$ be a partition 
of unity subordinate to this locally finite open subcovering such that 
  \begin{equation*}
   1 = \sum_{j\in \mathds{N}} \chi_{\eta_j} + \chi_{\eta_0},
  \end{equation*}
in a neighborhood of $\Omega$. Here, for 
$j\in \mathds{N}$, $\chi_{\eta_j} \in \mathcal{C}^{\infty}_0(\widetilde{\Omega}_{\eta}^a)$, 
supported in either $\widetilde{\Omega}_{\eta}^{a}$. Furthermore, 
$\chi_{\eta_0} \in \mathcal{C}^{\infty}(\Omega_i)$. This partition 
of unity together with Proposition \ref{prop:formulaFor1Intensity} 
yields 
  \begin{align*}
   \erw\big[~\Xi(\varphi)&\mathds{1}_{B(0,R)}~\big] = 
   \sum_{j\in\mathds{N}} \erw\left[~\Xi(\varphi\chi_{\eta_j})\mathds{1}_{B(0,R)}~\right] + 
   \erw\left[~\Xi(\varphi\chi_0)\mathds{1}_{B(0,R)}~\right]  \notag \\
  & =\sum_{j\in\mathds{N}}\int \varphi(z)\chi_{\eta_j}(z) 
  \frac{1+ \mathcal{O}\!\left(\eta_j^{-1/4}\delta h^{-3/2}\right)}{\pi \lVert X \rVert^2} \Psi(z;h,\delta) 
       \e^{-\Theta_j} L(dz)
       \notag \\
   & +\int \varphi(z)\chi_0(z) \frac{1+ \mathcal{O}\!\left(\delta h^{-3/2}\right)}{\pi \lVert X \rVert^2} \Psi(z;h,\delta) 
       \e^{-\Theta_0 } L(dz)
      + \mathcal{O}\!\left(\e^{-\frac{D}{h^2}}\right).
  \end{align*}
where 
  \begin{equation*}
   \Theta_j := \frac{\left|E_{-+}(z) + 
   \mathcal{O}\!\left(\eta_j^{-1/4}\delta^2 h^{-5/2}\right)\right|^2}{\delta^2 \lVert X \rVert^2 }, ~
   \Theta_0 := \frac{|E_{-+}(z) + \mathcal{O}\!\left(\delta^2 h^{-5/2}\right)|^2}{\delta^2 \lVert X \rVert^2 }.
  \end{equation*}
Note that to gain the exponentially small error estimate 
in the above we used that the bound on the distribution $T_h\in D'(\C)$ 
(cf. Proposition \ref{prop:formulaFor1Intensity}) is independent 
of $\eta$. Thus,
  \begin{equation*}
    \left|\sum_{j\in\mathds{N_0}}\langle T_h,\varphi\chi_{\eta_j}\rangle\right|
    = |\langle T_h,\varphi\rangle | \leq C\lVert \varphi\rVert_{\infty} \e^{-\frac{D}{h^2}}.
  \end{equation*}
$\phantom{-}$\\
\textbf{Analysis of the density $\Psi$} Recall the formula for the density of eigenvalues given in Proposition 
\ref{prop:formulaFor1Intensity}. Define
  \begin{equation}
  \label{eqn:Psi1111}
   \Psi_1(z;h,\delta):=  (\partial_z X | \partial_z X) - \frac{1}{\lVert X \rVert^2}\left|(\partial_z X | X) \right|^2
			+  \mathcal{O}\!\left(\delta^3 h^{-3}\right)
  \end{equation}
Since the error above is of order $\mathcal{O}(1)$, it 
follows from Proposition \ref{lem:1MomSymplVol} that 
  \begin{equation*}
   \Psi_1(z,h,\delta)=  \frac{1}{h}\left\{\frac{i}{\{p,\overline{p}\}(\rho_+(z))} - \frac{i}{\{p,\overline{p}\}(\rho_-(z))}\right\}
			+\mathcal{O}\!\left(\dist(z,\partial\Sigma)^{-2}\right),
  \end{equation*}
where we used that $\Ima z \asymp \eta_j$ for $z\in\Omega_{\eta_j}^{a}$. Proposition 
\ref{cor:1MomSymplVol} implies 
  \begin{equation*}
    \Psi_1(z,h,\delta)L(dz) = \frac{1}{2h}p_*(d\xi\wedge dx) 
   + \mathcal{O}\!\left(\dist(z,\partial\Sigma)^{-2}\right)L(dz).
  \end{equation*}
Furthermore, Proposition \ref{prop:formulaFor1Intensity} and Proposition \ref{lem:1MomSymplVol} 
yield that  
  \begin{equation*}
    \eta^{-\frac{n+m}{2}}h^{n+m}\partial_z^n\partial_{\overline{z}}^m\mathcal{O}\!\left(\eta_j^{-2}\right) 
    = \mathcal{O}_{}\!\left(\eta_j^{-2}\right),
  \end{equation*}
where $\mathcal{O}_{}\!\left(\eta_j^{-2}\right)$ is the error term of 
$\Psi_1$. Next, let us turn to the second part of $\Psi$: 
  \begin{align*}
    &\delta^{-2}\left|(e_0|f_0)(1+\mO\left(h^{\infty}\right))
                          +\mathcal{O}\!\left(\eta_j^{1/4}\delta^2 h^{-7/2}\right)\right|^2 \notag \\
           &=  \delta^{-2}\left|(e_0|f_0)\right|^2(1+\mO\left(h^{\infty}\right)) 
              + \mathcal{O}\!\left(\eta_j^{1/2}\delta^2 h^{-7}\right)
              + \mathcal{O}\!\left(\eta_j^{1/4}h^{-7/2}\left|(e_0|f_0)\right|\right) \notag \\
           &=\delta^{-2}\left|(e_0|f_0)\right|^2(1+\mO\left(h^{\infty}\right)) 
    +\mathcal{O}\!\left(\eta_j h^{-4}\e^{-\frac{S}{h}}+\eta_j^{1/2}\delta^2 h^{-7}\right).
  \end{align*}
In the last line, we applied an estimate on $\left|(e_0|f_0)\right|$ which follows from 
Proposition \ref{prop:AbsSPe0f0}. The error term $\mO(\eta_jh^{-4}\e^{-\frac{S}{h}})$ is 
bounded by $\mO(\eta_j)$ because $\eta \gg (-h\ln h)^{2/3}$. We then absorb $\mO(\eta_j)$ 
into the error term $\mathcal{O}(\eta_j^{-2})$ of $\Psi_1$ as well as 
the error term $\mathcal{O}(\eta_j^{1/2}\delta^2 h^{-7})\leq \mathcal{O}(\eta_j^{1/2})$. Then, 
one defines
\begin{align}\label{eqn:Psi2Nice}
   \Psi_2(z;h,\delta)&:= \frac{\left|(e_0|f_0)\right|^2}{\delta^2}
    \left(1 + \mathcal{O}\!\left(\eta_j^{-3/4}h^{1/2}\right)\right).
  \end{align} 
As in (\ref{eqn:EstimDerAB}), the error estimates don't change if 
we apply $\eta^{-\frac{n+m}{2}}h^{n+m}\partial_z^n\partial_{\overline{z}}^m$. \\
\\
\textbf{Analysis of the exponential $\Theta$ } Recall from Proposition \ref{propGrushUnpertOp}
that $-\alpha_0= E_{-+}$ and use (\ref{eqn:alphaEFFTUNN}) to find that
  \begin{equation*}
   E_{-+}(z) = ([P_h,\chi ]e_0|f_0)\left(1 + 
    \mathcal{O}\!\left(\exp\left[-\frac{\eta_j^{3/2}}{h}\right]\right)\right).
  \end{equation*}
Here $\chi\in\mathcal{C}^{\infty}_0(S^1)$ with $\chi\equiv 1$ in a 
small open neighborhood of $\overline{\{x_-(z);z\in\Omega\}}$. Thus, 
using $\lVert X \rVert = (1 + \mathcal{O}(h^{\infty}))$ 
(cf. Proposition \ref{prop:EstimFourrierCoeffE0F0}), we have the following 
equation for $\Theta$ given in Proposition \ref{prop:formulaFor1Intensity}
  \begin{align}\label{eqn:ThetaNIce}
  \Theta(z,h,\delta)& =
  \frac{\left|E_{-+}(z) + \mathcal{O}\!\left(\eta_j^{-1/4}\delta^2 h^{-5/2}\right)\right|^2}{\delta^2\lVert X \rVert^2}
  \notag \\
   &= \frac{\left|([P_h,\chi ]e_0|f_0)+ \mathcal{O}\!\left(\eta_j^{-1/4}\delta^2 h^{-5/2}\right)\right|^2}
    {\delta^2(1 + \mathcal{O}(h^{\infty}))}\left(1 + 
    \mathcal{O}\!\left(\e^{-\frac{\eta_j^{3/2}}{h}}\right)\right).
  \end{align}
As in (\ref{eqn:EstimDerAB}), the error estimates stay invariant under the action of \\
$\eta_j^{-\frac{n+m}{2}}h^{n+m}\partial_z^n\partial_{\overline{z}}^m$. Finally, to 
conclude the density given in the Theorem, note that 
\begin{equation*}
   \frac{1+ \mathcal{O}\!\left(\eta_j^{-1/4}\delta h^{-3/2}\right)}{\pi\lVert X \rVert^2} 
   =    \frac{1+ \mathcal{O}\!\left(\dist(z,\partial\Sigma)^{-1/4}\delta h^{-3/2}\right)}{\pi}. \qedhere
  \end{equation*}
\end{proof}
In the case of the operator $P_h^{\delta}$, it is possible to state more 
explicit formulas for the different parts of the density of eigenvalues 
given in Theorem \ref{thm:ModelFirstIntensity}: \par 
It follows by Propositions \ref{lem:1MomSymplVol} and \ref{cor:1MomSymplVol} that 
  \begin{align*}
   \frac{1}{2h}p_*(d\xi\wedge dx) 
   &=  
   \frac{1}{h}\left\{\frac{i}{\{p,\overline{p}\}(\rho_+(z))} + \frac{i}{\{\overline{p},p\}(\rho_-(z))} \right\}
   L(dz)
   \notag\\
  & \asymp  \frac{1}{h\!\sqrt{\dist(z,\partial\Sigma)}}L(dz)
  \end{align*}
where we used that $\Ima z \asymp \eta_j$ for $z\in\Omega_{\eta}^{a}$.
For our purposes we can assume that $|\Ima z - \langle \Ima g\rangle | > 1/C$, $C\gg 1$, 
since inside this tube $\Psi_2$ and $\Theta$ are exponentially small in $h>0$. In the case 
of $\Psi_2$, this follows from the assumptions on $\delta$ (cf. Definition \ref{defn:CoupConstDel}) 
and from Remark \ref{rem:MoreExplSP13}. In the case of $\Theta$, this follows from the assumptions on $\delta$ 
and Proposition \ref{cor:E-+Repres} and (\ref{eqn:ThetaNIce}). Thus, applying 
Proposition \ref{prop:AbsSPe0f0} to (\ref{eqn:Psi2Nice}) yields
   \begin{equation}
    \label{eqn:SplitPsi3}
    \Psi_2(z;h,\delta)=
     \frac{\left(\frac{i}{2}\{p,\overline{p}\}(\rho_+)\frac{i}{2}\{\overline{p},p\}(\rho_-)\right)^{\frac{1}{2}}}
	{\pi h\delta^2\exp\{\frac{2S}{h}\}}|\partial_{\Ima z}S(z)|^2
      \left(1 + \mathcal{O}\!\left(\eta^{-3/4}h^{1/2}\right)\right).
   \end{equation}
As in (\ref{eqn:EstimDerAB}), the error estimates don't change if we apply 
$\eta^{-\frac{n+m}{2}}h^{n+m}\partial_z^n\partial_{\overline{z}}^m$. 
Moreover, since 
$\Ima z \asymp \eta_j$ for $z\in\Omega_{\eta}^{a}$,
\begin{equation*}
   \Psi_2^0(z;h,\delta) \asymp \frac{(\dist(z,\partial\Sigma))^{3/2}\e^{-\frac{2S}{h}}}{h\delta^2}.
  \end{equation*}
Apply Proposition \ref{cor:E-+Repres} to (\ref{eqn:ThetaNIce}) gives that
  \begin{align}\label{eqn:Therough}
   \Theta(z,h,\delta)
   =& V(z,h)^2\frac{\e^{-\frac{2S}{h}}}{\delta^2}
	  \left(1 +\mathcal{O}\!\left(h^{\infty}\right)
	  +\mathcal{O}\!\left(\e^{-\frac{\asymp\eta_j^{3/2}}{h}}\right)\right)
   \notag \\
   &+ \mathcal{O}\!\left(\eta_j^{-1/2}\delta^2 h^{-5}\right) 
      + \mathcal{O}\!\left(V h^{-5/2}\e^{-\frac{S}{h}} \right).
  \end{align}
Since $0\leq V = \mO\!\left(\eta_j^{1/4}h^{1/2}\right)$ by (\ref{eqn:Prefac2}), it 
follows that 
  \begin{align*}
   \Theta(z,h,\delta)
   =& \frac{h\left(\frac{i}{2}\{p,\overline{p}\}(\rho_+)
	  \frac{i}{2}\{\overline{p},p\}(\rho_-)\right)^{\frac{1}{2}}}{\pi}\frac{\e^{-\frac{2S}{h}}}{\delta^2}
	 \left(1 +\mathcal{O}\!\left(\eta_j^{-1/4}h^{\frac{3}{2}}\right)\right)
   \notag \\
   &+ \mathcal{O}\!\left(\eta_j^{-1/2}\delta^2 h^{-5}\right) 
      + \mathcal{O}\!\left(\eta_j^{1/4}h^{-2}\e^{-\frac{S}{h}}\right) .
  \end{align*}
Furthermore, for $\e^{-\frac{2S}{h}}\delta^{-2}\leq 1$, 
the error term $\mathcal{O}\!\left(\eta_j^{1/4} h^{-2}\e^{-\frac{S}{h}}\right)$ 
is bounded by $\mO(\eta_j^{1/4}h^{-2}\delta)$ since there we have 
that $\e^{-\frac{S}{h}}\leq \delta$. For $\e^{-\frac{2S}{h}}\delta^{-2}\leq 1$, 
we have that 
  \begin{equation*}
   \mathcal{O}\!\left(\eta_j^{1/4} h^{-2}\e^{-\frac{S}{h}}\right)
   \leq
   \mathcal{O}\!\left(\eta_j^{1/4} h^{-2}\delta \e^{-\frac{2S}{h}}\delta^{-2}\right)
   \leq 
   \mathcal{O}\!\left(\eta_j^{1/4} h^{2} \e^{-\frac{2S}{h}}\delta^{-2}\right).
  \end{equation*}
Thus, 
 \begin{align}\label{eqn:ThetaFinal}
  \Theta(z,h,\delta)
   = &\frac{h\left(\frac{i}{2}\{p,\overline{p}\}(\rho_+)
	  \frac{i}{2}\{\overline{p},p\}(\rho_-)\right)^{\frac{1}{2}}}{\pi}\frac{\e^{-\frac{2S}{h}}}{\delta^2}
	 \left(1 +\mathcal{O}\!\left(\eta_j^{-1/4}h^{\frac{3}{2}}\right)\right)\notag \\
   &+ \mathcal{O}\!\left(\eta_j^{1/4}h^{-2}\delta
   + \eta_j^{-1/2}\delta^2 h^{-5}\right).
  \end{align}
Analogous to (\ref{eqn:EstimDerAB}), the error estimates stay invariant under the action of 
$\eta_j^{-\frac{n+m}{2}}h^{n+m}\partial_z^n\partial_{\overline{z}}^m$. Moreover, 
  \begin{equation*}
   \Theta^0(z;h,\delta)\asymp h\sqrt{\dist(z,\partial\Sigma)}
   \frac{\e^{-\frac{2S}{h}}}{\delta^2}. 
  \end{equation*}
We have thus proven Proposition \ref{prop:GrowPropDens} and Proposition \ref{prop:DensExplicit}. 
Since we will need it later on we will state the following formulas:
\begin{lem}\label{lem:DerPsiThet}
 Under the assumptions of Theorem \ref{thm:ModelFirstIntensity} and for 
 $(h\ln h^{-1})^{2/3}\ll \eta <\mathrm{const.}$ we have 
 \begin{align*}
   \partial_{\Ima z}\Psi_1 = -\frac{1}{4h}\left( \frac{\Ima g''(x_-)}{(\Ima g'(x_-))^3}
   -\frac{\Ima g''(x_+)}{(\Ima g'(x_+))^3}\right) +\mO\left(\eta^{-2}\right)
   = \mO\left(\eta^{-3/2}h^{-1}\right)
  \end{align*}
  and for $|\Ima z - \langle\Ima g\rangle| >1/C$, $C>0$ large enough,
  \begin{align*}
   &\partial_{\Ima z}\Psi_2(z,h) = 
   \frac{2\left(\frac{i}{2}\{p,\overline{p}\}(\rho_+)\frac{i}{2}\{\overline{p},p\}(\rho_-)\right)^{\frac{1}{2}}}{\pi h^2}
   |\partial_{\Ima z}S(z)|^2(-\partial_{\Ima z}S(z))
   \frac{\e^{-\frac{2S}{h}}}{\delta^2} 
    \notag \\
    &\phantom{\partial_{\Ima z}\Psi_2(z,h) = }\cdot\left(1 +\mathcal{O}\!\left(\eta^{-3/4}h^{\frac{1}{2}}\right) \right),
  \end{align*}
  \begin{align*}
   \partial_{\Ima z}\Theta(z,h) =& 
   \frac{2\left(\frac{i}{2}\{p,\overline{p}\}(\rho_+)\frac{i}{2}\{\overline{p},p\}(\rho_-)\right)^{\frac{1}{2}}}
   {\pi\delta^2\exp\{-\frac{2S}{h}\}}
   (-\partial_{\Ima z}S(z))
   \left(1 +\mathcal{O}\!\left(\eta^{-1/4}h^{\frac{3}{2}}\right)\right) \notag \\
   &+\mathcal{O}\!\left(\eta^{3/4}h^{-3}\delta + \delta^2 h^{-6}\right),
  \end{align*}
\end{lem}
\begin{proof}
 Let us first treat $\Psi_1$: Recall from the proof of 
 Proposition \ref{prop:EstimAndFormulaDerzE-+} that $\Psi_1$ was 
 given by an oscillatory integral where the phase vanishes at 
 the critical point. Thus, the $\partial_{\Ima z}$ derivative of 
 the error term $\mO\left(\eta^{-2}\right)$ grows at most by $\eta^{-1}$.
 Thus, taking the derivative of (\ref{eqn:Psi1111}) yields
  \begin{equation*}
   \partial_{\Ima z}\Psi_1 = -\frac{1}{4h}\left( \frac{\Ima g''(x_-)}{(\Ima g'(x_-))^3}
   -\frac{\Ima g''(x_+)}{(\Ima g'(x_+))^3}\right) +\mO\left(\eta^{-3}\right)
   = \mO\left(\eta^{-3/2}h^{-1}\right),
  \end{equation*}
where the last estimate follows from $|2\Ima g'(x_{\pm}|=|\{p,\overline{p}\}(\rho_{\pm}|\asymp \sqrt{\eta}$ 
(cf. Proposition \ref{lem:1MomSymplVol}) and from the fact that the $z$- and $\overline{z}$-derivative of 
  the error term grow at most by a factor of $\mO(\eta^{1/2}h^{-1})$. \\ 
  \par 
 Now let us turn to $\Psi_2$: one calculates from (\ref{eqn:SplitPsi3}) 
 that for $|\Ima z - \langle\Ima g\rangle| >1/C$
  \begin{align*}
   \partial_{\Ima z}\Psi_2(z,h) =& 
   \frac{2\left(\frac{i}{2}\{p,\overline{p}\}(\rho_+)\frac{i}{2}\{\overline{p},p\}(\rho_-)\right)^{\frac{1}{2}}}{\pi h^2}
   |\partial_{\Ima z}S(z)|^2(-\partial_{\Ima z}S(z))
   \frac{\e^{-\frac{2S}{h}}}{\delta^2} \notag \\
   &\cdot\left(1 +\mathcal{O}\!\left(\eta^{-3/4}h^{\frac{1}{2}}\right) \right).
  \end{align*}
  Here we used that the $z$- and $\overline{z}$-derivative of the error terms grow at most by 
  a factor of $\mO(\eta^{1/2}h^{-1})$. \par 
  Finally, let us turn to $\Theta$: as in the proof of Proposition \ref{prop:EstimAndFormulaDerzE-+_pt2} 
  one calculates the formula for $\partial_{\Ima z}\Theta$ from (\ref{eqn:ThetaFinal}).
\end{proof}
%
% Chaper 10
%
\section{Properties of the density}\label{sec:BehavDensity}
In this section we will discuss and prove the results stated 
in Section \ref{suse:MainResPropDens}. 
\subsection{Local maximum of the average density}\label{susec:10.1}
First, we prove the resolvent estimate given in Proposition \ref{prop:Resolvent}. 
\begin{proof}[Proof of Proposition \ref{prop:Resolvent}]
Recall that the operator $Q(z)$ is self-adjoint and that $|t_0(z)|=|\alpha_0(z)|$; 
see Section \ref{sec:AuxOpe}. It follows that 
  \begin{equation*}
   \lVert (P_h - z)^{-1}\rVert = |t_0(z)|^{-1} = |\alpha_0(z)|^{-1}. 
  \end{equation*}
Recall the Grushin problem posed in Proposition \ref{propGrushUnpertOp}. 
Since $E_{-+}^{-1}=-\alpha_0$, it follows by Proposition \ref{cor:E-+Repres} 
that 
  \begin{align}\label{eqn:ResolvBound}
   \lVert (P_h - z)^{-1}\rVert 
     = \frac{\exp\left\{\frac{S}{h}\right\}}{V(z)|1-\e^{\Phi(z)}|
    \left(1 + \mathcal{O}\!\left(\e^{-\frac{\asymp\eta^{\frac{3}{2}}}{h}}\right)\right)},
  \end{align}
which together with \eqref{eqn:Prefac2} implies \eqref{res_id}. The result 
about the asymptotic behavior of the resolvent follows from the above together 
with the fact that $|\{p,\overline{p}\}(\rho_{\pm})| \asymp \sqrt{\eta}$ 
(cf. Proposition \ref{lem:1MomSymplVol}).
\end{proof}
We have split the proof of Proposition \ref{prop:the2curves} into the following 
two Lemmata:
\begin{lem}\label{prop:ExZone2}
  Let $z\in\Omega\Subset\Sigma_{c,d}$ with $\Sigma_{c,d}$ as in (\ref{def:stripSigma}),  
  let $S(z)$ be as in Definition \ref{defn:ZeroOrderDensity}. Let $\delta>0$ and $\varepsilon(h)$ be as 
  in Definition \ref{defn:CoupConstDel} with $\kappa > 0$ large enough. Moreover, let $E_{-+}(z)$ be as 
  in Proposition \ref{propGrushUnpertOp}. Then,
  \begin{itemize}
   \item for $0<h\ll 1$, there exist numbers $y_{\pm}(h)$ such that $\varepsilon_0=S(y_{\pm}(h))$ with 
      \begin{align*}
	\frac{1}{C}(h\ln h^{-1})^{\frac{2}{3}} &\ll y_{-}(h) < \langle \Ima g\rangle - ch\ln h^{-1}
	\notag\\
	&<\langle \Ima g\rangle + ch\ln h^{-1} < y_{+}(h) \ll \Ima g(b) - \frac{1}{C}
	(h\ln h^{-1})^{\frac{2}{3}},
      \end{align*}
     for $c>1$. Furthermore, 
       \begin{align*}
        y_{-}(h), (\Ima g(b) - y_+(h)) \asymp (\varepsilon_0(h))^{2/3};
       \end{align*}
   \item there exists $h_0>0$ and a family of smooth curves, indexed by $h\in]h_0,0[$,
	\begin{equation*}
	 \gamma_{\pm}^h: ~ ]c,d[ \longrightarrow \mathds{C} 
	 ~\text{with}~ \Rea \gamma_{\pm}^h(t) =t
	\end{equation*}
	such that
	\begin{equation*}
	  |E_{-+}(\gamma_{\pm}^h(t))| = \delta,
	\end{equation*}
	and 
	\begin{equation*}
	  \Ima \gamma^h_{\pm}(t) = y_{\pm}(\varepsilon_0(h))
	  \left(1 + \mathcal{O}\left(\frac{h}{\varepsilon_0(h)}\right)\right).
	\end{equation*}
	Furthermore, there exists a constant $C>0$ such that 
      \begin{equation*}
	  \frac{d\Ima \gamma_{\pm}^h}{d t}(t) 
	  = \mathcal{O}\!\left(\exp\left[-\frac{\varepsilon_0(h)}{Ch}\right]\right).
	\end{equation*}
  \end{itemize}
\end{lem}

\begin{lem}\label{prop:NoName}
 Assume the same hypothesis as in Lemma \ref{prop:ExZone2} and let 
  \begin{equation*}
   D(z,h):=\frac{1+ \mathcal{O}\!\left(\delta h^{-\frac{3}{2}}\dist(z,\partial\Sigma)^{-1/4}\right)}{\pi} 
	     \Psi(z;h,\delta)\exp\{-\Theta(z;h,\delta)\}
  \end{equation*}
 be the average density of eigenvalues of the operator of $P_h^{\delta}$ given in 
 Theorem \ref{thm:ModelFirstIntensity}. Then, 
 there  exists $h_0>0$ and a family of smooth curves, indexed by $h\in]h_0,0[$,
	\begin{equation*}
	 \Gamma_{\pm}^h: ~ ]c,d[ \longrightarrow \mathds{C},~ \Rea \Gamma_{\pm}^h(t) = t,
	\end{equation*}
	with $\Gamma_-\subset\{\Ima z <  \langle \Ima g\rangle\}$ and 
	$\Gamma_+\subset\{\Ima z >  \langle \Ima g\rangle\}$, along which 
	$\Ima z \mapsto D(z,h)$ takes its local 
	maxima on the vertical line $\Rea z = \mathrm{const.}$ and 
	  \begin{align*}
	   \frac{d}{dt}\Ima\Gamma_{\pm}^h(t) = 
	   \mO\!\left(\frac{h^{4}}{\varepsilon_0(h)^{4}}\right).
	  \end{align*}
	Moreover, for all $c < t < d$
	\begin{equation*}
	  |\Gamma_{\pm}^h(t) - \gamma_{\pm}^h(t)|
	  \leq \mathcal{O}\left(\frac{h^{5}}{\varepsilon_0(h)^{13/3}}\right).
	  \end{equation*}
\end{lem}
\begin{proof}[Proof of Proposition \ref{prop:the2curves}]
The first two points of the proposition follow from Lemma \ref{prop:ExZone2} together 
with the observations that $|E_{-+}(z)|=|\alpha_0|=|t_0(z)|$ (cf. Proposition 
\ref{propGrushUnpertOp}) and that by (\ref{eqn:ResolvBound})
\begin{equation*}
   \lVert (P_h - \gamma_{\pm}^h)^{-1}\rVert 
   = \delta^{-1}.
  \end{equation*}
The third point has been proven with Lemma \ref{prop:NoName}.
\end{proof}

\begin{proof}[Proof of Lemma \ref{prop:ExZone2}]
Recall from Proposition \ref{prop:SProp} that $S$ is strictly monotonous above and below 
the spectral line, i.e. $\Ima z = \langle\Ima g\rangle$. Furthermore, recall from 
Definition \ref{defn:CoupConstDel} that 
$ -\left(\kappa - \frac{1}{2}\right)h\ln h +Ch \leq \epsilon_0(h) < S(\langle \Ima g\rangle)$. 
Thus, the implicit function theorem implies that there exist $y_{\pm}(\varepsilon_0(h))\in\mathds{R}$ such 
that $S(y_{\pm}(\varepsilon_0(h)))=\varepsilon_0(h)$. Note that in the case where $\varepsilon_0(h)$ is independent of $h$, 
the same holds true for $y_{\pm}(\varepsilon_0)$. For the rest of the proof we will only treat the 
case where $\Ima z \leq \langle\Ima g \rangle$ (corresponding to $y_-$) since the other case is similar. 
\par
Consider $z\in\Omega\Subset\Sigma_{c,d}$ with $\Rea z=\mathrm{const}.$ 
First, let us prove some 
a priori estimates: assume that there exists a $\zeta_{-}$ with 
$h^{2/3} \ll \zeta_{-}\leq \langle\Ima g\rangle$ such 
that $|E_{-+}(\Rea z + i\zeta_{-})|\delta^{-1}=1$. Recall Proposition \ref{prop:SProp} 
and note that
 \begin{align}\label{eqn:Sepsilondependence}
     S(z) -\varepsilon_0(h) 
     &= \int_{\langle\Ima g\rangle }^{\Ima z}(\partial_{\Ima z}S) (t)dt+S(\langle\Ima g\rangle )-\varepsilon_0(h) \notag \\
     &=\int_{y_{-}(\varepsilon_0(h))}^{\Ima z}(\partial_{\Ima z}S) (t)dt+S(y_{-}(\varepsilon_0(h)))-\varepsilon_0(h).
 \end{align}
Recall Proposition \ref{cor:E-+Repres} and Definition \ref{defn:CoupConstDel}.
It follows by (\ref{eqn:Sepsilondependence}), that 
if $|\zeta_{-} - \langle\Ima g\rangle |\leq \frac{1}{C}$, $C>0$ large enough, then 
$|E_{-+}(\Rea z +i\zeta_{-})|\delta^{-1}\leq \mathcal{O}\left(\eta^{1/4}\e^{-\frac{1}{Dh}}\right)$ for 
some $D>0$ large.
Thus, we may assume that, in case it exists, 
  \begin{equation}\label{eqn:BFPbound1}
   |\zeta_{-} - \langle\Ima g\rangle | > \frac{1}{C}.
  \end{equation}
We conclude from (\ref{eqn:Sepsilondependence}) that
  \begin{equation}\label{eqn:YEps}
   y_{-}(h)\asymp (\varepsilon_0(h))^{2/3}
  \end{equation} 
and that for $C>0$ large enough
  \begin{equation}\label{eqn:BFPbound2}
    |\langle \Ima g\rangle - y_{-}(\varepsilon)|>\frac{1}{C}.
  \end{equation}
\eqref{eqn:YEps}, \eqref{eqn:BFPbound2} and Definition \ref{defn:CoupConstDel}
imply, for $\kappa >0 $ large enough, the first point of the Lemma. 
\\
\par
Now let us prove the existence of the points $\zeta_{-}$. More precisely, 
we will prove that for $z\in\Omega\Subset\Sigma_{c,d}$ with $\Ima z <\langle\Ima g\rangle-1/C$ 
(cf. (\ref{eqn:BFPbound1})) and fixed $\Rea z$ 
there exist exactly one $\zeta_{-}$ such that $|E_{-+}(\Rea z,\zeta_{-})|\delta^{-1} =1$.  
For $z\in\Omega\cap\Omega_{\eta}^a \Subset\Sigma_{c,d}$ one calculates from 
by Proposition \ref{cor:E-+Repres} that 
  \begin{align}\label{eqn:EImaDEr} 
   \partial_{\Ima z} |E_{-+}(z)| =  
    &\left\{- V(z)\frac{\partial_{\Ima z}S(z)}{h}|1-\e^{\Phi(z)}|\left(1 + 
      \mathcal{O}\!\left(\e^{-\frac{\asymp\eta^{\frac{3}{2}}}{h}}\right)\right) \right. \notag \\
     & +\left. \partial_{\Ima z}\left[ V(z)|1-\e^{\Phi(z)}|\left(1 + 
      \mathcal{O}\!\left(\e^{-\frac{\asymp\eta^{\frac{3}{2}}}{h}}\right)\right)\right]\right\}
      \e^{-\frac{S(z)}{h}},
  \end{align} 
Recall that $V$ is the product of the normalization factors of the 
quasimodes $e_{wkb}$ and $f_{wkb}$ when $z\in\Omega$ with $\dist(\Omega,\partial\Sigma)>1/C$ and 
the product of the normalization factors of the quasimodes $e^{\eta}_{wkb}$ and $f^{\eta}_{wkb}$ 
when $z\in\Omega\cap\Omega_{\eta}^a $ (cf. (\ref{eqn:Prefac})). Since the derivative with 
respect to $\Ima z$ of the imaginary part of their phase function $\Ima\phi_{\pm}$ is equal to 
zero at $x_{\pm}$, it follows that 
  \begin{equation}\label{eqn:DerImazV}
   |\partial_{\Ima z} V(z)| = \mathcal{O}(h^{1/2}\eta^{-3/4}).
  \end{equation}
The a priori bound (\ref{eqn:BFPbound1}) implies that there exists a constant $C>1$ such that 
  \begin{equation}\label{eqn:PhiTerm}
   |1 - \e^{\Phi(z)}| = 1 + \mathcal{O}\left(\e^{-\frac{1}{Ch}}\right),~\text{and}~ 
   \partial_{\Ima z}|1 - \e^{\Phi(z)}| = \mathcal{O}\left(\e^{-\frac{1}{Ch}}\right).
  \end{equation}
The fact that $\partial_{\Ima z}S(z)> 0$ (cf. (\ref{prop:SProp})) implies that 
$\partial_{\Ima z} |E_{-+}(z)| < 0$. Note that in the case where $\dist(\Omega,\partial\Sigma)>1/C$ 
one sets in the above $\eta=1$. Recall from Propositions \ref{prop:GrushHager} 
and \ref{prop:GrushBM} that $V$ is independent of $\Rea z$. Using
  \begin{equation*}
   \partial_{\Rea z }|1-\e^{\Phi(z)}| =   \mathcal{O}\!\left(\e^{-\frac{1}{Ch}}\right),
  \end{equation*}
we conclude that
  \begin{align}\label{eqn:EReaDEr}
   \partial_{\Rea z} |E_{-+}(z)| &= 
     \partial_{\Rea z}\left[ V(z)|1-\e^{\Phi(z)}|\left(1 + 
      \mathcal{O}\!\left(\e^{-\frac{\asymp\eta^{\frac{3}{2}}}{h}}\right)\right)\right]\e^{-\frac{S(z)}{h}}
      \notag \\
      & =  \mathcal{O}\!\left(\e^{-\frac{\asymp\eta^{\frac{3}{2}}}{h}}\right) \e^{-\frac{S(z)}{h}}.
  \end{align}
This implies that the gradient $|E_{-+}(z)|$ is non-zero for all $z$ with 
$|\Ima z - \langle\Ima g \rangle| > 1/C$ (cf. (\ref{eqn:BFPbound1})) and thus we may 
conclude by the implicit function theorem, that for $\delta$ as above there 
exist locally smooth curves $\gamma_{-}^h(\Rea z) := (\Rea z, \zeta_{-}(\varepsilon_0(h),\Rea z)$ such that 
$|E_{-+}(\gamma_{-}^h)|=\delta$. Furthermore, we may extend 
$\gamma_-(\Rea z)$ smoothly for $c < \Rea z < d$. By the mean value theorem applied to $|E_{-+}(z)|$, there 
exists a $\zeta$ between $y_-(h)$ and $\Ima \gamma_{-}^h(\Rea z)$ such that 
  \begin{align*}
   &\big||E_{-+}(\Rea z +iy_-(h))| - |E_{-+}(\gamma_{-}^h(\Rea z))|\big|
   \notag \\
   &=  
   |(\partial_{\Ima z} |E_{-+}(z)|)(\Rea z+i\zeta)|\cdot|y_-(h) - \Ima \gamma_{-}^h(\Rea z)|.
  \end{align*}
Since $|E_{-+}| =\mO(\sqrt{h}\eta^{1/4}\e^{-\frac{S}{h}})$ 
(cf. Proposition \ref{prop:EstimDerzE-+}) 
and since $\partial_{\Ima z}|E_{-+}|\asymp - h^{-1/2}\eta^{3/4}\e^{-\frac{S}{h}}$ 
(cf. (\ref{eqn:EImaDEr})), it follows that 
  \begin{equation*}
   |y_-(h) - \Ima \gamma_{-}^h(\Rea z)| = \mO\left(\eta^{-1/2}h\right).
  \end{equation*}
$\eta \asymp y_{-}(h)\asymp (\varepsilon_0(h))^{2/3}$ implies that also 
$ \Ima \gamma_{-}^h(\Rea z) \asymp \eta \asymp (\varepsilon_0(h))^{2/3}$, 
and we conclude that 
\begin{equation*}
  \Ima \gamma^h_{-}(\Rea z) = y_{-}(\varepsilon_0(h))
  \left(1 + \mathcal{O}\left(\frac{h}{\varepsilon_0(h)}\right)\right).
\end{equation*}
Finally, by 
  \begin{align*}
   0 &= \frac{d}{d\Rea z}|E_{-+}(\gamma_{-}^h(\Rea z))| \notag \\
   &= \partial_{\Rea z}|E_{-+}(\gamma_{-}^h(\Rea z))| 
   +\partial_{\Ima z}|E_{-+}(\gamma_{-}^h(\Rea z))| \frac{d\Ima \gamma_{-}^h}{d \Rea z}(\Rea z).
  \end{align*}
and by (\ref{eqn:EImaDEr} ) and (\ref{eqn:EReaDEr}) we may then conclude 
  \begin{equation}\label{eqn:estimDergammapm}
   \frac{d\Ima \gamma_{-}^h}{d \Rea z}(\Rea z) = 
   \mathcal{O}\!\left(\e^{-\frac{\asymp\eta^{3/2}}{h}}\right)
  \end{equation}
which, using $\eta \asymp y_{-}(h)\asymp (\varepsilon_0(h))^{2/3}$, 
yields the last statement of the Lemma. 
\end{proof}
\begin{proof}[Proof of Lemma \ref{prop:NoName}]
The idea of this proof is to search for the critical points of the 
average density of eigenvalues via the Banach fix point theorem. We 
shall only consider the case where $\Ima z \leq \langle \Ima g\rangle$ 
since the other case is similar. \\
\par
Recall from Proposition \ref{prop:DensExplicit} the explicit form the density 
given in Theorem \ref{thm:ModelFirstIntensity}. Proposition \ref{lem:1MomSymplVol} 
and the fact that $\Ima g$ has exactly two critical points imply that $\Psi_1$ is 
strictly monotonously decreasing. Thus, we may assume similar to 
(\ref{eqn:BFPbound1}) that for $C>0$ large enough
  \begin{equation}
   |\Ima z - \langle \Ima g\rangle| > \frac{1}{C}.
  \end{equation}
since else $\Psi_2 = \mO(\e^{-\frac{1}{Dh}})$ with $D>0$ large. Now, to find the 
critical points of the density of eigenvalues consider
  \begin{align}\label{eqn:DerImazDens}
   \pi \partial_{\Ima z}D(z,h) 
   &= \left(\partial_{\Ima z}\Psi(z;h,\delta)\exp\{-\Theta(z;h,\delta)\}\right)
     \left(1 +\mO\left(\delta\eta^{-1/4}h^{-3/2}\right) \right) \notag \\
   &\phantom{=}+\Psi(z;h,\delta)\exp\{-\Theta(z;h,\delta)\}\mO\left(\delta\eta^{1/4}h^{-5/2}\right) \notag \\ 
   & = 0.
  \end{align}
Here we used that the $z$- and $\overline{z}$-derivative of the error 
term $\mO\left(\delta\eta^{-1/4}h^{-3/2}\right)$ increases its order 
of growth at most by a term of order $\mO(\eta^{1/2}h^{-1})$ 
(cf. Theorem \ref{thm:ModelFirstIntensity}). By
  \begin{equation*}
   \partial_{\Ima z}\Psi(z;h,\delta)\e^{-\Theta(z;h,\delta)} = 
   \left(\partial_{\Ima z}\Psi_1 +\partial_{\Ima z}\Psi_2
   - (\Psi_1+\Psi_2)\partial_{\Ima z}\Theta\right)\e^{-\Theta(z;h,\delta)},
  \end{equation*}
and by Lemma \ref{lem:DerPsiThet} and Proposition \ref{prop:DensExplicit}, 
we can write (\ref{eqn:DerImazDens}) as 
\begin{align}\label{eqn:DerImazDensity2}
   & h^{-3}F(z,h,\delta) + 2\frac{\e^{-\frac{2S}{h}}}{\delta^2}|\partial_{\Ima z}S(z)|^2(-\partial_{\Ima z}S(z))
   \frac{\left(\frac{i}{2}\{p,\overline{p}\}(\rho_+)\frac{i}{2}\{\overline{p},p\}(\rho_-)\right)^{\frac{1}{2}}}{\pi h^2}
    \notag \\
   &\cdot\left(1 + \mO\left(\eta^{-3/4}h^{1/2}\right)\right)
   \left((1+ \mO(\eta^{-3/2}h))  - \frac{h\left(\frac{i}{2}\{p,\overline{p}\}(\rho_+)\frac{i}{2}\{\overline{p},p\}(\rho_-)\right)^{\frac{1}{2}}}
   {\pi\delta^2\exp\{-\frac{2S}{h}\}}
   \right) 
   \notag\\
   &=0,
  \end{align}
where $F(z,h,\delta)$ is a function depending smoothly on $z$, satisfying the bound 
  \begin{equation*}
   F(z,h,\delta) \asymp -\frac{h^2}{\eta^{3/2}}.
  \end{equation*}
Here we used $\partial_{\Ima z } \Psi_1 \asymp - (\eta^{3/2}h)^{-1}$ which 
follows from Lemma \ref{lem:DerPsiThet} using the fact that $\Ima g$ has only 
two critical points: a minimum at $a$ and a maximum at $b$. 
\begin{rem}
 In the case $\Ima z > \langle\Ima g\rangle$ we find similarly that 
 $F(z,h,\delta) \asymp \frac{h^2}{\eta^{3/2}}$.
\end{rem}
Furthermore, the functions in (\ref{eqn:DerImazDensity2}) are smooth in 
$z$ and the $z$- and $\overline{z}$-derivate increase their order 
of growth at most by $\mO(\eta^{1/2}h^{-1})$. Recall $|E_-+(z)|$ as given 
in Proposition \ref{cor:E-+Repres} and define
  \begin{equation*}
   l(z):=|E_{-+}( z)| 
   = \frac{h\left(\frac{i}{2}\{p,\overline{p}\}(\rho_+)\frac{i}{2}\{\overline{p},p\}(\rho_-)\right)^{\frac{1}{2}}}{\pi}
   \frac{\e^{-\frac{2S}{h}}}{\delta^2}(1+ \mO(\eta^{-3/2}h)) 
  \end{equation*}
Thus, (\ref{eqn:DerImazDensity2}) is equal to zero if and only if 
 \begin{align}\label{eqn:DerImazDensity3}
    G(z,h,\delta) +  l\left(1 - l\right) =0,
  \end{align}
where $G(z,h,\delta)$ is a function depending smoothly on $z$, satisfying
  \begin{equation*}
   G(z,h,\delta)= \frac{F(z,h,\delta)}{2|\partial_{\Ima z}S(z)|^2
   (-\partial_{\Ima z}S(z))}\left(1 + \mO\left(\eta^{-3/4}h^{1/2}\right)\right) 
    \asymp \frac{h^2}{\eta^{3}}.
  \end{equation*}
The $z$- and $\overline{z}$-derivate increase the order 
of growth of $G$ at most by $\mO(\eta^{\frac{1}{2}}h^{-1})$. For $l\geq0$ to be a 
solution to (\ref{eqn:DerImazDensity3}), it is necessary that 
  \begin{equation*}
   l = 1 + \frac{h^2}{\mO(1)\eta^{3}}.
  \end{equation*}
Thus, $l\asymp 1$. Define the smooth function
  \begin{equation*}
    z \mapsto  t(z) : = \frac{\eta^{3}}{h^2}(l( z) -1),
  \end{equation*}
with $- c_0 \leq t \leq C_0$ and $c_0,C_0>0$ large enough. 
As in (\ref{eqn:EImaDEr}) on calculates
  \begin{align*}
   \frac{h^2}{\eta^{3}}\partial_{\Ima z} t 
   = -\frac{2\partial_{\Ima z} S}{h}\left(1 + \mO(\eta^{-3/2}h)\right)l(\Ima z)
   \asymp -\frac{\eta^{1/2}}{h} ,
  \end{align*}
where we used that $\partial_{\Ima z} S\asymp \sqrt{\eta}$ (cf. Proposition \ref{prop:SProp}) and 
that the $\partial_{\Ima z}$ derivative of 
$\left(\frac{i}{2}\{p,\overline{p}\}(\rho_+)\frac{i}{2}\{\overline{p},p\}(\rho_-)\right)^{\frac{1}{2}}$ 
is of order $\mO(\eta^{-1/2})$ due to the scaling $\tilde{z}=z\eta$ as in the proof of 
Proposition \ref{quasimodes}.
The implicit function theorem then implies that we may locally invert and that $t\mapsto (\Ima z)(t)$ is smooth. Since 
$-c_0 \leq t\leq C_0$ we may continue $(\Ima z)(t)$ smoothly to all open subsets of the domain of $t$. Furthermore, we 
conclude that 
  \begin{equation}\label{eqn:f_max_dydt} 
   \frac{d(\Ima z)}{dt}  \asymp -\eta^{-7/2}h^{3}
  \end{equation}
Substitute $\Ima z = \Ima z(t)$ in (\ref{eqn:DerImazDensity3}). To find the critical points, 
it is then enough to consider 
  \begin{equation*}
   t-\widetilde{G}(t,\Rea z,h,\delta) =0, \quad 
   \widetilde{G}(t,\Rea z,h,\delta):= 
   \frac{G(\Ima z(t),\Rea z,h,\delta))}{\eta^{-3}h^{2}(1+ \eta^{-3}h^{2}t)} 
  \end{equation*}
and one finds 
  \begin{equation*}
   \frac{d}{dt}\widetilde{G}(t,\Rea z,h,\delta))
   =\mO(h^{2}\eta^{-3}).
  \end{equation*}
Thus, using $t(\gamma_{-}^h)=0$ as starting point, which corresponds 
to $l(\gamma_{-}^h)=1$, the Banach fixed-point theorem implies that 
for each $\Rea z$ there exist a unique zero, $t_{-}^*(\Rea z)$, of
(\ref{eqn:DerImazDensity2}), it depends smoothly on $\Rea z$ and satisfies
  \begin{equation}\label{eqn:tstartbound}
   |t_{-}^*(\Rea z) - t(\gamma_{-}^h)| \leq \mO(h^{2}\eta^{-3}).
  \end{equation}
and
  \begin{align*}
    \frac{dt_{-}^*(\Rea z)}{d \Rea z} 
    &= \frac{1}{1 - \left(\frac{d}{dt}\widetilde{G}\right)(t_{-}^*,\Rea z,h,\delta)} 
    (\partial_{\Rea z} \widetilde{G})(t_{-}^*,\Rea z,h,\delta))
    \notag \\
    &= \frac{1}{1 + \mO(h^{2}\eta^{-3})} (\partial_{\Rea z} \widetilde{G})(t_{-}^*,\Rea z,h,\delta)).
  \end{align*}
Since the $z$- and $\overline{z}$-derivate applied to $G$ increase its order 
of growth at most by $\mO(\eta^{1/2}h^{-1})$, we conclude that 
  \begin{align*}
    \frac{dt_{-}^*(\Rea z)}{d \Rea z} 
    = \mO(\eta^{1/2}h^{-1}).
  \end{align*}
Taylor's formula applied to $(\Ima z)(t)$ yields that 
  \begin{equation*}
   (\Ima z)(t_{\pm}^*(\Rea z)) =  \Ima z(t(\Ima\gamma_{\pm}^h(\Rea z))) 
   + \int_{t(\Ima\gamma_{\pm}^h(\Rea z))}^{t_{\pm}^*(\Rea z)}\frac{d\Ima z}{dt}(\tau)d\tau.
   \end{equation*}
By (\ref{eqn:tstartbound}) and (\ref{eqn:f_max_dydt}) we conclude that 
  \begin{equation}\label{eqn:TwoWayEstimate}
   (\Ima z)(t_{\pm}^*(\Rea z)) =  \Ima\gamma_{\pm}^h(\Rea z) + \mathcal{O}(\eta^{-13/2}h^{5})
   \end{equation}
and using (\ref{eqn:estimDergammapm}) that 
  \begin{equation*}
   \frac{d}{d\Rea z}(\Ima z)(t_{\pm}^*(\Rea z)) =  \mO\!\left(\eta^{-6}h^{4}\right).
   \end{equation*}
It follows by Proposition \ref{prop:WedgePoisson} that the density has local 
maxima along the curves $\Gamma_{\pm}^h(\Rea z) := (\Rea z,\Ima z(t_{\pm}^*(\Rea z)))$. 
Applying this definition to (\ref{eqn:TwoWayEstimate}) yields that
	\begin{equation*}
	  |\Ima \Gamma_{\pm}^h(\Rea z) - \Ima \gamma_{\pm}^h(\Rea z)| \leq \mathcal{O}(\eta^{-13/2}h^{5})
	  \end{equation*}
for all $z\in\Sigma_{c,d}$. By Lemma \ref{prop:ExZone2} we have that 
$\Ima \gamma_{\pm}^h(\Rea z)\asymp \varepsilon_0(h)^{2/3}$. Thus, 
  \begin{equation*}
   \Ima \Gamma_{\pm}^h(\Rea z) = \Ima \gamma_{\pm}^h(\Rea z)\left(1 +\mathcal{O}(\varepsilon_0(h)^{-5}h^{5}) \right),
  \end{equation*}
which in particular implies that $ \Ima \Gamma_{\pm}^h(\Rea z) \asymp \varepsilon_0(h)^{2/3}$. This concludes 
the proof of the lemma.
\end{proof}
\begin{proof}[Proof of Proposition \ref{prop:WedgePoisson}]
Proposition \ref{prop:DensExplicit} implies that for $|\Ima z - \langle \Ima g \rangle |> 1/C$
  \begin{equation}\label{eqn:Psi2_constawayFromSpecLine}
    \Psi_2(z,h,\delta) = 
    \frac{\left(\frac{i}{2}\{p,\overline{p}\}(\rho_+)\frac{i}{2}\{\overline{p},p\}(\rho_-)\right)^{\frac{1}{2}}}
    {\pi h\delta^2\exp\{\frac{2S}{h}\}}
   |\partial_{\Ima z}S(z)|^2\left(1 + \mO\left(\eta^{-3/4}h^{1/2}\right)\right)
  \end{equation}
and 
  \begin{align*}
    \Theta(z;h,\delta) = \frac{h\left(\frac{i}{2}\{p,\overline{p}\}(\rho_+)
	  \frac{i}{2}\{\overline{p},p\}(\rho_-)\right)^{\frac{1}{2}}}{\pi}&\frac{\e^{-\frac{2S}{h}}}{\delta^2}
     \left(1 +\mathcal{O}\!\left(\eta^{-1/4}h^{\frac{3}{2}}\right)\right) \notag \\
     &+\mathcal{O}\!\left(\eta^{1/4}h^{-2}\delta  + \eta^{-1/2}\delta^2 h^{-5}\right).
  \end{align*}
Thus, one calculates 
  \begin{align*}
   \left|\Psi_2 - \frac{|\partial_{\Ima z}S|^2}{h^2}\Theta \right| \leq &\frac{\left(\frac{i}{2}\{p,\overline{p}\}(\rho_+)\frac{i}{2}\{\overline{p},p\}(\rho_-)\right)^{\frac{1}{2}}}{\pi h\delta^2}
   \e^{-\frac{2S}{h}}|\partial_{\Ima z}S(z)|^2\mO\left(\eta^{-\frac{3}{4}}h^{\frac{1}{2}}\right)\notag \\
     &+\mathcal{O}\!\left(\eta^{5/4}h^{-4}\delta  + \eta^{1/2}\delta^2 h^{-7}\right),
  \end{align*}
which implies the result given in Proposition \ref{prop:Resolvent}.
\end{proof}
\begin{proof}[Proof of Proposition \ref{prop:DensityProperties}]
We will only consider the case $z\in\Sigma_{c,d}$ with $\Ima z \leq \langle\Ima g\rangle$.\\
\par
\textbf{A priori restrictions on the domain of integration } Let $y_{-}(h)$ and $\gamma_{-}(\Rea z)$ 
be as in Lemma \ref{prop:ExZone2} and note that similarly to (\ref{eqn:Sepsilondependence}), we have 
  \begin{equation}\label{eqn:Sinanewway}
   S(\Ima z) - \varepsilon_0(h)
   = \int_{y_-(h)}^{\Ima \gamma_{-}^h}(\partial_{\Ima z}S)(t)dt 
   + \int_{\Ima \gamma_{-}^h}^{\Ima z}(\partial_{\Ima z}S)(t)dt.
  \end{equation}
Recall from Lemma \ref{prop:ExZone2} that 
$(\Ima \gamma_{-}^h - y_-(h)) \asymp h \varepsilon_0^{-1/3}$. Then, 
one calculates using the mean value theorem and Proposition 
\ref{prop:SProp}, similar as in the proof of Lemma \ref{prop:ExZone2} 
(cf. (\ref{eqn:YEps})) that 
  \begin{equation*}
   \int_{y_-(h)}^{\Ima \gamma_{-}^h}(\partial_{\Ima z}S)(t)dt =\mO( h)
  \end{equation*}
and that
  \begin{align*}
   \int_{\Ima \gamma_{-}^h}^{\Ima z}(\partial_{\Ima z}S)(t)dt 
   \asymp  (  \Ima z - \Ima \gamma_{-}^h )\eta^{1/2},
  \end{align*}
where $\eta$ should be set to $1$ in case of $\dist(z,\partial\Sigma_{c,d})> 1/C$. 
Next, (\ref{eqn:Sinanewway}) and Proposition \ref{prop:DensExplicit} implies that 
   \begin{align*}
   \Theta(z;h,\delta) = &\frac{\eta^{1/2}}{\mO(1)}
   \exp\left\{-\asymp \frac{ (\Ima z - \Ima \gamma_{-}^h)\eta^{1/2}}{h}\right\}
   \notag \\
   &+ \mathcal{O}\!\left(\eta^{1/4}h^{-2}\delta + \eta^{-1/2}\delta^2 h^{-5}\right).
  \end{align*}
Here, we used that $\delta = \sqrt{h}\exp\{-\frac{\varepsilon_0(h)}{h}\}$; see 
Definition \ref{defn:CoupConstDel}. Thus, for 
$\Ima \gamma_{-}^h <\Ima z < \langle \Ima g\rangle$
  \begin{align}\label{eqn:ThetaEstimates1}
   \exp\{-\Theta(z;h,\delta)\}= 
        \left(1 
        + \mathcal{O}\!\left(\eta^{1/2}\exp\left\{-\frac{(\Ima z - \Ima \gamma_{-}^h)\eta^{1/2}}{Ch}\right\}
        +\eta^{1/4}h^{2}
        \right)\right)                         
  \end{align}
and for $\Ima z \leq \Ima \gamma_{-}^h$
  \begin{multline}\label{eqn:ThetaEstimates2}
  \frac{1}{C}\exp\left\{-C\eta^{1/2}\exp\left[-\frac{(\Ima z - \Ima \gamma_{-}^h)\eta^{1/2}}{Ch}\right]\right\} 
  \leq 
      \exp\{-\Theta(z;h,\delta)\} \\
  \leq
  C\exp\left\{-\frac{\eta^{1/2}}{C}\exp\left[-\frac{C(\Ima z - \Ima \gamma_{-}^h)\eta^{1/2}}{h}\right]\right\}.
  \end{multline}
Similarly, by Proposition \ref{prop:DensExplicit} 
\begin{equation*}
  \Psi_2(z;h,\delta) 
  \leq 
  \frac{\eta^{3/2}}{\mO(1)h^{2}}\left(1 + \mO(\eta^{-1})\e^{\Phi(z,h)}\right)
  \exp\left\{- \frac{ (\Ima z - \Ima \gamma_{-}^h)\eta^{1/2}}{Ch}\right\}.
\end{equation*}
Thus, for 
$\Ima\gamma_{-}(\Rea z)+ \alpha h\eta^{-1/2}\ln\frac{\eta^{1/2}}{h} \leq\Ima z \leq \langle\Ima g \rangle $ 
with $\alpha>0$ large enough, we see that the average density of 
eigenvalues (cf. Theorem \ref{thm:ModelFirstIntensity}) 
  \begin{equation}\label{eqn:firstapprox}
   D(z,h,\delta)L(dz)   = \frac{1}{2h}p_*(d\xi\wedge dx) + \mO(\eta^{-2})L(dz).
  \end{equation}
We then conclude the first statement of the proposition. \\ \par 
Next, recall from Proposition \ref{lem:truncation} that restricting the probability 
space to the ball $B(0,R)$ of radius $R=Ch^{-1}$ implies that $\lVert Q_{\omega}\rVert \leq C/h$ 
with probability $\geq \left(1-\e^{-\frac{1}{Ch^2}}\right)$. It follows from 
  \begin{equation*}
   \lVert (P_h^{\delta} - z)^{-1}\rVert = \left\lVert(P_h -z)^{-1} \sum_{n\geq 1} 
   (-\delta)^n\left(Q_{\omega}(P_h -z)^{-1}\right)^n\right\rVert 
  \end{equation*}
that for $z\notin \sigma(P_h)$ such that $\delta \lVert Q_{\omega}\rVert \lVert (P_h -z)^{-1}\rVert <1$ 
$z\notin \sigma(P_h^{\delta})$ with probability $\geq \left(1-\e^{-\frac{1}{Ch^2}}\right)$. 
Proposition \ref{prop:Resolvent} implies that
with probability $\geq \left(1-\e^{-\frac{1}{Ch^2}}\right)$
  \begin{align*}
   \delta \lVert Q_{\omega}\rVert \lVert (P_h -z)^{-1}\rVert \leq 
     \frac{C\left|1 - \e^{\Phi(z,h)}\right|^{-1}}
     {h^{3/2}\left(\frac{i}{2}\{p,\overline{p}\}(\rho_+)\frac{i}{2}\{\overline{p},p\}(\rho_-)\right)^{\frac{1}{4}}}
     \exp\left\{\frac{S(z)-\varepsilon_0(h)}{h}\right\}.
  \end{align*}
Since $S(z)\asymp \eta^{3/2}$, it follows that $\eta \asymp \varepsilon_0(h)^{2/3}$. 
Using the mean value theorem together with 
Lemma \ref{prop:ExZone2} implies that with probability 
$\geq \left(1-\e^{-\frac{1}{Ch^2}}\right)$ there are no eigenvalues 
of $P_h^{\delta}$ with
  \begin{align*}
   \Ima z \leq \beta_1 := 
	\Ima \gamma_-^h 
	- C\frac{h}{\varepsilon_0(h)^{1/3}}\ln \left(\frac{\varepsilon_0(h)^{1/6}}{h}\right), \quad C\gg 1.            
  \end{align*} 
  Thus, to count 
eigenvalues, it is sufficient to integrate the density 
given in Theorem \ref{thm:ModelFirstIntensity} over subsets 
of	
  \begin{equation*}
   \Sigma_{c,d}'=\left\{z\in\Sigma_{c,d}|~\beta_1 \leq
	      \Ima z \leq \langle\Ima g\rangle, ~ c <\Rea z< d\right\}.
  \end{equation*}
Similarly, for an $\alpha$ large enough as above, define 
  \begin{equation*}
   \alpha_1:=\Ima\gamma_{-}(\Rea z) + \alpha \frac{h}{\varepsilon_0^{1/3}}\ln\frac{\varepsilon_0(h)^{1/3}}{h}
  \end{equation*}
and note that (\ref{eqn:firstapprox}) implies the second statement of 
the proposition for $\Ima z \geq \alpha_1$. 
\\
\\
\textbf{Approximate Primitive } Define $d(z):=\dist(z,\partial\Sigma)$ and recall from 
(\ref{def:etaOmega}) that $\eta \asymp d(z)$. Recall that 
the density of eigenvalues given in Theorem \ref{thm:ModelFirstIntensity} is given by $\Psi_1$, 
$\Psi_2$ and $\Theta$ which are expressed explicitly in Proposition \ref{prop:DensExplicit} and 
Theorem \ref{thm:ModelFirstIntensity}. 
Since $\Ima g(x_{\pm})  = \Ima z $ and $\xi_{\pm} = \Rea z - \Rea g(x_{\pm})$ (cf. Section \ref{sec:Intro}), 
we conclude together with Proposition \ref{cor:1MomSymplVol} that for $\beta_1 \leq \Ima z\leq \alpha_1$ 
  \begin{equation*}
   \Psi_1(z;h) = \frac{1}{2h} \partial_{\Ima z}( x_-(z) -  x_+(z)) + \mathcal{O}(d(z)^{-2})
   = \frac{1}{2h}\partial_{\Ima z}^2 S(z) + \mathcal{O}(d(z)^{-2}).                        
  \end{equation*}
Next, it follows by (\ref{eqn:Psi2_constawayFromSpecLine}) and Lemma \ref{lem:DerPsiThet} that 
  \begin{align*}
   \big|2h\Psi_2 - (\partial_{\Ima z} S)(-\partial_{\Ima z}\Theta) \big|  
      = \mO\left(d(z)^{3/4}h^{1/2}\frac{\e^{-\frac{2S}{h}}}{\delta^2}\right) + 
      \mO(d(z)^{3/4}h^{-3}\delta).
  \end{align*}   
Thus, 
  \begin{align}\label{eqn:derDenstiy}
   &\frac{1+\mathcal{O}\!\left(\delta d(z)^{-1/4}h^{-3/2}\right)}{\pi }
   \left\{\Psi_1(z;h) + \Psi_2(z;h,\delta)\right\}\e^{-\Theta(z;h,\delta)} \notag \\
   &= \frac{1}{2\pi h}\partial_{\Ima z}\left[(\partial_{\Ima z} S(z))\e^{-\Theta(z;h,\delta)}\right]
   +R(z;h,\delta)\e^{-\Theta(z;h,\delta)},           
  \end{align}
where
  \begin{align*}
   R(z;h,\delta) :=  \mathcal{O}\!\left(d(z)^{-2} + d(z)^{3/4}h^{-1/2}\frac{\e^{-\frac{2S}{h}}}{\delta^2}\right).
  \end{align*}
Let $\beta_1 \leq \beta_2 \leq \alpha_1$. Let us first treat the error term $R$. Similar 
as for (\ref{eqn:ThetaEstimates1}), it follows that 
  \begin{align*}
   R(z;h,\delta) 
   =  
     \mathcal{O}\!\left(d(z)^{-2}
   + d(z)^{-3/4}h^{-1/2}
     \exp\left\{- \frac{ (\Ima z - \Ima \gamma_{-}^h)d(z)^{1/2}}{Ch}\right\}\right).
  \end{align*}
Hence,
  \begin{align}\label{eqn:final2}
     \bigg|\int_{\beta_1}^{\alpha_1} R(z;h,\delta)&\e^{-\Theta(z;h,\delta)} d(\Ima z)\bigg| 
      \leq \frac{[d(z)^{-1}]_{\beta_1}^{\alpha_1}}{\mathcal{O}\!\left(1\right)}
	     \exp\{-\Theta(\Rea z,\alpha_1 ;h,\delta)\} \notag \\
	     &+\frac{d(z)^{1/4}h^{1/2}}{\mathcal{O}\!\left(1\right)}
	     \exp\left[-\exp\left\{- \frac{ (\Ima z - \Ima \gamma_{-}^h)d(z)^{1/2}}{Ch}\right\}\right]
	     \bigg|_{\beta_1}^{\alpha_1} \notag \\ 
      & = \frac{\beta_1^{-1}}{\mathcal{O}\!\left(1\right)}
          \exp\left[-\exp\left\{- \frac{ (\alpha_1 - \Ima \gamma_{-}^h)\alpha_1 ^{1/2}}{Ch}\right\}\right]
           \notag \\
      &= \frac{\varepsilon_0(h)^{-2/3}}{\mathcal{O}\!\left(1\right)}.
  \end{align}
Next, 
  \begin{align}\label{eqn:final1}
  \frac{1}{2\pi h}\int_{\beta_2}^{\alpha_1}
    \partial_{\Ima z}\big[(\partial_{\Ima z} S(z))&\e^{-\Theta(z;h,\delta)}\big]L(\Ima z) \notag \\
   & =  \frac{1}{2\pi h}(x_-(\Ima z)-x_+(\Ima z))\e^{-\Theta(z;h,\delta)}\Big|_{\beta_2}^{\alpha_1}.
  \end{align}
Since, 
  \begin{equation*}
   \int\limits_{\substack{\Sigma_{c,d}\\ 0 \leq \Ima z \leq \alpha_1}} 
   \frac{1}{2\pi h}p_*(d\xi\wedge dx)(dz)
   = \frac{1}{2\pi h}(x_-(\alpha_1)-x_+(\alpha_1))\int_c^d d\Rea z
  \end{equation*}
we conclude by (\ref{eqn:ThetaEstimates2}) the second statement of the proposition for
  \begin{equation*}
   \beta_2 = \Ima\gamma_{-}(\Rea z)- \frac{h}{\varepsilon_0(h)^{1/3}}
	      \ln  \left(\beta\ln\frac{\varepsilon_0(h)^{1/3}}{h}\right) 
  \end{equation*} 
with $\beta>0$ large enough. The last statement of the proposition can be 
deduced similarly from (\ref{eqn:ThetaEstimates2}), (\ref{eqn:final1}) and 
(\ref{eqn:final2}).
\end{proof}
\providecommand{\bysame}{\leavevmode\hbox to3em{\hrulefill}\thinspace}
\providecommand{\MR}{\relax\ifhmode\unskip\space\fi MR }
\providecommand{\MRhref}[2]{%
  \href{http://www.ams.org/mathscinet-getitem?mr=#1}{#2}
}
\providecommand{\href}[2]{#2}


\begin{thebibliography}{10}
%
\bibitem{BM}
W.~Bordeaux-Montrieux, \emph{{Loi de Weyl presque s\^{u}re et r\'{e}solvent
  pour des op\'{e}rateurs diff\'{e}rentiels non-autoadjoints, Th\'{e}se}},
  pastel.archives-ouvertes.fr/docs/00/50/12/81/PDF/manuscrit.pdf (2008).

\bibitem{BoSj09}
W.~Bordeaux-Montrieux and J.~Sj\"{o}strand, \emph{{Almost sure Weyl asymptotics
  for non-self-adjoint elliptic operators on compact manifolds}}, Ann. Fac.
  Sci. Toulouse \textbf{19} (2010), no.~3--4, 567--587.

\bibitem{ZwChrist10}
T.J. Christiansen and M.~Zworski, \emph{{Probabilistic Weyl Laws for Quantized
  Tori}}, Communications in Mathematical Physics \textbf{299} (2010).

\bibitem{DaJo}
D.J. Daley and D.~Vere-Jones, \emph{{An Introduction to the Theory of Point
  Processes}}, vol. I-II, Springer, 2008.

\bibitem{Da97}
E.~B. Davies, \emph{{Pseudospectra of Differential Operators}}, J. Oper. Th
  \textbf{43} (1997), 243--262.

\bibitem{Da07}
\bysame, \emph{{Non-Self-Adjoint Operators and Pseudospectra}}, {Proc. Symp.
  Pure Math.}, vol.~76, Amer. Math. Soc., 2007.

\bibitem{Da99}
E.B. Davies, \emph{{Pseudo{\textendash}spectra, the harmonic oscillator and
  complex resonances}}, Proc. of the Royal Soc.of London A \textbf{455} (1999),
  no.~1982, 585--599.

\bibitem{DaHa09}
E.B. Davies and M.~Hager, \emph{{Perturbations of Jordan matrices}}, J. Approx.
  Theory \textbf{156} (2009), no.~1, 82--94.

\bibitem{NSjZw}
N.~Dencker, J.~Sj\"{o}strand, and M.~Zworski, \emph{{Pseudospectra of
  semiclassical (pseudo-) differential operators}}, Communications on Pure and
  Applied Mathematics \textbf{57} (2004), no.~3, 384--415.

\bibitem{TrefEmbr}
M.~Embree and L.~N. Trefethen, \emph{{Spectra and Pseudospectra: The Behavior
  of Nonnormal Matrices and Operators}}, Princeton University Press, 2005.

\bibitem{Ha06b}
M.~Hager, \emph{{Instabilit\'{e} Spectrale Semiclassique d{\rq}Op\'{e}rateurs
  Non-Autoadjoints II}}, Annales Henri Poincare \textbf{7} (2006), 1035--1064.

\bibitem{Ha06}
\bysame, \emph{{Instabilit\'{e} spectrale semiclassique pour des op\'{e}rateurs
  non-autoadjoints I: un mod\`{e}le}}, Annales de la facult\'{e} des sciences
  de Toulouse S\'{e}. 6 \textbf{15} (2006), no.~2, 243--280.

\bibitem{HaSj08}
M.~Hager and J.~Sj\"{o}strand, \emph{{Eigenvalue asymptotics for randomly
  perturbed non-selfadjoint operators}}, Mathematische Annalen \textbf{342}
  (2008), 177--243.

\bibitem{FourPeop09}
J.B. Hough, M.~Krishnapur, Y.~Peres, and B.~Vir\'{a}g, \emph{{Zeros of Gaussian
  Analytic Functions and Determinantal Point Processes}}, American Mathematical
  Society, 2009.

\bibitem{Kato}
T.~Kato, \emph{{Perturbation Theory for Linear Operators}}, 2nd ed., {Classics
  in Mathematics}, Springer-Verlag Berlin Heidelberg, 1995.

\bibitem{Sh08}
B.~Shiffman, \emph{{Convergence of random zeros on complex manifolds}}, Science
  in China Series A: Mathematics \textbf{51} (2008), 707--720.

\bibitem{ShZe99}
B.~Shiffman and S.~Zelditch, \emph{{Distribution of Zeros of Random and Quantum
  Chaotic Sections of Positive Line Bundles}}, Communications in Mathematical
  Physics \textbf{200} (1999), 661--683.

\bibitem{SZ03}
\bysame, \emph{{Equilibrium distribution of zeros of random polynomials}}, Int.
  Math. Res. Not. (2003), 25--49.

\bibitem{ShZe08}
\bysame, \emph{{Number Variance of Random Zeros on Complex Manifolds}},
  Geometric And Functional Analysis \textbf{18} (2008), 1422--1475.

\bibitem{Sj09}
J.~Sj\"{o}strand, \emph{{Eigenvalue distribution for non-self-adjoint operators
  with small multiplicative random perturbations}}, Annales Fac.~Sci.~Toulouse
  \textbf{18} (2009), no.~4, 739--795.

\bibitem{SjAX1002}
\bysame, \emph{{Spectral properties of non-self-adjoint operators}}, Actes des
  Journ\'{e}es d'\'{e}.d.p. d'\'{E}vian (2009).

\bibitem{Sj08}
\bysame, \emph{{Eigenvalue distribution for non-self-adjoint operators on
  compact manifolds with small multiplicative random perturbations}}, Ann. Fac.
  Toulouse \textbf{19} (2010), no.~2, 277--301.

\bibitem{Sj_4VM10}
\bysame, \emph{{Resolvent Estimates for Non-Selfadjoint Operators via
  Semigroups}}, {Around the Research of Vladimir Maz'ya III}, {International
  Mathematical Series}, no.~13, Springer, 2010, pp.~359--384.

\bibitem{Sj13}
\bysame, \emph{{Weyl law for semi-classical resonances with randomly perturbed
  potentials}}, M\'{e}moires de la SMF \textbf{136} (2014).

\bibitem{SjZwor07}
J.~Sj\"{o}strand and M.~Zworski, \emph{{Elementary linear algebra for advanced
  spectral problems}}, Annales de l'Institute Fourier \textbf{57} (2007),
  2095--2141.

\bibitem{So00}
M.~Sodin, \emph{{Zeros of Gaussian Analytic Functions and Determinantal Point
  Processes}}, Mathematical Research Letters (2000), no.~7, 371--381.

\bibitem{Tr97}
L.N. Trefethen, \emph{{Pseudospectra of linear operators}}, SIAM Rev.
  \textbf{39} (1997), no.~3, 383--406.

\end{thebibliography}
\end{document}